\documentclass{amsart}
\usepackage{etex}
\usepackage{graphicx}
\usepackage{bm,caption,amsbsy,enumitem,amsmath,amsthm,
amssymb,mathtools,amsfonts,multirow,verbatim,tikz,tikz-cd,thmtools,thm-restate,xr,tensor,xcolor}
\usepackage[alphabetic]{amsrefs}
\usetikzlibrary{matrix,arrows,decorations.pathmorphing}
\usepackage[top=1.25in, bottom=1in, left=1.25in, right=1.25in,marginpar=1in]{geometry}
\usepackage{url}
\DefineSimpleKey{bib}{myurl}
\usepackage[hidelinks]{hyperref}
\usepackage{chngcntr}
\usepackage[all]{xy}
\counterwithin{figure}{section}
\numberwithin{equation}{section}

\newcommand\myurl[1]{\url{#1}}

\BibSpec{webpage}{
  +{}{\PrintAuthors} {author}
  +{,}{ \textit} {title}
  +{}{ \parenthesize} {date}
  +{,}{ \myurl} {myurl}
  +{,}{ } {note}
  +{.}{ } {transition}
}

\newtheorem{thm}{Theorem}[section]
\newtheorem{prop}[thm]{Proposition}

\newtheorem{lem}[thm]{Lemma}

\theoremstyle{definition}
\newtheorem{define}[thm]{Definition}

\theoremstyle{remark}
\newtheorem{rem}[thm]{Remark}
\newtheorem{example}[thm]{Example}

\newcommand{\ve}[1]{\boldsymbol{\mathbf{#1}}}
\newcommand{\R}{\mathbb{R}}

\newcommand{\C}{\mathbb{C}}

\newcommand{\T}{\mathbb{T}}

\renewcommand{\d}{\partial}
\renewcommand{\subset}{\subseteq}

\renewcommand{\tilde}{\widetilde}
\renewcommand{\hat}{\widehat}
\newcommand{\iso}{\cong}

\DeclareMathOperator{\cl}{{cl}}

\DeclareMathOperator{\cotr}{{cotr}}

\DeclareMathOperator{\Diff}{{Diff}}

\DeclareMathOperator{\ext}{{ext}}

\DeclareMathOperator{\gr}{{gr}}

\DeclareMathOperator{\Hom}{{Hom}}

\DeclareMathOperator{\id}{{id}}
\DeclareMathOperator{\ind}{{ind}}
\DeclareMathOperator{\Int}{{int}}

\DeclareMathOperator{\rank}{{rk}}

\DeclareMathOperator{\Spin}{{Spin}}

\DeclareMathOperator{\Sym}{{Sym}}

\DeclareMathOperator{\tr}{{tr}}

\newcommand{\bD}{\mathbb{D}}

\newcommand{\bF}{\mathbb{F}}

\newcommand{\bL}{\mathbb{L}}

\newcommand{\bT}{\mathbb{T}}
\newcommand{\bU}{\mathbb{U}}

\newcommand{\cA}{\mathcal{A}}

\newcommand{\cC}{\mathcal{C}}
\newcommand{\cD}{\mathcal{D}}

\newcommand{\cH}{\mathcal{H}}
\newcommand{\cI}{\mathcal{I}}

\newcommand{\cM}{\mathcal{M}}

\newcommand{\cS}{\mathcal{S}}
\newcommand{\cT}{\mathcal{T}}

\newcommand{\cW}{\mathcal{W}}
\newcommand{\cX}{\mathcal{X}}

\newcommand{\frs}{\mathfrak{s}}

\newcommand{\as}{\ve{\alpha}}
\newcommand{\bs}{\ve{\beta}}
\newcommand{\gs}{\ve{\gamma}}
\newcommand{\ds}{\ve{\delta}}
\newcommand{\es}{\ve{\epsilon}}
\newcommand{\Ds}{\ve{\Delta}}
\newcommand{\xs}{\ve{x}}
\newcommand{\ys}{\ve{y}}
\newcommand{\ws}{\ve{w}}
\newcommand{\zs}{\ve{z}}
\newcommand{\ts}{\boldsymbol{\tau}}
\newcommand{\sigmas}{\ve{\sigma}}
\newcommand{\taus}{\ve{\tau}}
\newcommand{\etas}{\ve{\eta}}
\newcommand{\xis}{\ve{\xi}}
\newcommand{\zetas}{\ve{\zeta}}

\newcommand{\SFH}{\mathit{SFH}}
\newcommand{\HKM}{\mathit{HKM}}
\newcommand{\CF}{\mathit{CF}}
\newcommand{\HF}{\mathit{HF}}

\newcommand{\CFL}{\mathit{CFL}}
\newcommand{\HFLh}{\widehat{\mathit{HFL}}}
\newcommand{\HFL}{\mathit{HFL}}

\newcommand{\EH}{\mathit{EH}}
\renewcommand{\a}{\alpha}
\renewcommand{\b}{\beta}
\newcommand{\g}{\gamma}
\renewcommand{\S}{\Sigma}
\renewcommand{\H}{\mathcal{H}}
\newcommand{\x}{\mathbf{x}}
\newcommand{\y}{\mathbf{y}}

\DeclareMathOperator{\tb}{{tb}}

\newcommand{\flip}[1]{\rotatebox{180}{\reflectbox{#1}}}

\title{Contact handles, duality, and sutured Floer homology}

\author{Andr\'as Juh\'asz}
\address{Mathematical Institute, University of Oxford, Andrew Wiles Building,
Radcliffe Observatory Quarter, Woodstock Road, Oxford, OX2 6GG, UK}
\email{juhasza@maths.ox.ac.uk}

\author{Ian Zemke}
\address{Department of Mathematics\\Princeton University\\  Princeton, NJ 08544, USA}
\email{izemke@math.princeton.edu}
\begin{document}


\begin{abstract}
  We give an explicit construction of the Honda--Kazez--Mati\'c gluing maps in terms of contact handles.
  We use this to prove a duality result for turning a sutured manifold cobordism around,
  and to compute the trace in the sutured Floer TQFT.
  We also show that the decorated link cobordism maps on the hat version of link Floer homology
  defined by the first author via sutured manifold cobordisms and by the second author via elementary cobordisms agree.
\end{abstract}

\maketitle

\tableofcontents

\section{Introduction}

The purpose of this paper is to provide an explicit construction of the Honda--Kazez--Mati\'c gluing map~\cite{HKMTQFT} in terms of contact handles, and use this to prove several results about the sutured Floer TQFT defined by the first author~\cite{JCob}.
Additionally, we show that the decorated link cobordism maps on the hat version of link Floer homology
defined via sutured manifold cobordisms by the first author~\cite{JCob}, and the
maps defined using elementary link cobordisms by the second author~\cite{ZemCFLTQFT} agree.

\subsection{The contact gluing map}
Sutured manifolds were introduced by Gabai~\cite{Gabai} to construct taut foliations on 3-manifolds,
and are also ubiquitous in contact topology.
In this paper, a sutured manifold is a pair $(M,\g)$, where $M$ is a compact
oriented 3-manifold with boundary, and the set of sutures $\g \subset \d M$ is an oriented 1-manifold
that divides $\d M$ into subsurfaces $R_+(\g)$ and $R_-(\g)$ that meet along $\g$.
For example, if $M$ carries a contact structure such that $\d M$ is convex with dividing set $\g$,
then $(M,\g)$ is a sutured manifold.

We say that $(M,\g)$ is balanced if $M$ has no closed components, each component of
$M$ contains a suture, and $\chi(R_+(\g)) = \chi(R_-(\g))$.
Sutured Floer homology, defined by the first author~\cite{JDisks}, assigns an $\bF_2$-vector space $\SFH(M,\g)$ to
a balanced sutured manifold $(M,\g)$. It is a common extension of the hat version of
Heegaard Floer homology of closed 3-manifolds and link Floer homology, both due to
Ozsv\'ath and Szab\'o~\cites{OSDisks, OSProperties, OSLinks}, to 3-manifolds with boundary.

Let $(M,\g)$ and $(M',\g')$ be sutured manifolds such that $M \subset \Int(M')$.
Given a contact structure $\xi$ on $M' \setminus \Int(M)$ such that
$\d M \cup \d M'$ is convex with dividing set $\g \cup \g'$,
Honda--Kazez--Mati\'c~\cite{HKMTQFT} define a gluing map
\[
\Phi_\xi \colon \SFH(-M,-\g) \to \SFH(-M',-\g')
\]
using partial open book decompositions that satisfy a ``contact compatibility'' condition
near the boundary. However, the contact compatibility condition makes working with and computing
the gluing map impractical. In the first part of this paper, we give a new definition of the
contact gluing map based on contact handle attachments, prove invariance via contact cell decompositions,
and show that our map agrees with the Honda--Kazez--Mati\'c gluing map.
In particular, this allows us to give a simple diagrammatic description of the
gluing map for a single contact handle attachment.
For a precise statement about the gluing map associated to a contact handle attachment, see Proposition~\ref{prop:computemorsetypehandlemap}.
Contact handles were introduced by Giroux~\cite{Giroux}; see Definition~\ref{def:Morsetypecontacthandle}.

\begin{figure}[ht!]
    \centering
    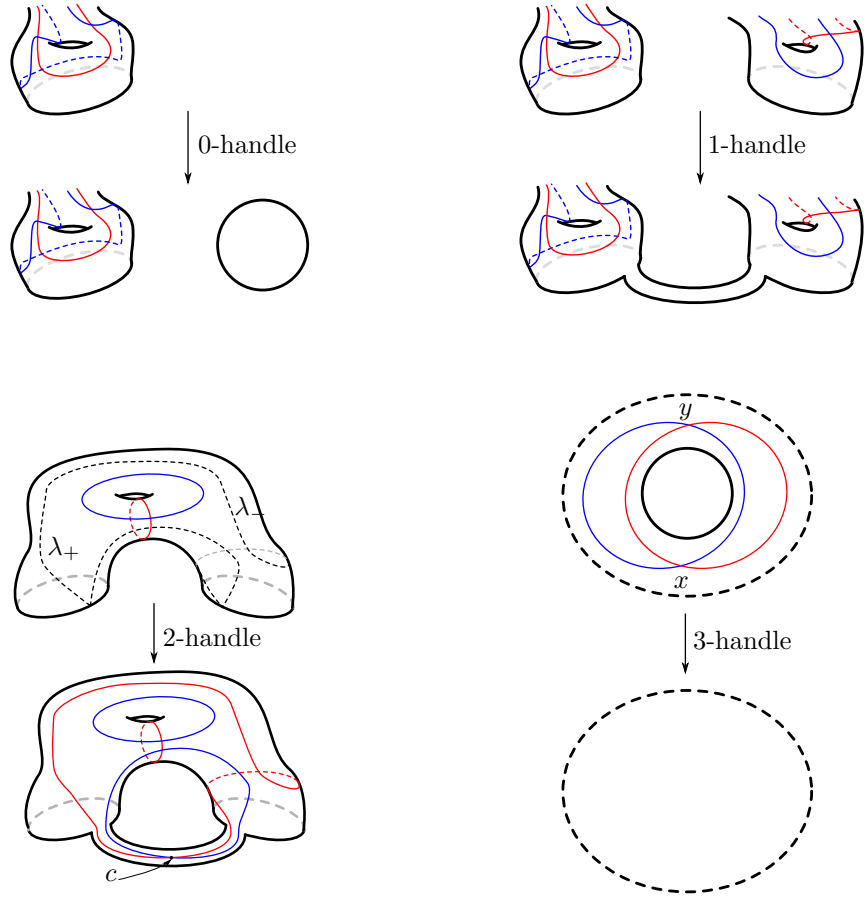
    \caption{The diagrams used in the definition of the contact handle maps. \label{fig::handles}}
\end{figure}

We now describe the map $C_{h^i}$ that we assign to attaching a contact $i$-handle~$h^i$
for $i \in \{0,1,2,3\}$. Let $(\S, \as, \bs)$ be a diagram of $(M,\g)$;
then $(\bar{\S},\as,\bs)$ is a diagram of $(-M,-\g)$.
Attaching a contact 0-handle corresponds to taking the disjoint union of
$\S$ with a disk. A contact 1-handle corresponds to attaching a 1-handle
to $\d \S$; see Figure~\ref{fig::handles}.
Adding a disk or a 1-handle to $\d \Sigma$ does not change the sutured Floer complex,
and we define $C_{h^i}$ to be the tautological map on intersection points.
A contact 2-handle is attached to $\d M$
along a curve $l$ that intersects $\g$ in two points.
Let $\lambda_\pm \subset \S$ be the properly embedded arc corresponding
to $l \cap R_\pm(\g)$.
As in Figure~\ref{fig::handles}, we glue a 1-handle $H$ to $\S$ along $\d \S$,
and add a curve $\a$ to $\as$ and a curve $\b$ to $\bs$ that intersect in $H$ in
a single point $c$, and such that $\a \cap \S = \lambda_-$
and $\b \cap \S = \lambda_+$. Then, given a generator $\x \in \T_\a \cap \T_\b$,
we let $C_{h^2}(\x) = \x \times \{c\}$.
Finally, suppose that we attach a contact 3-handle $h^3$ along an $S^2$ component $S$ of $\d M$
containing the suture $\g_S = \g \cap S$, giving rise to the sutured manifold $(M',\g')$.
We then choose a diagram where $\g_S$ is encircled
by a curve $\a \in \as$ and a curve $\b \in \bs$ such that $\a \cap \b = \{x, y\}$,
and such that there are no other $\as$ or $\bs$ curves between $\a$ and $\g_S$
or $\b$ and $\g_S$. Let $\S'$ be the result of gluing a disk
to $\S$ along $\g_S$. Then $(\S', \as \setminus \{\a\}, \bs \setminus \{\b\})$
is a diagram of $(M',\g')$; see Figure~\ref{fig::handles}.
We let $C_{h^3}(\x \times \{x\}) = 0$ and $C_{h^3}(\x \times \{y\}) = \x$,
where $\mu(\x \times \{x\}, \x \times \{y\}) = 1$ in $(\bar{\S}, \as, \bs)$.

Note that Zarev~\cite{ZarevGluing} has also defined a type of gluing map in sutured Floer homology,
corresponding to a convex decomposition. Combining this with the $\EH$ invariant~\cite{HKMSutured},
one can define a map for gluing a contact structure to a sutured manifold.
Zarev conjectured that this map agrees with the Honda--Kazez--Mati\'{c} gluing map, though we will
not address Zarev's construction in this paper.

\subsection{The sutured Floer TQFT}
The first author~\cite{JCob} defined the category of balanced sutured manifolds and sutured manifold
cobordisms, and extended $\SFH$ to a functor on this category.
A sutured manifold cobordism from $(M_0,\g_0)$ to $(M_1,\g_1)$ is a triple $\cW = (W,Z,[\xi])$,
where $W$ is a 4-manifold with boundary and corners, $Z$ is a codimension-0 compact submanifold of $\d W$
such that $\d W \setminus \Int(Z) = -M_0 \sqcup M_1$, and $[\xi]$ is a
certain equivalence class of a contact structure $\xi$ on $Z$ such that
$\d Z$ is convex with dividing set $\g_0 \sqcup \g_1$. The sutured cobordism map
\[
F_\cW \colon \SFH(M_0,\g_0) \to \SFH(M_1,\g_1)
\]
is the composition of the contact gluing map for $-\xi$ and 4-dimensional handle maps.

Let $\xi_{I \times \d M}$ denote the $I$-invariant contact structure
on $-I\times \d M $ that induces the dividing set $\g$ on~$\d M$.
Consider the trace cobordism
\[
\Lambda_{(M,\g)} = (I \times M, \xi_{I \times \d M})
\]
from $(M,\g) \sqcup( -M,\g)$ to $\emptyset$, and the cotrace cobordism
\[
V_{(M,\g)} = (I \times M, \xi_{I \times \d M})
\]
from $\emptyset \to (-M,\g) \sqcup(M,\g)$.
In Theorem~\ref{thm:traceandcotracecobordismmaps},
we positively answer \cite{JCob}*{Conjecture 11.13}:

\begin{thm}\label{thm:traceformula}
The trace cobordism $\Lambda_{(M,\g)}$ induces the canonical trace map
\[
\tr \colon \SFH(M,\g) \otimes \SFH(-M,\g)\to \bF_2,
\]
obtained by evaluating cohomology on homology.
The cotrace cobordism $V_{(M,\g)}$ induces the canonical cotrace map
\[
\cotr \colon \bF_2 \to \SFH(-M,\g) \otimes \SFH(M,\g).
\]
\end{thm}

The proof relies on the following deep technical result, which is Theorem~\ref{thm:trianglemapiscobmap}.
Before stating it, recall that, given a sutured triple diagram $(\S,\as,\bs,\gs)$,
we can associate to it a sutured manifold cobordism $\cW_{\as,\bs,\gs}$
from $(M_{\as,\bs}, \g_{\as,\bs}) \sqcup (M_{\bs,\gs}, \g_{\bs,\gs})$ to $(M_{\as,\gs}, \g_{\as,\gs})$.

\begin{thm}
Let $\cT = (\S,\as,\bs,\gs)$ be an admissible balanced sutured triple diagram.
Then the cobordism map
\[
F_{\cW_{\as,\bs,\gs}} \colon \CF(\S,\as,\bs) \otimes \CF(\S,\bs,\gs) \to \CF(\S,\as,\gs)
\]
is chain homotopic to the map $F_{\as,\bs,\gs}$ defined in \cite{JCob}*{Definition~5.13}
that counts holomorphic triangles on the triple diagram~$\cT$.
\end{thm}

One can obtain from Theorem~\ref{thm:traceformula} a positive answer to \cite{JCob}*{Question~11.9}:

\begin{thm}\label{thm:turningaroundcobordism}
If $\cW \colon (M,\g) \to (M',\g')$ is a balanced cobordism of sutured manifolds, and
$\cW^\vee$ is the cobordism obtained by turning around $\cW$, then
\[
F_{\cW^\vee} = (F_{\cW})^\vee,
\]
with respect to the trace pairing.
\end{thm}

We also give a self-contained proof of this result, without invoking Theorem~\ref{thm:traceformula}.
As a special case, we obtain that the decorated link cobordism maps $F_{\cX}^J$ of the first author satisfy
an analogous duality property when we turn a decorated link cobordism $\cX$ around;
see \cite{JMComputeCobordismMaps}*{Section~5.7}.
Indeed, these maps are defined by assigning a sutured manifold cobordism to a decorated link cobordism,
and applying the $\SFH$ functor.

The second author~\cite{ZemCFLTQFT} later gave a different construction of link Floer cobordism maps~$F_{\cX}^Z$
by composing maps defined for elementary link cobordisms, and showing independence of the decomposition.
Note that this construction makes sense for all versions of link Floer homology, not just the hat version.
In the last section, we prove that the two maps agree:

\begin{thm}
Given a decorated link cobordism $\cX$, we have $F_{\cX}^J = \hat{F}_{\cX}^Z$.
\end{thm}

A key technical lemma that we use throughout the paper gives a simple formula
for the naturality map for a compound stabilization operation on a sutured diagram
(called a $(k,0)$- or $(0,l)$-stabilization in~\cite{JTNaturality}), which consists of
a simple stabilization, followed by handle sliding some $\a$-curves over the new $\a$-curve, or some
$\b$-curves over the new $\b$-curve; see Proposition~\ref{prop:compoundstabilizationmap}.

In an upcoming paper, we will use Theorem~\ref{thm:traceformula} to
compute the effect of a generalization of the Fintushel--Stern knot surgery
operation using a self-concordance of a knot, called concordance surgery~\cite{Akbulut}*{Section~2},
on the Ozsv\'ath--Szab\'o 4-manifold invariant.
The formula involves the graded Lefschetz number of the concordance map on knot Floer homology.
In another work, we will apply Theorem~\ref{thm:traceformula} to compute the invariant
due to Marengon and the first author~\cite{JMConcordance}
of a slice disk obtained by the deform-spinning construction of Litherland~\cite{Litherland}.
Hence, we will show that this invariant can effectively distinguish different slice disks of a knot,
answering \cite{JMConcordance}*{Question~1.4}.

\subsection{Notation and conventions}
Throughout this paper, if $A$ and $B$ are smooth manifolds,
then we write $A \cong B$ if $A$ and $B$ are diffeomorphic.
Given a submanifold $A$ of $C$, we write $N(A)$ for a regular neighborhood of~$A$ in~$C$. We denote Heegaard diagrams by $\cH$,
and handle decompositions by $\ve{H}$. If $M$ is an oriented $n$-manifold, then we will
denote the same manifold with its orientation reversed by $\bar{M}$ when $n$ is even,
and by $-M$ when $n$ is odd. The closure of a set $X$ is $\cl(X)$. If $\xi$ is a co-oriented 2-plane field on $M$, we will write $-\xi$ for the co-orientated 2-plane field obtained by reversing the co-orientation of $\xi$.

We orient the boundary of a manifold using the ``outward normal first'' convention.
To be consistent with this convention, all our cobordisms go from left to right.

\subsection{Acknowledgements}
We would like to thank Ko Honda and Jacob Rasmussen for helpful discussions on the contact gluing maps.

The first author was supported by a Royal Society Research Fellowship,
and the second author by an NSF Postdoctoral Research Fellowship (DMS-1703685).
This project has received funding from the European Research Council (ERC)
under the European Union's Horizon 2020 research and innovation programme
(grant agreement No 674978).
\maketitle

\tableofcontents

\section{1-handle and 3-handle maps and triangle maps; compound stabilizations}
\label{sec:compoundstab}

In this section, we describe several results about the interactions between holomorphic triangles and 1-handle, 3-handle, and stabilization maps. These will be used in later sections.

\subsection{1-handle and 3-handle maps}

Let $(\S,\as,\bs)$ be an admissible sutured diagram,
and let $p_1$, $p_2 \in \S$ be a pair of points that are both in components of $\S\setminus (\as\cup \bs)$ that intersect $\d \S$.
We construct the admissible sutured diagram $(\S',\as\cup \{\a_0\},\bs\cup \{\b_0\})$
by removing disks centered at $p_1$ and $p_2$, and adding an annulus~$A$ connecting the
boundaries of the disks.
Furthermore, $\a_0$ and $\b_0$ are homologically nontrivial curves in $A$
that intersect transversely at two points $\theta_{\a_0,\b_0}^+$ and $\theta_{\a_0,\b_0}^-$, such that $\theta_{\a_0,\b_0}^+$
has the larger relative Maslov grading.
The first author~\cite{JCob}*{Section~7} defined the 1-handle map
\[
\begin{split}
F_1^{\a_0,\b_0} \colon \CF(\S,\as,\bs) &\to \CF(\S',\as\cup \{\a_0\},\bs\cup \{\b_0\}) \\
\ve{x} &\mapsto \ve{x} \times \theta^+_{\a_0,\b_0}
\end{split}
\]
and the 3-handle map
\[
\begin{split}
F_3^{\a_0,\b_0} \colon  \CF(\S',\as \cup \{\a_0\}, \bs \cup \{\b_0\}) &\to  \CF(\S,\as,\bs) \\
\x \times \theta^+_{\a_0,\b_0} &\mapsto 0 \\
\x \times \theta^-_{\a_0,\b_0} &\mapsto \x.
\end{split}
\]

If $\cT = (\S,\as,\bs,\gs)$ is an admissible sutured triple,
then it induces a holomorphic triangle map
\[
F_\cT \colon \CF(\S,\as,\bs) \otimes \CF(\S,\bs,\gs) \to \CF(\S,\as,\gs).
\]
If $p_1$, $p_2 \in \S \setminus (\as\cup \bs\cup\gs)$
are distinct points such that they are in components of
$\S\setminus (\as\cup \bs\cup \gs)$ that intersect $\d \S$,
then we can similarly form the admissible Heegaard triple
\[
\cT' := (\S', \as' = \as \cup \{\a_0\}, \bs' = \bs \cup \{\b_0\}, \gs' = \gs \cup \{\g_0\}),
\]
where $\S'$ is obtained by adding a 1-handle $A$ with feet at $p_1$ and $p_2$,
and three new curves, $\a_0$, $\b_0$, and $\g_0$,
that are homologically nontrivial in $A$ and pairwise intersect in two points; see Figure~\ref{fig::one-handle}.
If $x$, $y$ are sets of attaching curves on some Heegaard surface $S$, then
we denote the diagram $(S, x, y)$ by $\H_{x, y}$.

\begin{figure}
    \centering
    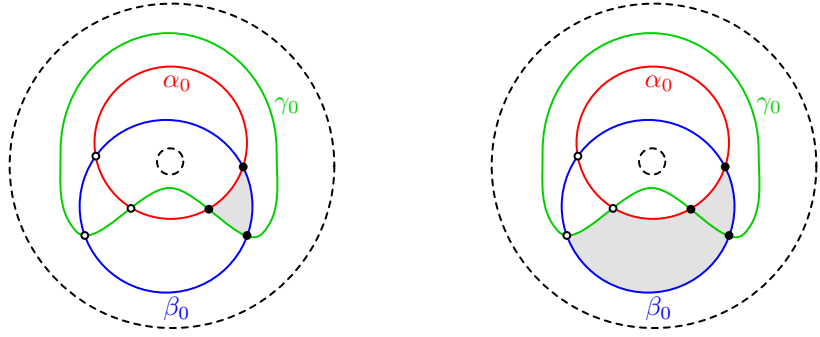
    \caption{The annulus $A$ is bounded by the two dashed circles.
    The intersection points $\theta^+_{\a_0,\b_0}$, $\theta^+_{\b_0,\g_0}$, and $\theta^+_{\a_0,\g_0}$
    are marked by solid circles, and $\theta^-_{\a_0,\b_0}$, $\theta^-_{\b_0,\g_0}$, and $\theta^-_{\a_0,\g_0}$
    by empty circles. On the left, the only index 0 triangle connecting $\theta^+_{\a_0,\b_0}$, $\theta^+_{\b_0,\g_0}$, and $\theta^+_{\a_0,\g_0}$ is shaded. On the right, the only index 0 triangle connecting
    $\theta^+_{\a_0,\b_0}$, $\theta^-_{\b_0,\g_0}$, and $\theta^-_{\a_0,\g_0}$ is shaded.
    \label{fig::one-handle}}
\end{figure}

\begin{prop}\label{prop:1handletrianglecount}
With the above notation, the following diagrams are commutative:
\[
  \xymatrix{
  \CF(\H_{\as,\bs}) \otimes \CF(\H_{\bs,\gs}) \ar[r]^-{F_{\cT}}
  \ar[d]|-{F_1^{\a_0,\b_0} \otimes  F_1^{\b_0,\g_0}} &
  \CF(\H_{\as,\gs}) \ar[d]^{F_1^{\a_0,\g_0}} \\
  \CF(\H_{\as',\bs'}) \otimes \CF(\H_{\bs',\gs'})  \ar[r]^-{F_{\cT'}} &
  \CF(\H_{\as',\gs'}),}
\]

\[\xymatrixcolsep{6pc}
  \xymatrix{
  \CF(\H_{\as,\bs}) \otimes \CF(\H_{\bs',\gs'}) \ar[r]^-{\id_{\CF(\H_{\as,\bs})} \otimes F_3^{\b_0,\g_0}}
  \ar[d]|-{F_1^{\a_0,\b_0} \otimes  \id_{\CF(\H_{\bs',\gs'})}} &
  \CF(\H_{\as,\bs}) \otimes \CF(\H_{\bs,\gs}) \ar[r]^-{F_{\cT}} &
  \CF(\H_{\as,\gs}), \\
  \CF(\H_{\as',\bs'}) \otimes \CF(\H_{\bs',\gs'})  \ar[r]^-{F_{\cT'}} &
  \CF(\H_{\as',\gs'}) \ar[ur]_-{F_3^{\a_0,\g_0}}}
\]

\[\xymatrixcolsep{6pc}
  \xymatrix{
  \CF(\H_{\as',\bs'}) \otimes \CF(\H_{\bs,\gs}) \ar[r]^-{F_3^{\a_0,\b_0} \otimes \id_{\CF(\H_{\bs,\gs})}}
  \ar[d]|-{\id_{\CF(\H_{\as',\bs'})} \otimes F_1^{\b_0,\g_0}} &
  \CF(\H_{\as,\bs}) \otimes \CF(\H_{\bs,\gs}) \ar[r]^-{F_{\cT}} &
  \CF(\H_{\as,\gs}). \\
  \CF(\H_{\as',\bs'}) \otimes \CF(\H_{\bs',\gs'})  \ar[r]^-{F_{\cT'}} &
  \CF(\H_{\as',\gs'}) \ar[ur]_-{F_3^{\a_0,\g_0}}}
\]
\end{prop}

\begin{proof}
Consider the first diagram.
The assumption that the points $p_1$ and $p_2$ are in components of
$\S\setminus (\as\cup \bs\cup \gs)$
that intersect $\d \S$ allows one to reduce the claim to the model computation
\begin{equation}\label{eqn:triangle}
  F_{\a_0,\b_0,\g_0}(\theta^+_{\a_0,\b_0} \otimes \theta^+_{\b_0,\g_0}) = \theta^+_{\a_0,\g_0}
\end{equation}
in the annulus $A$, which was established in the proof of \cite{JCob}*{Theorem~7.6};
see the left-hand side of Figure~\ref{fig::one-handle}.

We now show commutativity of the second diagram. Let $\x \in \T_\a \cap \T_\b$ and $\y \in \T_\b \cap \T_\g$.
Then
\[
F_{\cT} \circ (\id_{\CF(\H_{\as,\bs})} \otimes F_3^{\b_0,\g_0})
(\x \otimes (\y \times \theta^-_{\b_0,\g_0})) = F_\cT(\x \otimes \y).
\]
On the other hand,
\[
\begin{split}
&F_3^{\a_0,\g_0} \circ F_{\cT'} \circ \left(F_1^{\a_0,\b_0} \otimes \id_{\CF(\H_{\bs',\gs'})}\right)
\left(\x \otimes (\y \times \theta^-_{\b_0,\g_0})\right) = \\
&F_3^{\a_0,\g_0}\left( F_\cT(\x,\y) \times F_{\a_0,\b_0,\g_0}(\theta^+_{\a_0,\b_0} \otimes \theta^-_{\b_0,\g_0})\right).
\end{split}
\]
Hence, commutativity for the generator $\x \otimes (\y \times \theta^-_{\b_0,\g_0})$ follows from
\[
F_{\a_0,\b_0,\g_0}(\theta^+_{\a_0,\b_0} \otimes \theta^-_{\b_0,\g_0}) = \theta^-_{\a_0,\g_0},
\]
which can be shown similarly to equation~\eqref{eqn:triangle}; see the right-hand side of
Figure~\ref{fig::one-handle}. Note that there is a unique index 0 pseudo-holomorphic
triangle in $A$ connecting $\theta^+_{\a_0,\b_0}$, $\theta^-_{\b_0,\g_0}$, and $\theta^-_{\a_0,\g_0}$,
and there is none connecting $\theta^+_{\a_0,\b_0}$, $\theta^-_{\b_0,\g_0}$, and $\theta^+_{\a_0,\g_0}$.

On a generator of the form $\x \otimes (\y \times \theta^+_{\b_0,\g_0})$, we have
\[
F_{\cT} \circ (\id_{\CF(\H_{\as,\bs})} \otimes F_3^{\b_0,\g_0})
(\x \otimes (\y \times \theta^+_{\b_0,\g_0})) = 0,
\]
and, using equation~\eqref{eqn:triangle},
\[
\begin{split}
&F_3^{\a_0,\g_0} \circ F_{\cT'} \circ \left(F_1^{\a_0,\b_0} \otimes \id_{\CF(\H_{\bs',\gs'})}\right)
\left(\x \otimes (\y \times \theta^+_{\b_0,\g_0})\right) = \\
&F_3^{\a_0,\g_0}\left( F_\cT(\x,\y) \times F_{\a_0,\b_0,\g_0}(\theta^+_{\a_0,\b_0} \otimes \theta^+_{\b_0,\g_0})\right) =
F_3^{\a_0,\g_0}\left( F_\cT(\x,\y) \times \theta^+_{\a_0,\g_0}\right) = 0.
\end{split}
\]
This establishes commutativity of the second diagram. Commutativity of the third diagram is analogous.
\end{proof}

\subsection{Compound stabilization}
\label{sec:compoundstabilization}
In this section, we describe an elaboration of the usual stabilization operation on Heegaard diagrams.
Suppose that $\cH=(\S,\as,\bs)$ is an admissible sutured diagram,
and that $\lambda$ is an embedded path on $\S$ between two distinct points on $\d \S$
that avoids the $\as$ curves.
We define the \emph{compound stabilization} of $\cH$ along $\lambda$, as follows.

First, construct a surface $\S'$ by pushing $\lambda$ into the sutured compression body $U_{\as}$,
and add a tube that is the boundary of a regular neighborhood of $\lambda$.
Let $\a_0$ be a longitude of the tube,
concatenated with a portion of the curve $\lambda$ on $\S$.
Furthermore, let $\b_0$ be a meridian of the tube.
The curve $\a_0$ may intersect other $\bs$ curves;
however, $\b_0$ intersects only $\a_0$.
The construction is shown in Figure~\ref{fig::25}.
Let us denote by $\cH'$ the Heegaard diagram $(\S',\as \cup \{\a_0\},\bs \cup \{\b_0\})$.
This is an instance of a $(k,0)$-stabilization using the terminology of \cite{JTNaturality}*{Definition~6.26},
where $k = |\a_0 \cap \bs|$.
If $\lambda$ avoids $\bs$, then we can perform an analogous operation, with
the roles of $\a$ and $\b$ swapped, which is an instance of a $(0,l)$-stabilization.
We also call this a compound stabilization.
In the opposite direction, we say that $\H$ is obtained from $\H'$ by a \emph{compound
destabilization}.

We denote the unique intersection point of $\a_0$ and $\b_0$ by $c_{\a_0, \b_0}$.
There is a map
\[
\sigma^{\a_0,\b_0} \colon \SFH(\cH)\to \SFH(\cH'),
\]
defined by
\[
\sigma^{\a_0,\b_0}(\ve{x})=\ve{x}\times c_{\a_0, \b_0},
\]
which is a chain isomorphism since the tube is added near $\d\S$.

   \begin{figure}[ht!]
    \centering
    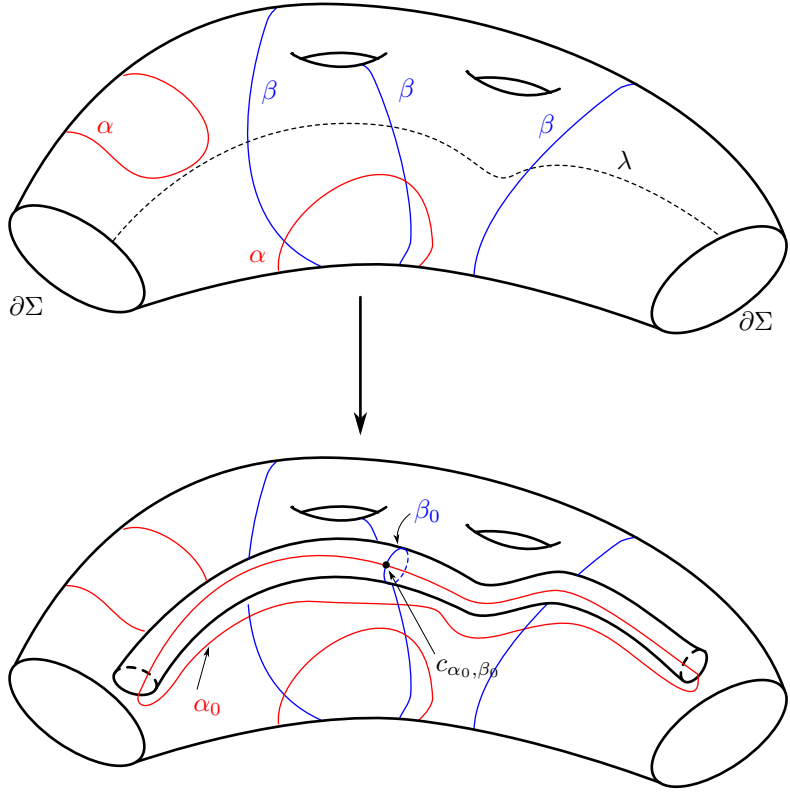
    \caption{An example of the compound stabilization operation along a path $\lambda$. \label{fig::25}}
    \end{figure}

On the other hand, there is also a naturality map $\Psi_{\cH \to \cH'}$.
One would expect these to be equal. Indeed, we prove the following (compare~\cite{HKMTQFT}*{Proposition~3.7}):

\begin{prop}\label{prop:compoundstabilizationmap}
The compound stabilization map $\sigma^{\a_0,\b_0}$ is chain homotopic
to the naturality map $\Psi_{\cH\to \cH'}$.
\end{prop}

One strategy to prove the above theorem would be to handle slide
all the $\bs$ curves that intersect $\a_0$ across $\b_0$.
The map from naturality induced by these handle slides can be computed by counting holomorphic triangles.
To prove Proposition~\ref{prop:compoundstabilizationmap}, one could analyze how holomorphic
triangles degenerate as one stretches two necks
(one on each end of the tube we are adding). While this can be done,
we will give a somewhat indirect argument that avoids performing a neck-stretching argument.

Suppose that $\cT = (\S,\as,\bs,\bs')$ is an admissible sutured triple
with a path $\lambda$ from $\d \S$ to itself that does not intersect any $\as$ curves.
Then we can perform the compound stabilization procedure on $(\S,\as,\bs,\bs')$
to obtain a Heegaard triple
\[
\cT' = (\S', \as \cup \{\a_0\}, \bs \cup \{\b_0\}, \bs' \cup \{\b_0'\}),
\]
where $(\S', \as \cup \{\a_0\}, \bs \cup \{\b_0\})$ is the compound stabilization
of $(\S, \as, \bs)$ along $\lambda$. Furthermore, the curve
$\b_0'$ is isotopic to $\b_0$ and $|\b_0 \cap \b_0'|=2$, while
$|\a_0\cap \b_0| = |\a_0\cap \b_0'| = 1$. An example is shown in Figure~\ref{fig::29}.
Let $\theta^+_{\b_0, \b_0'}$ be the point of $\b_0 \cap \b_0'$ with the higher relative
Maslov grading, and write $\a_0 \cap \b_0 = \{c_{\a_0,\b_0}\}$ and
$\a_0 \cap \b_0' = \{c_{\a_0,\b_0'}\}$.

\begin{lem}\label{lem:compoundtrianglemap1}
If $\cT=(\S,\as,\bs,\bs')$ is an admissible sutured triple and
\[
\cT'=(\S',\as\cup \{\a_0\}, \bs\cup \{\b_0\}, \bs'\cup \{\b_0'\})
\]
is a compound stabilization of $\cT$, as described in the previous paragraph, then
\[
F_{\cT'}(\ve{x} \times c_{\a_0,\b_0}, \ve{y}\times \theta^+_{\b_0,\b_0'}) =
F_{\cT}(\ve{x},\ve{y})\times c_{\a_0,\b_0'}.
\]
\end{lem}

\begin{proof}
Since the tube is added near $\d \S$, the result is obtained by a model computation inside
the tube. This is shown in Figure~\ref{fig::29}.
\end{proof}

\begin{figure}[ht!]
 \centering
 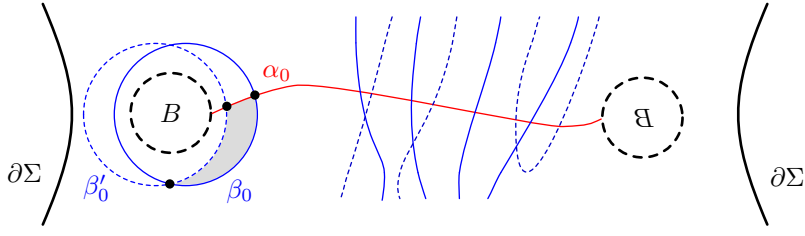
 \caption{A compound stabilization of a sutured triple, and the model
 computation of Lemma~\ref{lem:compoundtrianglemap1}. The $\as$ curves are
 shown as red solid lines, the $\bs$ curves are shown as blue solid lines, and
 the $\bs'$ curves are shown as blue dashed lines. The circles marked $B$ are
 identified. \label{fig::29}}
\end{figure}

\begin{rem}
Despite the notation, the triangle map computation of Lemma~\ref{lem:compoundtrianglemap1}
does not assume that the curves $\bs$ and $\bs'$ appearing in the triple $\cT$
are related by a sequence of handle slides or isotopies.
However, we will only need the result for examples where that is the case.
\end{rem}

Analogously, we need to consider moves of the $\a_0$ curve appearing in a compound stabilization.
To this end, suppose that $\cT = (\S,\as',\as,\bs)$ is a sutured triple with two paths,
$\lambda$ and $\lambda'$, from $\d \S$ to itself, such that $\lambda$ and $\lambda'$ have the same endpoints
and disjoint interiors.
Furthermore, suppose that $\lambda$ avoids $\as$ and $\lambda'$ avoids $\as'$.
We can construct a compound stabilization of the triple $(\S,\as',\as,\bs)$ to obtain
\[
\cT' = (\S',\as'\cup \{\a_0'\},\as\cup \{\a_0\},\bs\cup \{\b_0\}),
\]
where $(\S', \as \cup \{\a_0\}, \bs \cup \{\b_0\})$ is the compound stabilization
of $(\S, \as, \bs)$ along $\lambda$.
Furthermore, the curve $\a_0'$ is a concatenation of a portion of the curve $\lambda'$ on $\S$
with a longitude of the tube $\S' \setminus \S$, such that
$|\a_0 \cap \a_0'| = 2$ and $|\a_0 \cap \b_0| = |\a_0' \cap \b_0| = 1$.
In the tube, $\a_0$, $\a_0'$, and $\b_0$ are configured as in Figure~\ref{fig::30}.
Let $\theta^+_{\a_0',\a_0}$ be the point of $\a_0' \cap \a_0$ with the larger
relative Maslov grading, and write $\a_0 \cap \b_0 = \{c_{\a_0,\b_0}\}$
and $\a_0' \cap \b_0 = \{c_{\a_0',\b_0}\}$.

\begin{lem}\label{lem:compoundtrianglemap2}
If $\cT = (\S, \as', \as, \bs)$ is an admissible sutured triple and
\[
\cT' = (\S',\as'\cup \{\a_0'\},\as\cup \{\a_0\},\bs\cup \{\b_0\})
\]
is a compound stabilization, as described in the previous paragraph, then
\[
F_{\cT'}(\ve{x} \times \theta^+_{\a_0',\a_0},\ve{y} \times c_{\a_0,\b_0}) =
F_{\cT}(\ve{x},\ve{y})\times c_{\a'_0,\b_0}.
\]
\end{lem}

\begin{proof}
As before, since the ends of the tube $\S' \setminus \S$ are near $\d \S$, we obtain constraints on the multiplicities of any homology class of triangles which has holomorphic representatives.
An easy model computation shows that triangles with representatives have homology class
$\psi \sqcup \psi_0$, where $\psi$ is a homology class on $(\S,\as',\as,\bs)$
and $\psi_0$ is a homology class supported entirely on the tube.
The appropriate model computation is shown in Figure~\ref{fig::30}.
\end{proof}

\begin{figure}[ht!]
 \centering
 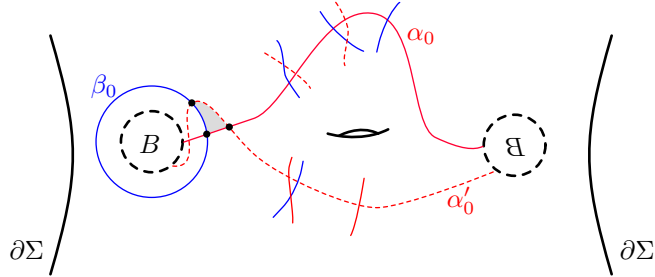
 \caption{The model computation of Lemma \ref{lem:compoundtrianglemap2}.
 The $\as'$ curves are shown as dashed red, the $\as$ curves are shown as solid red,
 and the $\bs$ curves are shown as solid blue. \label{fig::30}}
\end{figure}

Using the above two lemmas, we now prove Proposition~\ref{prop:compoundstabilizationmap}.

\begin{proof}[Proof of Proposition~\ref{prop:compoundstabilizationmap}]
Let $\cH'=(\S', \as\cup \{\a_0\},\bs\cup \{\b_0\})$
denote a compound stabilization of $(\S, \as, \bs)$ using a path $\lambda$ with ends on $\d \S$,
and let $D$ be the tube attached. Let $\hat{\cH}=(\hat{\S}, \as\cup \{\a_1\},\bs\cup \{\b_1\})$
be another compound stabilization of $\cH$, along a path that is parallel to $\lambda$.
Write $B$ for the attached tube. Let
$\hat{\cH}' = (\hat{\S}',\as,\cup \{\a_0,\a_1\}, \bs \cup \{\b_0,\b_1\})$
denote the two-fold compound stabilization of $\S$ along both paths.
Write $\{c_0\} = \a_0 \cap \b_0$ and $\{c_1\}=\a_1 \cap \b_1$.

We claim that
\begin{equation}\label{eq:transitionmaps1}
\Psi_{\cH' \to \hat{\cH}'}(\ve{x} \times c_0) = \Psi_{\cH \to \hat{\cH}}(\ve{x}) \times c_0.
\end{equation}
To see this, we note that a sequence of diagrams from~$\cH'$ to~$\hat{\cH}'$ can be constructed
by starting with~$\cH'$, performing a simple stabilization near the boundary,
and then moving one foot of the new tube along $\S$, parallel to~$\a_0$.
At various points, we will have to handle slide a $\bs$ curve across~$\b_1$.
It is not obvious what the holomorphic triangle count will be for each handle slide.
However, by the holomorphic triangle count from Lemma~\ref{lem:compoundtrianglemap1},
it is unchanged by the presence of the compound stabilization along $\a_0$.
In particular, the triangles counted by going from $\cH'$ to $\hat{\cH}'$
are the same as the ones counted in the analogous sequence of
diagrams from $\cH$ to $\hat{\cH}$, so equation~\eqref{eq:transitionmaps1} follows.

We now consider the path of Heegaard diagrams from $\hat{\cH}'$ to $\hat{\cH}$
shown in Figure~\ref{fig::27}. The diagram $\hat{\cH}''$ is obtained by handle sliding
$\b_1$ over $\b_0$. We let $\b_1'$ denote the curve resulting from this handle slide.
The diagram $\hat{\cH}'''$ is obtained by isotopying the Heegaard surface by sliding
the foot of the tube marked  $D$ inside $\b_1'$ over the tube marked $B$,
carrying $\a_0$ and $\b_0$ along,
and then handle sliding $\a_0$ over $\a_1$, giving rise to a new curve $\a_0'$.
Note that $\hat{\cH}'''$ is a simple stabilization of $\hat{\cH}$ (in \cite{JTNaturality},
this type of stabilization was also referred to as a $(0,0)$-stabilization),
and hence there is a destabilization from $\hat{\cH}'''$ to $\hat{\cH}$.

\begin{figure}[ht!]
 \centering
 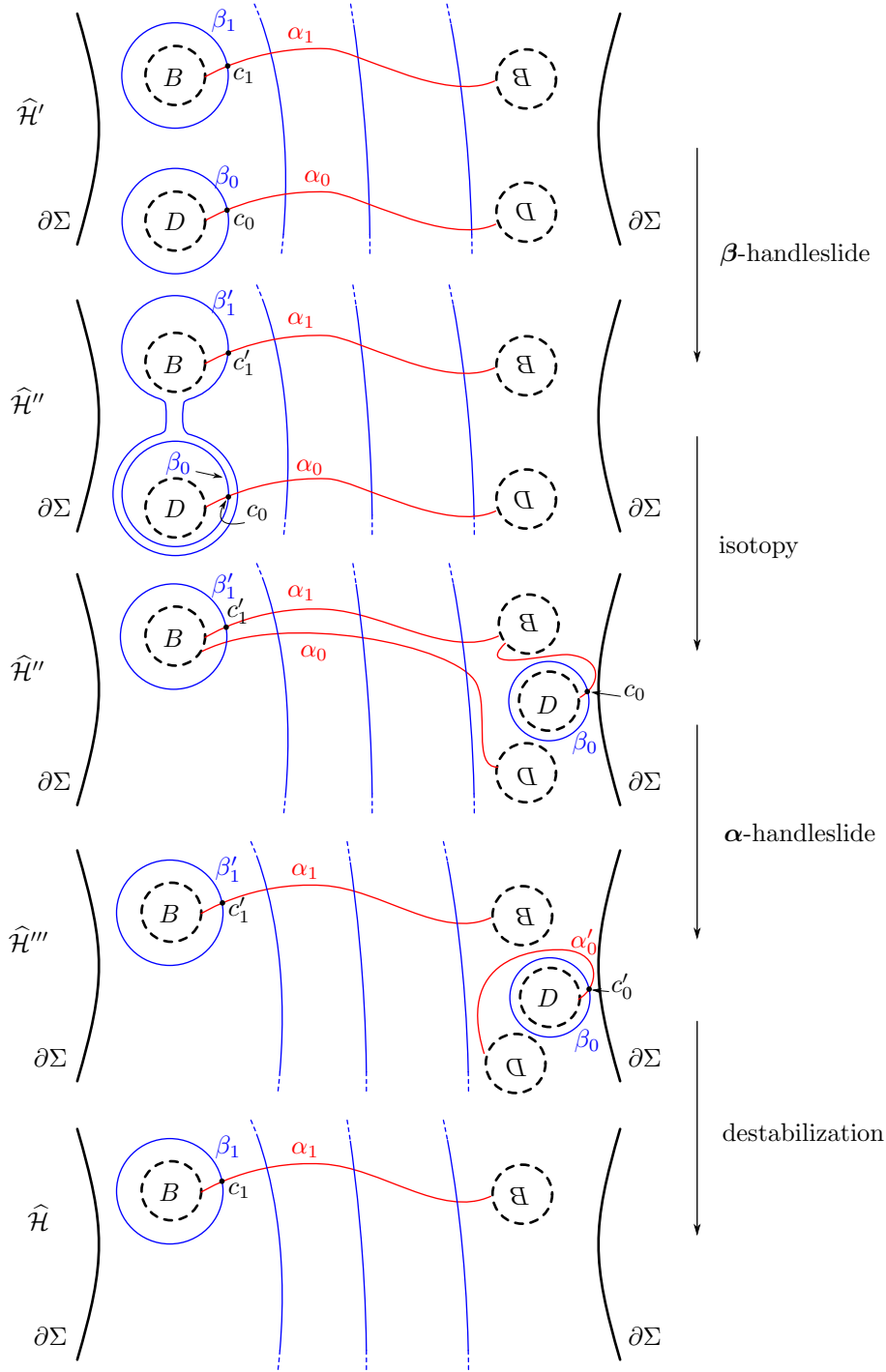
 \caption{A sequence of Heegaard diagrams from $\hat{\cH}'$ to $\hat{\cH}$. \label{fig::27}}
\end{figure}

Using the presence of the boundary $\d \S$ to simplify the computation, one can see that
the only holomorphic triangles contributing to the change of diagrams map $\Psi_{\hat{\cH}'\to\hat{\cH}''}$
have homology class $\psi\sqcup \psi_0$, where $\psi$ is a holomorphic triangle on an unstabilized
Heegaard triple $(\S,\as,\bs,\bs')$, where $\bs'$ is a small isotopy of $\bs$, and $\psi_0$
is the homology class shown in Figure~\ref{fig::28} .
Using an additional triangle map to move the $\bs'$ back to $\bs$ (and only isotoping $\b_0$
and $\b_1'$ a small amount) that can be analyzed similarly, we have that
\begin{equation}
\Psi_{\hat{\cH}'\to \hat{\cH}''}(\ve{x}\times c_0\times c_1)=\ve{x}\times c_0\times c_1',\label{eq:transitionmaps2}
\end{equation}
where $c_1' = \a_1 \cap \b_1'$.

\begin{figure}[ht!]
 \centering
 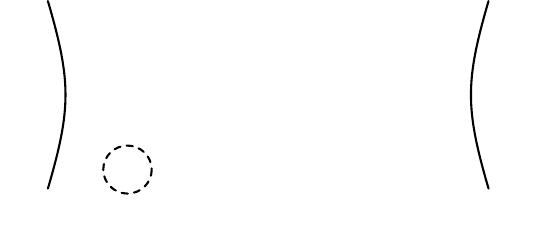
 \caption{Computing $\Psi_{\hat{\cH}'\to \hat{\cH}''}$. \label{fig::28}}
\end{figure}

By Lemma~\ref{lem:compoundtrianglemap2}, we have that
\begin{equation}
\Psi_{\hat{\cH}''\to \hat{\cH}'''}(\ve{x}\times c_0\times c_1')=\ve{x}\times c_0'\times c_1',\label{eq:transitionmaps3}
\end{equation}
where $c_0' = \a_0' \cap \b_0$. Equations \eqref{eq:transitionmaps2} and \eqref{eq:transitionmaps3} imply that
\begin{equation}\Psi_{\hat{\cH}'\to \hat{\cH}'''}(\ve{x}\times c_0\times c_1)=\ve{x}\times c'_0\times c_1'.\label{eq:transitionmaps5}\end{equation} The diagrams $\hat{\cH}'''$ and $\hat{\cH}$ are related
by a destabilization of the curves $\a_0'$ and $\b_0$. Hence
\begin{equation}
\Psi_{\hat{\cH}'''\to \hat{\cH}}(\ve{x}\times c_0'\times c_1') = \ve{x}\times c_1.\label{eq:transitionmaps4}
\end{equation}

In $\hat{\cH}$, the only $\a$-curve that intersects $\b_1$ is $\a_1$,
and $\a_1 \cap \b_1 = \{c_1\}$, hence every generator in $\hat{\cH}$ is
of the form $\ys \times c_1$ for some $\y \in \T_{\a} \cap \T_{\b}$.
So, there are constants $c_{\xs,\ys}\in \bF_2$ such that
\[
\Psi_{\cH \to \hat{\cH}}(\xs) = \sum_{\ys \in \T_{\a} \cap \T_{\b}} c_{\xs,\ys} (\ys \times c_1).
\]
If we substitute this into equation~\eqref{eq:transitionmaps1}, we get that
\[
\Psi_{\cH'\to \hat{\cH}'}(\ve{x} \times c_0) =
\sum_{\ys \in \T_\a \cap \T_\b} c_{\xs,\ys} (\ys \times c_0 \times c_1).
\]
Together with equations~\eqref{eq:transitionmaps5} and~\eqref{eq:transitionmaps4},
we arrive at the equality
\begin{equation*}
 \begin{split}
 \Psi_{\cH'\to \hat{\cH}}(\xs \times c_0) &=
 (\Psi_{\hat{\cH}'''\to \hat{\cH}}\circ \Psi_{\hat{\cH}'\to \hat{\cH}'''}
 \circ \Psi_{\cH'\to \hat{\cH}'})(\xs \times c_0) = \\
 \sum_{\ys \in \T_\a \cap \T_\b} c_{\xs,\ys} (\ys \times c_1) &= \Psi_{\cH\to \hat{\cH}}(\xs).
 \end{split}
\end{equation*}
Hence, we obtain that
\[
\Psi_{\cH'\to \cH}(\xs\times c_0) = (\Psi_{\hat{\cH}\to \cH} \circ \Psi_{\cH'\to \hat{\cH}})(\xs\times c_0) =
(\Psi_{\hat{\cH}\to \cH}\circ \Psi_{\cH\to \hat{\cH}})(\xs)=\xs.
\]
We note that the last equality follows from naturality.
Hence $\Psi_{\cH'\to \cH}(\xs\times c_0) = \ve{x}$, completing the proof
of Proposition~\ref{prop:compoundstabilizationmap}.
\end{proof}

\section{Contact cell decompositions and the gluing map}
\label{sec:definingthegluingmap}
In this section, we give a definition of the contact gluing map using contact cell decompositions, and prove invariance.
The construction is similar to the one due to Honda, Kazez, and Mati\'c~\cite{HKMTQFT}.
On a formal level, the gluing map is described as follows. Suppose that $(M,\g)$
is a sutured submanifold of $(M',\g')$ and that $\xi$ is a cooriented contact structure
on $M' \setminus \Int(M)$ such that $\d M$ is a convex surface with dividing set $\g$ and
$\d M'$ is a convex surface with dividing set $\g'$. Note that this implies that
a contact vector field positively transverse to $\d M$ lies on the positive side of $\xi$
along $R_-(M,\g)$, and on the negative side of $\xi$ along $R_+(M,\g)$.
In this situation, there is an induced map
\[
\Phi_{\xi} \colon  \SFH(-M,-\g) \to \SFH(-M',-\g'),
\]
called the gluing map.

\subsection{Sutured cell decompositions}

In order to discuss cell decompositions of contact 3-manifolds,
we need the following notion of cell decomposition for surfaces with divides:

\begin{define}\label{def:almostcelldecomp}
Let $F$ be a closed, orientable surface, and  $\g \subset F$ a dividing set.
A \emph{sutured cell decomposition} of $(F,\g)$ consists of the following:
\begin{itemize}

\item (\emph{Fattened 0-cells}) A collection of pairwise disjoint disks $B_1, \dots, B_n \subset F$
  such that each $B_i \cap \g$ is an arc and each component of $\g$ intersects some $B_i$.

\item (\emph{1-cells}) A collection of pairwise disjoint, properly embedded arcs
\[
\lambda_1,\dots, \lambda_m \subset F \setminus \bigcup_{i = 1}^n \Int(B_i)
\]
disjoint from $\g$. Furthermore, each component of
$F \setminus (B_1 \cup \dots \cup B_n \cup \lambda_1 \cup \dots \cup \lambda_m)$
intersects $\gamma$ non-trivially and is homeomorphic to an open disk.
\end{itemize}
\end{define}

A sutured cell decomposition of a torus with two parallel divides
is illustrated in Figure~\ref{fig::77}.

\begin{figure}[ht!]
 \centering
 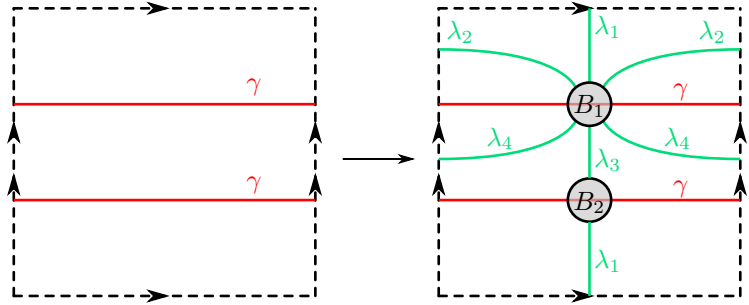
 \caption{A torus $F$ with two parallel divides $\gamma$ (left),
 as well as a sutured cell decomposition (right).\label{fig::77}}
\end{figure}

\begin{rem} \label{rem:disk}
  Let $F_0\subset F$ be the closure of a connected component of $F \setminus
  (B_1 \cup \dots \cup B_n \cup \lambda_1 \cup \dots \cup \lambda_m)$. Since
  each component of $\g$ intersects some $B_i$, the dividing set $\g \cap
  F_0$ contains no closed curves.  Hence, according to Giroux's
  criterion~\cite{HondaClassI}*{Theorem~3.5}, if $F$ is a convex surface in
  the contact manifold $(M,\xi)$, and $\d B_1\cup \dots \cup \d B_n\cup
  \lambda_1\cup \cdots \cup \lambda_m$ is a Legendrian graph, then $F_0$ has
  a tight neighborhood in $M$.
\end{rem}

We now describe two moves between sutured cell decompositions $\cD$ and $\cD'$
of a surface $F$ with dividing set $\gamma$:
\begin{enumerate}[label=($\cM$-\arabic*)]
\item\label{submove:1} $\cD'$ is obtained from $\cD$ by adding or removing
  a 1-cell $\lambda \subset F\setminus \gamma$ that has both
  ends on boundaries of fattened 0-cells in $\cD$, but is
  otherwise disjoint from the 0- and 1-cells of $\cD$.
\item\label{submove:2} $\cD'$ is obtained from $\cD$ by adding or removing a fattened 0-cell $B$
  and a 1-cell $\lambda$ that connects $B$ to another 0-cell in $\cD$,
  but is otherwise disjoint from the 0- and 1-cells of $\cD$.
\end{enumerate}
Moves~\ref{submove:1} and~\ref{submove:2} are illustrated in Figure~\ref{fig::43}.

\begin{figure}[ht!]
 \centering
 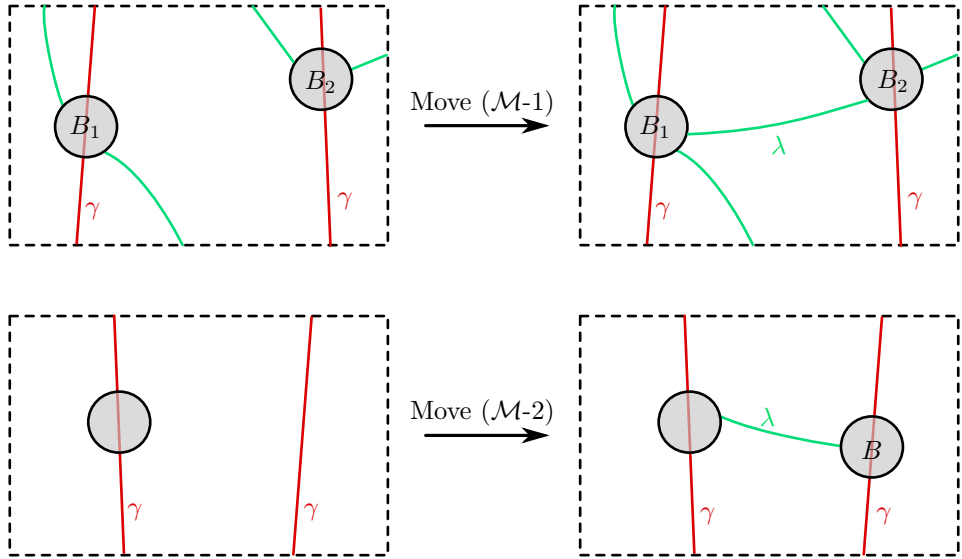
 \caption{Moves~\ref{submove:1} and~\ref{submove:2} between sutured cell decompositions of $(F,\gamma)$.
 The grey disks are the fattened 0-cells, and the green arcs marked $\lambda$ are the Legendrians.
 The red arcs marked $\g$ form the dividing set.  \label{fig::43}}
\end{figure}

\begin{lem}\label{lem:suturedcellmoves}If $F$ is a closed, orientable surface and $\gamma\subset F$ is a dividing set,
 then any two sutured cell decompositions of $(F,\gamma)$ can be connected
 by a sequence of moves~\ref{submove:1} and \ref{submove:2}.
\end{lem}

To prove Lemma~\ref{lem:suturedcellmoves}, it is convenient to consider the following notion of cell decomposition for surfaces with divides which is more general than Definition~\ref{def:almostcelldecomp}:

\begin{define} A \emph{generalized sutured cell decomposition} $\cD^\bullet$ of a surface $F$
with divides $\g$ consists of the following:
\begin{itemize}
\item  (\emph{Fattened 0-cells}) A collection of pairwise disjoint disks $B_1, \dots, B_n \subset F$
  such that each $B_i \cap \g$ is an arc and each component of $\g$ intersects some $B_i$.
\item (\emph{1-cells}) A collection of pairwise disjoint, properly embedded arcs
\[
\lambda_1,\dots, \lambda_m\subset F\setminus \bigcup_{i=1}^n \Int(B_i)
\]
which are transverse to $\gamma$. Furthermore, each component of $F\setminus (B_1\cup \cdots \cup B_n\cup \lambda_1\cup \cdots \cup \lambda_m)$ is homeomorphic to an open disk (which may be disjoint from $\gamma$).

 \item (\emph{Splitting arcs}) A collection of oriented, properly embedded, pairwise disjoint arcs $c_1,\dots, c_l\subset \bigcup_{i=1}^n  B_i$, disjoint from each of the $\lambda_i$, such that $|c_i\cap \gamma|=1$ for all $i$, and such that each component of the closure of $F\setminus (B_1\cup \cdots \cup B_n\cup \lambda_1\cup \cdots \cup \lambda_m)$ which is disjoint from $\gamma$ intersects the terminal endpoint of at least one $c_i$.
 \end{itemize}
\end{define}

Note that a sutured cell decomposition can be viewed as a generalized sutured
cell decomposition with no splitting arcs. In the other direction, if
$\cD^\bullet$ is a generalized sutured cell decomposition of a surface $F$
with dividing set $\g$, we can define the \emph{split} of $\cD^\bullet$,
denoted $\cS(\cD^\bullet)$, as the (genuine) sutured cell decomposition
obtained by performing the following two modifications to $\cD^\bullet$:
\begin{enumerate}[label=($\cS$-\arabic*)]
\item\label{def:split1} A small, fattened 0-cell is added at each intersection point between a 1-cell $\lambda_i$
and the dividing set $\gamma$.
\item\label{def:split2} The fattened 0-cells are split in half along the arcs $c_1,\dots, c_l$.
A new 1-cell is added for each of the arcs $c_i$, as shown in Figure~\ref{fig::71}.
The new 1-cell is on the initial side of $c_i$, with respect to the orientation of $c_i$.
\end{enumerate}

\begin{figure}[ht!]
 \centering
 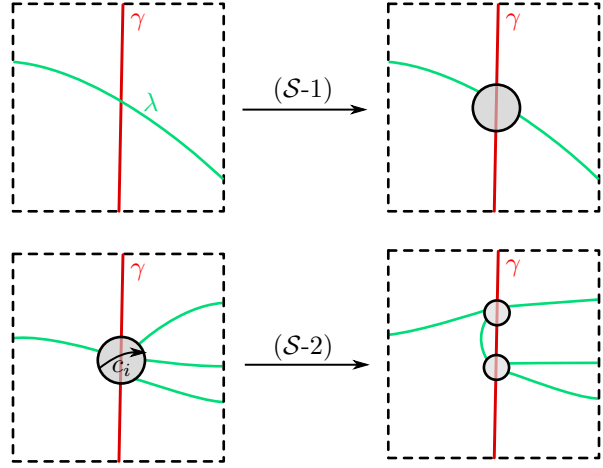
 \caption{Splitting a generalized sutured cell decomposition $\cD^\bullet$.
 On the left is $\cD^\bullet$, and on the right is the split, $\cS(\cD^\bullet)$.\label{fig::71}}
\end{figure}

\begin{proof}[Proof of Lemma~\ref{lem:suturedcellmoves}]Suppose $\cD_1$ and $\cD_2$
are two sutured cell decompositions of $(F,\gamma)$.
By move \ref{submove:2}, we can increase the number of fattened
0-cells in both, and hence we can assume that they have the
same number of fattened 0-cells.  It is not hard to see
that we can isotope or rescale a fattened 0-cell by a
sequence of moves~\ref{submove:1} and~\ref{submove:2}.
The procedure is illustrated in Figure~\ref{fig::70}.
So we can assume that $\cD_1$ and $\cD_2$ have the same fattened 0-cells.

\begin{figure}[ht!]
 \centering
 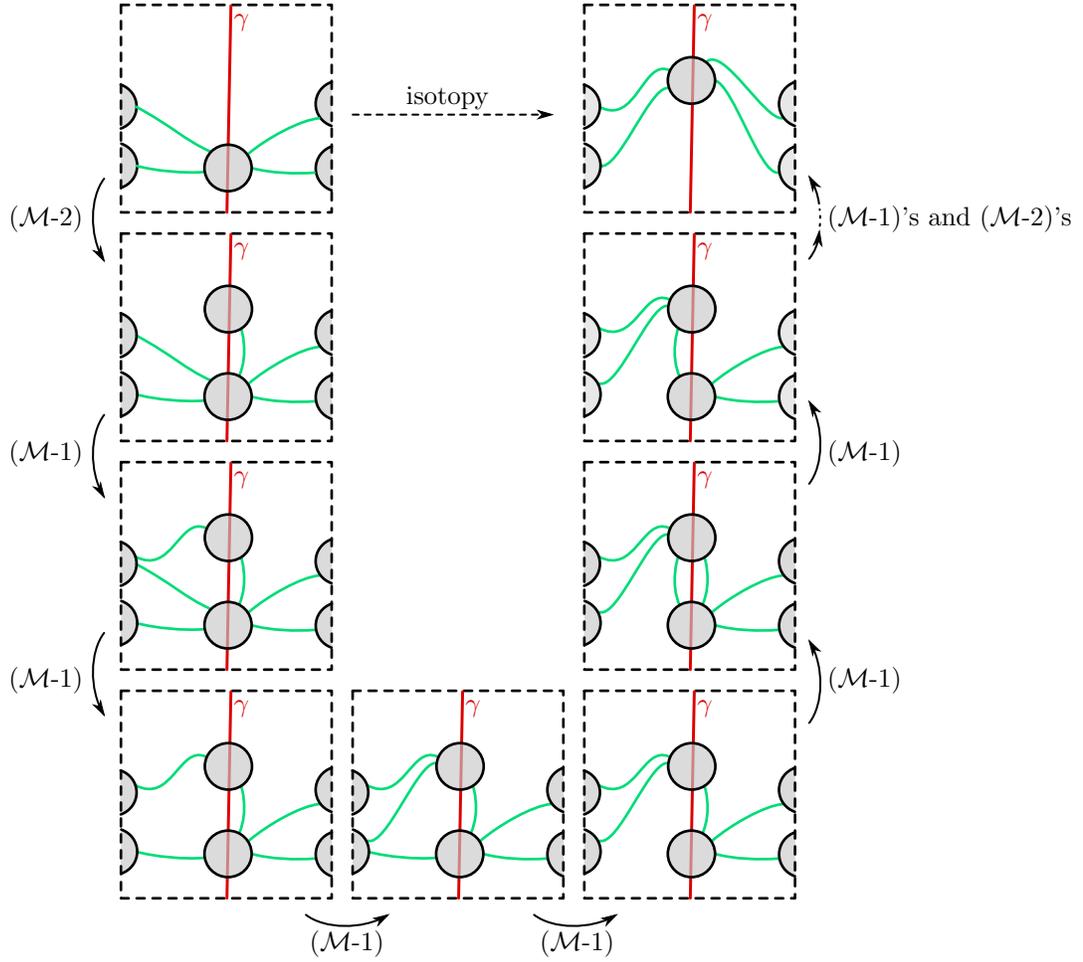
 \caption{Isotoping a fattened 0-cell via a sequence of
  moves~\ref{submove:1} and~\ref{submove:2}.\label{fig::70}}
\end{figure}

We can view the fattened 0-cells as 0-handles of
the surface, and the 1-cells $\lambda_1,\dots, \lambda_n$
as the cores of 1-handles. Hence, a
generalized sutured cell decomposition with fattened 0-cells
$B_1, \dots, B_n$ corresponds to a Morse function
\[
f\colon \Sigma\setminus \Int(B_1\cup \cdots \cup B_n)\to [0,\infty)
\]
 that has no index 0 critical points, achieves its minimal
 value of $0$ on $\d B_1\cup \cdots \cup \d B_n$, and such
 that the union of the descending manifolds of the index
 1 critical points is $\lambda_1\cup \cdots \cup \lambda_m$.
 The collection of splitting arcs $c_1,\dots, c_{\l}$ is some additional combinatorial
 data, which can always be constructed, given such a Morse
 function. Two such Morse functions can always be connected
 by a path of smooth functions with no index~0 critical points
 relative to $\d B_1 \cup \cdots \cup \d B_n$ (see Lemma~\ref{lem:allowablearcslides}
 for a proof of a closely related result that
 adapts to our present setting). Consequently, two generalized
 sutured cell decompositions can be connected by the
 following moves:
\begin{enumerate}[label=($g\cM$-\arabic*)]
\item\label{generalmove:1} Adding or removing a splitting arc
  disjoint from all the other splitting arcs and 1-cells.
\item\label{generalmove:2} Isotoping a 1-cell.
\item\label{generalmove:3} Adding or removing a 1-cell.
\item\label{generalmove:4} Arc sliding a 1-cell $\lambda_i$ across another 1-cell $\lambda_j$.
\end{enumerate}

We can view any sutured cell decomposition as a
generalized sutured cell decomposition with no splitting arcs, and with no intersections
between the 1-cells and the dividing set.
Hence, to prove the main claim, it suffices to show that if two generalized sutured cell
decompositions $\cD_1^\bullet$ and
$\cD_2^\bullet$ differ by moves~\ref{generalmove:1}--\ref{generalmove:4},
then their splits $\cS(\cD_1^\bullet)$ and $\cS(\cD_2^\bullet)$
differ by a sequence of moves \ref{submove:1} and \ref{submove:2}.

We first address move~\ref{generalmove:1}. We illustrate
in Figure~\ref{fig::72} an example of how to connect
the splits of two generalized sutured cell decompositions
that differ by move~\ref{generalmove:1}.

\begin{figure}[ht!]
 \centering
 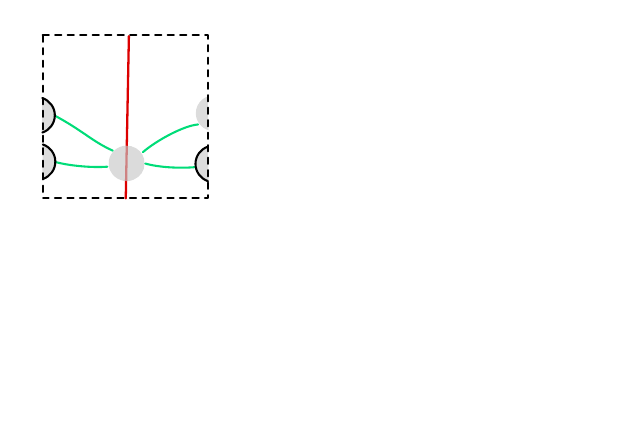
 \caption{Connecting $\cS(\cD_1^\bullet)$ and $\cS(\cD_2^\bullet)$
 by a sequence of moves \ref{submove:1} and \ref{submove:2}
  when $\cD_1^\bullet$ and $\cD_2^\bullet$ differ by move~\ref{generalmove:1}. \label{fig::72}}
\end{figure}

Next, we consider move~\ref{generalmove:2}, isotopies of a 1-cell $\lambda_i$.
Moves of type~\ref{generalmove:2} can be further broken down into three subtypes:
\begin{enumerate}
\item\label{it:isot1} Isotopies supported outside a neighborhood of $\gamma$.
\item\label{it:isot2} Isotopies supported in a neighborhood of $\gamma$ that are fixed on the circles $\d B_j$,
and which either create or cancel a pair of intersection points between $\gamma$ and $\lambda_i$.
\item\label{it:isot3} Isotopies supported in a neighborhood of a single $\d B_i$
that isotope an end of a 1-cell $\lambda_i$ across $\gamma$.
\end{enumerate}

Isotopies of type~\eqref{it:isot1} can be addressed using
a manipulation similar to the one shown in Figure~\ref{fig::70},
so we leave this case to the reader. In case~\eqref{it:isot2},
a sequence of moves~\ref{submove:1} and~\ref{submove:2} suffice;
an example is shown in Figure~\ref{fig::73}.
We leave case~\eqref{it:isot3} to the reader, since the manipulation is similar.

\begin{figure}[ht!]
 \centering
 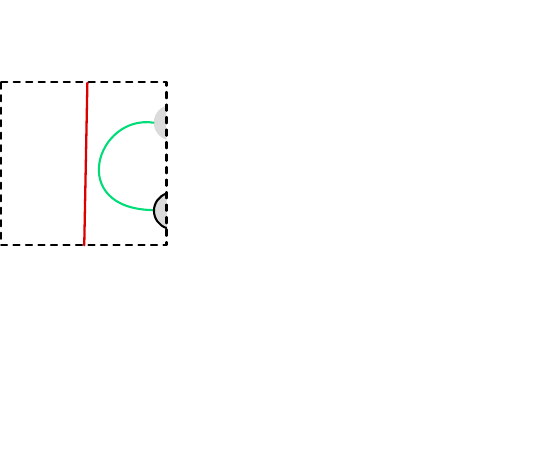
 \caption{An example of connecting $\cS(\cD_1^\bullet)$ and $\cS(\cD_2^\bullet)$
 by a sequence of moves \ref{submove:1} and
 \ref{submove:2} when $\cD_1^\bullet$ and $\cD_2^\bullet$
 differ by move~\ref{generalmove:2}. \label{fig::73}}
\end{figure}

Sufficiency of moves~\ref{submove:1} and~\ref{submove:2}
to connect splits of generalized sutured cell
decompositions differing by moves~\ref{generalmove:3} and \ref{generalmove:4}
is proven in an analogous fashion. We leave the argument as an easy exercise for the reader.
\end{proof}

\subsection{Contact cell decompositions}

In this section, we describe some background on contact cell decompositions. The technical content is due to
Honda, Kazez, and Mati\'c~\cite{HKMSutured}*{Section~1.1} and  Giroux~\cite{GirouxCorrespondence}.

\begin{define}\label{def:contactcelldecomposition}
Suppose that $(M,\g)$ is a sutured submanifold of $(M',\g')$ and
$\xi$ is a contact structure on $Z := M'\setminus \Int(M)$ such that $\d Z$ is convex with
dividing set $\g \cup \g'$. A \emph{contact cell decomposition of~$(Z,\xi)$}
consists of the following data:
\begin{enumerate}
\item \label{it:vector} A non-vanishing contact vector field $v$ defined on a neighborhood of $\d Z$ in $Z$
  and transverse to $\d Z$ such that it induces the dividing set $\g \cup \g'$.
  The flow of $v$ induces a diffeomorphism of
  $\d Z \times I$ with  a collar neighborhood of $\d Z$ in $Z$.
  Under this diffeomorphism, $v$ corresponds to $\d / \d t$,
  the boundary $\d M$ is identified with $\d M \times \{0\}$,
  and $\d M'$ is identified with $\d M' \times \{1\}$.
  We let $\nu = v|_{\d M \times I}$ and $\nu' = v|_{\d M' \times I}$.
\item \label{it:barrier} ``Barrier'' surfaces $S\subset \d M\times (0,1)$ and $S'\subset \d M'\times (0,1)$
  in $Z$ that are isotopic to $\d M$ and $\d M'$, respectively,
  and are transverse to $v$.
  Write $N$ and $N'$ for the collar neighborhoods of $\d M$ and $\d M'$ in
  $Z$ that are bounded by $S$ and $S'$, respectively, and set $Z' = Z \setminus \Int(N \cup N')$.
\item A Legendrian graph $\Gamma \subset Z'$
  that intersects $\d Z'$ transversely in a finite collection of points along the dividing
  set of $\d Z'$ with respect to the vector field $v$.
  Furthermore, $\Gamma$ is tangent to the vector field $v$ near $\d Z'$.
\item A choice of regular neighborhood $N(\Gamma)$ of $\Gamma$ such that $\xi|_{N(\Gamma)}$
  is tight and $\d N(\Gamma) \setminus \d Z'$
  is a convex surface.  Furthermore, $N(\Gamma) \cap \d Z'$
  is a collection of disks $D$ with Legendrian boundary such that $\tb(\d D) = -1$.
  We assume that the edge rounding procedure of Honda~\cite{HondaClassI}*{Section~3.3.2}
  has been performed such that $N(\Gamma)$ meets $\d Z'$ tangentially along
  the Legendrian unknots forming $\d N(\Gamma)$.
\item A collection of convex 2-cells $D_1, \dots, D_n$ inside
  $Z' \setminus \Int(N(\Gamma))$
  with Legendrian boundary on $\d (Z' \setminus \Int(N(\Gamma)))$ and $\tb(\d D_i) = -1$.
\end{enumerate}
Furthermore, the following hold:
\begin{enumerate} [label=(\alph*)]
\item \label{condition:a} The complement of $N(\Gamma) \cup D_1 \cup \cdots \cup D_n$
  in $Z'$ is a finite collection of topological 3-balls, and $\xi$ is tight on each.
\item \label{condition:b} The disks in $N(\Gamma) \cap \d Z'$
  and the Legendrian arcs $\d D_i \cap \d Z'$ induce a sutured
  cell decomposition of $\d Z'$, with dividing set $\{\, x \in \d Z' : v_x \in \xi_x \,\}$
  (Definition~\ref{def:almostcelldecomp}).
\end{enumerate}
\end{define}

\begin{rem}
  Given surfaces $S$ and $S'$ in $Z$ and a transverse contact
  vector field $v$ defined in a neighborhood of $\d Z$ that
  satisfy~\eqref{it:vector} and \eqref{it:barrier},
  it is not always possible to construct a contact cell decomposition of
  $Z$ with $S$ and $S'$ as barrier surfaces. For example,
  the characteristic foliations on $S$ and $S'$ may obstruct the existence of
  sutured cell decompositions (Definition~\ref{def:almostcelldecomp}) such that
  the arcs $\lambda_i$ and the curves $\d B_i$ are Legendrian.
\end{rem}

We now describe an important example of contact cell decompositions:

\begin{example} \label{ex:productcelldecomp}
Suppose that $(M,\g)$ is a sutured submanifold of $(M',\g')$ and
$\xi$ is a compatible contact structure on $Z = M'\setminus \Int(M)$. Furthermore,
suppose that $(Z, \xi)$ is contactomorphic to $F \times I$ with an
$I$-invariant contact structure for a surface $F \iso \d M \iso \d M'$.
Suppose that the disks $B_1, \dots, B_n \subset F$ and the (not necessarily
Legendrian) properly embedded arcs $\lambda_1, \dots, \lambda_m \subset F \setminus \Int(B_1\cup \cdots B_n)$
give a sutured cell decomposition of $F \times \{0\}$, as in Definition~\ref{def:almostcelldecomp}.

There is an induced contact cell decomposition of $Z$,
called the \emph{product contact cell decomposition}, that we describe presently.
For an illustration, see Figure~\ref{fig::42}.
Let $G \subset F$ be the graph consisting of
\[
G = \d B_1 \cup \cdots \cup \d B_n \cup \lambda_1 \cup \dots \cup \lambda_m;
\]
this is not necessarily Legendrian. However, using Giroux's flexibility theorem
\cite[Proposition~3.6]{Giroux},
we can find a surface $S \subset F \times I$ that is the image of a $C^0$
small isotopy of $F \times \{\tfrac{1}{3}\}$ that is transverse to $\d/\d t$
and such that the image of $G \times \{\tfrac{1}{3}\}$ is Legendrian. We let
$S'$ be the translation of $S$ by $\tfrac{1}{3}$ in the $\d/\d t$ direction.
We use $S$ and $S'$ for the barrier surfaces of our contact cell
decomposition. We will also write $B_1, \dots, B_n$ and $\lambda_1, \dots, \lambda_m$
for the images on $S$ of the corresponding cells on $F \times \{\frac13\}$, under the
chosen $C^0$ small isotopy. Let $N$ and $N'$ be the collars of $\d M$ and $\d M'$
bounded by $S$ and $S'$, respectively, and write $Z' = Z \setminus \Int(N \cup N')$.

For every $i \in \{\,1, \dots, n \,\}$, choose a point $p_i \in \Int(\g \cap B_i)$.
We define the Legendrian graph $\Gamma$ as
\[
\bigcup_{i = 1}^n (\{p_i\} \times I) \cap Z',
\]
and set $N(\Gamma) = \bigcup_{i=1}^n (B_i \times I) \cap Z'$.
The contact 2-cells of the decomposition are
defined to be $D_i := (\lambda_i \times I) \cap Z'$. Using Giroux's flexibility theorem on
$\d N(\Gamma)$, we can assume that each $D_i$ has Legendrian boundary. The disks $D_i$
have $\tb(\d D_i) = -1$ since their boundaries intersect the dividing
set exactly twice. If $C$ is a component of $S \setminus (B_1 \cup \cdots \cup
B_n \cup \lambda_1 \cup \dots \cup \lambda_m)$, then
$(C \times I) \cap Z'$ is a tight contact
ball in the complement of $N(\Gamma) \cup D_1 \cup \dots \cup D_n$ in $Z'$.
\end{example}

\begin{figure}[ht!]
 \centering
 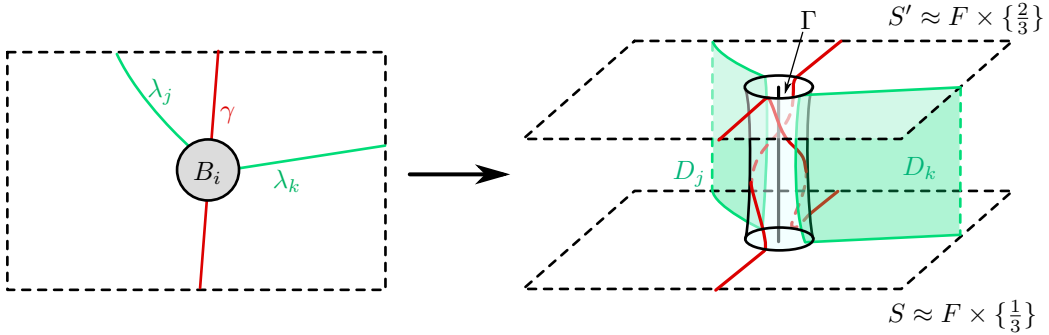
 \caption{The left shows a sutured cell decomposition of a surface $F$ with dividing set $\g$.
 The right shows the induced product contact cell decomposition of an
 $I$-invariant contact structure on $F \times I$ with dividing set $\g$ from
 Example~\ref{ex:productcelldecomp}. The barrier surfaces are $S$ and $S'$,
 which are $C^0$ close to $F \times \{\tfrac{1}{3}\}$ and $F \times \{\tfrac{2}{3}\}$,
 respectively. The fattened
 0-cell $B_i$ is shown on the left in gray. The dividing set $\g$ is
 shown in red. The Legendrian arcs $\lambda_j$ and $\lambda_k$, as well as
 the corresponding convex 2-cells $D_j$ and $D_k$, are shown in green.
 \label{fig::42}}
\end{figure}

To show invariance of the gluing map, we need to describe how two contact cell
decompositions are related. We have the following; cf.~\cite{HKMSutured}*{Theorem~1.2}:

\begin{prop}\label{prop:movesbetweencelldecompositions}
Suppose that $(M,\g)$ is a sutured submanifold of $(M',\g')$, and
$\xi$ is a contact structure on $Z = M'\setminus \Int(M)$ such that $\d Z$
is convex with dividing set $\g \cup \g'$. If $\cC_1$ and
$\cC_2$ are contact cell decompositions of $Z$, then there
is a sequence $\cC_{(1)}, \dots, \cC_{(\ell)}$ of contact cell decompositions
with $\cC_1 = \cC_{(1)}$ and $\cC_2 = \cC_{(\ell)}$, such that $\cC_{(i+1)}$ is
obtained from $\cC_{(i)}$ by one of the following moves, or its inverse:
\begin{enumerate}[label=(C\arabic*)]
\item\label{move:isotopy} (Isotopy) Replacing a contact cell decomposition
  $\cC$ with a contact cell decomposition of the form $\phi(\cC)$, where
  $\phi \in \Diff(Z)$ (not necessarily a contactomorphism)
  fixes $\d Z$ pointwise and is isotopic to the identity relative to $\d Z$.
\item\label{move:index0/1} (Index 0/1 cancellation) Subdividing a Legendrian
  edge of $\Gamma$, or adding a Legendrian edge that has one endpoint on
  $\Gamma$, but which is otherwise disjoint from $\Gamma$ and the 2-cells of $\cC$.
\item\label{move:index1/2} (Index 1/2 cancellation)
  Adding a Legendrian edge $\lambda$ to $\Gamma$, and adding a convex 2-cell
  $D$ with $\tb(\d D)=-1$, such that the interiors of $\lambda$ and $D$ are
  disjoint from all the other cells, and $\d D = c' \cup c''$, where $c'$ is a
  Legendrian arc on a neighborhood of $\lambda$, and $c'' \subset \d (Z' \setminus \Int  N(\Gamma)) $
  is a Legendrian arc that is disjoint from the dividing set on $\d (Z' \setminus \Int N(\Gamma))$.
\item\label{move:index2/3} (Index 2/3 cancellation)
  Adding a convex disk $D$ with $\d D \subset \d (Z' \setminus \Int N(\Gamma))$ and
  $\tb(\d D)=-1$ disjoint from the other cells.
\end{enumerate}
\end{prop}

\begin{proof}
The result follows from an adaptation of the subdivision techniques due to
Giroux~\cite{GirouxCorrespondence} and Honda, Kazez, and Mati\'c~\cite{HKMSutured}*{Theorem~1.2}.
Giroux's  technique involves the following move:
Replace a convex 2-cell $D$ with a pair of convex  2-cells $D_1$ and
$D_2$ with $\tb(\d D_i) = -1$ that meet along a Legendrian arc $\lambda$
such that $D_1\cup D_2=D$ and $\lambda\subset D$ is an arc that intersects
the dividing set of $D$ transversely at a single point.

\begin{lem}\label{lem:subdivision}
The above subdivision of a 2-cell can be achieved using Moves~\ref{move:isotopy}--\ref{move:index2/3}.
\end{lem}

\begin{proof}
Let $D'$ be a parallel copy
of $D$, intersecting $D$ only along $\d D$. Let $D_1'\subset D'$, $D_2'\subset D'$,
and $\lambda'\subset D'$ be the images of $D_1$, $D_2$, and $\lambda$,
respectively.  We can add $D_1'$ and $\lambda'$ to the decomposition using
Move~\ref{move:index1/2}. Then we can add $D_2'$ to the decomposition using
Move~\ref{move:index2/3}. Then we remove $D$ from the decomposition using the
inverse of Move~\ref{move:index2/3}. Finally, we use Move~\ref{move:isotopy}
to move $D_1'$, $D_2'$, and $\lambda'$ into the position of $D_1$, $D_2$, and
$\lambda$, while preserving all the other cells.
\end{proof}

\textbf{Step 1:} \emph{Using Move~\ref{move:isotopy}, we can change the vector field $v_1$ of $\cC_1$
to the vector field $v_2$ of $\cC_2$.}

We will focus on changing the contact vector field $\nu_1 = v_1|_{\d M \times I}$
to the vector field $\nu_2 = v_2|_{\d M \times I}$, as changing
$\nu_1' = v_1|_{\d M' \times I}$ to the vector field $\nu_2' = v_2|_{\d M' \times I}$ is analogous.
The idea is that the space of germs of contact vector fields defined
on open neighborhoods of $\d M$ that induce the dividing set $\g$ is
convex and hence contractible. There is an open neighborhood of $\d M$ where the convex
combinations $\nu_t = (2-t) \nu_1 + (t-1)\nu_2$ are non-vanishing for $t \in [1,2]$.
Then $\nu_t$ is transverse to $\d M$ for every $t \in [1,2]$. We can
define an isotopy $\psi_t$ of a neighborhood of $\d M$, by flowing a point
$p \in M$ backward along $\nu_1$ until it hits $\d M$ (say, at a time
$-T(p)$), then flowing along $\nu_t$ for time $T(p)$. This
yields a 1-parameter family of contactomorphic embeddings of a neighborhood $U$
of $\d M$ into $Z$. We can extend this to a family of
contactomorphisms $\psi_t$ of all of $M$ for $t \in [1,2]$ by writing
$\psi_t$ as the integral of a time-dependent contact vector field $X_t$ (i.e.,
$X_t$ is a contact vector field for each $t$) defined over $\d M \times I$,
and then extending $X_t$ to all of $Z$ using
time-dependent contact Hamiltonians that vanish outside a neighborhood of $\d M$.
Using Move~\ref{move:isotopy}, we can isotope $S_1$ into $U$ using Move~\ref{move:isotopy},
and then push $\cC_1$ forward under $\psi_1$.

\textbf{Step 2:} \emph{If $\cC_1$ and $\cC_2$ are contact cell
decompositions with the same vector fields $v$ and $v'$, then
we can achieve that $\cC_1$ and $\cC_2$ have the same barrier surfaces with identical
induced sutured cell decompositions using Moves~\ref{move:isotopy}--\ref{move:index2/3}.}

Let $\cD_i$ and $\cD_i'$ be the sutured cell decompositions induced on the barrier surfaces $S_i$ and $S_i'$
of the contact cell decomposition $\cC_i$ for $i \in \{1,2\}$.
Write $\pi \colon \d M \times I \to \d M$ and $\pi' \colon \d M' \times I \to \d M'$ for the projections.
Note that  $\pi(\cD_1)$ and $\pi(\cD_2)$ are sutured cell decompositions of $\d M$,
and $\pi'(\cD_1')$ and $\pi'(\cD_2')$ are sutured cell decompositions of $\d M'$.
Let us focus on the barrier surfaces $S_1$ and $S_2$ near $\d M$, and their sutured cell
decompositions $\cD_1$ and $\cD_2$.
By Lemma~\ref{lem:suturedcellmoves}, it follows that $\pi(\cD_1)$ and $\pi(\cD_2)$ can be connected
by Moves~\ref{submove:1} and~\ref{submove:2}.

We will show that, using Moves~\ref{move:isotopy}--\ref{move:index2/3},
we can connect $\cC_1$ to a contact cell decomposition
$\hat{\cC}_2$ with the same contact vector field as $\cC_1$
and $\cC_2$, whose induced sutured cell decomposition
of $\d M$ is $C^0$ close to $\pi(\cD_2)$.
It is sufficient to show the claim when $\pi(\cD_2)$ is obtained from
$\pi(\cD_1)$ by a single instance of Move~\ref{submove:1} or~\ref{submove:2}.

Firstly, we can construct two contact automorphisms $\varphi_1$ and
$\varphi_2$ of $(Z, \xi)$ that are isotopic to the
identity relative to $\d Z$, such that  $\varphi_1(\cC_1)$ and
$\varphi_2(\cC_2)$ are cell decompositions with Legendrian graphs
$\varphi_1(\Gamma_1)$ and $\varphi_2(\Gamma_2)$, respectively, that intersect
$\d M \times I$ and $\d M' \times I$ along arcs of the form $\{p\}\times [a,1]$
or $\{p'\}\times [0,a]$ for various $p \in \g$, $p'\in \g'$, and $a > 0$.
Furthermore, by applying Move~\ref{move:isotopy}, we may assume that each component of the intersection
of every 2-cell with $\d Z \times I$ is $C^0$ close
to a set of the form $(\lambda \times I) \cap Z'$ for some arc $\lambda$ in $\d Z$.

Consider first the case when $\pi(\cD_2)$ is obtained
from $\pi(\cD_1)$ by adding a 1-cell, as in Move~\ref{submove:1}.
Write $\lambda\subset \d M$ for the new 1-cell, which
is added to the complement of the other 1-cells of
$\pi(\cD_1)$. To construct $\hat{\cC}_2$, we add a new Legendrian edge $e$ to
the graph $\Gamma_1$ of $\cC_1$, as well as a new 2-cell $c$, as follows.
Let $\lambda_1 \subset S_1$ denote the preimage of $\lambda$ under
$\pi|_{S_1}$. Using Legendrian realization that only changes $S_1$
in the $\d/\d t$-direction, we may assume that the resulting arc
$\lambda_1 \subset S_1$ is Legendrian, and projects $C^0$ close to
$\lambda$ under $\pi$. Note that the isotopy of $S_1$ used for
Legendrian realization can be chosen to fix the 0- and 1-cells
of $\cD_1$, and hence can be realized by an instance of
Move~\ref{move:isotopy} according to the following result:

\begin{lem}\label{lem:perturb}
Let $\cC$ be a contact cell decomposition with
barrier surface $S$ and induced sutured cell decomposition $\cD$ of $S$.
Then any isotopy of $S$ through surfaces that are transverse to $\nu = \d/\d t$
that fixes the 0- and 1-cells of $\cD$ can be achieved using Move~\ref{move:isotopy}.
The same holds for $S'$, $\nu'$, and $\cD'$.
\end{lem}

\begin{proof}
As above, we can assume that the graph $\Gamma$ is a union of Legendrians of the form
$\{p\}\times [a,1]$ for various $p\in \g$ and $0<a<1$, and the 2-cells are
small perturbations of sets of the form $(\lambda\times I) \cap Z'$ for arcs
$\lambda \subset \d M$. Suppose that $f \colon S\to \d M\times I$ is an embedding
that is the identity on the 0- and 1-cells of $\cD$, and whose image is transverse to $\nu$.
We can extend $f$ to an automorphism of $Z$ that is the identity outside $\d M \times (0,1)$.
It is straightforward to see that $f(\cC)$ is a contact cell decomposition, and
that $\cC$ and $f(\cC)$ are related by Move~\ref{move:isotopy}.
\end{proof}

Let $\varepsilon > 0$ be small, and let
\[
\theta \colon S_1 \times (-\varepsilon, \varepsilon)\to \d M\times I
\]
denote the map induced by the flow of $\nu=\d/ \d t$; i.e., translation in the $I$-direction.
We obtain the new edge $e$ by extending
$\theta(\lambda_1 \times \{\varepsilon\})$ into $N(\Gamma_1)$.
The new 2-cell $c$ is the intersection of $\theta(\lambda_1 \times [0,\varepsilon])$
with the complement of $N(\Gamma_1 \cup e)$. A schematic is shown
in Figure~\ref{fig::76}. Adding $e$ and $c$ to $\cC_1$
is an instance of Move~\ref{move:index1/2}.

\begin{figure}[ht!]
 \centering
 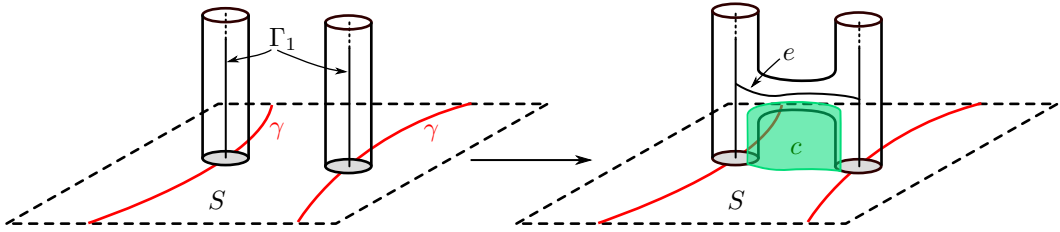
 \caption{Adding a new edge $e$ and a new 2-cell $c$
 to a contact cell decomposition to realize Move~\ref{submove:1}
 on the induced sutured cell decomposition of $\d M$.
 Shown on the left is a neighborhood of the Legendrian graph
 $\Gamma_1$, as well as the barrier surface $S$. Shown on the
 right is a neighborhood of $\Gamma_1\cup e$, as well as the
  new 2-cell $c$.\label{fig::76}}
\end{figure}

In the opposite direction, suppose that $\pi(\cD_2)$ is obtained from $\pi(\cD_1)$ by removing
a 1-cell $\lambda$, as in Move~\ref{submove:1}. Let $c$ be the 2-cell of $\cC_1$
corresponding to $\lambda_1 = (\pi|_{S_1})^{-1}(\lambda)$.
Note that the 2-cells of $\pi(\cD_1)$ neighboring $\lambda$ are disjoint,
since otherwise there would be a component of the complement of $\pi(\cD_2)$ that were not a disk.
We subdivide $\cC_1$ away from $S \cup S'$ such that the 3-cells of $\cC_1$ corresponding to
the 2-cells of $\cD_1$ on the two sides of $\lambda_1$ become distinct.
We then subdivide $c$ along $e := \theta(\lambda_1 \times \{\varepsilon\})$ for
$\varepsilon$ small using Lemma~\ref{lem:subdivision},
and cancel the 2-cell $\theta(\lambda_1 \times [0,\varepsilon]) \setminus N(\Gamma_1 \cup e)$
with one of the neighboring 3-cells using Move~\ref{move:index2/3}.

Next, we consider the case when $\pi(\cD_2)$ is obtained from
$\pi(\cD_1)$ by adding a fattened 0-cell $B$ and a 1-cell $\lambda$, as in Move~\ref{submove:2}.
To construct the new contact cell decomposition $\hat{\cC}_2$, we add two
new Legendrian edges $e_1$ and $e_2$ to $\cC_1$ that meet at a valence 2 vertex,
as well as a new 2-cell $c$, as follows. Let $\lambda_1 \subset S_1$ denote
the Legendrian realization of the preimage under $\pi|_{S_1}$ of
an extension of the Legendrian $\lambda$ to a point $p \in \Int(B)$.
The Legendrian realization can be achieved using Move~\ref{move:isotopy} according to Lemma~\ref{lem:perturb}.
We define the edge $e_1$ to be $\theta(\lambda_1 \times \{\varepsilon\})$ for some small $\varepsilon > 0$, and
$e_2$ as $\theta(\{p\} \times [0,\varepsilon])$.
We let the 2-cell $c$ be $\theta(\lambda_1 \times [0,\varepsilon]) \setminus N(\Gamma_1 \cup e_1 \cup e_2)$. Note
that adding $e_1$, $e_2$, and $c$ to $\cC_1$ can be achieved by
Moves~\ref{move:index0/1} and~\ref{move:index1/2}.
Writing $\hat{\cD}_2$ for the sutured cell decomposition of
$S$ induced by $\hat{\cC}_2$, we note that $\pi(\hat{\cD}_2)$ is $C^0$ close
to the one obtained from $\pi(\cD_1)$ by Move~\ref{submove:2}.

In the opposite direction, suppose that $\pi(\cD_2)$ is obtained from $\pi(\cD_1)$ by
removing a fattened 0-cell $B$ and a 1-cell $\lambda$, as in Move~\ref{submove:2}.
Then we split the 2-cell of $\cC_1$ corresponding to $\lambda_1 = (\pi|_{S_1})^{-1}(\lambda)$
along $e_1 := \theta(\lambda_1 \times \{\varepsilon\})$,
and cancel the 2-cell $\theta(\lambda_1 \times [0,\varepsilon]) \setminus N(\Gamma_1 \cup e_1)$ and the 1-cell
$\theta(\{p\} \times [0,\varepsilon])$ for $p = B \cap \Gamma_1$ using Move~\ref{move:index1/2}.

Having obtained $\hat{\cC}_2$ as above,
we can construct an isotopy of $M$ that is supported in a neighborhood of $\d M$,
and moves $\hat{\cC}_2$ to a contact cell decomposition $\cC_2''$ that
shares the same barrier surface and induced sutured cell
decomposition of $\d M$ as $\cC_2$. It follows that $\hat{\cC}_2$
and $\cC_2''$ are related by Move~\ref{move:isotopy}.
This completes Step~2.

\textbf{Step 3:} \emph{Suppose $\cC_1$ and $\cC_2$ are contact cell
decompositions with the same barrier surfaces $S$ and $S'$, the same contact
vector fields $v$ and $v'$, such that $\cC_1$ and $\cC_2$ induce the same
sutured cell decompositions $\cD$ and $\cD'$ of $S$ and $S'$, and such that
the Legendrian 1-skeletons of $\cC_1$ and $\cC_2$ intersect $S$ and $S'$ at
the same points. Then  $\cC_1$ and $\cC_2$ can be connected by Moves
\ref{move:isotopy}--\ref{move:index2/3}.}

The proof follows from the strategy
of \cite{HKMSutured}*{Theorem~1.2}, which we summarize using our present notation.
As in Step~2, we can assume that the graph $\Gamma$ is a union of Legendrians of the form
$\{p\}\times [a,1]$ or $\{p'\} \times [0,a]$ for various $p \in \g$, $p' \in \g'$, and $0<a<1$,
and the 2-cells are small perturbations of sets of the form $(\lambda \times I) \cap Z'$ for arcs
$\lambda \subset \d Z$.

Using Lemma~\ref{lem:subdivision}, we subdivide each 2-cell $D$ of $\cC_1$ and $\cC_2$
that intersects $\d Z \times I$ by adding a Legendrian arc obtained by perturbing
$D \cap (\d M \times \{1\} \cup \d M' \times \{0\})$ along $v$ relative to its endpoints.
Furthermore, using Move~\ref{move:index2/3}, we add a new convex 2-cell $D$
with Legendrian boundary and $\tb(\d D) = -1$ near $\d M \times \{1\}$ for
each 2-cell of the sutured cell decomposition $\cD$,
and near $\d M \times \{0\}$ for each 2-cell of $\cD'$.
Then we apply the subdivision procedure of Giroux~\cite{GirouxCorrespondence}
away from $\d Z \times I$.
\end{proof}

\subsection{Contact handle maps} \label{sec:contacthandlemap}

We first recall the definition of contact handles in dimension~3 due to Giroux~\cite{Giroux};
see also Ozbagci~\cite{Ozbagci}.

\begin{define} \label{def:contacthandle}
 For $k \in \{\,0,1,2,3\,\}$, a \emph{3-dimensional contact handle of index $k$}
attached to a sutured manifold $(M,\g)$ is a tight contact ball $(B_0,\xi_0)$
with convex boundary (possibly with corners)
that is attached via a map $\phi \colon S \to \d M$ for some subset $S \subset \d B_0$.
Furthermore, the dividing set of $\xi_0$ on $S$ is mapped into $\g$ under $\phi$,
and we have the following requirements, depending on the index:
\begin{enumerate}[leftmargin=20mm, label = \text{(Index \arabic*)}:]
\setcounter{enumi}{-1}
\item  $B_0 = D^3$ has no corners and $S = \emptyset$. The dividing set on $\d B_0$ is a single circle.
\item As a manifold with corners, $B_0$ is $I \times D^2$, and $S = \d I \times D^2$.
    The dividing set on $\d I \times D^2$ consists of one arc on each component.
    The dividing set on $I\times \d D^2$ consists of two arcs,
    each connecting the two components of $\d I \times \d D^2$.
\item As a manifold with corners, $B_0$ is $D^2 \times I$ and $S = \d D^2 \times I$.
    The dividing set is the same as on a contact 1-handle.
\item $B_0 = D^3$ with no corners, and $S = \d D^3$. The dividing set on $\d B_0$ is a single circle.
\end{enumerate}
\end{define}

In Figure~\ref{fig::20}, we have drawn the dividing sets on contact handles.
Note that, for handles of index~1 and~2, the curves $\d I \times \d D^2$ and $\d D^2 \times \d I$,
respectively, are not Legendrian.

\begin{figure}[ht!]
 \centering
 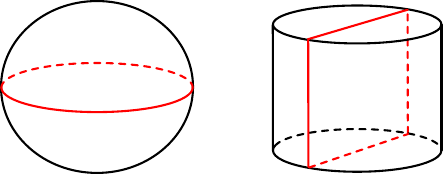
 \caption{Contact handles. On the left is a picture of a contact 0-handle or 3-handle.
 On the right is a contact 1-handle or 2-handle. \label{fig::20}}
\end{figure}

\subsubsection{Contact 0-handle map}
Adding a contact 0-handle $h^0$ amounts to adding a copy of the product sutured manifold
$(D^2\times [-1,1], S^1 \times \{0\})$ to $(M, \g)$. Noting that
\[
\SFH(-D^2\times [-1,1], -S^1 \times \{0\})\iso \bF_2,
\]
the contact 0-handle map is simply the tautological map
\[
\begin{split}
C_{h^0} \colon &\SFH(-M,-\g) \stackrel{\cong}{\longrightarrow}
\SFH(-M,-\g) \otimes \bF_2 \stackrel{\cong}{\longrightarrow} \\
&\SFH(-M,-\g) \otimes \SFH(-D^2\times [-1,1], -S^1 \times \{0\}) \stackrel{\cong}{\longrightarrow} \\
&\SFH((-M,-\g) \sqcup (-D^2\times [-1,1], -S^1 \times \{0\})).
\end{split}
\]
On the chain level, it is given as follows. Choose a diagram $\cH = (\bar{\S}, \as, \bs)$ of $(-M,-\g)$.
As $\cH_0 = (\bar{D}^2, \emptyset, \emptyset)$ is a diagram of $(-D^2 \times [-1,1], -S^1 \times \{0\})$,
the disjoint union $\cH \sqcup \cH_0$ is a diagram
of $(-M,-\g) \sqcup (-D^2\times [-1,1], -S^1 \times \{0\})$.
Let $i \colon \bar{\S} \to \bar{\S} \sqcup \bar{D}^2$ be the inclusion.
Then, for $\x = (x_1, \dots, x_d) \in \T_\a \cap \T_\b$,
we set $C_{h^0}(\x) = (i(x_1), \dots, i(x_d))$. This clearly induces $C_{h^0}$ on homology.

\subsubsection{Contact 1-handle map}
A contact 1-handle $h^1$ determines two points along the sutures. If
$(\S,\as,\bs)$ is a Heegaard surface for $(M,\g)$, glue a strip to $\d \S$
where the feet of the 1-handle are attached; see Figure~\ref{fig::1}. Add no
new $\as$ or $\bs$ curves. The map on complexes is the tautological map
induced by the inclusion of Heegaard surfaces. Let us call this map $C_{h^1}$.
This is an isomorphism of chain complexes, since the domains of the curves
counted by the boundary map on $\CF(\S,\as,\bs)$ have coefficient zero along $\d \S$,
where the strip is attached; see \cite[Lemma~9.13]{JDisks}.

\begin{figure}[ht!]
 \centering
 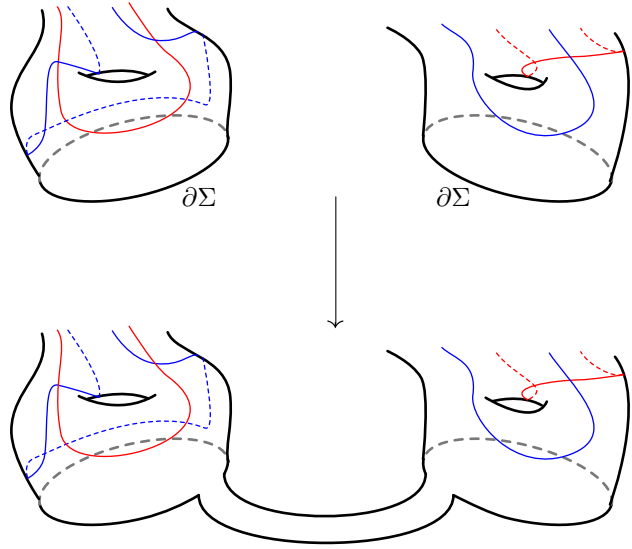
 \caption{A contact 1-handle, on the level of diagrams. \label{fig::1}}
\end{figure}

\begin{lem}
The map $C_{h^1}$ is natural; i.e., it commutes with change of diagrams maps.
\end{lem}

\begin{proof}
The proof is straightforward, since the change of diagrams maps are either
tautological (stabilization, isotopy) or involve counting holomorphic curves
that do not intersect  $\d \S \times \Sym^{k-1}(\S)\subset \Sym^{k}(\S)$
(continuation maps for changes of the almost complex structure, triangle maps
for changes of the $\as$ and $\bs$ curves).
\end{proof}

\subsubsection{Contact 2-handle map}
\label{sec:2-handle}

We now define a map for contact 2-handles. Let $(\S,\as,\bs)$ be an admissible diagram of $(M,\g)$.
Suppose that $l$ is the curve on $\d M$ along which we add a 2-handle $h^2$,
and $l$ intersects $\g$ (the sutures) exactly twice.
Let $p_1$, $p_2\in \g$ be the two points of intersection
and $l_\pm = l \cap R_\pm(\g)$. We denote the result of the 2-handle attachment by $(M',\g')$.

We now construct a diagram of $(M',\g')$.
Choose a sutured Morse function $f$ and gradient-like vector field $v$ on $(M,\g)$
that induce $(\S,\as,\bs)$,
and follow the flow of $v$ from $l_+$ and $l_-$ onto $\S$.
Writing $\lambda_+$ and $\lambda_-$ for the resulting arcs on $\S$,
we note that $\lambda_+$ avoids the $\bs$ curves and $\lambda_-$ avoids the $\as$ curves.
Hence, we can form a diagram $(\S',\as\cup \{\a'\},\bs\cup \{\b'\})$, as in Figure~\ref{fig::2}.
The surface $\S'$ is obtained by adding a band~$B$ to the boundary of $\S$ at $p_1$ and $p_2$.
The curve $\a'$ is obtained by concatenating the curve $\lambda_-$ with an arc in the band.
The curve $\b'$ is obtained by concatenating the curve $\lambda_+$ with an arc in the band.
We assume that $\a'$ and $\b'$ intersect in a single point in the band
(and possibly other places outside the band).
Furthermore, we assume that the intersection point of $\a'$ and $\b'$ has the configuration
shown in Figure~\ref{fig::2}; i.e., locally, it looks like there is a holomorphic disk on
$(\bar{\S}', \as \cup \{\a'\}, \bs \cup \{\b'\})$
going towards the intersection point that does not intersect $\d \S'$.

To see that $(\S', \as \cup \{\a'\}, \bs \cup \{\b'\})$ is indeed a diagram of $(M',\g')$, let
\[
J = [-1,1] \times \{0\} \subset D^2 = \{\, (x,y) \in \R^2 : x^2 + y^2 \le 1 \,\}.
\]
Consider the arc $a = J \times \{1/2\}$ in the 2-handle $D^2 \times I \subset M'$
connecting $p_1$ and $p_2$.
Furthermore, consider the neighborhood
\[
N(a) = (([-1,1] \times (-\varepsilon,\varepsilon)) \cap D^2) \times I
\]
of $a$, where $\varepsilon \in (0,1)$. Then $N(a) \cap R_\pm(\g')$ is a regular neighborhood of
$\g' \cap (D^2 \times I) = J \times \{0,1\}$,
and $N(a)$ is a product 1-handle. On the sutured diagram level,
attaching $N(a)$ corresponds to adding the band~$B = J \times I$.
Then $(D^2 \times I) \setminus N(a)$ is the disjoint
union of two 3-dimensional 2-handles, attached to $(M \cup N(a),\g')$ along
$l_\pm \cup s_\pm$, where $s_\pm$ is a properly embedded arc in $N(a) \cap R_\pm(\g')$.
These attaching curves correspond to $\a'$ and $\b'$ on the Heegaard surface $\S'$.
The diagram $(\S', \as \cup \{\a'\}, \bs \cup \{\b'\})$ is also admissible.

\begin{figure}[ht!]
 \centering
 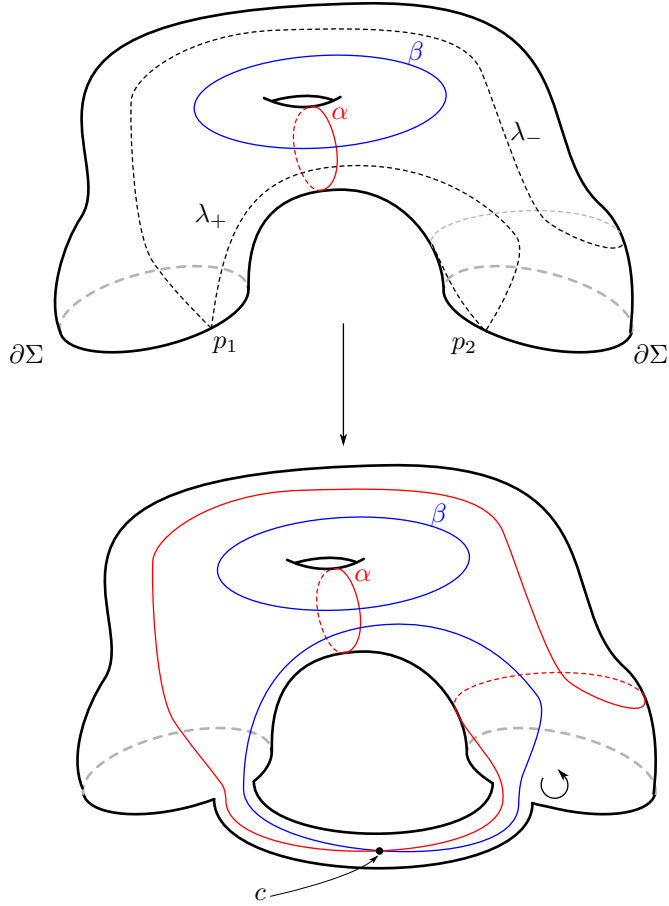
 \caption{The contact 2-handle map. The orientation of $\bar{\S}$ is shown on the bottom
  of the surface. \label{fig::2}}
\end{figure}

The contact 2-handle map
\[
C_{h^2} \colon \SFH(\bar{\S}, \as,\bs)\to \SFH(\bar{\S}', \as \cup \{\a'\},\bs \cup \{\b'\})
\]
is defined on $\ve{x} \in \T_{\as} \cap \T_{\bs}$ by the formula
\[
C_{h^2}(\ve{x}):= \ve{x} \times c,
\]
where $c \in \a' \cap \b'$ is the intersection point in the band.

\begin{lem}The following hold:
\begin{enumerate}
\item The map $C_{h^2}$ is a chain map.
\item The map $C_{h^2}$ is independent of the choices made in the construction;
    i.e., the choice of diagram $(\S, \as, \bs)$, and the choice of arcs $\lambda_+$ and $\lambda_-$
    obtained by projecting the two arcs $l_+$ and $l_-$ onto $\Sigma$.
\end{enumerate}
\end{lem}

\begin{proof}
The claim that $C_{h^2}$ is a chain map is straightforward. We wish to show
that $\d \circ C_{h^2} = C_{h^2} \circ \d$, and hence we need to check that
$\d (\ve{x} \times c) = \d \ve{x} \times c$. To this end, we note that any
disk counted by $\d (\ve{x}\times c)$ must have zero multiplicity around $c$,
since the boundary $\d \S$ is nearby. Hence, the homology class of the disk is equal to a
class on $(\bar{\S},\as,\bs)$ with the constant class at $c$ added.

We now show that $C_{h^2}$ is independent of the choices made in the
construction (i.e., commutes with the change of diagrams maps,
appropriately). We note that there are two sources of ambiguity: the curves
$\lambda_+$ and $\lambda_-$ in $\S$, and the diagram $(\S,\as,\bs)$.
Let us address the ambiguity of $\lambda_+$ and $\lambda_-$. Since they are
gotten by flowing curves in $\d M$  under the gradient-like vector field of a
Morse function, any two choices of $\lambda_+$ are related by handle slides
over $\bs$ curves (as well as isotopies of~$\lambda_+$, relative $\d \S$,
such that it never intersects any $\bs$ curves). Similarly, any two choices
of $\lambda_-$ are related by handle slides over the $\as$ curves and
isotopies within $\S \setminus \as$. To see that the change of diagrams
map commutes with $C_{h^2}$, one realizes the handle slide and isotopy maps as
triangle maps, and uses the local computation  at the end of the proof of
\cite{HKMSutured}*{Lemma~3.5}, which is illustrated by
\cite{HKMSutured}*{Figure~8}. For the reader's convenience, we have
reproduced the relevant picture in Figure~\ref{fig::50}, since we will use
the local computation later, as well.
\end{proof}

\begin{figure}[ht!]
 \centering
\begingroup%
  \makeatletter%
  \providecommand\color[2][]{%
    \errmessage{(Inkscape) Color is used for the text in Inkscape, but the package 'color.sty' is not loaded}%
    \renewcommand\color[2][]{}%
  }%
  \providecommand\transparent[1]{%
    \errmessage{(Inkscape) Transparency is used (non-zero) for the text in Inkscape, but the package 'transparent.sty' is not loaded}%
    \renewcommand\transparent[1]{}%
  }%
  \providecommand\rotatebox[2]{#2}%
  \newcommand*\fsize{\dimexpr\f@size pt\relax}%
  \newcommand*\lineheight[1]{\fontsize{\fsize}{#1\fsize}\selectfont}%
  \ifx\svgwidth\undefined%
    \setlength{\unitlength}{165.26455808bp}%
    \ifx\svgscale\undefined%
      \relax%
    \else%
      \setlength{\unitlength}{\unitlength * \real{\svgscale}}%
    \fi%
  \else%
    \setlength{\unitlength}{\svgwidth}%
  \fi%
  \global\let\svgwidth\undefined%
  \global\let\svgscale\undefined%
  \makeatother%
  \begin{picture}(1,0.80693687)%
    \lineheight{1}%
    \setlength\tabcolsep{0pt}%
    \put(0,0){\includegraphics[width=\unitlength,page=1]{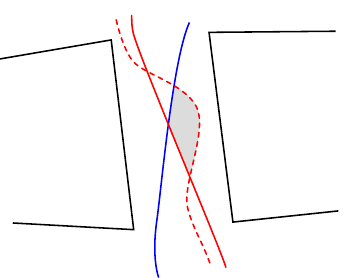}}%
    \put(0.46483398,0.43634732){\color[rgb]{0,0,0}\makebox(0,0)[rt]{\lineheight{1.25}\smash{\begin{tabular}[t]{r}$c$\end{tabular}}}}%
    \put(0.67910036,0.35690772){\color[rgb]{0,0,0}\makebox(0,0)[lt]{\lineheight{1.25}\smash{\begin{tabular}[t]{l}$\theta_{\alpha_2',\alpha_1'}^+$\end{tabular}}}}%
    \put(0.52235322,0.56112564){\color[rgb]{0,0,0}\makebox(0,0)[lt]{\lineheight{1.25}\smash{\begin{tabular}[t]{l}$c'$\end{tabular}}}}%
    \put(0.65244107,0.08137536){\color[rgb]{1,0,0}\makebox(0,0)[lt]{\lineheight{1.25}\smash{\begin{tabular}[t]{l}$\alpha'_1$\end{tabular}}}}%
    \put(0.57534209,0.05815231){\color[rgb]{1,0,0}\makebox(0,0)[rt]{\lineheight{1.25}\smash{\begin{tabular}[t]{r}$\alpha'_2$\end{tabular}}}}%
    \put(0.4579637,0.27019616){\color[rgb]{0,0,1}\makebox(0,0)[rt]{\lineheight{1.25}\smash{\begin{tabular}[t]{r}$\beta'$\end{tabular}}}}%
    \put(0.8096375,0.73860084){\color[rgb]{0,0,0}\makebox(0,0)[lt]{\lineheight{1.25}\smash{\begin{tabular}[t]{l}$\d \Sigma$\end{tabular}}}}%
    \put(0,0){\includegraphics[width=\unitlength,page=2]{fig50.pdf}}%
  \end{picture}%
\endgroup%

 \caption{The local computation required to show that the contact 2-handle map
 is independent of handle slides of the $\a'$ curve used in the definition of the contact 2-handle map.
 A local computation using the vertex multiplicities (identical to the one in \cite{HKMSutured}*{Lemma~3.5})
 shows that any holomorphic triangle with vertices at $c$ and $\theta_{\a_2',\a_1'}^+$ also has a vertex at $c'$,
 and has domain equal to the shaded region. The orientation of $\bar{\S}$ is shown.\label{fig::50}}
\end{figure}

\subsubsection{Contact 3-handle maps}\label{sec:3-handlemap}
We now define the contact 3-handle map.
Let $(M',\g')$ be the result of gluing a contact 3-handle $h^3$ to the balanced sutured
manifold $(M,\g)$, and suppose that $(M',\g')$ is also balanced.
The 3-handle is attached to a 2-sphere component $S \subset \d M$ such that
$\g_S = \g \cap S$ is a single closed curve.
One defines the contact 3-handle map $C_{h^3}$ to be equal to the composition of the 4-dimensional
3-handle map obtained by surgering $(-M, -\g)$ along the 2-sphere obtained by pushing $S$
into $\Int(M)$, followed by the inverse of the contact 0-handle map corresponding to
removing the resulting copy of $(-D^2 \times [-1,1], -S^1 \times \{0\})$.

On the diagram level, we choose a diagram $(\S, \as, \bs)$ of $(M,\g)$ such that there
are curves $\a \in \as$ and $\b \in \bs$ that are parallel to $\g_S \subset \partial \S$,
intersect transversely in two points $x$ and $y$, and there are no other $\as$ or $\bs$ curves
between $\a$ and $\g_S$ or $\b$ and $\g_S$. Let $\S'$ be the result of gluing a disk
to $\S$ along $\g_S$. Then $(\S', \as \setminus \{\a\}, \bs \setminus \{\b\})$
is a diagram of $(M',\g')$. Suppose that $x$ has larger relative grading than $y$ in $(\bar{\S}, \as, \bs)$.
For $\x \in \T_\a \cap \T_\b$, we set $C_{h^3}(\x \times \{x\}) = 0$
and $C_{h^3}(\x \times \{y\}) = \x$.

\subsection{Definition of the contact gluing map} \label{sec:def-gluingmap}

We now define the gluing map, in terms of a contact cell decomposition
(Definition~\ref{def:contactcelldecomposition}).  Suppose that $(M,\g)$
and $(M',\g')$ are balanced sutured manifolds, $(M,\g)$ is a
sutured submanifold of $(M',\g')$, and $\xi$ is contact structure on
$Z = M'\setminus \Int(M)$ that has dividing set $\g \sqcup \g'$. Let
$\cC$ be a contact cell decomposition of~$Z$. Let $\nu$ and
$\nu'$ be the chosen contact vector fields, and $S$ and $S'$ the chosen
barrier surfaces near $\d M$ and $\d M'$, respectively. Recall that we write
$N$ and $N'$ for the collar neighborhoods of $\d M$ and $\d M'$ bounded by
$S$ and $S'$, respectively.  Let $\Gamma\subset Z \setminus \Int(N \cup N')$ denote
the Legendrian 1-skeleton and $D_1, \dots, D_n$ the convex 2-cells.

A neighborhood of each vertex of $\Gamma$ that is not contained in $S$ or
$S'$ determines a contact 0-handle. A neighborhood of each edge of $\Gamma$
is a contact 1-handle. A neighborhood of each convex 2-cell $D_i$ is a
contact 2-handle. Finally, after removing neighborhoods of the graph $\Gamma$
and the disks $D_i$, we are left with a collection of tight contact 3-balls
that we view as a collection of contact 3-handles. Write $h_1,\dots, h_n$
for these handles, ordered such that their indices are nondecreasing.

Let $\g_0$ be the dividing set on $S$ with respect to $\nu$ and
$\xi$, and write $\g_0'$ for the dividing set on $S'$ with respect to
$\nu'$ and $\xi$. The flow of $\nu$ induces a diffeomorphism
\[
\psi^{\nu} \colon (M,\g)\to (M\cup N,\g_0)
\]
that is well defined up to isotopy.
Note that $(N', \g_0' \cup \g')$ is a balanced sutured
manifold, and $\xi$ induces the dividing set $\g_0' \cup \g'$ on
$\d N'$. There is a canonical isomorphism
\[
\SFH(-M \sqcup -N', -\g \cup -\g_0' \cup -\g') \iso
\SFH(-M, -\g) \otimes \SFH(-N', -\g_0' \cup -\g').
\]
Hence, we can define a map
\[
\Phi_{\sqcup N'} \colon \SFH(-M, -\g) \to
\SFH(-M \sqcup -N', -\g \cup -\g_0' \cup -\g')
\]
via the formula
\[
\Phi_{\sqcup N'}(\ve{x}) = \ve{x} \otimes \EH(N', \g_0' \cup \g', \xi|_{N'}),
\]
where $\EH(N', \g_0' \cup \g',\xi|_{N'})\in \SFH(-N',-\g_0'\cup -\g')$ is the
contact invariant of $\xi|_{N'}$, constructed using a partial open book
decomposition, as defined by Honda, Kazez, and Mati\'c~\cite{HKMSutured}.
We describe the invariant
$\EH(N', \g_0' \cup \g', \xi|_{N'})$ in more detail in
Section~\ref{sec:partialopenbooks}, when we prove some properties of the gluing map.
We will show in Lemma~\ref{lem:EH} that $\Phi_{\sqcup N'}$ can be written
as a composition of contact handle maps.

We now define the contact gluing map as the composition
\begin{equation}
\Phi_{\xi,\cC} := C_{h_n} \circ \cdots \circ C_{h_1} \circ
(\psi_*^{\nu} \otimes \id)\circ \Phi_{\sqcup N'}, \label{eq:defgluingmap}
\end{equation}
where $C_{h_i}$ is the map induced by the contact handle $h_i$,
as in Section~\ref{sec:contacthandlemap}.

\subsection{Invariance of the contact gluing map}

In this section, we prove the following:

\begin{thm}\label{thm:invarianceofcontacthandlemap}
The contact gluing map $\Phi_{\xi,\cC}$ is independent of the contact cell decomposition $\cC$.
\end{thm}

As a first step, the following lemma is helpful:

\begin{lem}\label{lem:disjointhandlescommute}
If $h_1$ and $h_2$ are two contact handles (of arbitrary index)
that are disjoint and are attached to $(M,\g)$, then
\[
C_{h_1}\circ C_{h_2}=C_{h_2}\circ C_{h_1}.
\]
\end{lem}

\begin{proof}
The proof follows by analyzing the formulas for the two maps. If one of $h_1$
and $h_2$ is a 0-handle or a 1-handle, then the statement is obvious. Let us
consider the case when $h_1$ and $h_2$ are disjoint 2-handles. Let us recall
how the maps $C_{h_i}$ are defined, for $i \in \{1 ,2\}$. We first attach a band
$B_i$ to the boundary of a sutured diagram for $(-M,-\g)$. The
band $B_i$ is attached where the attaching circle of $h_i$ intersects
$\g \subset \d M$. We pick curves $\a_i$ and $\b_i$, according to
the attaching circle of $h_i$, that intersect at a single point
$c_i$ in $B_i$. The map $C_{h_i}$ is defined by $C_{h_i}(\ve{x}) = \ve{x}\times c_i$.
Since the handles $h_1$ and $h_2$ are disjoint, the bands $B_1$ and
$B_2$ can be chosen to be disjoint, and the curves $\a_i$ and $\b_i$
can be assumed to not intersect the band $B_j$ when $i \neq j$. Hence, it
follows that we can use the curves $\a_1$, $\b_1$, $\a_2$, and $\b_2$
to compute both compositions $C_{h_2} \circ C_{h_1}$ and $C_{h_2} \circ C_{h_1}$,
and since the formulas for the two compositions clearly agree, we
conclude that $C_{h_2} \circ C_{h_1} = C_{h_1}\circ C_{h_2}$. In a similar
manner, one can show that the same formula holds if at least one of $h_1$
and $h_2$ is a 3-handle.
\end{proof}

\begin{proof}[Proof of Theorem~\ref{thm:invarianceofcontacthandlemap}]
Independence of the relative ordering of cells of the same index follows
from Lemma~\ref{lem:disjointhandlescommute}, so it is sufficient to check
invariance under the moves in Proposition~\ref{prop:movesbetweencelldecompositions}.

First consider the case when $\cC_1$ and $\cC_2$ differ by Move~\ref{move:isotopy} (isotopy).
The maps $\Phi_{\xi,\cC_1}$ and $\Phi_{\xi,\cC_2}$ differ by post-composition
with a map $\phi_* \colon \SFH(M',\g')\to \SFH(M',\g')$
for a diffeomorphism $\phi \colon M'\to M'$ that is isotopic to the identity
relative to $\d M'$. By naturality of sutured Floer homology~\cite{JTNaturality}*{Theorem~1.9},
the map $\phi_*$ is the identity.

We now consider Move~\ref{move:index0/1} (index 0/1 cell cancellation). We
need to check that subdividing an edge of $\Gamma$, or adding a Legendrian edge
$\lambda$ to $\Gamma$ that meets $\Gamma$ at a single vertex and intersects none
of the other cells does not change the map. Let us first consider
subdivision. The contact 0-handle map adds a disk to the Heegaard surface
with no new $\a$ or $\b$ curves, and the contact 1-handle map
attaches a band to the boundary of the Heegaard surface with no new $\a$
or $\b$ curves. When we subdivide a Legendrian edge into two Legendrian
edges that meet at a single vertex, the contact handle map changes by first
adding a disk to $S$ (the 0-handle map), followed by two bands, each of which
has one foot on the new disk, and another foot at a foot of the original
band. Clearly, the induced maps agree. Invariance under adding a Legendrian
edge $\lambda$ that intersects $\Gamma$ at a single vertex follows similarly.

We now consider the case when $\cC_1$ and $\cC_2$ differ by Move~\ref{move:index1/2}
 (index 1/2 contact cell cancellation). The proof is a model computation, that we now describe.
The computation is summarized in Figure~\ref{fig::40}. Note that the contact cell decompositions
$\cC_1$ and $\cC_2$ have the same barrier surfaces $S$ and $S'$ near $\d M$ and $\d M'$, respectively.
Recall that we write $N$, $N'\subset Z$ for the neighborhoods
of $\d M$ and $\d M'$ that are bounded by $S$ and $S'$, and $Z' := Z \setminus \Int(N \cup N')$.
By definition of an index 1/2 contact cell cancellation, the graph $\Gamma_2$ is formed by adding a Legendrian edge $\lambda$ to $\Gamma_1$, and a new 2-cell $D$ is added that intersects $\d (Z'\setminus \Int N(\Gamma_1))$ along an arc that
does not intersect the dividing set.

For the sake of demonstration, let us
assume that the new 2-cell $D$ intersects $\d (Z'\setminus \Int N(\Gamma_1))$ along
$R_+$. If $(\S,\as,\bs)$ is a sutured diagram for $M \cup N \cup
N(\Gamma_1) \cup N'$, the effect of adding the 1-handle $h^1$ corresponding to $\lambda$ is
to attach a band $B$ to the boundary of $\bar{\S}$. The attaching circle of the
2-handle $h^2$ corresponding to $D$ intersects the dividing set of $\d N(\lambda)$ at two points. On the
level of diagrams, the effect is to attach a band $B'$ to $B$, and add new
curves $\a_0$ and $\b_0$, as in Figure~\ref{fig::40}. Note that the
effect of adding $B$ and $B'$ to $\bar{\S}$ is to attach a tube to two points
along the interior of $\bar{\S}$ near $\d \bar{\S}$, and the curve $\b_0$ is a meridian of the
tube, while the curve $\a_0$ is the concatenation of a longitude of the tube
with a path on $\bar{\S}$ between the two ends of the tube. The curves
$\a_0$ and $\b_0$ intersect at a single point $c$, and the composition
of the contact 1-handle and 2-handle maps is
\[
(C_{h^2}\circ
C_{h^1})(\ve{x})=\ve{x}\times c.
\]
This is the compound stabilization map from Section~\ref{sec:compoundstabilization}. By
Proposition~\ref{prop:compoundstabilizationmap}, this is equal to the transition map from
the naturality of sutured Floer homology.

\begin{figure}[ht!]
 \centering
 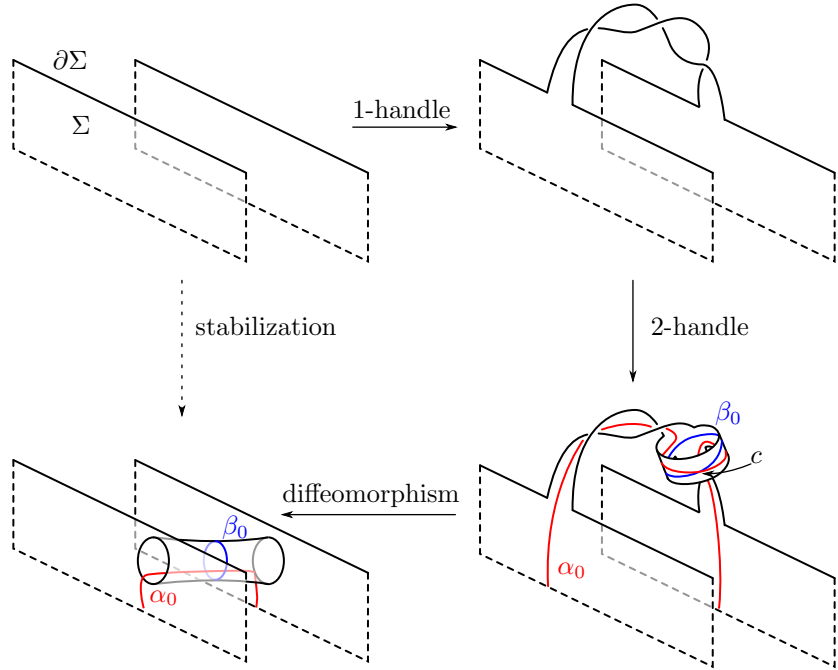
 \caption{The composition of the contact 1-handle map followed by a contact 2-handle map,
 for a canceling pair of contact 1- and 2-cells. The composition is the compound stabilization map
 described in Section~\ref{sec:compoundstabilization}.\label{fig::40}}
\end{figure}

Finally, we consider invariance under Move~\ref{move:index2/3} (index 2/3
contact cell cancellation). In this move, we
add a convex 2-cell $D$ to the decomposition that has Legendrian boundary
and $\tb(\d D) = -1$, and such that $\d D$ intersects the dividing set on
$\d (Z'\setminus \Int N(\Gamma))$ exactly twice. By definition of a contact cell
decomposition, the disk $D$ cuts one of the contact 3-cells in our
decomposition into two contact 3-cells. Hence, the effect on the contact
gluing map from equation~\eqref{eq:defgluingmap} is to insert a
contact 2-handle map $C_{h^2}$, followed by a contact 3-handle map $C_{h^3}$.
The composition is easily seen to be a diffeomorphism map,
as demonstrated in Figure~\ref{fig::41}.

Indeed, for notational simplicity, assume that $h^2$ is attached to $(M,\g)$.
Let $(M',\g')$ be the result of attaching $h^2$ to $(M,\g)$, and $(M'',\g'')$ the result
of attaching $h^3$ to $(M',\g')$. Given a diagram $(\S,\as,\bs)$ of $(M,\g)$,
we get a diagram of $(M',\g')$ by attaching a band $B$ to a component of $\partial \S$,
and add curves $\a'$ to $\as$ and $\b'$ to $\bs$ parallel to the suture $\g_S$ along which
$h^3$ is attached, and such that $\a' \cap \b' \cap B = \{c\}$.
For $\x \in \T_\a \cap \T_\b$, we have $C_{h^2}(\x)= \x \times c$.
 We can choose $\a'$ and $\b'$ so that $|\a' \cap \b'| = 2$,
 and  we write $\a'\cap \b'=\{\theta^+,\theta^-\}$,
where $\theta^+$ and $\theta^-$ are distinguished by
the relative Maslov grading.
By construction, on $(\bar{\S}, \as \cup \{\a'\}, \bs \cup \{\b'\})$
there are two bigons connecting the two points of $\a' \cap \b'$, which we can use to
determine that $c$ has lower
relative grading, so $c=\theta^-$. It follows that $(C_{h^3} \circ C_{h^2})(\x) = C_{h^3}(\x \times c) = \x$.
\end{proof}

\begin{figure}[ht!]
 \centering
 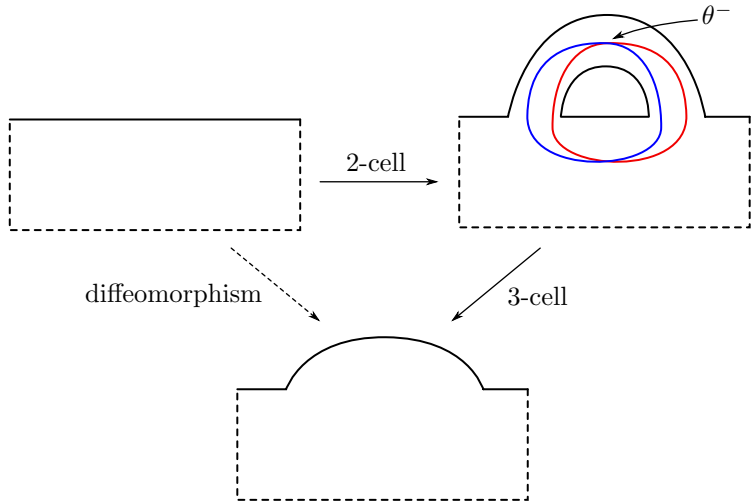
 \caption{The composition of a contact 2-handle map followed by a contact 3-handle map, for a canceling pair of contact 2- and 3-cells \label{fig::41}}
\end{figure}

\begin{rem}
It is possible to directly show invariance of the contact gluing map under
subdividing a contact 2-cell into two contact 2-cells that meet along a Legendrian
arc. Indeed, the topological manipulation described in the proof of Lemma~\ref{lem:subdivision}
gives a recipe for doing so. Nonetheless, the model computations required to show invariance
of the gluing map under Moves~\ref{move:isotopy}--\ref{move:index2/3} are simpler.
\end{rem}

\section{Contact handles and partial open book decompositions}
\label{sec:partialopenbooks}

\subsection{Partial open book decompositions and sutured Heegaard diagrams}
\label{sec:partialopenbookandHD}

In order to prove basic properties of the gluing map, and to compare our
construction to the one due to Honda, Kazez, and Mati\'c~\cite{HKMTQFT},
we need the following definition from~\cite{HKMSutured}*{Section~1.2}:

\begin{define}
A \emph{partial open book decomposition} is a triple $(S,P,h)$ consisting
of a compact, connected, oriented surface $S$ with non-empty boundary, a
compact subsurface $P\subset S$ (the
\emph{page}), such that $S$ is obtained from $\cl(S \setminus P)$ by successively attaching
1-handles, and a smooth embedding $h \colon P\to S$ (the \emph{monodromy})
such that $h|_{\d S\cap P}=\id$.
\end{define}

An abstract (non-embedded) partial open book decomposition $(S,P,h)$ defines a sutured manifold
$(M,\g)$ via the formula
\[
M = \left( S\times [0,\tfrac{1}{2}] \cup P \times [\tfrac{1}{2},1] \right)/{\sim}_h,
\]
where $\sim_h$ is the equivalence relation defined as
\[
\begin{split}
&(x,t)\sim_h (x,t') \text{ if } x \in \d S \text{ and } t \text{, } t' \in [0,\tfrac{1}{2}],\\
& (x,t)\sim_h (x,t') \text{ if } x \in \d P\cap \d S \text{ and } t \text{, } t'\in [\tfrac{1}{2},1], \text{ and }\\
&(x,1)\sim_h (h(x),0) \text{ if } x \in P.
\end{split}
\]

The  manifold $M$ contains the properly embedded surface
\[
\S := \left( S\times \{\tfrac{1}{4}\} \right) \cup \left( P\times \{\tfrac{3}{4}\} \right).
\]
The curves $\g=\d M\cap \S$ divide $\d M$ into two subsurfaces that meet along $\g$.
Furthermore, $(M,\g)$ is a balanced sutured manifold and $\S$ is a
sutured Heegaard surface for $(M,\g)$. Using our orientation
conventions, $R_+(M)$ deformation retracts onto $(S\setminus P)\times
\{\tfrac{1}{2}\}$ and $R_-(M)$ deformation retracts onto $(S\setminus
h(P))\times \{0\}$.

\begin{define}
If $(S,P,h)$ is a partial open book decomposition, a \emph{basis of arcs
for $P$ in $S$} is a collection of pairwise disjoint, properly embedded arcs $\ve{a}=\{a_1,\dots, a_k\}$ on $P$
with ends on $P \cap \d S$ such that $P \setminus (a_1\cup \cdots \cup a_k)$
deformation retracts onto $\d P\setminus \d S$ (or, equivalently,
$S \setminus (a_1\cup \cdots \cup a_k)$ deformation retracts onto $\cl(S \setminus P)$).
\end{define}

Given a basis of arcs for $P$ in $S$, we can construct attaching curves $\as$ and
$\bs$ on $\S$ that make $(\S,\as,\bs)$ a Heegaard diagram for
$(M,\g)$. Let $b_i$ be an isotopic copy of $a_i$ obtained by moving the ends
of $a_i$ along $\d \S$ in the positive direction such that $a_i \cap b_j = \delta_{ij}$.
We then set
\[
\a_i = \left(a_i \times \{\tfrac14\} \right) \cup \left(a_i \times \{\tfrac34\}\right)
\text{ and } \b_i = \left(b_i \times \{\tfrac34\} \right) \cup \left( h(b_i) \times \{\tfrac14\} \right).
\]

A partial open book decomposition $(S,P,h)$ determines a unique contact structure $\xi$ on $(M,\g)$,
up to equivalence, as follows; see \cite{OzbagciEtguContactClass}*{Proposition~1.2}.  Let
\begin{equation}\label{eq:U1U2contacthandlebodies}
U_1 := \left( S \times [0,\tfrac{1}{2}] \right)/{\sim_1} \text{ and }
U_2:= \left( P \times [\tfrac{1}{2},1] \right)/{\sim_2},
\end{equation}
where $\sim_1$ and $\sim_2$ are  the relations defined by
\[
(x,t)\sim_1 (x,t') \text{ if } x \in \d S, \text{ and }
(x,t)\sim_2(x,t') \text{ if } x \in \d P \cap \d S.
\]
Then $\xi|_{U_1}$ is the unique tight contact structure on $U_1$
with dividing set $\d S \times \{\tfrac14\}$, and $\xi|_{U_2}$
is the unique tight contact structure on $U_2$ with dividing set $\d P \times \{\tfrac34\}$.
Hence, we say that $(S,P,h)$ is a partial open book decomposition of
the contact 3-manifold $(M,\g,\xi)$ if we are given a contactomorphism
between the contact 3-manifold defined by $(S,P,h)$ and $(M,\g,\xi)$.

Given a partial open book decomposition $(S,P,h)$ of $(M,\g,\xi)$
and a basis of arcs $\{a_1, \dots, a_k\}$, let $x_i = (a_i \times \{\tfrac34\}) \cap (b_i \times \{\tfrac34\})$
for $i \in \{1, \dots, k\}$ and $\x_\xi = x_1 \times \dots \times x_k$.
Then Honda, Kazez, and Mati\'c~\cite{HKMSutured} showed that
$\x_\xi$ is a cycle whose homology class $\EH(M,\g,\xi) \in \SFH(-M,-\g)$
is independent of the choice of partial open book and basis of arcs,
and is hence an invariant of~$\xi$.

\subsection{Partial open books and contact handles}
\label{sec:partialopenbookdecompositionsandhandledecomps}

Contact handle decompositions were defined by Giroux~\cite{Giroux}.
In this section, we describe some useful relations between contact handle decompositions
and the Honda--Kazez--Mati\'{c} definition of a partial open book; compare~\cite{OzbagciEtguContactClass}.

Given a partial open book decomposition $(S,P,h)$ for the contact sutured manifold $(M, \g,\xi)$,
we can naturally construct a contact handle decomposition of $(M, \g)$ with no 3-handles, as follows.
Recall that the manifold $M$ is obtained by gluing the contact handlebodies $U_1$ and $U_2$ together, as described in Section~\ref{sec:partialopenbookandHD}, along a portion of their boundaries, using the map $h$.

We start by constructing a contact handle decomposition of $U_1$ consisting
of only 0-handles and 1-handles. Such a description is obtained by giving the
surface $S$ a decomposition into 2-dimensional 0-handles and 1-handles.
Next, we extend this decomposition to a contact handle decomposition of all
of $M$. To do this, pick a basis of arcs $\ve{a}$ for $P$ in $S$. The closed
curves
\[
l_i := \left(a_i\times \{\tfrac{1}{2}\}\right) \cup \left( h(a_i) \times \{0\} \right)
\]
bound disks in $U_2$ that intersect the dividing set $\d S \times \{0\}$
of $\d U_1$ exactly twice. By perturbing $U_1$ and $U_2$ slightly inside
$M$, we can assume that the curves $l_i$ are Legendrian. As the curve
$l_i$ intersects the dividing set of $\d U_1$  exactly twice, it follows that
they bound convex disks with $\tb = -1$, and that neighborhoods of
these convex disks are contact 2-handles. Furthermore, after attaching these
contact 2-handles, we obtain the sutured manifold $M$.

In the opposite direction, given a contact handle decomposition $\ve{H}$ of $(M,\g,\xi)$
with handles ordered with nondecreasing index and no 3-handles,
viewed as a cobordism from $\emptyset$ to $\d M$, we can construct a partial open book
decomposition $(S,P,h)$, as follows. The handlebody $U_1$ is the union of the 0- and 1-handles.
If $\g_1$ is the dividing set of $\xi$ on $\d U_1$, then $(U_1, \g_1)$ is diffeomorphic to the product
sutured manifold $((S \times [0,\tfrac12])/{\sim}, \d S \times \{\tfrac14\})$ where $S:=  R_+(\g_1)\subset \d U_1$.
Let $U_2$ be the union of the 2-handles, and let $\g_2$ be the dividing set of $\xi$ on $\d U_2$.
Then $(U_2, \g_2)$ is a product sutured manifold of the form $((P \times [\tfrac12,1])/{\sim}, \d P \times \{\tfrac34\})$
for $P := U_2 \cap S$. We finally set $h$ to be $\pi_1 \circ i_2 \colon P \to S$,
where $\pi_1 \colon S \times [0,\tfrac12] \to S$ is the projection and
\[
i_2 \colon P \to P \times \{1\} \subset R_-(\g_1) \approx S \times \{0\}
\]
is the canonical embedding.

\begin{lem} \label{lem:EH}
  Let $h_1, \dots, h_n$ be the handles of a contact handle decomposition $\ve{H}$ of $(M,\g,\xi)$,
  ordered such that their indices are nonincreasing. If there are no 3-handles, then
  \[
  \EH(M,\g,\xi) = (C_{h_1} \circ \dots \circ C_{h_n})(1) \in \SFH(-M,-\g),
  \]
  where $1 \in \bF_2 \cong \SFH(\emptyset)$.
\end{lem}

\begin{proof}
  Let $(S,P,h)$ be the partial open book corresponding to $\ve{H}$, as above.
  Suppose that $h_1, \dots, h_k$ are the 2-handles.
  On the level of diagrams, $C_{h_{k+1}} \circ \cdots \circ C_{h_n}$ corresponds to adding
  a disk for each 0-handle, and a band for each 1-handle. The union of these is $S \times \{\tfrac14\}$.
  Adding a 2-handle $h_i$ for $i \in \{1, \dots, k\}$ corresponds to attaching a band $B_i$ to $S$,
  and adding curves $\a_i$ and $\b_i$. Then $B_1 \cup \dots \cup B_k = P \times \{\tfrac34\}$,
  and $\a_i$ and $\b_i$ are obtained from the basis of arcs $\{a_1, \dots, a_k\}$
  as described in Section~\ref{sec:partialopenbookandHD}, where $a_i$ is the core of $B_i$.
  Let $x_i = \a_i \cap \b_i \cap B_i$ for $i \in \{1,\dots,k\}$.
  Then the element $(C_{h_1} \circ \dots \circ C_{h_n})(1) = x_1 \times \dots \times x_k$ tautologically agrees
  with the cycle $\x_\xi$ representing $\EH(M,\g,\xi)$,
  as defined by Honda, Kazez, and Mati\'c~\cite{HKMSutured} for the partial open book decomposition $(S,P,h)$
  and the basis of arcs $\{a_1, \dots, a_k\}$.
\end{proof}

We now describe the effect of attaching a single contact handle on the level of partial open books.
We begin with the effect of attaching a contact 1-handle:

\begin{lem}\label{lem:partialopenbookcont1-handle}
Suppose that we obtain the contact sutured manifold $(M',\g',\xi')$
from $(M,\g,\xi)$ by attaching a contact 1-handle $h^1$.
Let $(S,P,h)$ be a partial open book decomposition for the contact
structure $\xi$ on $(M,\g)$. A partial open book decomposition
$(S',P',h')$ for the contact structure $\xi'$ on $(M', \g')$ can be obtained
by setting
\begin{enumerate}
\item $S' = S \cup B$,
  where $B$ is a band attached to $\d S$,
\item $P' = P$, and
\item $h'=\iota_S\circ h$, where $\iota_S \colon S \to S'$ is the embedding.
\end{enumerate}
\end{lem}

\begin{proof}
First, we isotope the partial open book $(S,P,h)$ in $M$ such that the attaching
sphere of $h^1$ lies in $\d S \setminus P$. We set $S' = S \cup B$,
where the band $B\subset h^1$ has boundary equal to the dividing set on $h^1$. Then
\[
\left(S' \times [0,\tfrac12] \cup P \times [\tfrac12, 1]\right)/{\sim}_{h'} =
\left(S \times [0,\tfrac12] \cup P \times [\tfrac12, 1]\right)/{\sim}_{h'} \cup
\left(B \times [0,\tfrac12]\right)/{\sim}_{h'}.
\]
Furthermore,
\[
\left(S \times [0,\tfrac12] \cup P \times [\tfrac12, 1]\right)/{\sim}_{h'} =
\left(S \times [0,\tfrac12] \cup P \times [\tfrac12, 1]\right)/{\sim}_h = M
\]
and $\left(B \times [0,\tfrac12]\right)/{\sim}_{h'} = h^1$.
\end{proof}

We now consider the effect of attaching a contact 2-handle $h^2$ to $(M,\g)$.
Let $\xi'$ denote the contact structure on $M \cup h^2$ with dividing set $\g'$,
obtained by gluing the tight contact structure $\xi_2$ on $h^2$ to $\xi$. Let $p_1$, $p_2 \in \g$
denote the two points of intersection of the attaching circle of $h^2$ with $\g$.
Furthermore, the attaching circle of $h^2$ consists of a path
$l_+$ in $R_+(\g)$ from $p_1$ to $p_2$, concatenated with the reverse of a path
$l_-$ in $R_-(\g)$ from $p_1$ to $p_2$. We can isotope the partial open
book decomposition $(S,P,h)$ such that $p_1$ and $p_2$ lie in $\d S \setminus P$.
Since we have identifications
\[
S \setminus P \iso R_+(M) \text{ and } S \setminus h(P) \iso R_-(M),
\]
we can view $l_+$ as a path $\lambda_+$ in
$S \setminus P$, and $l_-$ as a path $\lambda_-$ in $S \setminus h(P)$. Using the
above notation, we are now prepared to describe a partial open book
decomposition for the contact structure $\xi'$ on $(M \cup h^2, \g')$.

\begin{lem}\label{lem:partialopenbookcont2-handle}
Suppose that $h^2$ is a contact 2-handle attached to $(M, \g, \xi)$,
and let $\xi'$ be the resulting contact structure on $(M \cup h^2, \g')$. Given a
partial open book decomposition $(S,P,h)$ for $\xi$ on $(M, \g)$, a partial open
book decomposition for $\xi'$ on $(M \cup h^2, \g')$ is given by $(S',P',h')$, where
\begin{enumerate}
\item $S' = S$,
\item $P' = P \cup N(\lambda_+)$,
\item $h'|_P = h$, and $h'$ maps $N(\lambda_+)$ to $N(\lambda_-)$.
\end{enumerate}
\end{lem}

\begin{proof}
Using equation~\eqref{eq:U1U2contacthandlebodies}, we write $M$ as the union
of the subsets $U_1 = \left(S \times [0,\tfrac{1}{2}]\right)/{\sim}_1$ and
$U_2 = \left(P \times [\tfrac{1}{2},1]\right)/{\sim_2}$. We note that
\[
\left(P' \times [\tfrac{1}{2}, 1]\right)/{\sim_2} = \left(P \times [\tfrac{1}{2},1]\right)/{\sim_2}
\cup \left(N(\lambda_+) \times [\tfrac{1}{2},1]\right)/{\sim_2}.
\]
Let $U_2' = \left(P' \times [\tfrac{1}{2}, 1] \right)/{\sim_2}$ and
$h^2 = \left(N(\lambda_+) \times [\tfrac{1}{2}, 1] \right)/{\sim_2}$.
Then $(U_1 \cup U_2')/{\sim_{h'}}$ is obtained
from $(U_1 \cup U_2)/{\sim_{h}}$ by attaching $h^2$ along a neighborhood of the
circle $\left(\lambda_+ \times \{\tfrac{1}{2}\}\right) \cup \left(\lambda_- \times \{0\}\right)$. By
definition of a partial open book decomposition, it follows that $(S',P',h')$
is a partial open book for~$\xi'$ on $(M \cup h^2, \g')$.
\end{proof}

\subsection{Positive stabilizations and contact handle cancellations}

Honda, Kazez, and Mati\'{c} \cite{HKMSutured} extended the notion of positive
stabilizations to partial open book decompositions,
adapting Giroux's construction~\cite{GirouxCorrespondence}
for open books of closed manifolds. In this section, as an instructive example,
we show how to interpret their construction in terms of canceling pairs of contact handles.

We begin with Honda, Kazez, and Mati\'{c}'s definition of a positive stabilization of a partial open book:

\begin{define}
Suppose that $(S,P,h)$ is a partial open book decomposition of the contact 3-manifold $(M, \g, \xi)$,
and suppose that $c$ is a properly embedded arc on $S$. The arc $c$ is allowed to
intersect $P$; however, we require $\d c \subset \d S\setminus P$.
The \emph{positive stabilization} $(S',P',h')$ of $(S,P,h)$
along the arc $c$ is defined as follows. Let $S' := S \cup B$, where $B$ is a band
that we attach along $\d c \subset \d S$. Furthermore, let $P' := P \cup B$ be the new page.
Let $\tau \subset S'$ be the curve obtained by concatenating the arc $c$ with a core of $B$.
The new partial monodromy map $h' \colon P' \to S'$ is defined as
\[
h' := R_\tau \circ (h \cup \id_B),
\]
where $R_{\tau}$ is a right-handed Dehn twist along $\tau$,
with respect to the orientation of $S$.
\end{define}

We now wish to relate positive stabilizations to canceling pairs of contact
handles. As described in Section~\ref{sec:partialopenbookdecompositionsandhandledecomps},
the partial open book decomposition $(S,P,h)$, together with a choice of handle decomposition
of the surface $S$ into 0-handles and 1-handles, as well as a basis of arcs $\ve{a}$
for $P$, determine a contact handle decomposition of $(M, \g)$. We now show
that the contact handle decomposition arising from a positive stabilization
$(S',P',h')$ of $(S,P,h)$ can be obtained from a contact handle decomposition
arising from $(S,P,h)$ by inserting a pair of canceling index 1 and 2 contact handles.

\begin{lem} \label{lem:stab}
Suppose $(S,P,h)$ is a partial open book decomposition of $(M, \g, \xi)$, and that
$(S',P',h')$ is a positive stabilization of $(S,P,h)$ along a properly
embedded arc $c \subset S$. Let $U_1$ and $U_2$ be the tight contact
handlebodies defined in equation~\eqref{eq:U1U2contacthandlebodies}
associated to $(S,P,h)$, whose union is $M$. Consider a contact handle decomposition
$\ve{H}$ of $(M,\g,\xi)$ arising from the partial open book $(S,P,h)$ as in
Section~\ref{sec:partialopenbookdecompositionsandhandledecomps}.
Then a handle decomposition $\ve{H}'$ arising from the positive stabilization $(S',P',h')$
can be obtained from $\ve{H}$ by adding a pair of
canceling index 1 and 2 contact handles between $U_1$ and $U_2$.
\end{lem}

\begin{proof}
We attach a pair of canceling
contact 1- and 2-handles to $U_1$. Let $h^1$ denote a 1-handle,
attached with feet along $\d c\times \{\tfrac{1}{4}\}$.
We then attach a contact 2-handle $h^2$ with attaching
circle equal to the concatenation of $c\times \{0\}\subset \d U_1$,
together with a longitude of the 1-handle $h^1$; see the middle
row of Figure~\ref{fig::38}. The contact manifolds $U_1$ and
$U_1\cup h^1\cup h^2$ are equivalent, since $h^1$ and $h^2$
are canceling contact handles.

Note that $(S,\emptyset,\emptyset)$ is a partial
open book decomposition for $U_1$. By Lemma~\ref{lem:partialopenbookcont1-handle},
if $B$ denotes a band attached at $\d c\subset \d S$,
then the surface $(S\cup B,\emptyset,\emptyset)$ is
a partial open book decomposition for $U_1\cup h^1$.

Let $h^2_1,\dots, h^2_n$ denote the contact 2-handles
obtained by picking an arc basis of $P$. To build $M$
from $U_1\cup h^1$, we first attach the contact 2-handle $h^2$,
followed by $h^2_1, \dots, h^2_n$.
We can determine the effect of this on the partial open book
using Lemma~\ref{lem:partialopenbookcont2-handle}.
The partial monodromy map after attaching~$h^2$ (but before attaching the other 2-handles)
is determined by the attaching circle of the 2-handle~$h^2$.
Upon examining Figure~\ref{fig::38}, the partial monodromy map
is the composition of the inclusion of $B$ into $S$, followed by
a right-handed Dehn twist along the curve $\tau$ obtained by concatenating $c$ with the core of $B$.

We now need to compute the effect on the partial open book of
attaching $h_1^2,\dots, h^2_n$. Since the contact 2-handles
$h_1^2,\dots,h_n^2$ were obtained by picking a basis of arcs
of $P$, it follows that the attaching circle of the 2-handle
$h_i^2$ is equal to the concatenation of properly embedded arcs
$\lambda_{+,i}\times \{\tfrac{1}{2}\}\subset S\times \{\tfrac{1}{2}\}$
and $\lambda_{-,i}\times \{0\}\subset S\times \{0\}$,
where $\lambda_{+,i}$, $\lambda_{-,i} \subset S$.

 The monodromy map $h$ is determined up to isotopy relative
 to $\d P\cap \d S$ by requiring $\lambda_{+,i}$ to be sent to
 $\lambda_{-,i}$ by $h$.  It is clear that, after attaching
 $h^1$ and $h^2$, the attaching arcs $\lambda_{-,i}$ must be
 modified if they intersect the arc $c$ on $S$, since now they
 are attached on top of $h^1$ and $h^2$. Indeed, by examining
 Figure~\ref{fig::38}, we see that the effect is to replace each
 $\lambda_{-,i}$ by $R_{\tau}(\lambda_{-,i})$. It follows that
 the new partial diffeomorphism map is simply
 \[
R_{\tau}\circ (h\cup \id_B),
 \]
 completing the proof.
\end{proof}

\begin{figure}[ht!]
 \centering
 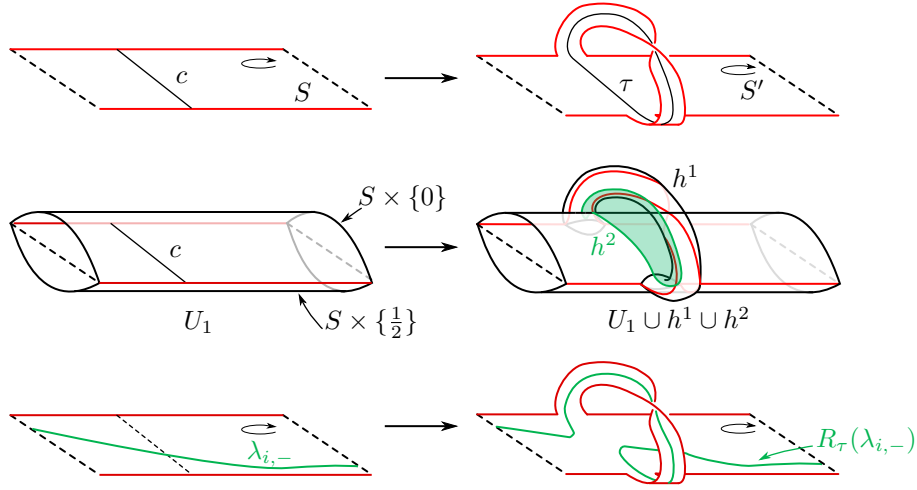
 \caption{Positively stabilizing a partial open book
 is the same as inserting a pair of canceling contact
 1- and 2-handles. On the top row, we show a schematic of
 the new partial open book. In the middle row, we show the
 contact handlebody $U_1=S\times [0,\tfrac{1}{2}]/{\sim}_1$
 (left), as well as the equivalent contact handlebody $U_1\cup h^1\cup h^2$ (right).
 On the bottom row, we show that the attaching circle of
 any contact 2-handles in the previous decomposition are
 changed by a positive Dehn twist along $\tau$. Note that the
 Dehn twist looks negative; however, the picture is of
 $\left(S \times [0,\tfrac{1}{2}] \right)/{\sim}_1$ turned ``upside down,''
  since we are drawing the images of the contact 2-handles on $R_-(U_1)=S\times \{0\}$.
 \label{fig::38}}
\end{figure}

\section{Properties of the gluing map}

\subsection{The gluing map for \texorpdfstring{$I$}{I}-invariant contact structures}
\label{sec:gluingmapIinvariant}
In this section, we prove that gluing on a copy of $\d M\times I$ induces
the identity map, in a sense that we now describe. Suppose that $(M,\g)$
is a sutured submanifold of $(M',\g')$ and $\xi$ is a contact structure
on $M'\setminus \Int(M)$. Furthermore, suppose that there is a Morse function
$f$ on $M'\setminus \Int(M)$ such that $f|_{\d M}\equiv 0$, $f|_{\d M'}\equiv 1$,
$f$ has no critical points, and there is a contact vector field $\nu$ such
that the derivative $\nu(f) > 0$. Furthermore, suppose that the dividing set of $\xi$ on $\d
M\cup \d M'$, with respect to $\nu$, is equal to $\g \cup \g'$. The
vector field $\nu$ induces a diffeomorphism
\[
\phi^{\nu} \colon (-M,-\g) \to (-M',-\g'),
\]
which is well-defined up to isotopy, relative to $\d M$. The
induced diffeomorphism map
\[
\phi^{\nu}_* \colon \SFH(-M,-\g)\to \SFH(-M',-\g')
\]
has a simple description. If $(\bar{\S},\as,\bs)$ is a sutured Heegaard diagram
for $(-M,-\g)$, then we can construct a Heegaard diagram for $(-M',-\g')$ as
\[
(\bar{\S} \cup \bar{A}, \as,\bs),
\]
where $A$ is the \emph{characteristic surface} of $\xi$, with respect to $\nu$; i.e.,
\[
A:=\{\, p\in M'\setminus \Int(M) \,\colon\,  \nu_p\in \xi_p \,\}.
\]
The surface $A$ is a collection of annuli. With
respect to these two diagrams, the diffeomorphism map takes the form
\[
\phi^{\nu}_*(\ve{x})=\ve{x}.
\]
Our gluing map satisfies the following analogue of \cite{HKMTQFT}*{Theorem~6.1}.

\begin{prop}\label{prop:glueinproduct}
Suppose that $(M,\g)$ is a sutured submanifold of $(M',\g')$ and
$\xi$ is a contact structure on $M'\setminus \Int(M)$. Furthermore, suppose
that there is a Morse function $f$ on $M'\setminus \Int(M)$ such that $f|_{\d
M} \equiv 0$, $f|_{\d M'} \equiv 1$, and there is contact vector field $\nu$
such that $\nu(f)>0$. Under the above assumptions, the contact gluing map
$\Phi_{\xi}$ satisfies
\[
\Phi_{\xi}=\phi^{\nu}_*,
\]
where $\phi^\nu \colon (-M,-\g)\to (-M',-\g')$ is the diffeomorphism described above.
\end{prop}

Before we begin with the proof, we need the following definition regarding
sutured cell decompositions of surfaces:

\begin{define}\label{def:dualcelldecomposition}
We say that the sutured cell decompositions
$\cD=(B_1,\dots, B_n, \lambda_1,\dots, \lambda_m)$ and $\cD^*=(B_1',\dots,
B_n', \lambda_1',\dots, \lambda_m')$ of the surface with divides $(F, \g)$ are
\emph{dual} if the following hold:
\begin{enumerate}
\item $B_i\cap B_j'=\emptyset$ for all $i$ and $j$.
\item Each component of $F \setminus (B_1\cup \cdots \cup B_n\cup \lambda_1\cup \cdots \cup \lambda_m)$
  intersects $\g$ in a single arc, and contains a single fattened 0-cell $B_k'$.
  The same statement holds with the roles of $\cD$ and $\cD^*$ reversed.
\item $|\lambda_i \cap \lambda_j'| = \delta_{ij}$, where $\delta_{ij}$ denotes the Kronecker delta.
\end{enumerate}
A sutured cell decomposition $\cD$ that admits a dual is called \emph{dualizable}.
\end{define}

As an example, the sutured cell decomposition of a torus with
 two parallel sutures shown in Figure~\ref{fig::77} is dualizable.

\begin{rem}
Not all sutured cell decompositions $\cD$ admit dual cell decompositions. For
example, if the intersection of a component of $F\setminus (B_1 \cup \cdots \cup B_n \cup
\lambda_1\cup \cdots \cup \lambda_m)$ with $\g$ is disconnected,
then $\cD$ is not dualizable.
\end{rem}

\begin{rem}
Every surface with divides $(F,\g)$ such that $\pi_0(\g) \to \pi_0(F)$
is surjective admits a dualizable cell decomposition. To
construct one, we pick sets of arcs $A_0$ and $A_1$ along $\g$, such that
the arcs in $A_0 \cup A_1$ are pairwise disjoint, and $A_0$ and $A_1$ each
contain exactly one arc in every component of $\g$. We can then view $R_+(\g)$
as a cobordism with product boundary from $A_0$ to $A_1$. By picking a Morse function
on $R_+(\g)$ which is minimized along $A_0$, maximized along $A_1$, and which only
has index 1 critical points, we get a collection of arcs $\lambda_1,\dots, \lambda_k$
(the stable manifolds of the index~1 critical points) with boundary on $A_0$
which cut $R_+(\g)$ into a collection of disks,
each of which contains exactly one component of $A_1$. By picking a similar Morse function on $R_-(\g)$,
we obtain another collection of arcs $\lambda_{k+1}, \dots, \lambda_m$ in $R_-(\g)$
with boundary on $A_0$ that cut $R_-(\g)$ into disks, each of which contains exactly one arc of $A_1$.
We get a sutured cell decomposition with a fattened 0-cell $B_i$ along each arc of $A_0$, and we
use the collection of arcs $\lambda_1, \dots, \lambda_m$. The 0-cells
$B_1, \dots, B_n$ and the arcs $\lambda_1,\dots,\lambda_m$ determine a handle
decomposition of $F$, and the dual sutured cell decomposition is obtained by
using the dual handle decomposition obtained by turning the Morse functions upside down.
\end{rem}

\begin{proof}[Proof of Proposition~\ref{prop:glueinproduct}]
Let $\cD$ and $\cD^*$ be dual sutured cell decompositions of $(\d M,\g)$.
Let $\cC$ be the product contact cell decomposition of
$M'\setminus \Int(M)$ constructed from $\cD$ as in Example~\ref{ex:productcelldecomp},
with barrier surfaces $S$ and $S'$. Write $N$ and $N'$ for the collar
neighborhoods of $\d M$ and $\d M'$ in $M' \setminus \Int(M)$
that are bounded by $S$ and $S'$, respectively.
We denote the dividing set on $S'$ by $\g_0'$.

For each fattened 0-cell $B$ of $\cD$, there is a contact 1-handle $h_{B}$ of
$\cC$. For each arc $\lambda$ of $\cD$, there is a contact 2-handle
$h_{\lambda}$ of $\cC$. For each 2-cell $c$ of $\cD$, there is a contact
3-handle $h_{c}$ of $\cC$. Let $h_1,\dots, h_n$ be an enumeration
of these handles with nondecreasing index. Using equation~\eqref{eq:defgluingmap},
the gluing map is defined as
\begin{equation}
\Phi_{\xi}(\ve{x}):= (C_{h_n}\circ \cdots \circ C_{h_1})
\left(\psi_*^\nu(\ve{x})\otimes \EH(N', \g'\cup \g_0', \xi|_{N'})\right).
\label{eq:defcontactgluingmap}
\end{equation}

The element $\EH(N',\g'\cup \g_0,\xi|_{N'})$ is defined using a partial
open book decomposition of $(N',\g'\cup \g_0,\xi|_{N'})$.
In Section~\ref{sec:partialopenbookdecompositionsandhandledecomps},
we described how a contact handle decomposition of $N'$ with no 3-handles,
viewed as a cobordism from $\emptyset$ to $\d N'$, can be used to
construct a partial open book.
By turning around our construction of a contact cell decomposition
$\cC$ from $\cD$, we can also construct a contact handle
decomposition of $N'$ from a sutured cell decomposition of
$(\d M,\g)$; however, the indices of the corresponding handles will be
different. For our argument to work, we will actually consider the handle
decomposition $\ve{H}'$ of $N'$ induced by the dual sutured cell decomposition $\cD^*$.
For each fattened 0-cell $B'$ of $\cD^*$, there is a contact
2-handle $h_{B'}$ of $\ve{H}'$. For each arc $\lambda'$ of $\cD^*$, there is a
contact 1-handle $h_{\lambda'}$ of $\ve{H}'$. For each 2-cell $c'$ of $\cD^*$,
there is a contact 0-handle $h_{c'}$ of $\ve{H}'$.

By the above construction, there is a correspondence between
the $k$-handles of $\ve{H}'$ and the $(2-k)$-cells of $\cD^*$,
which, in turn, correspond to the $k$-cells of $\cD$.
Finally, there is a correspondence between the $k$-cells of $\cD$ and the $(k+1)$-handles of $\cC$.
Combining these, we get a correspondence between the $k$-handles of $\ve{H}'$ and
the $(k+1)$-handles of $\cC$. Let $h_1',\dots, h_n'$
be an enumeration of the contact handles of $\ve{H}'$
such that $h_i'$ corresponds to $h_i$ under the above correspondence. By Lemma~\ref{lem:EH},
\[
\EH(N',\g'\cup \g_0',\xi) = \left(C_{h_n'}\circ \cdots \circ C_{h_1'} \right)(1),
\]
where $1$ is the generator of $\SFH(\emptyset) \iso \bF_2$. Using
Lemma~\ref{lem:disjointhandlescommute}, we can commute contact handle maps for
disjoint contact handles, so we can rearrange equation~\eqref{eq:defcontactgluingmap} as
\begin{equation}
\Phi_{\xi}(\ve{x})=\left((C_{h_n}\circ C_{h_n'})\circ \cdots \circ (C_{h_1}\circ C_{h_1'})\right)\left(\psi_*^\nu(\ve{x})\right).\label{eq:reordercontacthandles}
\end{equation}

The contact handles $h_i$ and $h_i'$ do not always form a canceling pair
in the sense of Proposition~\ref{prop:movesbetweencelldecompositions}.
However, they are close enough to a
canceling pair to allow us to reduce the above composition to the
diffeomorphism map $\psi^\nu_*$ by performing a sequence of handle
cancellations and isotopies, as we now describe.

Let us first consider the case when $h_i$ and $h_i'$ correspond to a 0-cell
$B$ of $\cD$. Under the previously described correspondence, there is a
2-cell $B^*$ of $\cD^*$ that $B$ corresponds to. Using the previous
notation, we have $h_i = h_B$ and $h_i' = h_{B^*}$.
Then $h_B$ and $h_{B^*}$ form a pair of canceling index 0 and 1 contact handles;
see Figure~\ref{fig::45}. Hence $C_{h_i} \circ C_{h_i'}$ is a diffeomorphism map by
the computation in the proof of Theorem~\ref{thm:invarianceofcontacthandlemap}.

\begin{figure}[ht!]
 \centering
\begingroup%
  \makeatletter%
  \providecommand\color[2][]{%
    \errmessage{(Inkscape) Color is used for the text in Inkscape, but the package 'color.sty' is not loaded}%
    \renewcommand\color[2][]{}%
  }%
  \providecommand\transparent[1]{%
    \errmessage{(Inkscape) Transparency is used (non-zero) for the text in Inkscape, but the package 'transparent.sty' is not loaded}%
    \renewcommand\transparent[1]{}%
  }%
  \providecommand\rotatebox[2]{#2}%
  \newcommand*\fsize{\dimexpr\f@size pt\relax}%
  \newcommand*\lineheight[1]{\fontsize{\fsize}{#1\fsize}\selectfont}%
  \ifx\svgwidth\undefined%
    \setlength{\unitlength}{258.84855669bp}%
    \ifx\svgscale\undefined%
      \relax%
    \else%
      \setlength{\unitlength}{\unitlength * \real{\svgscale}}%
    \fi%
  \else%
    \setlength{\unitlength}{\svgwidth}%
  \fi%
  \global\let\svgwidth\undefined%
  \global\let\svgscale\undefined%
  \makeatother%
  \begin{picture}(1,0.50357959)%
    \lineheight{1}%
    \setlength\tabcolsep{0pt}%
    \put(0,0){\includegraphics[width=\unitlength,page=1]{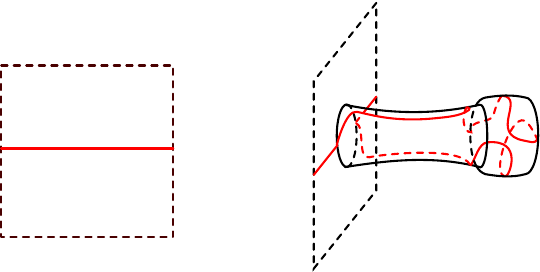}}%
    \put(0.05342505,0.2426715){\color[rgb]{1,0,0}\makebox(0,0)[t]{\lineheight{2.49999976}\smash{\begin{tabular}[t]{c}$\gamma$\end{tabular}}}}%
    \put(0,0){\includegraphics[width=\unitlength,page=2]{fig45.pdf}}%
    \put(0.1634017,0.22241914){\color[rgb]{0,0,0}\makebox(0,0)[t]{\lineheight{2.5}\smash{\begin{tabular}[t]{c}$B$\end{tabular}}}}%
    \put(0,0){\includegraphics[width=\unitlength,page=3]{fig45.pdf}}%
    \put(0.77489959,0.31578429){\color[rgb]{0,0,0}\makebox(0,0)[t]{\lineheight{2.5}\smash{\begin{tabular}[t]{c}$h_B$\end{tabular}}}}%
    \put(0.94021192,0.34112112){\color[rgb]{0,0,0}\makebox(0,0)[t]{\lineheight{2.5}\smash{\begin{tabular}[t]{c}$h_{B^*}$\end{tabular}}}}%
    \put(0,0){\includegraphics[width=\unitlength,page=4]{fig45.pdf}}%
  \end{picture}%
\endgroup%

 \caption{A fattened 0-cell $B$ of a sutured cell decomposition $\cD$
 induces a pair of canceling index 0 and 1 contact handles. We label the
 handles as $h_B$ and $h_{B^*}$. Here $B^*$ denotes the 2-cell of $\cD^*$ that is
 dual to $B$. A portion of the sutured cell decomposition $\cD$ is shown on
 the left, and the corresponding contact handles are shown on the
 right.\label{fig::45}}
\end{figure}

We now consider the case when $h_i$ and $h_i'$ correspond to a 1-cell
$\lambda$ of $\cD$. Let $\lambda^*$ denote the corresponding dual arc of
$\cD^*$. Let $B_1$ and $B_2$ be the fattened 0-cells of $\cD$ at $\d \lambda$
(note that we do not exclude the possibility that
$B_1 = B_2$). Let $h_{B_1}$ and $h_{B_2}$ denote the corresponding 1-handles of
$\cC$. The arc $\lambda$ corresponds to a 2-handle $h_{\lambda}$ of $\cC$.
Let $B_1^*$ and $B_2^*$ denote the 2-cells of $\cD^*$ that correspond to
$B_1$ and $B_2$, respectively. The 2-cells
$B_1^*$ and $B_2^*$ correspond to 0-handles $h_{B_1^*}$ and $h_{B_2^*}$ of $\ve{H}'$. The
dual arc $\lambda^*$ of $\cD^*$ induces a contact 1-handle of $\ve{H}'$. As
described above, the handles $h_{B_i}$ and $h_{B_i^*}$ form a canceling pair of index 1 and 2
contact handles. After canceling these two handles, the handles
$h_{\lambda}$ and $h_{\lambda^*}$ do not quite form a canceling pair of index 1 and 2
handles in the sense of Proposition~\ref{prop:movesbetweencelldecompositions},
because the attaching circle of the
2-handle $h_{\lambda}$ does not intersect the dividing set along
$h_{\lambda^*}$. Instead, it intersects the dividing set near the feet of the
1-handle $h_{\lambda^*}$.  This is shown in Figure~\ref{fig::46}. After
performing an isotopy of $h_{\lambda}$, the handles $h_{\lambda}$ and
$h_{\lambda^*}$ form a canceling pair of index 1 and 2 contact handles. Hence,
the composition $C_{h_i} \circ C_{h_i'}$ induces a diffeomorphism map, by the
computation in the proof Theorem~\ref{thm:invarianceofcontacthandlemap}.

\begin{figure}[ht!]
 \centering
 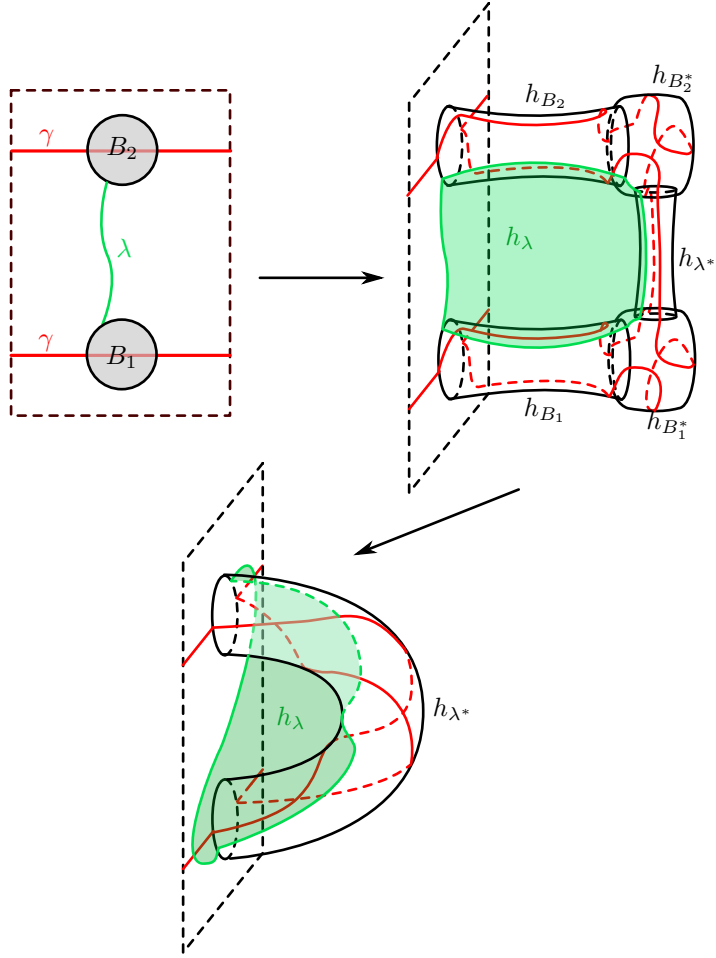
 \caption{A canceling pair of index 1 and 2 contact handles induced by a 1-cell
 $\lambda$ of $\cD$, and the dual 1-cell $\lambda^*$ of $\cD^*$. The first
 picture shows the sutured cell decomposition $\cD$ near $\lambda$.
 The second shows the contact handles associated to the
 cells $B_1$, $B_2$, $B_1^*$, $B_2^*$, $\lambda$, and $\lambda^*$. The third picture is
 obtained by canceling $h_{B_i}$ against $h_{B_i^*}$ for $i \in \{1,2\}$. After this
 cancellation and a small isotopy of $h_\lambda$, the handles $h_{\lambda}$ and $h_{\lambda^*}$ form a
 canceling pair of contact handles.\label{fig::46}}
\end{figure}

Finally, we consider the case when $h_i$ and $h_i'$ correspond to a 2-cell
$c$ of $\cD$. Corresponding to $c$, there is a fattened 0-cell $c^*$ of
$\cD^*$. The 2-cell $c$ induces a 3-handle $h_c = h_i$ of $\cC$, as well as a
2-handle $h_{c^*} = h_i'$ in $\ve{H}'$. The
handles $h_c$ and $h_{c^*}$ form a canceling pair of index 2 and 3 contact
handles. Hence, the composition $C_{h_c} \circ C_{h_{c^*}}$ is equal to a
diffeomorphism map, by the same computation as in the proof of
Theorem~\ref{thm:invarianceofcontacthandlemap}.

We have shown that $C_{h_i} \circ C_{h_i'}$ is equal to a diffeomorphism
map for every $i \in \{\, 1,\dots, n \,\}$, induced by cancelling the handles
$h_i$ and $h_i'$ that are stacked horizontally in the $\nu$-direction. It follows from
equation~\eqref{eq:reordercontacthandles} that $\Phi_{\xi}$ is equal to the
diffeomorphism map $\phi^\nu_*$.
\end{proof}

\subsection{Morse-type contact handles}
In Section~\ref{sec:contacthandlemap}, we defined maps for gluing a contact
handle~$h$ onto the boundary of a sutured manifold $(M,\g)$. We note that
$M$ is \emph{not} a sutured submanifold of $M \cup h$, so the
Honda--Kazez--Mati\'{c} framework does not assign a gluing map to the inclusion
$M \subset M \cup h$. Nonetheless, there is a natural notion of contact
handle that fits into the Honda--Kazez--Mati\'{c} TQFT framework:

\begin{define}\label{def:Morsetypecontacthandle}
Suppose that $(M,\g)$ is a sutured submanifold of $(M',\g')$, and
$\xi$ is a contact structure on $Z = M' \setminus \Int(M)$ with
dividing set $\g \cup \g'$. We say that $(Z,\xi)$
is a \emph{Morse-type contact handle of index $k$} if there is a
contact vector field $\nu$ on~$Z$ that points into~$Z$ on~$\d M$ and out
of~$Z$ on~$\d M'$, as well as a decomposition $Z = Z_0 \cup h$, such that
\begin{enumerate}
\item $Z_0$ is diffeomorphic to $\d M \times I$,
\item $\nu$ is non-vanishing on
  $Z_0$, points into $Z_0$ on $\d M \times \{0\}$ and out of $Z_0$ on
  $\d M \times \{1\}$, and each flowline of $\nu$ is an arc from $\d M \times \{0\}$
  to $\d M \times \{1\}$,
\item $h$ is a topological 3-ball with piecewise smooth boundary, and $\xi$ is tight on $h$.
\end{enumerate}
Furthermore, $h$ is a contact $k$-handle attached to $M \cup Z_0$, as in Definition~\ref{def:contacthandle},
with corners smoothed.
\end{define}

If $(M,\g)$ is a sutured submanifold of $(M',\g')$, and
$(Z,\xi) = (M' \setminus \Int(M), \xi)$ is a Morse-type contact handle of
index~$k$ attached to $(M,\g)$, then we call a choice of contact vector field
$\nu$ and decomposition $Z = Z_0 \cup h$ a \emph{parametrization} of $(Z,\xi)$. Given a
parametrization of $(Z,\xi)$, there is a natural candidate for the contact
gluing map $\Phi_{\xi}$, namely
\[
C_{h} \circ \phi^{\nu|_{Z_0}}_* \colon \SFH(-M,-\g) \to \SFH(-M',-\g').
\]
In the above equation,
\[
\phi^{\nu|_{Z_0}}_* \colon \SFH(-M,-\g) \to \SFH(-M \cup -Z_0,-\g_0)
\]
is the diffeomorphism map induced by the vector
field $\nu$, as discussed in Section~\ref{sec:gluingmapIinvariant},
where $\g_0$ is the dividing set of $\xi$ on $\d (M \cup Z_0)$.
Furthermore,
\[
C_h \colon \SFH(-M \cup -Z_0, -\g_0)\to \SFH(-M',-\g')
\]
is the contact handle map, as defined in Section~\ref{sec:contacthandlemap}.
Indeed, we will prove the following:

\begin{prop}\label{prop:computemorsetypehandlemap}
Suppose $(M,\g)$ is a sutured submanifold of $(M',\g')$, and
$(Z,\xi) = (M' \setminus \Int(M), \xi)$ is a Morse-type contact handle, with a
parametrizing contact vector field $\nu$ and decomposition $Z = Z_0 \cup h$.
Then the contact gluing map $\Phi_{\xi}$ is equal to the composition
$C_h \circ \phi^{\nu|_{Z_0}}_*$.
\end{prop}

\begin{proof}
The proof is essentially the same for all handle indices, so for definiteness
we will focus on 2-handles. By assumption, the contact vector field
$\nu$ is non-vanishing on $Z_0$, and on a collar neighborhood of $\d M'$. Let
$D \subset Z$ be a core of $h$. Pick an incoming barrier surface $S \subset Z_0$.
Extend $D$ down into $Z_0$ such that $\d D \subset S$ is Legendrian with $\tb(\d D) = -1$.
Then we can perturb $D$ while fixing $\d D$ such that it becomes convex.
Let $N$ denote the collar of $\d M$
bounded by $S$, and let
\[
\tilde{Z} := \cl(Z \setminus (N \cup N(D))).
\]
We can pick $N(D)$ such that $\nu$ is non-vanishing on $\tilde{Z}$. Using the flow of
$\nu$, one can construct a Morse function $f$ on $\tilde{Z}$ that is $0$ on
$\d \tilde{Z} \setminus \d M'$ and $1$ on $\d M'$, and such that $\nu(f) > 0$.

The image of $\d N(D) \setminus S$ in $\d M'$ under the flow of $\nu$
consists of two disks, $D_1$ and $D_2$. Pick a dualizable sutured
cell decomposition $\cD$ of $(\d M', \g')$ with no 0-cells or 1-cells
that intersect~$D_1$ or~$D_2$. Let $B_1, \dots, B_m$ be the 0-cells of $\cD$,
and let $\lambda_1, \dots, \lambda_n$ be the 1-cells. Write $c_1, \dots, c_k$
for the 2-cells. Adapting Example~\ref{ex:productcelldecomp}, after
performing a $C^0$-small isotopy of $S$, and picking a barrier surface $S'$
that bounds a collar neighborhood $N'$ of $\d M'$, we can construct a
contact cell decomposition~$\cC$ of~$Z$ that has barrier surfaces $S$ and
$S'$. The contact cell decomposition $\cC$ has no 0-cells, one 1-cell for
each 0-cell $B_1, \dots, B_m$ of $\cD$, one 2-cell for each 1-cell
$\lambda_1, \dots, \lambda_n$ of $\cD$, one 3-cell for each 2-cell $c_1, \dots, c_k$
of $\cD$, and also the 2-cell $D$. Let us write $h_1, \dots, h_n$ for the
handles induced by $\cD$, and write $h_D$ for $N(D)$, viewed as a contact
2-handle. By definition
\begin{equation}\label{eq:Morsehandlemap1}
\Phi_{\xi}(\ve{x}) =
(C_{h_{n}} \circ \cdots \circ C_{h_{1}} \circ C_{h_{D}})
\left(\phi^{\nu|_{N}}_*(\ve{x}) \otimes \EH(N', \xi|_{N'})\right).
\end{equation}

A dual sutured cell decomposition $\cD^*$ of $(\d M', \g')$ gives rise to
a contact handle decomposition~$\ve{H}'$ of $(N', \xi|_{N'})$ starting at the empty
sutured manifold. We can compute $\EH(N',\xi|_{N'})$ by applying Lemma~\ref{lem:EH} to $\ve{H}'$.
Exactly as in the proof of Proposition~\ref{prop:glueinproduct},
the handles of $\ve{H}'$ cancel $h_1, \dots, h_n$ pairwise,
and we can reduce equation~\eqref{eq:Morsehandlemap1} to
\[
\left(\phi^{\nu|_{\tilde{Z}}}_*\circ C_{h_{D}} \circ \phi^{\nu|_{N}}_*\right)(\ve{x}).
\]
Given the description of the diffeomorphism maps $\phi_*^{\nu|_{\tilde{Z}}}$ and
$\phi_*^{\nu|_N}$ from Section~\ref{sec:gluingmapIinvariant}, the above
expression is clearly equal to $C_h \circ \phi_*^{\nu|_{Z_0}}$.
\end{proof}

\subsection{Functoriality of the gluing map}

We now show that the gluing map defined in this paper satisfies the
functoriality property of the Honda--Kazez--Mati\'{c} construction
\cite{HKMTQFT}*{Proposition~6.2}. This property will be useful when we prove
the equivalence of our construction with the Honda--Kazez--Mati\'{c}
construction.

\begin{prop}\label{prop:compositionlaw}
Suppose that we have a chain of sutured submanifolds
\[
(M,\g)\subset (M',\g')\subset (M'',\g''),
\]
as well as a contact structure $\xi$ on $M'' \setminus \Int(M)$ such that
$\d M$, $ \d M'$, and $\d M''$ are convex with dividing sets $\g$, $\g'$, and
$\g''$, respectively. Writing $\xi'$ for $\xi|_{M'\setminus \Int(M)}$ and
$\xi''$ for $\xi|_{M''\setminus \Int(M')}$, we have
\[
\Phi_{\xi} = \Phi_{\xi''}\circ \Phi_{\xi'}.
\]
\end{prop}

\begin{proof}
Define $Z' := M' \setminus \Int(M)$ and $Z'' := M'' \setminus \Int(M')$. Let
$\cC'$ and $\cC''$ be contact cell decompositions of $(Z',\xi')$ and
$(Z'',\xi'')$. Let us write $h_{1}', \dots, h_n'$ for the contact handles of
$\cC'$, and $h_{1}'', \dots, h_m''$ for the contact handles of $\cC''$,
ordered such that their indices are nondecreasing.
Let $\nu'$ and $\nu''$ denote the contact vector fields chosen on the incoming
ends of $Z'$ and $Z''$, let $N_1$ and $N_2$ denote the incoming layers of
$Z'$ and $Z''$, and let $N_1'$ and $N_2'$ denote the outgoing layers,
respectively, as described in Definition~\ref{def:contactcelldecomposition}.
By definition, the composition $\Phi_{\xi''} \circ \Phi_{\xi'}$ is equal to
\begin{equation}
(C_{h_m''} \circ \cdots \circ C_{h_1''}) \circ \Phi_{\sqcup N_2'}
\circ \phi_*^{\nu''|_{N_2}}\circ (C_{h_n'} \circ \cdots \circ C_{h_1'})
\circ \Phi_{\sqcup N_1'} \circ \phi_*^{\nu'|_{N_1}}.
\label{eq:composition1}
\end{equation}
The map $\Phi_{\sqcup N_i'}$ is given by tensoring with
$\EH(N_i',\xi|_{N_i'})$, for $i \in \{1,2\}$. As in Lemma~\ref{lem:EH}, the element $\Phi_{\sqcup
N_1'}$ can be written as a composition of contact handle maps
$C_{h_1} \circ \cdots \circ C_{h_\ell}$, for a sequence of contact 0-, 1-,
and 2-handles $h_1,\dots, h_\ell$. Hence, the composition in
equation~\eqref{eq:composition1} can be written as
\begin{equation}
(C_{h_m''} \circ \cdots \circ C_{h_1''}) \circ \Phi_{\sqcup N_2'} \circ \phi_*^{\nu''|_{N_2}}
\circ (C_{h_n'} \circ \cdots \circ C_{h_1'}) \circ (C_{h_\ell} \circ \cdots \circ C_{h_1})
\circ \phi_*^{\nu'|_{N_1}}. \label{eq:composition2}
\end{equation}
By picking $\nu''$ and $N_2$ appropriately, we can assume that the
diffeomorphism $\phi^{\nu''|_{N_2}} \colon M' \to M' \cup N_2$  is a contactomorphism
on all of $Z'$, and is the identity on $\d M$. Write
$\bar{h}_{k} = \phi^{\nu''|_{N_2}}(h_k)$ and $\bar{h}_{k}' = \phi^{\nu''|_{N_2}}(h_k')$.
Using the diffeomorphism invariance of the contact
handle maps, we can rewrite equation~\eqref{eq:composition2} as
\begin{equation}
(C_{h_m''} \circ \cdots \circ C_{h_1''}) \circ \Phi_{\sqcup N_2'}
\circ (C_{\bar{h}_n'} \circ \cdots \circ C_{\bar{h}_1'})
\circ (C_{\bar{h}_\ell} \circ \cdots \circ C_{\bar{h}_1}) \circ \phi_*^{\nu'|_{N_1}}.
\label{eq:composition3}
\end{equation}
The map $\Phi_{\sqcup N_2'}$ can be commuted with all the contact
handle maps to the right of it by Lemma~\ref{lem:disjointhandlescommute}.
After possibly isotoping some of the remaining contact handles, we can apply
Lemma~\ref{lem:disjointhandlescommute} and reorder the handles such that they
are attached with nondecreasing index. Furthermore, after isotoping some of
the handles, we can assume that the handles in the above composition are
induced by a contact cell decomposition (i.e., the 0-handles and 1-handles are
induced by a Legendrian graph, and the 2-handles are induced by a sequence of
convex disks with $\tb = -1$ attached to a neighborhood of the graph and $\d
((Z' \cup Z'') \setminus \Int(N_1 \cup N_2')$). It follows that
equation~\eqref{eq:composition3} is equal to $\Phi_{\xi,\cC}$ for some
contact cell decomposition $\cC$ of $(Z,\xi)$, completing the proof.
\end{proof}

\subsection{Equivalence with the Honda--Kazez--Mati\'{c} construction}

In this section, we prove that our construction of the gluing map from
Section~\ref{sec:def-gluingmap} is equivalent to the original construction
due to Honda, Kazez, and Mati\'c~\cite{HKMTQFT}. We will write
$\Phi_{\xi}^{\HKM}$ for the map defined using their construction.

\begin{thm}\label{thm:Contacthandlemaps=HKMmaps}
Suppose $(M,\g)$ is a sutured submanifold of $(M',\g')$ with no isolated components,
and that $\xi$ is a contact structure on $M' \setminus \Int(M)$ with convex boundary
and dividing set $\g \cup \g'$.
Then the Honda--Kazez--Mati\'c gluing map $\Phi_{\xi}^{\HKM}$
is equal to the gluing map $\Phi_{\xi}$ we defined in Section~\ref{sec:def-gluingmap}.
\end{thm}

\begin{proof}
Using the composition law for both constructions of the gluing map
(Proposition~\ref{prop:compositionlaw} and \cite{HKMTQFT}*{Proposition~6.2}),
it is sufficient to show the claim when $M' \setminus \Int(M)$ consists of a
single Morse-type contact handle of index 0, 1, or 2.

For a Morse-type index 0 handle, the claim is straightforward. Write
$M' \setminus \Int (M)$ as $Z_0 \cup h^0$ where  $Z_0 \iso (I \times \d M)$, and
$h^0$ is a 3-ball. The contact structure $\xi$ is the union of the
$I$-invariant contact structure on $Z_0$ and the unique tight contact
structure on $h^0$. Write $\g_0'$ for the suture on $M \cup Z_0$, and suppose
that $\nu$ is a parametrizing contact vector field on $Z_0$.

Under the identification
\[
\SFH(M \cup Z_0 \cup h^0, \g_0' \cup \g_0) \iso
\SFH(M \cup Z_0,\g_0') \otimes \SFH(h^0,\g_0),
\]
where $\g_0$ consists of a single suture on $h^0$,  both gluing maps
\[
\Phi_{\xi}, \,\Phi_{\xi}^{\HKM} \colon \SFH(M,\g) \to \SFH(M\cup Z_0,\g_0') \otimes \SFH(h^0,\g_0)
\]
take the form
\[
\ve{x} \mapsto \phi^{\nu}_*(\ve{x}) \otimes \EH(h^0,\g_0,\xi_0) = \phi^{\nu}_*(\ve{x})\otimes 1,
\]
where $1$ denotes the generator of $\SFH(h^0,\g_0)\iso \bF_2$. This follows
from \cite{HKMTQFT}*{Proposition~6.1} for~$\Phi_{\xi}^{\HKM}$, and from
Proposition~\ref{prop:computemorsetypehandlemap} for our map $\Phi_{\xi}$.

Before we consider index 1 and 2 Morse-type contact handles, we must first
give a more detailed description of the construction of Honda, Kazez, and Mati\'c~\cite{HKMTQFT}.
The definition of the map~$\Phi_{\xi}^{\HKM}$ uses the description of partial
open books from~\cite{HKMSutured}.
A contact 3-manifold with convex boundary is called \emph{product disk decomposable} if
it is a union of handlebodies, and
it contains a collection of pairwise disjoint compressing disks, each intersecting
the dividing set in two points, whose complement is a union of standard contact balls.
Given a contact sutured manifold
$(M,\g,\xi)$, one picks a properly embedded Legendrian graph $K \subset M$
that intersects $\d M$ along a collection of univalent vertices in
$\g$, such that $M \setminus \Int(N(K))$ is product disk decomposable.
Here $N(K)$ denotes a standard contact neighborhood, which is also product disk decomposable.
It follows that $M \setminus \Int(N(K))$ is contactomorphic to
$(S \times I)/{\sim_1}$ for a compact surface $S$ with boundary, and $N(K)$ is
contactomorphic to $(P \times I)/{\sim_2}$ for a compact surface with boundary $P$
(for the definition of $\sim_1$ and $\sim_2$, see Section~\ref{sec:partialopenbookandHD}).
Here $P$ has piecewise smooth boundary whose edges can naturally be divided
into two types: those that intersect $\d N(K) \setminus \d M$, and those
that intersect $\d M$. By $(P \times I)/{\sim_2}$,  we mean the space obtained
by quotienting out the $I$ direction along the edges that are contained
in $\d N(K)$. Since $(P \times I)/{\sim_2}$ meets $(S \times  I)/{\sim_1}$
along $\d N(K) \setminus \d M$, the surfaces $P \times \{0\}$ and $P \times \{1\}$ give two
embeddings of $P$ into $S$. Using the projection of $S \times I$ onto $S$, we
identify $P \times \{0\} \subset S \times \{1\}$ with a subset of $S$, for which
we also write $P$. The surface $P \times \{1\}$ then gives another smooth
embedding $h \colon P \to S$, which is the monodromy map.

The Honda--Kazez--Mati\'{c} map is easiest to define if one picks a contact
structure $\zeta$ on $(M,\g)$, such that $\d M$ is convex with dividing set
$\g$. To define the map $\Phi_{\xi}^{\HKM}$, one picks Legendrian graphs
$K \subset M$ and $K' \subset M' \setminus \Int(M)$, whose complements are
product disk decomposable. After modifying $K'$ in a neighborhood of $\d M$,
one extends $K$ to a Legendrian graph on all of $M'$, whose complement is
product disk decomposable. Furthermore, outside a small neighborhood of
$\d M$, the extension agrees with $K$ and $K'$. Let us write $\bar{K}$ for
this Legendrian graph. The graph $K$ is also required to satisfy a
\emph{contact compatibility} condition near $\d M$, described in
\cite{HKMTQFT}, though the specific form is not important for our present
argument. The graph $K \subset M$ induces a partial open book $(S,P,h)$ for
$(M,\g,\zeta)$, which induces a diagram $(\Sigma,\as,\bs)$ for $(M,\g)$. The
graph $\bar{K}$ induces a partial open book $(S',P',h')$ for $(M',\g')$,
which gives rise to the diagram $(\Sigma',\as',\bs')$. Furthermore
$\Sigma' = \Sigma \cup \Sigma''$, $\as'=\as\cup \as''$, and $\bs=\bs\cup \bs''$,
for a surface $\Sigma''$ and a collection of curves $\as''$ and $\bs''$ on
$\Sigma'$. There is a canonical intersection point $\ve{x}_{\xi} \in
\bT_{\as''} \cap \bT_{\bs''}$, and the map $\Phi_{\xi}^{\HKM}$ on the chain level is defined by
the formula
\[
\Phi_{\xi}^{\HKM}(\ve{x}) = \ve{x} \times \ve{x}_{\xi}.
\]

We now consider a Morse-type contact 1-handle addition. In this case,
\[
(M' \setminus \Int(M), \xi) \iso (Z_0 \cup h_1,\xi_0 \cup \xi_1),
\]
where $(Z_0,\xi_0)$ is an $I$-invariant contact structure on $I \times \d M$, and
$(h^1,\xi_1)$ is a contact 1-handle. Let $\bar{K} \subset M \cup Z_0$ denote an
extension that can be used to compute the gluing map. We note that if
$K' \subset Z_0$ is a graph whose complement is product disk decomposable,
then the complement of $K'$ in $Z_0 \cup h^1$ is also product disk
decomposable.  It follows that we can use the same graph
\[
\bar{K} \subset M \cup Z_0 \subset M \cup Z_0 \cup h^1
\]
to compute the gluing map for $Z_0 \cup
h^1$. Write $(S,P,h)$ for the partial open book of $(M,\g)$ induced by $K$.
Write $(S'_0,P'_0,h'_0)$ for the partial open book of $(M \cup Z_0,\g_0')$
induced by $\bar{K} \subset M \cup Z_0$, and write $(S',P',h')$ for the partial
open book of $(M \cup Z_0 \cup h^1,\gamma')$ induced by $\bar{K}$.

Since $(M \cup Z_0 \cup h^1) \setminus N(\bar{K})$ is obtained by attaching a
contact 1-handle to $(M \cup Z_0) \setminus N(\bar{K})$, we can apply
Lemma~\ref{lem:partialopenbookcont1-handle} to see that the partial open book
$(S',P',h')$ is obtained from $(S'_0,P'_0,h'_0)$ by attaching a 1-handle to
$S_0'$ along $\d S_0' \setminus \d P_0'$, and setting $P' = P'_0$ and $h' = \iota_{S_0'} \circ h'_0$,
where $\iota_{S_0'} \colon S_0' \to S'$ is the embedding.
The same basis of arcs $\ve{a}'$ for $P'_0$ in $S'_0$ can be used for $P'$ in
$S'$, which we assume extends a basis of arcs $\ve{a}$ for $P$ in $S$. Let
$(\Sigma,\as,\bs)$, $(\Sigma_0',\as_0',\bs_0')$, and $(\Sigma',\as',\bs')$ be
the diagrams induced by $(S,P,h)$, $(S_0', P_0',h_0')$, and $(S',P',h')$
with the bases $\ve{a},$ $\ve{a}'$, and $\ve{a}'$, respectively.

The Heegaard surface $\Sigma'$ is thus obtained by attaching a 1-handle along
the boundary of $\Sigma'_0.$ Notice that, since $\as_0' = \as'$ and
$\bs_0' = \bs'$, and $\Sigma'$ is obtained from $\Sigma_0'$ by attaching a band
along $\d \Sigma'_0$, the groups $\SFH(\Sigma'_0,\as'_0,\bs'_0)$ and
$\SFH(\Sigma',\as',\bs')$ are naturally isomorphic, and the isomorphism is
given by the contact 1-handle map, defined in this paper. Furthermore, we
observe that
\begin{equation} \label{eq:HKM=contacthandle1}
\Phi_{\xi_0\cup \xi_1}^{\HKM} = C_{h^1} \circ \Phi_{\xi_0}^{\HKM}.
\end{equation}
By \cite{HKMTQFT}*{Proposition~6.1},  the above expression is equal to
$C_{h^1}\circ \phi^{\nu}_*$, where
\[
\phi_*^\nu \colon \SFH(\Sigma,\as,\bs) \to \SFH(\Sigma_0',\as_0',\bs_0')
\]
is the composition of the tautological map induced by a diffeomorphism, and the transition maps
induced by naturality. By Proposition~\ref{prop:computemorsetypehandlemap},  we see that
the expression in equation~\eqref{eq:HKM=contacthandle1} agrees with the
definition of $\Phi_{\xi_0\cup\xi_1}$, in this paper.

The argument when
\[
(M' \setminus \Int(M),\xi) = (Z_0 \cup h^2,\xi_0 \cup \xi_2)
\]
is a Morse-type contact 2-handle is similar. Suppose that $K \subset M$ is a Legendrian graph
such that $M \setminus \Int(N(K))$ is product disk decomposable, and
induces a partial open book that satisfies the contact compatibility
condition near $\d M$. We then let $\bar{K}_0$ denote a Legendrian extension
to $M \cup Z_0$, whose complement is product disk decomposable, and which can
be used to compute the map $\Phi_{\xi_0}^{\HKM}$ for gluing $(Z_0,\xi_0)$ to
$M$. We can define an extension $\bar{K}$ of $K$ into all of $M \cup Z_0 \cup h^2$
by setting
\[
\bar{K} := \bar{K}_0 \cup c,
\]
where $c$ is a Legendrian co-core of the 2-handle $h^2$. We note that $(M \cup
Z_0 \cup h^2) \setminus \Int(N(\bar{K}))$ is  product disk decomposable. Let
$(S,P,h)$, $(S'_0,P'_0,h'_0)$, and $(S',P',h')$ denote the partial open books
induced by $K \subset (M,\g)$, $\bar{K}_0 \subset (M \cup Z_0,\g'_0)$, and
$\bar{K} \subset (M \cup Z_0 \cup h^2, \g')$, respectively.  Let
$\pi \colon S\times I\to S$ be the projection.
Writing the 2-handle $h^2$ as $(B \times I)/{\sim_2}$,
where $B$ is a square, we observe that $P'$ is obtained by adding the 1-handle
$\pi(B \times \{0\})$ to $P'_0$. The monodromy is extended to $P'$ by mapping
$\pi(B \times \{0\})\subset S$ to $\pi(B \times \{1\}) \subset S$.

We start with a basis of arcs $\ve{a}$ for $P \subset S$, and extend $\ve{a}$ to a basis
$\ve{a}'_0$ for $P'_0 \subset S_0'$. A basis of arcs $\ve{a}'$ for $P' \subset
S'$ can then be obtained from $\ve{a}'_0$  by adding a new arc $a'$, which is
a co-core of the band $\pi(B\times \{0\})$. Write $(\Sigma,\as,\bs)$,
$(\Sigma'_0,\as'_0,\bs'_0)$, and $(\Sigma',\as',\bs')$ for the diagrams
obtained from the partial open books $(S,P,h)$, $(S'_0,P'_0,h'_0)$, and
$(S',P',h')$, with bases $\ve{a}$, $\ve{a}_0'$, and $\ve{a}'$, respectively.
Also, let us write $\as'_0 = \as \cup \as''_0$ and $\bs'_0 = \bs \cup \bs''_0$, where
$\as''_0$ and $\bs''_0$ are the curves induced by the arcs in
$\ve{a}'_0 \setminus \ve{a}$, and  $\ve{x}_{\xi_0}$ for the canonical
intersection point in $\bT_{\as''_0} \cap \bT_{\bs''_0}$. Finally, write
$\alpha'$ and $\beta'$ for the curves induced by the new basis arc $a'$, and
write $c'$ for the canonical intersection point of $\alpha'\cap \beta'$. The
map $\Phi^{\HKM}_{\xi_0\cup \xi_2}$ on $\ve{x} \in \T_{\as} \cap \T_{\bs}$ is defined by the formula
\[
\Phi^{\HKM}_{\xi_0\cup \xi_2}(\ve{x}) = \ve{x}\times \ve{x}_{\xi_0} \times c'.
\]
Noting that $\Phi_{\xi_0}^{\HKM}(\ve{x}) = \ve{x} \times \ve{x}_{\xi_0}$, we see that
\[
\Phi_{\xi_0 \cup \xi_2}^{\HKM}(\ve{x}) = (C_{h^2} \circ \Phi_{\xi_0}^{\HKM})(\ve{x}).
\]
By the same argument as for contact 1-handles, this is equal to
$(C_{h^2} \circ \phi_*^{\nu})(\ve{x}) = \Phi_{\xi_0 \cup \xi_2}(\ve{x})$,
completing the proof.
\end{proof}

\section{Turning around cobordisms of sutured manifolds and duality}

In this section, we compute the effect of turning around a cobordism of sutured manifolds,
proving Theorem~\ref{thm:turningaroundcobordism}.

\subsection{The canonical trace pairing}
\label{sec:tracepairing}
As described in \cite{decateg}*{Proposition~2.14} and
\cite{JCob}*{Section~11.2}, there is duality between $\SFH(M,\g)$ and
$\SFH(-M,\g)$. If $(\S,\as,\bs)$ is a diagram for $(M,\g)$, then
$(\S,\bs,\as)$ is a diagram for $(-M,\g)$. Since $\bT_{\as}\cap
\bT_{\bs}$ is equal to $\bT_{\bs}\cap \bT_{\as}$, we can define a map
\begin{equation}
\tr \colon \CF(\S,\as,\bs) \otimes \CF(\S,\bs,\as)\to \bF_2\label{eq:tracedef}
\end{equation}
by the formula
\[
\tr(\ve{x}\otimes \ve{y})=
\begin{cases}
1& \text{if } \ve{x}=\ve{y},\\
0& \text{otherwise}.
\end{cases}
\]
It is straightforward to see that $\tr$ is a chain map, since $J$-holomorphic
discs on $(\S,\as,\bs)$ from~$\xs$ to~$\ys$ are in bijection with
$J$-holomorphic discs on $(\S,\bs,\as)$ from~$\ys$ to~$\xs$.
Note that the trace pairing gives a natural isomorphism
\[
\CF(\S,\bs,\as)\iso \CF(\S,\as,\bs)^\vee:=\Hom_{\bF_2}(\CF(\S,\as,\bs),\bF_2).
\]
In particular, $\tr$ is the usual pairing between homology and cohomology.

In the opposite direction, there is the cotrace map,
\[
\cotr \colon \bF_2 \to \CF(\S,\as,\bs) \otimes \CF(\S,\bs,\as).
\]
The cotrace map is defined by the formula
\[
\cotr(1) = \sum_{\ve{x}\in \bT_{\as}\cap \bT_{\bs}} \ve{x}\otimes \ve{x}.
\]

We note that if $V$ is a finite dimensional
 vector space over $\bF_2$, then there are canonical isomorphisms
\[
\Hom_{\bF_2}(V,V) \iso \Hom_{\bF_2}(V\otimes V^\vee, \bF_2)\iso \Hom_{\bF_2}(\bF_2, V^\vee\otimes V).
\]
Under these isomorphisms, the $\tr$ and $\cotr$ maps are identified with $\id_V \in \Hom_{\bF_2}(V,V)$.

In \cite{JCob}, the first author defined a pairing
\[
\langle \, , \, \rangle \colon \SFH(M,\g) \otimes \SFH(-M,-\g) \to \bF_2,
\]
which at first glance appears to have a different domain than the trace map
defined in equation~\eqref{eq:tracedef}. However there is a canonical isomorphism
\begin{equation}
\SFH(-M,\g) \iso \SFH(-M,-\g),\label{eq:canonicaliso}
\end{equation}
which we describe presently. The diagram $(\Sigma,\bs,\as)$ represents $(-M,\g)$,
while $(-\Sigma,\as,\bs)$ represents $(-M,-\g)$. If $\phi\in \pi_2(\xs,\ys)$ is a homology class
of disks for $(\Sigma,\bs,\as)$, then there is
a uniquely determined homology class $\bar{\phi}\in \pi_2(\xs,\ys)$ for
$(-\Sigma,\as,\bs)$. Furthermore, by precomposing a holomorphic disk
$u\colon \bD\to \Sym^n(\Sigma)$ with the unique anti-holomorphic involution of $\bD$ that fixes
$\pm i\in \d \bD$, we obtain a bijection
\[
\cM_J(\phi)/\R\iso \cM_{-J}(\bar{\phi})/\R,
\]
establishing the isomorphism in Equation~\eqref{eq:canonicaliso}. We note that
$\tr$ agrees with $\langle\, ,\, \rangle$
under the isomorphism in Equation~\eqref{eq:canonicaliso}.

\subsection{Sutured manifold cobordisms and the induced maps}\label{sec:suturedcobordisms}

In this section, we review the definition of sutured manifold cobordisms,
special cobordisms, boundary cobordisms, and the construction of the sutured
cobordism maps. We finally give a simpler definition of the cobordism maps
using our gluing map from Section~\ref{sec:def-gluingmap}. The following is
\cite{JCob}*{Definition~2.3}.

\begin{define}\label{def:equivalence}
  We say that the contact structures $\xi_0$ and $\xi_1$ on the sutured
  manifold $(M,\g)$ are \emph{equivalent} if they can be connected by a
  1-parameter family $\{\xi_t : t \in I\}$ of contact structures on $(M,\g)$,
  such that $\d M$ is convex with dividing set $\gamma$ for each $\xi_t$. In
  this case, we write $\xi_1 \sim \xi_2$, and denote the equivalence class of
  $\xi$ by $[\xi]$.
\end{define}

Sutured manifold cobordisms were defined in \cite{JCob}*{Definition~2.4}.

\begin{define}
  Let $(M_0,\g_0)$ and $(M_1,\g_1)$ be sutured manifolds. A \emph{cobordism}
  from $(M_0,\g_0)$ to $(M_1,\g_1)$ is a triple $\cW = (W,Z,[\xi])$ such that
  \begin{itemize}
    \item $W$ is a compact, oriented 4-manifold with boundary and corners,
    \item $Z$ is a codimension-0 submanifold of $\d W$, and $\d W \setminus \Int(Z) = -M_0 \sqcup M_1$,
    \item $\xi$ is a positive contact structure on $(Z,\g_0 \cup \g_1)$.
  \end{itemize}
\end{define}

Note that equivalent contact structures can have different characteristic foliations on $\d M$,
which gives us enough flexibility to compose cobordisms.
The sutured manifold cobordism $\cW$ is \emph{balanced} if $(M_0,\g_0)$ and $(M_1,\g_1)$ are balanced
sutured manifolds. Furthermore, we say that $Z_0$ is an \emph{isolated component}
of $Z$ if $Z_0 \cap M_1 = \emptyset$. The following is \cite{JCob}*{Definition~5.1}

\begin{define}
  We say that the cobordism $\cW = (W,Z,[\xi])$ from $(M_0,\g_0)$ to $(M_1,\g_1)$ is \emph{special} if
  \begin{itemize}
    \item $\cW$ is balanced,
    \item $\d M_0 = \d M_1$, and $Z = -I \times \d M_0$ is the trivial cobordism between them,
    \item $\xi$ is an $I$-invariant contact structure on $Z$ such that each $\{t\} \times \d M_0$
    is convex with dividing set $\{t\} \times \g_0$ for every $t \in I$,
    with respect to the contact vector field $\d/\d t$.
  \end{itemize}
\end{define}

Given a special cobordism $\cW$ from $(M_0,\g_0)$ to $(M_1,\g_1)$, we define the map
\[
F_\cW \colon \SFH(M_0,\g_0) \to \SFH(M_1,\g_1)
\]
by composing maps
associated to 4-dimensional handle attachments along the interior of $M_0$;
see \cite{JCob}*{Section~8}. The following is equivalent to \cite{JCob}*{Definition~10.4}.

\begin{define}\label{def:boundarycobordism}
A sutured cobordism $(W,Z,[\xi])$ from $(M,\g)$ to $(M',\g')$ is called a
\emph{boundary cobordism} if $M \subset \text{Int}(M')$,
$W = I \times M'/{\sim}$, where $(t,x) \sim (t',x)$ for every $x \in \partial M'$ and $t \in I$,
and $\xi$ is a contact structure on $Z = - \{0\} \times (M' \setminus \Int(M))$
inducing the sutures $\{0\} \times \g$ on $\{0\} \times \d M$ and $\{0\} \times \g'$
on $\{0\} \times \d M'$.
\end{define}

If $\cW = (W,Z,[\xi])$ is a boundary cobordism from $(M,\g)$ to $(M',\g')$,
then we can view $(-M,-\g)$ as a sutured submanifold of $(-M',-\g')$,
and $-\xi$ is a positive contact structure on $Z = -M' \setminus \Int(-M)$
with dividing set $-\g_0 \cup -\g_1$. If $Z$ has no isolated components,
then the map $F_\cW \colon \SFH(M,\g) \to \SFH(M',\g')$ is defined
as the Honda--Kazez--Mati\'c gluing map $\Phi_{-\xi}$.

Every balanced cobordism $\cW = (W,Z,[\xi])$ from $(M_0,\g_0)$ to $(M_1,\g_1)$
can uniquely be written as a composition $\cW^s \circ \cW^b$, where
\[
\cW^b = (I \times (M_0 \cup -Z)/{\sim}, \{0\} \times Z, [\xi])
\]
is a boundary cobordism from $(M_0,\g_0)$ to $(M_0 \cup -Z,\g_1)$. Furthermore,
\[
\cW^s = (W, -I \times \d M_1, [\eta])
\]
is a special cobordism from $(M_0 \cup -Z,\g_1)$ to $(M_1,\g_1)$,
where $-I \times \d M_1$ is a collar of $\d M_1$ in $Z$, and $\eta$ is an $I$-invariant
contact structure with dividing set $\{t\} \times \g_1$ on $\{t\} \times \d M_1$
for every $t \in I$, with respect to $\d/\d t$.
We call $\cW^s$ the special part, and $\cW^b$ the boundary part of $\cW$.
If $Z$ has no isolated components, then the cobordism map $F_\cW$ is defined as $F_{\cW^s} \circ F_{\cW^b}$.

According to \cite{JCob}*{Definition~10.1}, in the general case,
we choose a standard contact ball $B_0 \subset \Int(Z_0)$ with convex boundary
and dividing set $\delta_0$ in each isolated component $Z_0$ of $Z$.
We write $(B,\delta)$ for the union of the balls $(B_0,\delta_0)$,
and consider the cobordism $\cW' = (W,Z',[\xi'])$ from $(M_0,\g_0)$ to
$(M_1,\g_1) \sqcup (B,\delta)$, where $Z' = Z \setminus \Int(B)$ and $\xi' = \xi|_{Z'}$.
Since $Z'$ has no isolated components and
\[
\SFH((M_1, \g_1) \sqcup (B,\delta)) \iso \SFH(M_1,\g_1),
\]
we can define $F_\cW := F_{\cW'}$. This is independent of the choice of $B$.

The gluing map that we defined in Section~\ref{sec:def-gluingmap} also assigns
maps to contact 3-handles, and hence $\Phi_{-\xi}$ makes sense even if
$Z$ has isolated components whenever $M_0 \cup -Z$ has no closed components,
giving rise to an alternative definition of $F_{\cW^b}$.
We now show that the two constructions agree.

\begin{prop}
  Let $\cW = (W,Z,\xi)$ be a sutured manifold cobordism from $(M_0,\g_0)$ to $(M_1,\g_1)$,
  possibly with $Z$ having isolated components, but such that $M_0 \cup -Z$ has no closed components.
  Then
  \[
  F_\cW = F_{\cW^s} \circ \Phi_{-\xi},
  \]
  where $F_\cW$ is the cobordism map defined
  in \cite{JCob}*{Definition~10.1} using the Honda--Kazez--Mati\'c gluing map,
  and $\Phi_{-\xi}$ is the gluing map from Section~\ref{sec:def-gluingmap}.
\end{prop}

\begin{proof}
  By definition, $F_\cW = i \circ F_{\cW'} = i \circ F_{(\cW')^s} \circ F_{(\cW')^b}$, where
  \[
  \cW' = (W,Z',[\xi']) \colon (M,\g) \to (M',\g') \sqcup (B,\delta)
  \]
  is the cobordism defined above, and
  \[
  i \colon \SFH((M_1, \g_1) \sqcup (B,\delta)) \stackrel{\sim}{\longrightarrow} \SFH(M_1,\g_1)
  \]
  is the canonical isomorphism.
  As $Z'$ has no isolated components,
  $F_{(\cW')^b}$ is the Honda--Kazez--Mati\'c gluing map $\Phi_{-\xi'}$,
  which agrees with our gluing map from Section~\ref{sec:def-gluingmap}
  by Theorem~\ref{thm:Contacthandlemaps=HKMmaps}.

  On the other hand, $\Phi_{-\xi} = \Phi_{-\xi_B} \circ \Phi_{-\xi'}$,
  where $\xi_B = \xi|_B$. Hence, it suffices to show that
  \[
  i \circ F_{(\cW')^s} = F_{\cW^s} \circ \Phi_{-\xi_B}.
  \]
  Let $\cW_B$ be the special cobordism from $(M_0 \cup -Z',\g_1 \cup \delta)$
  to $(M_0 \cup -Z,\g_1) \sqcup (B,\delta)$ obtained by pushing $\d B$ slightly
  into $Z'$, and attaching a 4-dimensional 3-handle along each component. Then
  \[
  (\cW')^s = (\cW^s \sqcup \id_{(B,\delta)}) \circ \cW_B.
  \]
  By definition, the contact 3-handle map $\Phi_{-\xi_B} = j \circ F_{\cW_B}$, where
  \[
  j \colon \SFH((M_0 \cup -Z,\g_1) \sqcup (B,\delta)) \stackrel{\sim}{\longrightarrow} \SFH(M_0 \cup -Z,\g_1)
  \]
  is the canonical isomorphism; see Section~\ref{sec:3-handlemap}.
  Hence, it suffices to show that
  \[
  i \circ F_{\cW^s \sqcup \id_{(B,\delta)}} = F_{\cW^s} \circ j.
  \]
  This holds since $F_{\cW^s \sqcup \id_{(B,\delta)}} = F_{\cW^s} \otimes \id_{\SFH(B,\delta)}$
  and $\SFH(B,\delta) \iso \bF_2$.
\end{proof}

\subsection{Turning around sutured cobordisms}

In this section, we use our contact handle maps to prove a first result about duality in sutured Floer homology.
If
\[
\cW=(W,Z,[\xi]) \colon (M,\g)\to (M',\g')
\]
is a cobordism of sutured manifolds, then we can form the cobordism
\[
\cW^\vee := (W,Z,[\xi]) \colon (-M',\g') \to (-M,\g)
\]
by reversing which ends of $W$ are viewed as incoming or outgoing.
The main result of this section is the following:

\begin{thm}\label{thm:firstduality}
If $\cW \colon (M,\g) \to (M',\g')$ is a balanced cobordism of sutured manifolds, and
$\cW^\vee$ is the cobordism obtained by turning around $\cW$, then
\[
F_{\cW^\vee} = (F_{\cW})^\vee,
\]
with respect to the trace pairing from Section~\ref{sec:tracepairing}.
\end{thm}

Suppose that $(M,\g)$ is a sutured submanifold of $(M',\g')$, and let $\xi$ be a contact structure
on $Z = -M' \setminus \Int(M)$ with dividing set $\g \cup \g'$ on the convex surface $\d Z$.
Consider the boundary cobordism
\[
\cW := (I \times M'/{\sim}, \{0\} \times Z,[\xi])
\]
from $(M,\g)$ to $(M',\g')$. Then $\cW^\vee$
is a sutured cobordism from $(-M',\g')$ to $(-M,\g)$. In general, $\cW^\vee$
will be neither a special cobordism nor a boundary cobordism.
It is the product of a boundary cobordism
\[
(\cW^\vee)^b \colon (-M',\g') \to (-M' \cup -Z,\g),
\]
and a special cobordism
\[
(\cW^\vee)^s \colon (-M' \cup -Z,\g) \to (-M,\g).
\]
The 4-manifold underlying $(\cW^\vee)^s$ is also $M'\times I/{\sim}$.
We need the following topological description of the special cobordism $(W^\vee)^s$:

\begin{lem}\label{lem:Morsetheoryturningaroundsuturedcobordism}
Suppose that $(M,\gamma)$ is a sutured submanifold of $(M',\g')$, and
\[
\cW^\vee := (I \times M'/{\sim}, \{0\} \times Z, [\xi]) \colon (-M',\g') \to (-M,\g)
\]
is the dual of the corresponding boundary cobordism, as described above,
and suppose that $(Z,\xi)$ has a contact handle decomposition
relative to $\d M$ with an associated Morse function $f \colon Z \to I$.
Then the special part $(W^\vee)^s$ of $\cW^\vee$ has a Morse function~$F$
whose critical points are in bijective correspondence with the critical points of~$f$.
Furthermore, if $p$ is a critical point of $f$
and $p'$ is the associated critical point of $F$, then
\[
\ind_{p'}(F)= 4 - \ind_{p}(f).
\]
The intersection of the descending manifold of a critical point
of $F$ with
\[
-M' \cup_{\d M'} -Z = -M \cup_{\d M} Z \cup_{\d M'} -Z
\]
is equal to the union of the ascending flow lines of the
corresponding critical point of $f$ in $Z$, together with their images in $-Z$.
\end{lem}

\begin{proof}
We first define an auxiliary function
\[
G \colon (I \times [-1,2])/(I \times \{2\}) \to I,
\]
where $I = [0,1]$. We require $G$ to satisfy the following:
\begin{itemize}
\item $G(t,s) = t$ for $s = -1$.
\item $\nabla G \neq 0$ for all $(t,s)$.
\item If $(t,s) \in I \times I$, then $(\d G/\d t)(t,s) = 0$ if and only if $t = \tfrac{1}{2}$.
\item If $(t,s) \in \{\tfrac{1}{2}\} \times I$, then $(\d G/ \d s) < 0$.
\item $G|_{\{0\} \times [-1,2]} \equiv G|_{I \times \{2\}} \equiv G|_{\{1\} \times [0,2]} \equiv 0$.
\end{itemize}
The graph of such a function $G(t,s)$ is shown in Figure~\ref{fig::54}.

\begin{figure}[ht!]
  \centering
  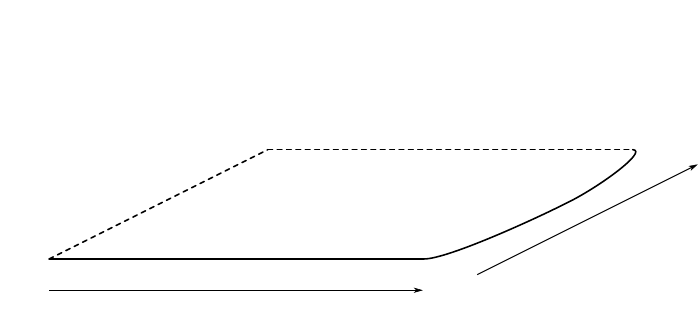
  \caption{An example of a function
  $G(t,s) \colon (I \times [-1,2])/(I \times \{2\}) \to I$ considered in
  Lemma~\ref{lem:Morsetheoryturningaroundsuturedcobordism}. \label{fig::54}}
\end{figure}

Let us view $M'$ as
\[
M' \iso M \cup  (\d M \times [-1,0]) \cup -Z \cup (\d M' \times [1,2]),
\]
and extend $f$ over all of $M'$ such that
\begin{itemize}
\item $f|_{M} \equiv -1$,
\item $f|_{\d M \times [-1,0]}(x,s) = s$, and
\item $f|_{\d M' \times [1,2]}(x,s) = s$.
\end{itemize}
We then consider the function $F \colon I \times M'/{\sim} \to I$ given by
\[
F(t,x) := G(t, f(x)).
\]
It is then straightforward to verify that $F$ has the stated properties.
\end{proof}

\begin{rem}
For an index 2 critical point of $f$, the attaching sphere of the
corresponding critical point of $F$ will be a knot $K$. The framing
of $K$ depends on some auxiliary choices, such as a choice of Riemannian
metric, and the precise choice of $G$. However, up to isotopy, the framing is
determined uniquely by the property that the framing of $K \cap Z$ is the
mirror of the framing of $K \cap (-Z)$.
\end{rem}

\begin{proof}[Proof of Theorem \ref{thm:firstduality}]
The claim was shown for special cobordisms in \cite{JCob}*{Theorem~11.8}.
Hence, it only remains to verify it for boundary cobordisms. By the
composition law for the gluing map, it is sufficient to prove the claim when
the boundary cobordism $\cW = (W,Z,[\xi])$ is formed by adding
a single contact $k$-handle for $k \in \{\,0,1,2,3\,\}$.
The cobordism map
\[
F_\cW \colon \SFH(M,\g) \to \SFH(M',\g')
\]
is the gluing map $\Phi_{-\xi}$, corresponding to the sutured submanifold
$(-M,-\g)$ of $(-M',-\g')$.

For a contact 0-handle, the map $\Phi_{-\xi}$ is the tautological one from $\SFH(M,\g)$ to
\[
\SFH(M',\g') \cong \SFH(M,\g) \otimes \SFH(D^2 \times [-1,1], S^1 \times \{0\})
\cong \SFH(M,\g) \otimes \bF_2.
\]
On the other hand, consider $\cW^\vee$ from $(-M',\g')$ to $(-M,\g)$.
Here $Z$ has an isolated component $Z_0$ corresponding to the contact 0-handle; i.e.,
$Z_0 \cap M = \emptyset$.
Hence, by \cite{JCob}*{Definition~10.1}, the map $F_{\cW^\vee}$ is defined by removing
a standard contact ball $B$ with connected dividing set $\delta$ on $\partial B$ from the interior of $Z_0$,
and adding $(B,\delta)$ to $(-M,\g)$.
The resulting cobordism is a product from $(-M',\g') = (-M,\g) \sqcup (B,\delta)$ to
$(-M',\g')$. Hence $F_{\cW^\vee}$ is $\id_{\SFH(-M',\g')}$, followed by the canonical
identification of $\SFH(-M',\g')$ with $\SFH(-M,\g)$, which is $\Phi_{-\xi}^\vee$.

Now suppose that $\cW$ is formed by adding a contact 1-handle to $(M,\g)$.
In this case, $F_\cW = \Phi_{-\xi}$, which is obtained by adding a strip to the boundary of
a Heegaard diagram for $(-M,-\g)$.
By Lemma~\ref{lem:Morsetheoryturningaroundsuturedcobordism}, the cobordism $\cW^\vee$
is obtained by gluing $(Z, -\xi)$ along $\d M'$ to $M' = M \cup_{\d M} -Z$, then attaching a
4-dimensional 3-handle.
Note that we attach $Z$ along $\d M'$, not $\d M$, so it becomes a contact 2-handle.
The 4-dimensional 3-handle is attached to a 2-sphere in $Z \cup_{\d M'} -Z$.
As in Lemma~\ref{lem:Morsetheoryturningaroundsuturedcobordism},
the 2-sphere is the union of the ascending flowlines of $f$ in $Z$
(i.e., the co-core of $Z$, viewed as a contact 1-handle attached to $-M$)
together with its image in the copy of $-Z$ that we glue onto $-M'$.
Diagrammatically, this is shown in Figure~\ref{fig::3}.
An easy model computation shows that this is equal to the dual of the contact 1-handle map.

\begin{figure}[ht!]
 \centering
 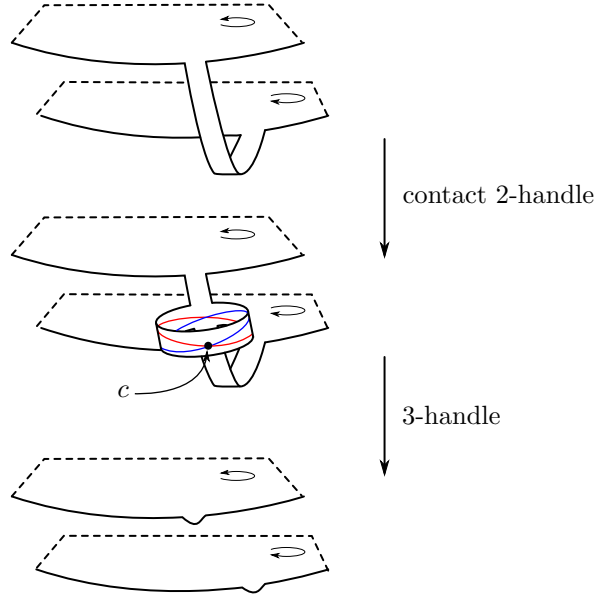
 \caption{Computing the cobordism map for a turned-around contact 1-handle.
 Orientations on the Heegaard surface are shown. \label{fig::3}}
\end{figure}

We now consider a sutured cobordism $\cW$ formed by a contact 2-handle
attachment to the sutured manifold $(M,\g)$. In this case, the dual
cobordism $\cW^\vee$ is formed by gluing $(Z,-\xi)$ to $M' = M \cup -Z$ along
$\d M'$, and then attaching a 4-dimensional 2-handle. Gluing $(Z,-\xi)$ to $M'$ is
now a contact 1-handle attachment. As described in
Lemma~\ref{lem:Morsetheoryturningaroundsuturedcobordism}, the knot that we
attach a 2-handle along is given as the union of the co-core of $Z$ (viewed
as a contact 2-handle) as well as its image in $-Z$.
Let $\cH'$ be an admissible diagram of $(-M',\g')$.
After adding the contact 1-handle and the 4-dimensional 2-handle, we get a diagram that is a compound
stabilization of $\cH'$. After performing a compound
destabilization, we get back to $\cH'$. An easy triangle map
computation for the 2-handle map shows that the composition of the triangle
map and the compound destabilization map is dual to the contact 2-handle map.
That the compound stabilization map agrees with the map from naturality is
shown in Proposition~\ref{prop:compoundstabilizationmap}. A schematic for the
turned-around contact 2-handle cobordism is shown in Figure~\ref{fig::4}. The
triangle map is shown in more detail in Figure~\ref{fig::5}.

\begin{figure}[ht!]
  \centering
  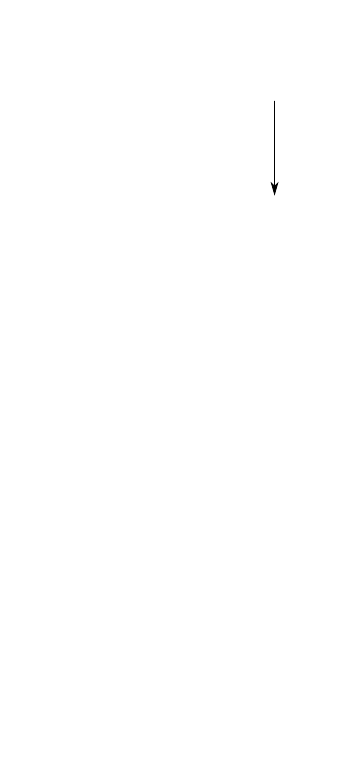
  \caption{Computing the cobordism map for a turned-around contact 2-handle.} \label{fig::4}
\end{figure}

\begin{figure}[ht!]
  \centering
  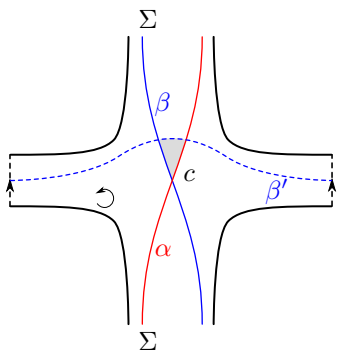
  \caption{The triangle map computation for the 4-dimensional 2-handle  map
  in the  cobordism map for a turned-around contact 2-handle. The orientation
  of the surface is shown. The two dashed lines with arrows on the left and
  right are identified. The only homology classes of triangles that have a
  vertex at $c$ have multiplicity 1 in the shaded region, and zero in the
  other regions shown.}
  \label{fig::5}
\end{figure}

Finally, we consider the case when $\cW$ is formed by adding a contact 3-handle.
In this case, $Z$ has an isolated component, and the cobordism map $F_\cW$
is obtained by removing a standard contact ball from the 3-ball we are adding.
Then the map is computed as a trivial gluing map, followed by a 4-dimensional 3-handle map.
The 4-dimensional 3-handle is attached along an embedded 2-sphere that is the
push-off of the boundary 2-sphere on which we are attaching the contact 3-handle.
The dual cobordism $\cW^\vee$ is formed by adding a contact 0-handle,
followed by a 4-dimensional 1-handle. Hence, the claim that
$F_\cW^\vee = F_{\cW^\vee}$ follows from the fact that
the 4-dimensional 1-handle and 3-handle maps are dual to each other,
as in \cite{OSTriangles}*{Theorem~3.5}.
\end{proof}

\section{Triangle cobordisms} \label{sec:trianglecobordisms}

If $\cT = (\S, \as, \bs, \gs)$ is a balanced sutured triple diagram,
then there is a natural sutured cobordism
\[
\cW_{\as,\bs,\gs} = (W_{\as,\bs,\gs}, Z_{\as,\bs,\gs}, \xi_{\as,\bs,\gs}),
\]
as in \cite[Section~5]{JCob}.
We note, however, that the construction of the contact structure
$\xi_{\as,\bs,\gs}$ in \cite[Section~5]{JCob} was incorrect, as it involved
gluing contact structures along annuli whose boundaries did not intersect the dividing set.
In this section, we will provide a different description, which we will take
as the definition.

Before describing the 4-manifold $W_{\as,\bs,\gs}$, we establish some
notation. For $\ts \in \{\as, \bs, \gs\}$, let $U_{\ts}$ be the sutured
compression body obtained from $\S \times I$ by attaching 3-dimensional 2-handles along
${\ts} \times \{0\} \subset \S \times \{0\}$.
We view $\S$ as being embedded into $\d U_{\ts}$ as $\S \times \{1\}$.
Similarly, we denote by $R_{\ts}$ the result of
surgering $\S \times \{0\}$ along $\ts \times \{0\}$. Using this notation,
\[
\partial U_{\ts} = \S \cup (\d \S \times I) \cup \bar{R}_{\ts},
\]
where $\bar{R}_{\ts}$ is the surface $R_{\ts}$ with the opposite orientation.

For $\ts$, $\ts' \in \{\as,\bs,\gs\}$, consider the 3-manifold
\[
M_{\ts,\ts'}:= U_{\ts} \cup_{\S} -U_{\ts'},
\]
and let $R_{\ts,\ts'}$ denote the surface
\[
\bar{R}_{\ts} \cup_{\d \S} R_{\ts'}.
\]
Here, we write $-U_{\ts'}$ for the 3-manifold $U_{\ts'}$ with the opposite
orientation. Note that, using this orientation convention, we have
\[
\d M_{\ts,\ts'} = R_{\ts,\ts'}.
\]
The oriented 1-manifold $\d \S$ has a natural embedding into
$\d M_{\ts,\ts'}$ that we denote $\g_{\ts,\ts'}$,
and $(M_{\ts,\ts'}, \g_{\ts,\ts'})$ is a sutured manifold with diagram $(\S,\ts,\ts')$.

\begin{figure}[ht!]
  \centering
\begingroup%
  \makeatletter%
  \providecommand\color[2][]{%
    \errmessage{(Inkscape) Color is used for the text in Inkscape, but the package 'color.sty' is not loaded}%
    \renewcommand\color[2][]{}%
  }%
  \providecommand\transparent[1]{%
    \errmessage{(Inkscape) Transparency is used (non-zero) for the text in Inkscape, but the package 'transparent.sty' is not loaded}%
    \renewcommand\transparent[1]{}%
  }%
  \providecommand\rotatebox[2]{#2}%
  \newcommand*\fsize{\dimexpr\f@size pt\relax}%
  \newcommand*\lineheight[1]{\fontsize{\fsize}{#1\fsize}\selectfont}%
  \ifx\svgwidth\undefined%
    \setlength{\unitlength}{112.76123625bp}%
    \ifx\svgscale\undefined%
      \relax%
    \else%
      \setlength{\unitlength}{\unitlength * \real{\svgscale}}%
    \fi%
  \else%
    \setlength{\unitlength}{\svgwidth}%
  \fi%
  \global\let\svgwidth\undefined%
  \global\let\svgscale\undefined%
  \makeatother%
  \begin{picture}(1,0.7745594)%
    \lineheight{1}%
    \setlength\tabcolsep{0pt}%
    \put(0,0){\includegraphics[width=\unitlength,page=1]{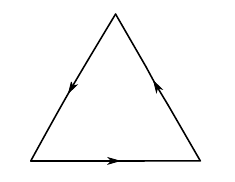}}%
    \put(0.69079499,0.42153744){\color[rgb]{0,0,0}\makebox(0,0)[lt]{\lineheight{0}\smash{\begin{tabular}[t]{l}$e_{\as}$\end{tabular}}}}%
    \put(0.4287637,0.00678979){\color[rgb]{0,0,0}\makebox(0,0)[lt]{\lineheight{0}\smash{\begin{tabular}[t]{l}$e_{\bs}$\end{tabular}}}}%
    \put(0.27481374,0.42153744){\color[rgb]{0,0,0}\makebox(0,0)[rt]{\lineheight{0}\smash{\begin{tabular}[t]{r}$e_{\gs}$\end{tabular}}}}%
    \put(0,0){\includegraphics[width=\unitlength,page=2]{fig56.pdf}}%
    \put(0.8727954,0.02862543){\color[rgb]{0,0,0}\makebox(0,0)[t]{\lineheight{0}\smash{\begin{tabular}[t]{c}$v_{\as,\bs}$\end{tabular}}}}%
    \put(0.12739514,0.02862543){\color[rgb]{0,0,0}\makebox(0,0)[t]{\lineheight{0}\smash{\begin{tabular}[t]{c}$v_{\bs,\gs}$\end{tabular}}}}%
    \put(0.48982642,0.74726168){\color[rgb]{0,0,0}\makebox(0,0)[t]{\lineheight{0}\smash{\begin{tabular}[t]{c}$v_{\gs,\as}$\end{tabular}}}}%
  \end{picture}%
\endgroup%

  \caption{A triangle $\Delta$, oriented as in the construction of
  $W_{\as,\bs,\gs}$. The boundary orientations of the edges $e_{\as},$
  $e_{\bs}$, and $e_{\gs}$ are also shown.}
  \label{fig::56}
\end{figure}

Let $\Delta$ be a regular triangle in $\C$ with edges $e_{\as}$, $e_{\bs}$, and $e_{\gs}$
appearing clockwise, and give $\Delta$ the complex orientation; see Figure~\ref{fig::56}.
We define the 4-manifold
\[
W_{\as,\bs,\gs} := (\Delta \times \S) \cup (e_{\as} \times U_{\as})
\cup (e_{\bs} \times U_{\bs}) \cup (e_{\gs} \times U_{\gs})/{\sim},
\]
where $\sim$ denotes gluing $\Delta \times \S$ to $e_{\ts} \times U_{\ts}$
along $e_{\ts} \times \S$ for $\ts \in \{\as, \bs, \gs\}$.
Furthermore, let
\[
Z_{\as,\bs,\gs}:= (\Delta \times \d \S) \cup (e_{\as}\times R_{\as})
\cup (e_{\bs}\times R_{\bs}) \cup (e_{\gs}\times R_{\gs}) \subset \d W_{\as,\bs,\gs},
\]
which we orient as the boundary of $W_{\as,\bs,\gs}$.
We then have the decomposition
\[
\d W_{\as,\bs,\gs} = -M_{\as,\bs}\cup -M_{\bs,\gs} \cup M_{\as,\gs} \cup Z_{\as,\bs,\gs}.
\]
Note that $M_{\as,\gs} = -M_{\gs,\as}$.
Using our orientation conventions,
\[
\d Z_{\as,\bs,\gs} = R_{\as,\bs} \sqcup R_{\bs,\gs} \sqcup \bar{R}_{\as,\gs}.
\]
There is a natural collection of sutures $\kappa_{\as,\bs,\gs}$ on $\d Z_{\as,\bs,\gs}$, defined as
\[
 \kappa_{\as,\bs,\gs}:= \{v_{\as,\bs}, v_{\bs,\gs}, v_{\gs,\as}\} \times \d \S =
 \g_{\as,\bs} \cup \g_{\bs,\gs} \cup \g_{\gs,\as},
\]
where $v_{\sigmas,\ts}$ is the vertex of $\Delta$ between $e_{\sigmas}$ and $e_{\taus}$.
Then
\[
R_+(\kappa_{\as,\bs,\gs}) = (\{v_{\as,\bs}\} \times R_{\bs}) \cup (\{v_{\bs,\gs}\} \times R_{\gs}) \cup
(\{v_{\gs,\as}\} \times R_{\as}).
\]

There is a natural contact structure $\xi_{\as,\bs,\gs}$ on
$Z_{\as,\bs,\gs}$ with dividing set $\kappa_{\as,\bs,\gs}$,
which is positive for the orientation of $Z_{\as,\bs,\gs}$ we have described. We delay our
description of $\xi_{\as,\bs,\gs}$ until Section~\ref{sec:xi-abgconstruction}.
The triple $\cW_{\as,\bs,\gs} = (W_{\as,\bs,\gs}, Z_{\as,\bs,\gs}, \xi_{\as,\bs,\gs})$
is a sutured manifold cobordism from $(M_{\as,\bs}, \g_{\as,\bs}) \sqcup (M_{\bs,\gs}, \g_{\bs,\gs})$
to $(M_{\as,\gs}, \g_{\as,\gs})$. The main goal of this section is to prove the following:

\begin{thm}\label{thm:trianglemapiscobmap}
Let $(\S,\as,\bs,\gs)$ be an admissible balanced sutured triple diagram.
Then the cobordism map
\[
F_{\cW_{\as,\bs,\gs}} \colon \CF(\S,\as,\bs) \otimes \CF(\S,\bs,\gs) \to \CF(\S,\as,\gs)
\]
is chain homotopic to the map $F_{\as,\bs,\gs}$ defined in \cite[Definition~5.13]{JCob}
that counts holomorphic triangles on the triple diagram~$(\S,\as,\bs,\gs)$.
\end{thm}

The proof occupies the remainder of this section. The first step is to
define the contact structure $\xi_{\as,\bs,\gs}$ in detail, and describe
some useful properties. Next, we introduce some special diagrams
called ``doubled diagrams'' and ``weakly conjugated diagrams'' that appear
when computing the cobordism map for $\cW_{\as,\bs,\gs}$.
We then compute convenient formulas for the contact gluing map for
$(Z_{\as,\bs,\gs}, \xi_{\as,\bs,\gs})$, and for the
4-dimensional handle attachments of the cobordism $\cW_{\as,\bs,\gs}$.
Finally, we put all the pieces together, using associativity of the triangle
maps and some other relations to show that the cobordism map
$F_{\cW_{\as,\bs,\gs}}$ is chain homotopic to the triangle map $F_{\as,\bs,\gs}$.

\subsection{The contact structure \texorpdfstring{$\xi_{\a,\b,\g}$}{} on \texorpdfstring{$Z_{\a,\b,\g}$}{Z}}
\label{sec:xi-abgconstruction}
The sutured manifold $(Z_{\as,\bs,\gs},\kappa_{\as,\bs,\gs})$ has a natural
contact structure $\xi_{\as,\bs,\gs}$ that we define by decomposing $Z_{\as,\bs,\gs}$
along convex surfaces. The construction generalizes to the case of sutured
multi-diagrams $(\S,\etas_1, \dots, \etas_n)$ for arbitrary $n \ge 1$;
though in this paper, we will only need it for $n \in \{1,2,3\}$.

Let $P_n$ be an $n$-gon, viewed as the complex unit disk with boundary divided into $n$ arcs
labeled $e_{\etas_1}, \dots, e_{\etas_n}$ clockwise. We write $v_{\etas_{i-1}, \etas_i}$
for the terminal endpoint of $e_{\etas_i}$ for $i \in \{1, \dots, n\}$, where $\etas_0 := \etas_n$.
Let
\begin{equation}
Z_{\etas_1,\dots,\etas_n} = (\Delta \times \d \S) \cup (e_{\etas_1} \times R_{\etas_1})
\cup \dots \cup (e_{\etas_n} \times R_{\etas_n})
\label{eq:Zabgdef}
\end{equation}
with sutures
\[
\kappa_{\etas_1, \dots, \etas_n} = \{v_{\etas_1, \etas_2}, \dots, v_{\etas_n, \etas_1}\} \times \d \S =
\g_{\etas_1, \etas_2} \cup \dots \cup \g_{\etas_n, \etas_1}
\]
and
\[
R_+(\kappa_{\etas_1, \dots, \etas_n}) = (\{v_{\etas_1,\etas_2}\} \times R_{\etas_2}) \cup \dots \cup
(\{v_{\etas_n,\etas_1}\} \times R_{\etas_1}).
\]

We now define the contact structure $\xi_{\etas_1, \dots, \etas_n}$ on
$(Z_{\etas_1, \dots, \etas_n}, \kappa_{\etas_1, \dots, \etas_n})$.
Let us write $Z_0$ for $P_n \times \d \S$, which is a union of $|\d \S|$ solid tori.
Let $\g_0$ consist of two parallel longitudinal sutures on each component of $\d Z_0$,
such that their projections to $P_n$ wind
positively around $\d P_n$ with respect to the orientation of $P_n$
(i.e., they wind counterclockwise around $\d P_n \subset \C$).
Since $(Z_0,\g_0)$ is product disc decomposable, it admits a unique
tight contact structure $\xi_0$, up to equivalence, which has $\d Z_0$ as a convex
surface with dividing set $\g_0$.

For $\ts \in \{\etas_1, \dots, \etas_n\}$, write $(Z_{\ts}, \g_{\ts})$ for the product sutured manifold
$(e_{\ts} \times R_{\ts}, \{m_{\ts}\} \times \d R_{\ts})$, where $m_{\ts}$ is the
midpoint of $e_{\ts}$. The sutured manifold $(Z_{\ts}, \g_{\ts})$ is
product disc decomposable, and hence admits a unique tight contact structure
$\xi_{\ts}$ with dividing set $\g_{\ts}$, up to equivalence. Let
$s_{\ts}$ denote a small translation of the suture
$\g_{\ts}$ for $\ts \in \{\etas_1, \dots, \etas_n\}$, such that
each component of $s_{\ts}$ intersects the corresponding
component of $\g_{\ts}$ transversely at exactly two points.
Let $N(s_{\ts})\subset \d Z_{\ts}$ denote
a small regular neighborhood of $s_{\ts}$ that intersects $\g_{\ts}$
along two arcs. Using Giroux's Legendrian realization principle, we may
assume that each $\d N(s_{\ts})$ is Legendrian.

We now describe how the subsurfaces $N(s_{\ts})\subset Z_{\ts}$ are glued to $Z_0$.
By picking $\g_0$ appropriately, we may assume that each component of
$\{m_{\ts}\} \times \d \S$ intersects $\g_0$ transversely at two points.
For each $\ts \in \{\etas_1, \dots, \etas_n\}$, we pick a small neighborhood
$N(\{m_{\ts}\}\times \d \S) \subset \d Z_0$, such that each component of
$N(\{m_{\ts}\}\times \d \S)$  intersects $\g_0$ along two arcs. Using
Legendrian realization, we may assume that $\d N(\{m_{\ts}\} \times \d \S )$ is Legendrian.

We glue $(Z_{\ts},\xi_{\ts})$ for every $\ts \in \{\etas_1, \dots, \etas_n\}$ to $(Z_0, \xi_0)$ by
identifying $N(s_{\ts})$ and $N(\{m_{\ts}\} \times \d \S)$.
We let $\xi_{\etas_1, \dots, \etas_n}$ be the resulting contact structure. After rounding the
Legendrian corners, the contact structure $\xi_{\etas_1, \dots, \etas_n}$ has  dividing
set isotopic to $\kappa_{\etas_1, \dots, \etas_n}$. This is shown in
Figure~\ref{fig::57} for $n = 3$ and $(\etas_1, \etas_2, \etas_3) = (\as, \bs, \gs)$.
As $N(s_{\ts})$ is unique up to isotopy, $\xi_{\etas_1, \dots, \etas_n}$ is well defined up to equivalence.

\begin{figure}[ht!]
  \centering
  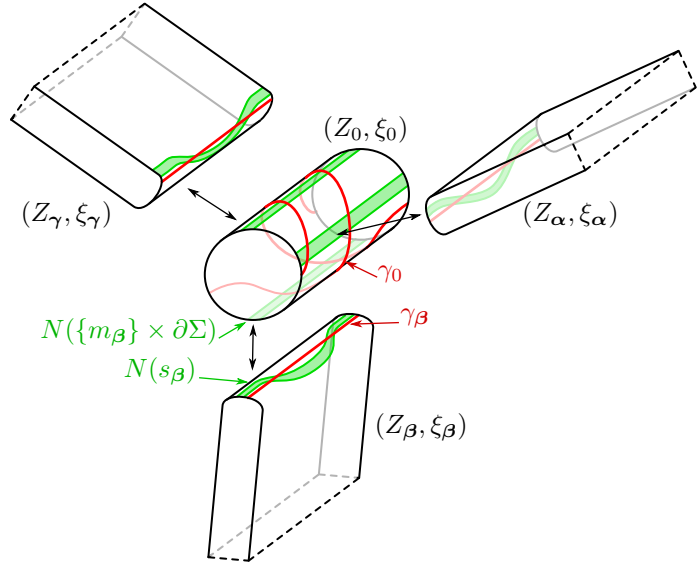
  \caption{Constructing the contact structure $\xi_{\as,\bs,\gs}$ on
  $Z_{\as,\bs,\gs}$ by gluing along convex surfaces. We glue
  $(Z_{\as},\xi_{\as})$, $(Z_{\bs},\xi_{\bs})$, and $(Z_{\gs},\xi_{\gs})$ to
  $(Z_0,\xi_0)$ along the green strips, which have Legendrian boundary. The
  faces on the front side are identified with the faces on the back.}
  \label{fig::57}
\end{figure}

The following will be useful throughout the paper:

\begin{lem}\label{lem:xi-ab-isproduct}
If $(\S,\as,\bs)$ is a sutured Heegaard diagram, then the sutured
manifold $(Z_{\as,\bs}, \kappa_{\as,\bs})$ is diffeomorphic to
\[
\left(I \times R_{\as,\bs} , (-\{0\} \times \g_{\as,\bs}) \cup (\{1\} \times \g_{\as,\bs})\right),
\]
and the contact structure $\xi_{\as,\bs}$ is isotopic to the $I$-invariant contact structure on
$I \times R_{\as,\bs} $ with dividing set $\{t\} \times \g_{\as,\bs}$ on $\{t\} \times R_{\as,\bs}$
for every $t \in I$.
\end{lem}

\begin{proof}
By construction, $Z_{\as,\bs}$ is obtained by gluing $e_{\as} \times R_{\as}$ and $e_{\bs} \times R_{\bs}$
to $P_2 \times \d \S$. This is diffeomorphic to gluing $I \times R_{\as}$ to $I \times R_{\bs}$
by identifying $(t,x) \in I \times \d \S$ with $(1-t,x) \in I \times \d \S$, where $\d R_{\as} = \d R_{\bs} = \d \S$.
Hence $Z_{\as,\bs}$ is diffeomorphic to $I \times (\bar{R}_{\as} \cup R_{\bs}) = I \times R_{\as,\bs}$.
Under this diffeomorphism, the sutures $\kappa_{\as,\bs} = \g_{\as,\bs} \cup \g_{\bs,\as}$ are mapped to
$(-\{0\} \times \g_{\as,\bs}) \cup (\{1\} \times \g_{\as,\bs})$, since $\g_{\bs,\as} = -\g_{\as,\bs}$.
In particular, $R_+(\kappa_{\as,\bs})$ is identified with $(\{0\} \times R_{\as}) \cup (\{1\} \times R_{\bs})$.

To see that the contact structure $\xi_{\as,\bs}$ is isotopic to the $I$-invariant contact
structure in the statement, we will construct a decomposition of
the latter along convex annuli that cuts
$I\times R_{\as,\bs} $ into the disjoint union of the three contact manifolds
$(Z_{\as}, \xi_{\as})$, $(Z_{\bs}, \xi_{\bs})$, and $(Z_0,\xi_0)$.
Let $s_1$, $s_2 \subset R_{\as,\bs}$ be two disjoint curves that
are both small translates of the dividing set $\g_{\as,\bs} = \d \S$.
We assume that $s_1$ and $s_2$ each intersects $\g_{\as,\bs}$
transversely at two points. Using Legendrian realization, we can assume that
both $s_1$ and $s_2$ are Legendrian.

We will cut $I\times R_{\as,\bs}$ along the two annuli $I\times s_1$ and
$I\times s_2$. We note that since $s_1$ and $s_2$ are Legendrians, the
characteristic foliations on $I\times s_1$ and $I\times s_2$ are simple to
describe. They consist of the horizontal leaves $\{t\}\times s_i$ for all
$t \in I$, as well as two vertical lines of singularities along $I\times \{p\}$ for
$p \in s_i \cap \g_{\as,\bs}$. This is shown in Figure~\ref{fig::59}. We
note that the characteristic foliation satisfies the Poincar\'{e}-Bendixson
property (the limit set of a flow line consists of a singular point, a
periodic orbit, or a finite union of singular points and connecting orbits), and has
no closed orbits or retrograde saddle connections. So, by the work of
Giroux~\cite{GirouxCharacteristicFoliations}, the surfaces $I\times s_i$ are convex.
Furthermore, the dividing set on $I\times s_i$
consists of two curves of the form $I\times \{q\}$, where
$q \in s_i \setminus \g_{\as,\bs}$. See also \cite[Example~2.24]{Etnyre}.

After rounding the Legendrian corners that appear when we cut along
$I\times s_i$, we obtain the disjoint union of the sutured manifolds
$(Z_{\as},\xi_{\as})$, $(Z_{\bs},\xi_{\bs})$, and $(Z_0,\g_0)$.
Furthermore, the contact structures obtained on the three pieces are isotopic
to $\xi_{\as}$, $\xi_{\bs}$, and $\xi_0$, since they are tight by Giroux's
criterion~\cite{HondaClassI}*{Theorem~3.5}, and $\xi_{\as}$, $\xi_{\bs}$, and
$\xi_0$ are the unique tight contact structures, by definition. A picture of
the convex decomposition of $I\times R_{\as,\bs}$ along $I\times s_1$ and
$I\times s_2$ is shown in Figure~\ref{fig::58}.
\end{proof}

\begin{figure}[ht!]
    \centering
    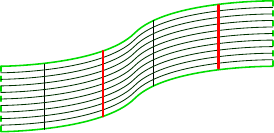
    \caption{The characteristic foliation (thin black) and the dividing set
    (thick red) on the annulus $I\times s_i \subset I\times  R_{\as,\bs}$.
    The vertical black lines consist of singularities. The
    left- and right-hand sides of the surface are identified to form an annulus.}
    \label{fig::59}
\end{figure}

\begin{figure}[ht!]
    \centering
    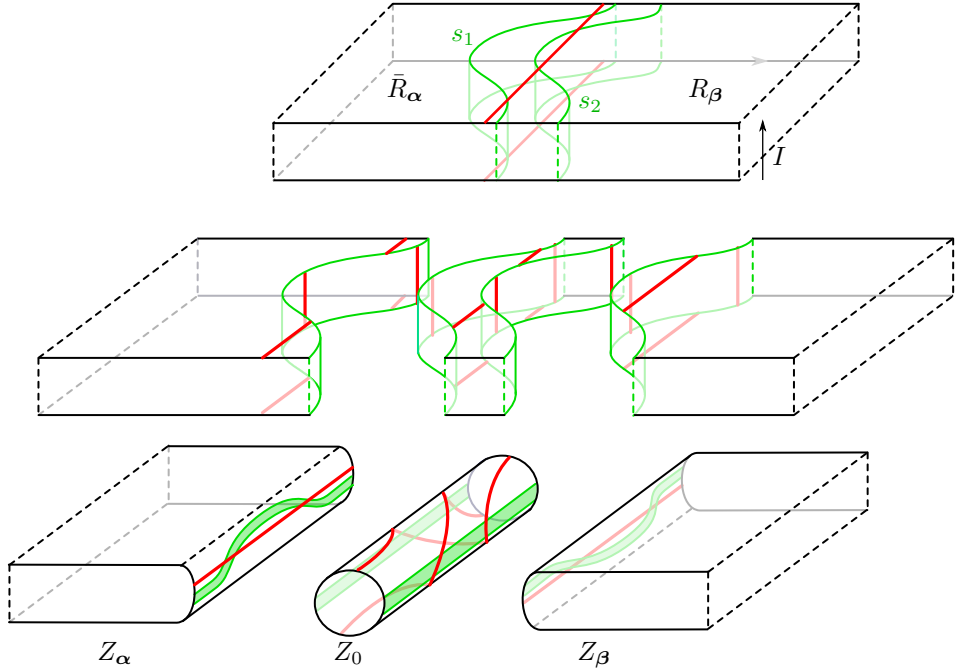
    \caption{Decomposing $I\times R_{\as,\bs}$ (top) along two convex annuli
    $I\times s_1$ and $I\times s_2$, we obtain the middle picture. After
    rounding the Legendrian corners, we obtain the bottom picture, which is
    the disjoint union of $Z_{\as}$, $Z_{\bs}$, and $Z_0$. The front and back
    sides of each picture are identified.}
    \label{fig::58}
\end{figure}

Finally, we need one additional description of the contact structure
$\xi_{\as,\bs,\gs}$ on $Z_{\as,\bs,\gs}$, in terms of gluing
$(Z_{\as,\bs},\xi_{\as,\bs})$ and $(Z_{\bs,\gs}, \xi_{\bs,\gs})$ together. By
Lemma~\ref{lem:xi-ab-isproduct}, the contact manifolds
$(Z_{\as,\bs},\xi_{\as,\bs})$ and $(Z_{\bs,\gs},\xi_{\bs,\gs})$ are
contactomorphic to $I\times R_{\as,\bs}$ and $I\times R_{\bs,\gs}$, respectively,
with $I$-invariant contact structures. We pick automorphisms of each of the
surfaces $R_{\as,\bs}$ and $R_{\bs,\gs}$, supported in
neighborhoods of $\g_{\as,\bs}$ and $\g_{\bs,\gs}$, that perturb
the dividing sets $\g_{\as,\bs}$ and $\g_{\bs,\gs}$ slightly. Write
$s_{\as,\bs} \subset R_{\as,\bs}$ and $s_{\bs,\gs} \subset R_{\bs,\gs}$ for the
images of $\g_{\as,\bs}$ and $\g_{\bs,\gs}$, respectively. We assume
that $s_{\as,\bs}$ and $\g_{\as,\bs}$ intersect transversely at two
points, and similarly for $s_{\bs,\gs}$ and $\g_{\bs,\gs}$. Let
$\bar{R}_{\bs}' \subset \{0\}\times R_{\as,\bs}$, oriented as the boundary
of $I\times R_{\as,\bs}$, denote the image of
$\{0\} \times \bar{R}_{\bs}  \subset \{0\}\times  R_{\as,\bs} $, and let
$R_{\bs}' \subset \{0\} \times R_{\bs,\gs} $ denote the image of
$\{0\}\times R_{\bs} \subset \{0\}\times  R_{\bs,\gs}$.
Using the Legendrian realization principle,
we may assume that $s_{\as,\bs}$ and $s_{\bs,\gs}$ are Legendrian.
The following description of $(Z_{\as,\bs,\gs},\xi_{\as,\bs,\gs})$
will be useful for our purposes:

\begin{figure}[ht!]
 \centering
 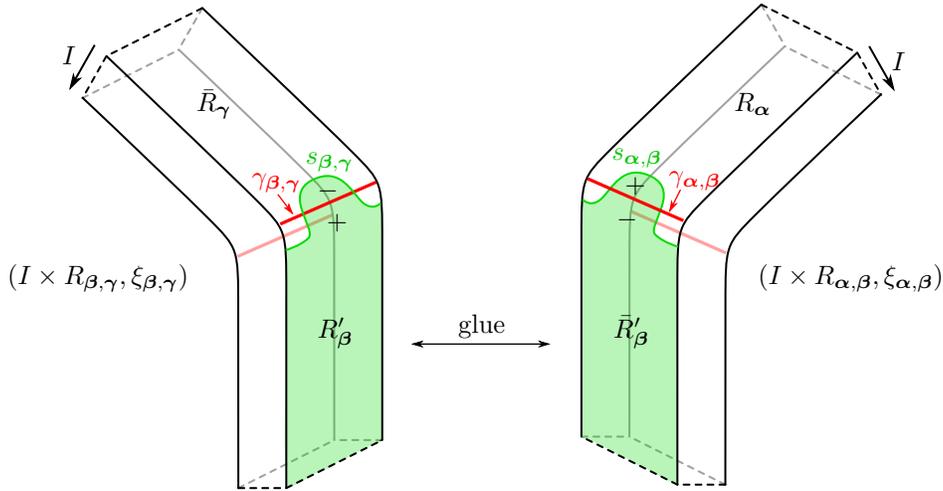
 \caption{Gluing $(I\times R_{\as,\bs} , \xi_{\as,\bs})$ and
 $(I\times R_{\bs,\gs} , \xi_{\bs,\gs})$ together to obtain
 $(Z_{\as,\bs,\gs}, \xi_{\as,\bs,\gs})$. We glue along the green shaded
 subsurfaces labeled $R_{\bs}'$ and $\bar{R}_{\bs}'$.  The dividing sets
 $\g_{\as,\bs}$ and $\g_{\bs,\gs}$ are shown in red. The curves
 $s_{\as,\bs}$ and $s_{\bs,\gs}$ are small perturbations of the dividing
 sets, and are Legendrian. \label{fig::51}}
\end{figure}

\begin{lem}\label{lem:xi-abg-istwoIinvariantgluedtogether}
The contact structure $(Z_{\as,\bs,\gs}, \xi_{\as,\bs,\gs})$ is equivalent to
the one obtained by gluing $(I\times R_{\as,\bs}, \xi_{\as,\bs})$ and
$(I\times R_{\bs,\gs}, \xi_{\bs,\gs})$ together along the surfaces $R_{\bs}'$
and $\bar{R}_{\bs}'$ (which are convex with Legendrian boundary), described
above; see Figure~\ref{fig::51}.
\end{lem}

\begin{proof}
As in Lemma~\ref{lem:xi-ab-isproduct}, we work backwards, and provide a
convex decomposition of $(Z_{\as,\bs,\gs},\xi_{\as,\bs,\gs})$ into the
disjoint union of $(I\times R_{\as,\bs}, \xi_{\as,\bs})$ and
$(I\times R_{\bs,\gs}, \xi_{\bs,\gs})$. By definition,
$Z_{\as,\bs,\gs}$ is the union of $Z_{\as}$, $Z_{\bs}$, $Z_{\gs}$, and $Z_0$.
Let us view $Z_{\bs}$ as $I\times R_{\bs}$ with $I$-invariant contact
structure, and Legendrian corners along $ \{0,1\}\times \d R_{\bs}$, as on the
right side of the middle level of Figure~\ref{fig::58}. We start with the
surface $S_{\bs} := \{\tfrac{1}{2}\}\times R_{\bs}$, which is convex with
Legendrian boundary. We can view $\d S_{\bs}$ as a Legendrian arc on
$\d Z_0$ that intersects the dividing set on $\d Z_0$ at two points. There is
an annulus $S_0 \subset Z_0$ with one boundary component on $\d S_{\bs}$, and
another boundary component on $\d Z_0$ between $Z_{\gs}$ and $Z_{\as}$.
Furthermore, after perturbing the surface $S_0$ to be convex, it cuts
$(Z_0,\xi_0)$ into two copies of $(Z_0,\xi_0)$. To see this, we consider
$(Z_0,\xi_0)$ as the $I$-invariant contact manifold (with Legendrian corners)
in the center of the middle of Figure~\ref{fig::58}. In this picture, an
example of an annulus cutting $(Z_0,\xi_0)$ into two copies of $(Z_0,\xi_0)$
would be the intersection of $Z_0$ with the slice $\{\tfrac{1}{2}\}\times R_{\as,\bs}$.
We let our decomposing surface $S$ be the union $S_{\bs} \cup S_0$.
This is shown schematically in Figure~\ref{fig::60}. When we cut along $S$, we get two
components, one of which is $Z_{\as} \cup Z_0 \cup Z_{\bs}$, and one of which is
$Z_{\bs} \cup Z_0 \cup Z_{\gs}$.  Lemma~\ref{lem:xi-ab-isproduct} identifies
these latter two contact manifolds with $I\times R_{\as,\bs}$ and
$I\times R_{\bs,\gs}$ with $I$-invariant contact structures.
\end{proof}

\begin{figure}[ht!]
 \centering
 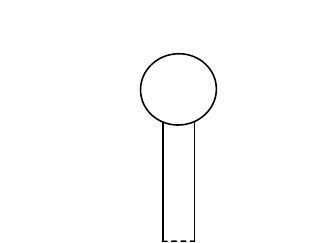
 \caption{Decomposing $Z_{\as,\bs,\gs}$ along the convex surface
 $S = S_{\bs} \cup S_0$ into $Z_{\as} \cup Z_0 \cup Z_{\bs}$ and $Z_{\bs} \cup
 Z_0 \cup Z_{\gs}$.  \label{fig::60}}
\end{figure}

\subsection{A handle decomposition of \texorpdfstring{$(W_{\a,\b,\g})^s$}{W}}

Recall that $\cW_{\as,\bs,\gs}$ is a sutured manifold cobordism from
$(M_{\as,\bs}, \g_{\as,\bs}) \sqcup (M_{\bs,\gs}, \g_{\bs,\gs})$ to
$(M_{\as,\gs}, \g_{\as,\gs})$. The boundary cobordism $(\cW_{\as,\bs,\gs})^b$
corresponds to gluing $(Z_{\as,\bs,\gs}, -\xi_{\as,\bs,\gs})$ to
\[
(-M_{\as,\bs}, -\g_{\as,\bs}) \sqcup (-M_{\bs,\gs}, -\g_{\bs,\gs}).
\]
In light of Lemma~\ref{lem:xi-abg-istwoIinvariantgluedtogether}, we can topologically
view this as gluing $\bar{R}_{\bs} \subset -\d M_{\as,\bs}$ to
$R_{\bs} \subset -\d M_{\bs,\gs}$. Hence, the special
cobordism $(W_{\as,\bs,\gs})^s$ goes from $(M_{\as,\bs} \cup_{R_{\bs}}
M_{\bs,\gs}, \g_{\as,\gs})$ to $(M_{\as,\gs},\g_{\as,\gs})$. In this section,
we give a topological description of the special cobordism
$(W_{\as,\bs,\gs})^s$ in terms of 4-dimensional handle attachments.

A handle decomposition of $(W_{\as,\bs,\gs})^s$ can be constructed from a
sutured Morse function on $U_{\bs}$, as we now describe. Let
$f_{\bs} \colon U_{\bs} \to I$ be a Morse function induced by the diagram
$(\S,\bs)$; i.e, we view $U_{\bs}$ as a collection of 2-handles glued to
$\S \times I$ along $\bs \times \{0\} \subset \S \times \{0\}$.
Furthermore, we pick $f_{\bs}$ such that $f_{\bs}^{-1}(1) = \S \times \{1\}$,
$f_{\bs}^{-1}(0) = R_{\bs}$, and $f_{\bs}(y,t) = t$ for $y \in \d \S$.
We also assume that $f_{\bs}$ has $|\bs|$ index 1 critical points, whose
ascending manifolds intersect $\S \times \{1\}$ along the $\bs$ curves,
and that $f_{\bs}$ has no other critical points. For a curve $\b_i \in \bs$,
let $\lambda_i \subset U_{\bs}$ denote the descending manifold of the
critical point of $f_{\bs}$ corresponding to $\b_i$. We have the following
topological description of $W_{\as,\bs,\gs}$:

\begin{lem}\label{lem:Wabgspecialcobordism2handle}
The special cobordism $(W_{\as,\bs,\gs})^s$ from $(M_{\as,\bs} \cup_{R_{\bs}}
M_{\bs,\gs}, \g_{\as,\gs})$ to $(M_{\as,\gs},\g_{\as,\gs})$ consists of
$|\bs|$ 2-handle attachments. The 2-handles are attached along the link formed
by concatenating the arcs $\lambda_i \subset U_{\bs}$
(that have boundary on $\bar{R}_{\bs}$) with their
reflections in $-U_{\bs}$. The framing on this link is determined by picking
an arbitrary framing on $\lambda_i$, and concatenating it with the
mirrored framing on the image of $\lambda_i$ in $-U_{\bs}$.
\end{lem}

\begin{proof}
This can be proven similarly to
Lemma~\ref{lem:Morsetheoryturningaroundsuturedcobordism}, by using a Morse
function built from $f_{\bs}$ as well as another auxiliary function. We leave
the details to the reader.
\end{proof}

\subsection{Arc slides and bases of arcs}

In this section, we describe some basic topological facts about arc decompositions
of surfaces with boundary.

\begin{define}\label{def:allowablearcslice} Suppose that $\Sigma$ is a
 compact, orientable surface with no closed components,
 and $\cI\subset \d \Sigma$ is a collection of subarcs of $\d \Sigma$,
 such that each component of $\d \Sigma$ contains at least one component of $\cI$.
 \begin{enumerate}
 \item We say that   $\ve{a}=\{a_1,\dots, a_n\}$ is an \emph{arc basis} for $(\Sigma,\cI)$
 if $\ve{a}$ is a set of pairwise disjoint,
 properly embedded arcs with boundary on $\cI$
 that form a basis of $H_1(\Sigma,\cI)$ (or, equivalently, $\ve{a}$ cuts $\Sigma$
  into a collection of closed disks,
 each containing exactly one component of $\d \Sigma \setminus \cI$).
 \item We say an arc $a_i'$
 is formed by an \emph{allowable arc slide} of $a_i$ across $a_j$ if
 there is an embedded hexagon $P_6\subset \Sigma$, whose boundary consists
 of $a_i$, $a_j$, $a_i'$, as well as three subarcs of $\cI$,
 and is otherwise disjoint from the arcs $a_1,\dots, a_n$; see Figure~\ref{fig::75}.
 \end{enumerate}
\end{define}

\begin{figure}[ht!]
 \centering
\begingroup%
  \makeatletter%
  \providecommand\color[2][]{%
    \errmessage{(Inkscape) Color is used for the text in Inkscape, but the package 'color.sty' is not loaded}%
    \renewcommand\color[2][]{}%
  }%
  \providecommand\transparent[1]{%
    \errmessage{(Inkscape) Transparency is used (non-zero) for the text in Inkscape, but the package 'transparent.sty' is not loaded}%
    \renewcommand\transparent[1]{}%
  }%
  \providecommand\rotatebox[2]{#2}%
  \newcommand*\fsize{\dimexpr\f@size pt\relax}%
  \newcommand*\lineheight[1]{\fontsize{\fsize}{#1\fsize}\selectfont}%
  \ifx\svgwidth\undefined%
    \setlength{\unitlength}{177.81975419bp}%
    \ifx\svgscale\undefined%
      \relax%
    \else%
      \setlength{\unitlength}{\unitlength * \real{\svgscale}}%
    \fi%
  \else%
    \setlength{\unitlength}{\svgwidth}%
  \fi%
  \global\let\svgwidth\undefined%
  \global\let\svgscale\undefined%
  \makeatother%
  \begin{picture}(1,0.3678823)%
    \lineheight{1}%
    \setlength\tabcolsep{0pt}%
    \put(0,0){\includegraphics[width=\unitlength,page=1]{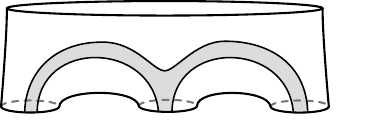}}%
    \put(0.26012138,0.16139171){\color[rgb]{0,0,0}\makebox(0,0)[t]{\lineheight{1.25}\smash{\begin{tabular}[t]{c}$a_i$\end{tabular}}}}%
    \put(0.61530422,0.158612){\color[rgb]{0,0,0}\makebox(0,0)[t]{\lineheight{1.25}\smash{\begin{tabular}[t]{c}$a_j$\end{tabular}}}}%
    \put(0.35957961,0.24495247){\color[rgb]{0,0,0}\makebox(0,0)[lt]{\lineheight{1.25}\smash{\begin{tabular}[t]{l}$a_i'$\end{tabular}}}}%
    \put(0.90793413,0.17640714){\color[rgb]{0,0,0}\makebox(0,0)[t]{\lineheight{1.25}\smash{\begin{tabular}[t]{c}$\Sigma$\end{tabular}}}}%
    \put(0,0){\includegraphics[width=\unitlength,page=2]{fig75.pdf}}%
    \put(0.07752587,0.00396513){\color[rgb]{0,0,0}\makebox(0,0)[t]{\lineheight{1.25}\smash{\begin{tabular}[t]{c}$\cI$\end{tabular}}}}%
    \put(0.44968224,0.00396513){\color[rgb]{0,0,0}\makebox(0,0)[t]{\lineheight{1.25}\smash{\begin{tabular}[t]{c}$\cI$\end{tabular}}}}%
    \put(0.81858506,0.00396513){\color[rgb]{0,0,0}\makebox(0,0)[t]{\lineheight{1.25}\smash{\begin{tabular}[t]{c}$\cI$\end{tabular}}}}%
  \end{picture}%
\endgroup%

 \caption{An allowable arc slide of $a_i$ across $a_j$.}
 \label{fig::75}
\end{figure}

The reader may compare the following to \cite{HKMSutured}*{Lemma~3.3}:

\begin{lem}\label{lem:allowablearcslides}
 Suppose $\Sigma$ is a compact, orientable
 surface with boundary and no closed components, and $\cI$ is a collection
 of pairwise disjoint, closed subintervals of $\d \Sigma$, such that each
 boundary component of $\d \Sigma$ contains at least one component
 of $\cI$. If $\ve{a}$ and $\ve{a}'$
 are two arc bases of $(\Sigma,\cI)$, then $\ve{a}$
 can be obtained from $\ve{a}'$
 by a sequence of allowable arc slides.
\end{lem}

\begin{proof}
Pick another collection of closed, pairwise
disjoint subarcs $\cI'\subset \Int(\d \Sigma \setminus \cI)$, such that each
component of $\d \Sigma\setminus \cI$ contains exactly
one component of $\cI'$. We view $\Sigma$ as a cobordism from
$\cI$ to $\cI'$, with horizontal boundary $\cI\cup \cI'$,
and vertical boundary $\d \Sigma\setminus \Int(\cI\cup \cI')$.
We can view each basis of arcs $\ve{a}$ as corresponding to a
Morse function $f$ and gradient-like vector field $v$ on $\Sigma$,
such that $f$ has only index 1 critical points with stable manifolds $\ve{a}$.
Furthermore, $f$ is minimized along $\cI$, maximized along $\cI'$,
and the components of $\d \Sigma\setminus \Int (\cI\cup \cI')$ are flow lines of $v$.

More generally, we can consider pairs $(f,v)$ of Morse functions $f$ with gradient-like vector fields $v$
such that $f$ is minimized on $\cI$, maximized on $\cI'$, has no critical points along $\d \Sigma$,
but is allowed to have critical points of index 0, 1, and 2. Assuming $(f,v)$ is also Morse--Smale,
we can construct a graph $\Gamma \subset \Sigma$ whose vertices are the index~0 critical points of $f$
and the points of $\cI$ that flow to index~1 critical points along $v$,
and whose edges are the stable manifolds of the index~1 critical points.
The graph $\Gamma$ intersects $\d \Sigma$ along a collection of valence~1 vertices in $\cI$.

With this in mind, we say a graph $\Gamma\subset \Sigma$ is a \emph{decomposing graph}
of $(\Sigma,\cI)$ if $\Gamma$ intersects $\d \Sigma$ in a collection of valence 1
vertices in $\cI$, and  the interior of each component of $\Sigma\setminus \Gamma$ is
homeomorphic to an open 2-ball that contains at most one component of $\d \Sigma\setminus \cI$.
Note that, for a decomposing graph of an orientable surface,
Definition~\ref{def:allowablearcslice} modifies easily to describe an arc slide
of two edges of $\Gamma$ that meet at a vertex in $\Int(\Sigma)$, and are consecutive with
respect to the cyclic ordering of the edges adjacent to the vertex.

By interpreting decomposing graphs in terms of Morse functions and gradient-like vector fields,
and considering the bifurcations occurring in generic 1-parameter families of smooth functions,
it follows that any two decomposing graphs can be
connected by a sequence of the following moves and their inverses:
\begin{enumerate}[label=($G$-\arabic*)]
\item\label{move:G1} (Index 0/1 births) Adding a vertex $v\in \Int(\Sigma)\setminus \Gamma$,
  as well as a new edge $e$ connecting $v$ to an existing vertex of
  $\Gamma\setminus \d \Sigma$, or to a point in $\cI\setminus \Gamma$.
\item\label{move:G2} (Index 1/2 births) Adding a new edge to $\Gamma$,
 whose interior is contained in $\Sigma\setminus \Gamma$,
 and which has both endpoints on the same vertex of $\Gamma\setminus \d \Sigma$.
\item\label{move:G3} (Arc slides) An allowable arc slide of two adjacent
  edges along $\cI$, or an arc slide of a pair of edges
  that meet at a vertex of $\Gamma \cap \Int (\Sigma)$, and are consecutive with
  respect to the cyclic ordering of the edges adjacent to the vertex.
 \end{enumerate}
Note that, if an index 1/2 birth occurs in a generic 1-parameter family of Morse functions
and gradient-like vector fields, then the corresponding decomposing graph $\Gamma$ changes
by adding an edge $e$ to a component $B$ of $\Sigma \setminus \Gamma$ with both endpoints
on $\d B \setminus (\d \S \setminus \cI)$. By assumption, $B$ is an open 2-ball
and $\d B \cap (\d \S \setminus \cI)$ is either empty or a single arc.
We can add $e$ using moves~\ref{move:G1}--\ref{move:G3}.
Indeed, using move~\ref{move:G1}, we add an edge $e'$ to $\Gamma$ with one endpoint along $\d B \cap \cI$
on the subarc of $\d B \setminus (\d \S \setminus \cI)$ between the endpoints of $e$,
and one endpoint at a new interior vertex $v \in B$. We then add a loop edge $e''$ at $v$ using move~\ref{move:G2}.
We slide the endpoints of $e''$ along the two sides of $e'$ to $\cI$, and finally
slide them to the endpoints of $e$ using move~\ref{move:G3}.

If $\ve{\Gamma} = (\Gamma_1,\dots, \Gamma_n)$ is a sequence of decomposing graphs
such that consecutive terms differ up to
isotopy by an instance of moves~\ref{move:G1}--\ref{move:G3} and their inverses,
then we say that $\ve{\Gamma}$ is a \emph{Cerf sequence}.
We define $n_1(\ve{\Gamma})$ to be the number of moves~\ref{move:G1},
and $n_2(\ve{\Gamma})$ to be the number of moves~\ref{move:G2} that occur in $\ve{\Gamma}$.
Note that, if $\Gamma_1$ and $\Gamma_n$ are arc bases, then $n_1(\ve{\Gamma}) = n_2(\ve{\Gamma}) = 0$
if an only if $\ve{\Gamma}$ consists only of allowable arc slides.

We will now show that if $\ve{\Gamma} = (\Gamma_1,\dots, \Gamma_n)$ is
a Cerf sequence  of decomposing graphs connecting two arc bases
$\ve{a}$ and $\ve{a}'$, and $n_1(\ve{\Gamma}) > 0$,
then there is a Cerf sequence $\ve{\Gamma}'$ connecting $\ve{a}$ and $\ve{a}'$ with
\begin{equation}
n_1(\ve{\Gamma}') = n_1(\ve{\Gamma}) - 1 \quad \text{and} \quad
n_2(\ve{\Gamma}') = n_2(\ve{\Gamma}).
\label{eq:inductioncontraction}
\end{equation}
Such a sequence $\ve{\Gamma}'$ will be constructed  by contacting
certain edges of the $\Gamma_i$.

We say  that an edge $e$ of $\Gamma_i$ is \emph{special} if $\d e$
consists of two distinct points, at least one of which
is contained in $\Int (\Sigma)$.
It is easy to see that if $\Gamma$ is a decomposing graph of $(\Sigma,\cI)$ with at least
one vertex in $\Int (\Sigma)$, then $\Gamma$ has
at least one special edge. If $e$ is a special edge of
$\Gamma$, we construct a new decomposing graph $C_e(\Gamma)$ of $\Sigma$,
which we call the \emph{contraction} of $\Gamma$ along $e$.
Outside a small neighborhood of $e$, we define
the graph $C_e(\Gamma)$ to be $\Gamma$.
If $e$ has two vertices in $\Int(\Sigma)$,
then $C_e(\Gamma)$ is formed by collapsing $e$ and its two
endpoints into a single vertex. If instead $e$ has one
endpoint on $\d \Sigma$, then $C_e(\Gamma)$ is formed by deleting $e$,
and performing an isotopy of the edges previously
incident to $e$ until they hit $\d \Sigma$.
The contraction operation is shown Figure~\ref{fig::78}.

\begin{figure}[ht!]
 \centering
 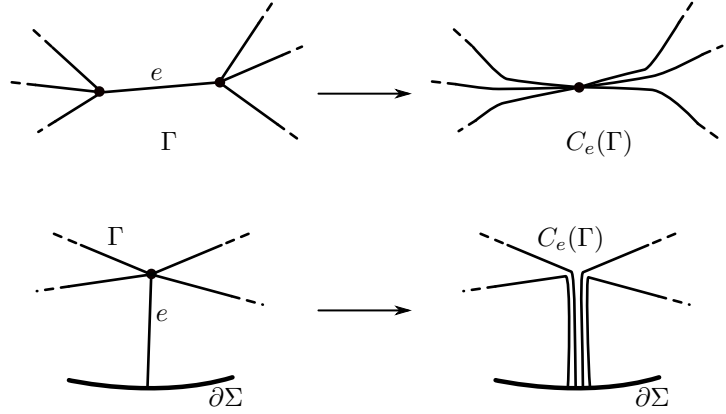
 \caption{Contracting a decomposing graph along a special edge.
 On the top is when $e$ has boundary equal to two
 vertices in $\Int (\Sigma)$. On the bottom is the
 case when $\d e$ has a vertex on $\d \Sigma$.\label{fig::78}}
\end{figure}

Now suppose $\ve{\Gamma}$ is a Cerf sequence connecting $\ve{a}$ and $\ve{a}'$ with
$n_1(\Gamma) > 0$. Let $k$ be the first index where $\Gamma_{k}$ is obtained from $\Gamma_{k-1}$
by move \ref{move:G1} (an index 0/1 birth).  We construct a
sequence of special edges $e_k,\dots, e_l$ in $\Gamma_k,\dots, \Gamma_l$,
respectively, for some $l\le n$. We define the first special edge, $e_{k}$,
to be the edge which is added to $\Gamma_{k-1}$
to form $\Gamma_k$. We construct the remaining special edges $e_i$ recursively.
Supposing $e_i$ has already been chosen,
we define $e_{i+1}$ as follows, depending on which move
is used to form $\Gamma_{i+1}$ from $\Gamma_i$:
\begin{enumerate}
\item $\Gamma_{i+1}$ is obtained  by move~\ref{move:G1} or its inverse:
  If $e_{i}$ is  deleted by this move, we declare $e_i$
  to be the final special edge in our sequence and set $l=i$.  If $e_i$ is unchanged, we set $e_{i+1}=e_i$.
\item $\Gamma_{i+1}$ is obtained by move~\ref{move:G2} or its inverse:
   Since $e_i$ is special, it cannot be deleted by this move. Noting this, we set $e_{i+1}=e_i$.
\item $\Gamma_{i+1}$ is obtained by move~\ref{move:G3}: There are two subcases:
\begin{enumerate}
\item \label{case:speciala}An edge $e$ is slid over another edge, and $e\neq e_i$:
  Noting that $e_i$ remains special, we set $e_{i+1}=e_i$.
\item \label{case:specialb} The edge $e_i$ is slid over another edge $e$:
  Write $e_i'$ for the edge obtained by arc sliding $e_i$ over $e$.
  If $e_i'$ is not special, then its easy to check that
  $e$ must be special, and we set $e_{i+1}=e$.  If $e_i'$ is special, we set $e_{i+1}=e_i'$.
\end{enumerate}
\end{enumerate}

We now demonstrate that consecutive terms in the sequence
\[
\ve{\Gamma}':=(\Gamma_1,\dots, \Gamma_{k-1}, C_{e_k}(\Gamma_k),
\dots,C_{e_l}(\Gamma_l), \Gamma_{l+1},\dots, \Gamma_n)
\]
are related by either a single instance of moves~\ref{move:G1},
\ref{move:G2}, their inverses,  by multiple instances
of move~\ref{move:G3}, or by no moves. Furthermore,
we will show that $\ve{\Gamma}'$ satisfies equation~\eqref{eq:inductioncontraction}.

We first consider the case when $\Gamma_{i+1}$ is obtained
from $\Gamma_i$ by adding or removing an edge $e$ via an instance of move~\ref{move:G1},
\ref{move:G2}, or their inverses. If $e = e_i$ is removed,
then, by definition, $e_i$ is the last special edge (i.e., $i = l$), and clearly
$C_{e_i}(\Gamma_i) = \Gamma_{i+1}$. If $e \neq e_i$,
then $C_{e_{i+1}}(\Gamma_{i+1})$ is also obtained from $C_{e_i}(\Gamma_i)$
by adding or removing the edge $e$ via a move of the same type as
the one relating $\Gamma_i$ and $\Gamma_{i+1}$.

We now consider the case when $\Gamma_{i+1}$ is obtained
from $\Gamma_i$ by an instance of move~\ref{move:G3}.
As indicated above, we break the argument into two cases.
The first case, labeled \eqref{case:speciala} above, occurs
when we slide an edge $e$ over another edge $e'$, and $e \neq e_i$.
Since $e_i$ is unchanged, it remains special.
If $e_i \neq e'$, then it is straightforward to see that
$C_{e_{i+1}}(\Gamma_{i+1})$ is also obtained from $C_{e_i}(\Gamma_i)$
by an arc slide. If $e_i = e'$, then, in fact, $C_{e_{i+1}}(\Gamma_{i+1})$
and $C_{e_i}(\Gamma_i)$ are isotopic.

We finally consider the most complicated case, labeled~\eqref{case:specialb} above,
when $\Gamma_{i+1}$ is obtained by sliding $e_i$
over another edge $e$. We write $e_i'$ for the edge resulting
from arc sliding $e_i$ over $e$. There are two further subcases,
depending on whether $e_i'$ is special or not.

Suppose first that $e_i'$ is not special, so $e_{i+1}=e$. By definition,
either $e_i'$ has both endpoints  on $\d \Sigma$, or $e_i'$ has both endpoints on the same vertex.
We claim that in both cases $C_{e_{i+1}}(\Gamma_{i+1})$
differs from $C_{e_i}(\Gamma_i)$ by a sequence of arc slides.
See Figure~\ref{fig::79} for an illustration when $e_i'$ has both endpoints on $\d \Sigma$.
A similar argument holds when $e_i'$ has both endpoints on the same vertex of $\Gamma_{i+1}$.

The other subcase of \eqref{case:specialb} occurs when $e_i'$ is special. Write
$\{v_1,v_2\}=\d e_i$ and $\{v_2,v_3\}=\d e$. Since $e_i$ and
$e_{i}'$ are special, it follows that all $v_1 \neq v_2$ and
$v_1 \neq v_3$. The argument can be further divided into the
cases that $v_2=v_3$ and $v_2\neq v_3$. In both cases, a
manipulation similar to the one shown in Figure~\ref{fig::79}
shows that $C_{e_{i+1}}(\Gamma_{i+1})$ is obtained from
$C_{e_i}(\Gamma_i)$ by a sequence of allowable arc slides.

\begin{figure}[ht!]
 \centering
 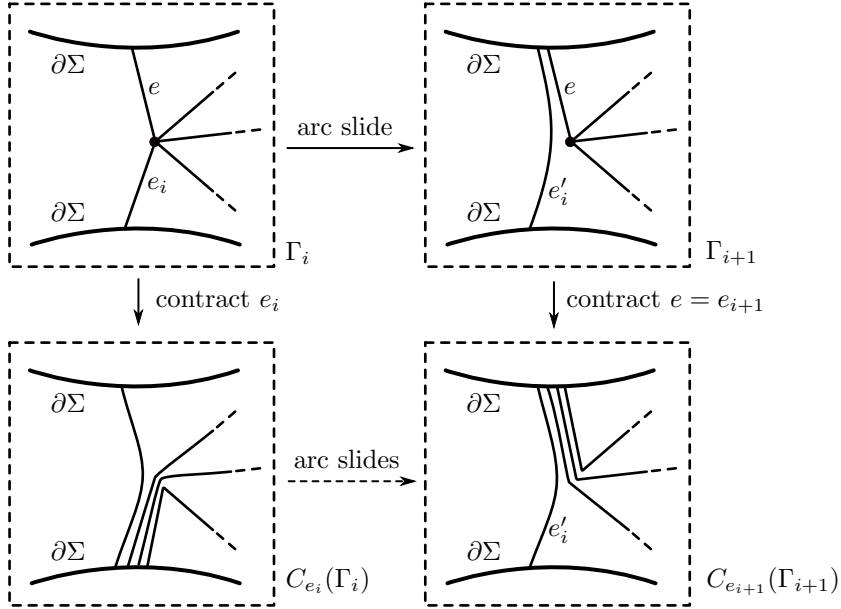
 \caption{A illustration of the fact that
 $C_{e_{i+1}}(\Gamma_{i+1})$ is obtained from $C_{e_i}(\Gamma_i)$ by
 a sequence of arc slides when
 $\Gamma_{i+1}$ is obtained from $\Gamma_i$ by an arc slide
 of the special edge $e_i$ across another edge $e$, and the
 resulting arc $e_i'$ is no longer special.}
 \label{fig::79}
\end{figure}

By the above, for $i \in \{k, \dots, l-1\}$, the graph $C_{e_{i+1}}(\Gamma_{i+1})$
is obtained from $C_{e_i}(\Gamma_i)$ using the same type of moves as
the one relating $\Gamma_{i+1}$ and $\Gamma_i$, and exactly one move is used
in case of a move of type~\ref{move:G1}.
Noting that $\Gamma_{k-1} = C_{e_k}(\Gamma_k)$ and $C_{e_l}(\Gamma_l) = \Gamma_{l+1}$,
equation~\eqref{eq:inductioncontraction} is immediate.

Repeating this procedure, starting with a fixed $\ve{\Gamma}$, we can construct a
Cerf sequence $\ve{\Gamma}''$ such that
\[
n_1(\ve{\Gamma}'') = 0 \quad \text{and} \quad n_2(\ve{\Gamma}'') = n_2(\ve{\Gamma}).
\]
By turning Morse functions corresponding to each decomposing graph
upside down (switching the roles of the index 0 and 2 critical
points) and  repeating the above procedure to the induced dual
decomposing graphs, we obtain a sequence of decomposing graphs
from $\ve{a}$ to $\ve{a}'$ which differ only by a sequence of
allowable arc slides, completing the proof.
\end{proof}

\subsection{Doubled diagrams}
\label{sec:doubleddiagrams}

Suppose that $\cH=(\S,\as,\bs)$ is a sutured Heegaard diagram for
the balanced sutured manifold $(M,\g)$. There is a special way of constructing a new diagram of $(M,\g)$ from
$\cH$, which will be important for computing the cobordism map
$F_{\cW_{\as,\bs,\gs}}$. Let us first pick two collections of subintervals
$\cI_0$ and $\cI_1$ of $\d \S$ such that each component of $\d \S$ contains
exactly one subinterval from $\cI_0$ and one subinterval from $\cI_1$ that are disjoint.

We now define the \emph{Heegaard surface that is doubled along $R_{\bs}$} to be
\[
D_{\bs}(\S):=\S\,\natural_{\cI_0} \bar{\S}\,\natural_{\cI_1} R_{\bs},
\]
where $\natural$ denotes the boundary connected sum operation.
Here $\bar{\S}$ denotes a pushoff of $\S$ into $U_{\bs}$,
with the opposite orientation.

Before we define the attaching curves, let us first describe how
$D_{\bs}(\S)$ is embedded in $M_{\as,\bs}$. A schematic is shown in
Figure~\ref{fig::21}. Strictly speaking, we have changed the sutures. An
isotopy supported in a neighborhood of the original sutures moves the new
sutures to the original sutures.

\begin{figure}[ht!]
 \centering
 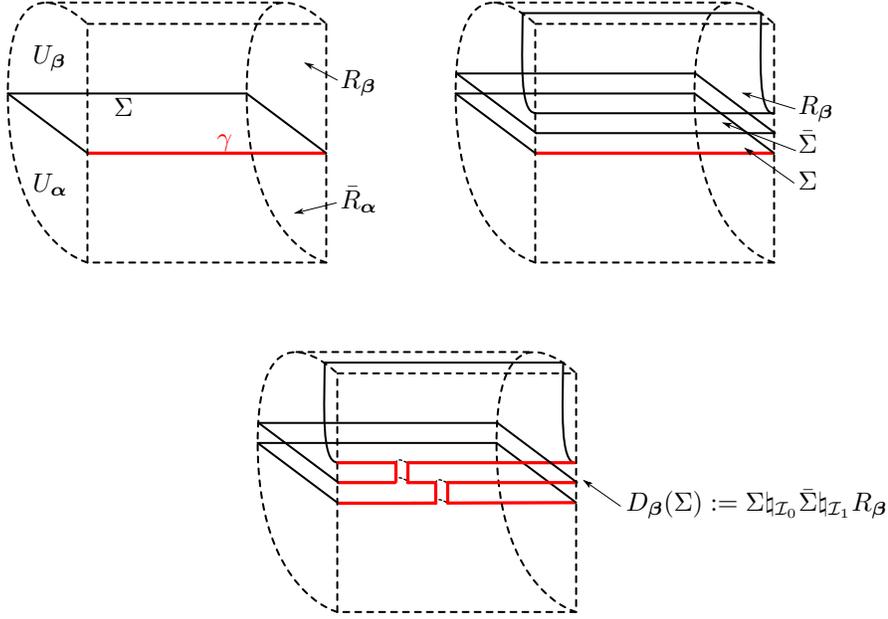
 \caption{The doubled Heegaard surface $D_{\bs}(\S)$.
 A neighborhood of a portion of the sutures $\g$ in $M$ is shown.
 The sutures are drawn in red. Strictly speaking, the new sutures $\d D_{\bs}(\S)$ are
 different from $\g$; however, an isotopy supported in a neighborhood of $\g$ moves
 $\d D_{\bs}(\S)$ back to $\g$.\label{fig::21}}
\end{figure}

We now describe compressing curves on $D_{\bs}(\S)$.
First, pick a collection of pairwise disjoint arcs on $\S$
with boundary on $\cI_0$ that form a basis of $H_1(\S, \cI_0)$.
One then doubles these across $\cI_0$, and obtains a collection of curves
$\Ds_{\S} \subset \S \,\natural_{\cI_0} \bar{\S} \subset D_{\bs}(\S)$.
Similarly, one picks a collection of pairwise disjoint arcs on $R_{\bs}$
(or equivalently, a collection of arcs on $\S$ that avoid the $\bs$ curves,
up to handle slides across $\bs$)
with boundary on $\cI_1$ that form a basis of $H_1(R_{\bs}, \cI_1)$.
Doubling these curves across $\cI_1$ yields closed curves
$\ds_{\bs} \subset \bar{\S}\,\natural_{\cI_1} R_{\bs} \subset D_{\bs}(\S)$.

Write $\bar{\bs}$ for the images of the $\bs$ curves on $\bar{\S}$.
Note that the $\as$, $\bar{\bs}$, and $\ds_{\bs}$ curves are all disjoint.
Since $(M,\g)$ is balanced, $|\as| = |\bs|$, and an easy computation shows that
\[
|\as|+|\bs|+|\ds_{\bs}|=|\Ds_{\S}|.
\]
The doubled Heegaard diagram is now defined as
\[
D_{\bs}(\cH) = (D_{\bs}(\S), D_{\bs}(\as), \Ds_\S) =
(\S\,\natural_{\cI_0} \bar{\S}\,\natural_{\cI_1} R_{\bs},
\as\cup \bar{\bs}\cup \ds_{\bs}, \Ds_{\S}).
\]

\begin{rem}\label{rem:alphadoubleddiagrams}
Note that there is some asymmetry in the above construction, since we took
$\S$ and connected it along $\cI_0$ to a surface $\bar{\S}\,\natural_{\cI_1}
R_{\bs}$ that was in the $U_{\bs}$ handlebody. We could instead connect
$\S$ along $\cI_0$ to the surface $R_{\as}\,\natural_{\cI_1}
\bar{\S}$ in the $U_{\as}$ handlebody, and construct analogous
attaching curves $\Ds_\S$ and $\ds_{\as} \cup \bar{\as} \cup \bs$ for the Heegaard surface
$D_{\as}(\S) := R_{\as}\,\natural_{\cI_1} \bar{\S}\,\natural_{\cI_0} \S$. We will write
\[
D_{\as}(\cH) = (D_{\as}(\S), \Ds_\S, \ds_{\as} \cup \bar{\as} \cup \bs)
\]
for this Heegaard diagram of $(M,\g)$. If there is any ambiguity, we
will call the Heegaard diagram $D_{\bs}(\cH)$ the \emph{$\b$-double}, and
$D_{\as}(\cH)$ the \emph{$\a$-double}.
\end{rem}

\subsection{Weakly conjugated diagrams}
\label{sec:weaklyconjugated}

Given a Heegaard diagram $(\S,\as,\bs,w)$ of a based 3-man\-i\-fold, we can consider
the conjugate diagram $(\bar{\S},\bs,\as,w)$ that represents the
same based 3-manifold. This was described by Ozsv\'ath and Szab\'o~\cite{OSDisks},
and was explored further by Hendricks and Manolescu~\cite{HMInvolutive}.
Given a sutured diagram $(\S,\as,\bs)$ for $(M,\g)$, one can consider the sutured
diagram $(\bar{\S},\bs,\as)$; however, this is now a diagram for $(M,-\g)$.
So, unlike in the case of closed 3-manifolds, this operation
does not induce a conjugation action on $\SFH(M,\g)$.

Nonetheless, a similar diagrammatic manipulation appears when we compute the
cobordism map for $\cW_{\as,\bs,\gs}$.  In analogy to the terminology for the
conjugation action on Heegaard diagrams for closed 3-manifolds, we will say
that the diagrams that appear are \emph{weakly conjugated.} We describe the
construction of weakly conjugated diagrams in this section.

As with the doubled diagrams, we pick collections of subintervals $\cI_0$ and
$\cI_1$ in $\d \S$ such that each component of $\d \S$ contains exactly one
subinterval from $\cI_0$ and from $\cI_1$ that are disjoint. We can then form
the weakly conjugated Heegaard surface
\[
C(\S) := R_{\as}\,\natural_{\cI_0}\bar{\S}\,\natural_{\cI_1} R_{\bs}.
\]
This is shown in Figure~\ref{fig::22}. As described,
$\d C(\S)$ is different from $\d \S$, but an isotopy supported in a
neighborhood of $\d \S$ moves $\d C(\S)$ to $\d \S$.

\begin{figure}[ht!]
 \centering
 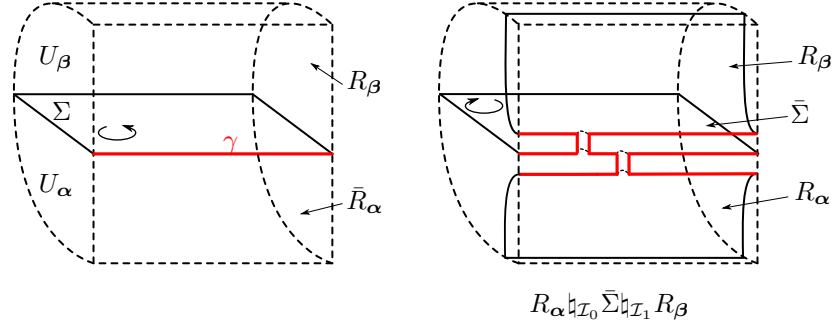
 \caption{The weakly conjugated Heegaard surface $C(\S)$.
 A neighborhood of a portion of the sutures is shown.
 The sutures are drawn in red. An isotopy supported in a neighborhood of the sutures moves the boundary
 of the new Heegaard surface, as we have drawn it,
 to the position of $\d \S$, the boundary of the old Heegaard surface.}
 \label{fig::22}
\end{figure}

We now describe compressing curves on $C(\S)$. Note that $\as$ and $\bs$
still bound compressing disks on~$\S$. As curves on $\bar{\S}$, we
denote them by $\bar{\as}$ and $\bar{\bs}$. However, these are not complete
collections of compressing curves, as we have increased the genus of the
Heegaard surface by attaching $R_{\as}$ and $R_{\bs}$.
Hence, we pick a collection of pairwise disjoint arcs on $R_{\as}$ that
form a basis of $H_1(R_{\as}, \cI_0)$, and double them across $\cI_0$, to get
a collection of curves $\ds_{\as}$ on $R_{\as}\,\natural_{\cI_0}
\bar{\S}$. Similarly, we pick a collection of pairwise disjoint arcs on
$R_{\bs}$ that form a basis of $H_1(R_{\bs}, \cI_1)$, and double them across
$\cI_1$ to get a collection of curves $\ds_{\bs}$ on
$\bar{\S}\,\natural_{\cI_1} R_{\bs}$. We define the weakly conjugated
Heegaard diagram of $\cH$ to be
\[
C(\cH) = (C(\S), C(\bs), C(\as)) :=
(R_{\as}\,\natural_{\cI_0} \bar{\S} \,\natural_{\cI_1}
R_{\bs},\bar{\bs}\cup \ds_{\bs},\bar{\as}\cup \ds_{\as}).
\]
Then $C(\cH)$ is also a diagram of $(M,\g)$.

\subsection{The change of diagrams map from \texorpdfstring{$D_{\b}(\cH)$}{D(H)} to \texorpdfstring{$\cH$}{H}}

In this section, we prove a relatively simple formula for the change of diagrams
map from the $\b$-double of a diagram $D_{\bs}(\cH)$ to the original
diagram $\cH = (\S,\as,\bs)$. Recall that
\[
D_{\bs}(\cH) = (\S\,\natural_{\cI_0} \bar{\S}\,
\natural_{\cI_1} R_{\bs}, \as\cup \bar{\bs} \cup \ds_{\bs}, \Ds_{\S}).
\]
We will write $\bs' := \bar{\bs} \cup \ds_{\bs}$. Note that
$(D_{\bs}(\S), \Ds_{\S}, \bs \cup \bs')$ is the $\a$-double of
the diagram $(\S,\bs,\bs)$ (see Remark~\ref{rem:alphadoubleddiagrams}),
which represents
\[
(M_{\bs,\bs}, \gamma_{\bs,\bs}),
\]
which we note is the sutured manifold obtained by
gluing two copies of the sutured handlebody $U_{\bs}$ together along $\Sigma$.

\begin{lem}\label{lem:topgradedgenerator}
There is a relatively graded isomorphism
\[
\SFH(M_{\bs,\bs},\gamma_{\bs,\bs}) \iso
\bigotimes_{i=1}^{|\bs|}\left((\bF_2)_{\tfrac{1}{2}}\oplus (\bF_2)_{-\tfrac{1}{2}}\right).
\]
In particular, there is a top-graded element
$[\Theta_{\Ds_{\S},\bs\cup \bs'}] \in \SFH(D_{\bs}(\cH))$.
\end{lem}

\begin{proof}
Let $\gs$ be small Hamiltonian translates
of the $\bs$ curves, such that $|\beta_i \cap \gamma_j| = 2\delta_{ij}$.
Then $(\Sigma,\bs,\gs)$ is an admissible diagram for
$(M_{\bs,\bs}, \gamma_{\bs,\bs})$ whose sutured Floer complex contains
$2^{|\bs|}$ generators. Furthermore, since every component of
$\Sigma \setminus \bs$ contains a component of $\d \Sigma$,
it is straightforward to see that the only homology classes
of disks on $(\Sigma,\bs,\gs)$ have domains which are
supported in the small bigons between $\beta_i$ and $\gamma_i$. From these considerations,
the only index 1 holomorphic disks are the classes whose domain consists of a bigon connecting
the higher graded point of $\beta_i \cap \gamma_i$ to the lower graded point.
Hence, modulo 2, the differential on $\CF(\Sigma,\bs,\gs)$ vanishes,
and as a relatively graded group, we have
\[
\SFH(\Sigma,\bs,\gs) = \CF(\Sigma,\bs,\gs) \iso
\bigotimes_{i=1}^{|\bs|}\left((\bF_2)_{\tfrac{1}{2}}\oplus (\bF_2)_{-\tfrac{1}{2}}\right),
\]
completing the proof.
\end{proof}

Using the holomorphic triangle map and the class
 $[\Theta_{\Ds_{\Sigma},\bs\cup \bs'}]$, we construct a map
\[
F_2 \colon \SFH(D_{\bs}(\S), \as\cup \bs', \Ds_\S)) \to
\SFH(D_{\bs}(\S), \as\cup \bs', \bs\cup \bs')
\]
using the formula
\[
F_2 :=  F_{\as\cup \bs',\Ds_{\S}, \bs\cup \bs'}(-, \Theta_{\Ds_{\S},\bs\cup \bs'}).
\]
We can also define a 3-handle map
\[
F_3 := F_3^{\bs',\bs'} \colon  \SFH(D_{\bs}(\S), \as \cup \bs', \bs \cup \bs') \to \SFH(\S,\as,\bs)
\]
(here, the second copy of $\bs'$ is a small Hamiltonian translate of $\bs'$,
though we omit this from the notation).

\begin{lem}\label{lem:changeofdiagramsdoubleddiagram}
The composition $F_3\circ F_2$ is chain homotopic to the change of diagrams map
from $\CF(D_{\bs}(\cH))$ to $\CF(\cH)$.
\end{lem}

\begin{proof}
The idea is simple: The map $F_2$ can be interpreted as a composition of
2-handle maps, and the map $F_3$ can be interpreted as a composition of
3-handle maps, for a collection of 4-dimensional 2-handles and 3-handles that
topologically cancel. Hence their composition represents the transition map from naturality,
by the well-definedness of the sutured cobordism maps~\cite{JCob}.

We now explain the technical details. Let us first describe
the framed link along which surgery induces $F_2$. Let $U_{\bs}'$ denote the
$\b$-handlebody of the diagram $D_{\bs}(\cH)$ (i.e., $\Ds_{\S}$
bounds compressing disks in $U_{\bs}'$). For each curve in $\bs$ and each
curve in $\ds_{\bs}$, we will construct a component of $\bL$. For a curve
$\b_i\in \bs$, there is a knot $K_{\b_i}\subset U_{\bs}'$, obtained by
pushing $\b_i$ into $U_{\bs}'$ slightly. We choose the framing of
$K_{\b_i}$ to be parallel to $\S$. Similarly, given $\delta_k\in
\ds_{\bs}$, we can construct a framed knot $K_{\delta_k}$ by pushing
$\delta_k$ into $U_{\bs}'$, and take the framing induced by the tangent space
of the surface $D_{\bs}(\S)$. The construction is illustrated in
Figure~\ref{fig::49}. We define the framed link
\[
\bL := \bigcup_{\tau\in \bs\cup \ds_{\bs}} K_\tau.
\]

\begin{figure}[ht!]
 \centering
 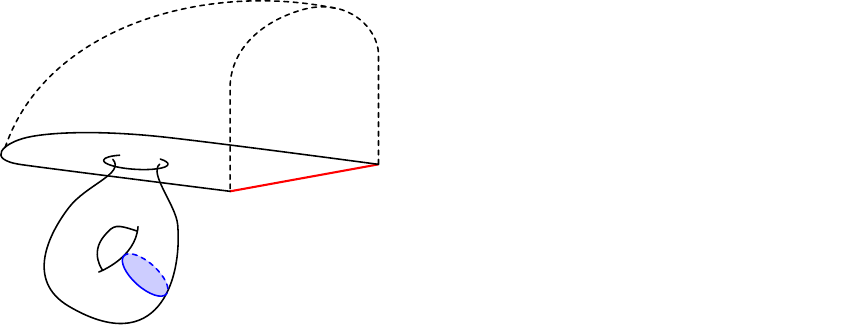
 \caption{Constructing the link $\bL \subset U_{\bs}$. The left shows
 a part of the handlebody $U_{\bs}$ for a diagram $(\S,\as,\bs)$,
 together with a curve $\b_i \in \bs$.
 The right depicts the $\b$-double $D_{\bs}(\cH)$ of $\cH$.
 The corresponding curve $\bar{\b}_i \in \bar{\bs}$ and a curve
 $\delta_k \in \ds_{\bs}$ are shown, as well as the corresponding components
 $K_{\b_i}$ and a portion of $K_{\delta_k}$ of~$\bL$. The
 shaded regions denote compressing disks bounded by the curves
 $\b_i$, $\bar{\b}_i$, and $\delta_k$. \label{fig::49}}
\end{figure}

The map $F_3$ is clearly the 3-handle map for a collection
of $|\bs'|$  2-spheres in $M(\bL)$ that topologically cancel the link $\bL$.
In the subsequent Lemma~\ref{lem:Heegaardtripleissubordinatetobouquet}, we will
show that the triple $(D_{\bs}(\Sigma), \as\cup \bs',\Ds_{\Sigma},\bs\cup \bs')$
can be related to a triple subordinate to a bouquet for $\bL$ by
a sequence of handle slides and isotopies of the $\Ds_{\Sigma}$
and $\bs\cup \bs'$ curves. Assuming this topological fact
for the moment, we now show that this implies that $F_2$ is the
2-handle map for surgery on $\bL$. Suppose
$(D_{\bs}(\Sigma),\as\cup\bs', \xis,\zetas)$ is a Heegaard triple
which is subordinate to a bouquet for $\bL$, and $\xis$ and
$\zetas$ are related to $\Ds_{\Sigma}$ and $\bs\cup \bs'$,
respectively, by a sequence of handleslides and isotopies.
To show that $F_2$ is the 2-handle map, it is sufficient to show that
\begin{equation}
F_{\as\cup \bs',\Ds_{\S}, \bs\cup \bs'}(-, \Theta_{\Ds_{\S},\bs\cup \bs'})\simeq \Psi_{\as\cup \bs'}^{\zetas\to \bs\cup \bs'}\circ  F_{\as\cup \bs', \xis,\zetas} \left(\Psi_{\as\cup \bs'}^{\Ds_{\Sigma}\to \xis}(-),\Theta_{\xis,\zetas}\right),\label{eq:F2=2-handle}
\end{equation}
the left-hand side being $F_2$, and the right-hand side being the
definition of the 2-handle map pre-composed with the
transition map for changing $\Ds_{\Sigma}$ to $\xis$,
and post-composed with the transition map for changing $\zetas$ to $\bs\cup \bs'$. We compute
\begin{equation}
\begin{split}
\Psi_{\as\cup \bs'}^{\zetas\to \bs\cup \bs'}\circ F_{\as\cup \bs', \xis,\zetas}
\left(\Psi_{\as\cup \bs'}^{\Ds_{\Sigma}\to \xis}(-),\Theta_{\xis,\zetas}\right)
&\simeq  \Psi_{\as\cup \bs'}^{\zetas\to \bs\cup \bs'}\circ F_{\as\cup \bs', \Ds_{\Sigma},\zetas}
\left( -,\Psi^{\zetas}_{\xis\to \Ds_{\Sigma}}(\Theta_{\xis,\zetas})\right)\\
 & \simeq F_{\as\cup \bs', \Ds_{\Sigma},\zetas} \left( -,\Psi_{\Ds_{\Sigma}}^{\zetas\to \bs\cup \bs'}\left(\Psi^{\zetas}_{\xis\to \Ds_{\Sigma}}(\Theta_{\xis,\zetas})\right)\right).
\end{split}
\label{eq:associativityfor2-handlemap}
\end{equation}
The first chain homotopy of equation~\eqref{eq:associativityfor2-handlemap}
follows from associativity for the 4-tuple $(\as\cup \bs', \Ds_{\Sigma},\xis, \zetas)$,
and the second follows from associativity applied to the
4-tuple $(\as\cup \bs',\Ds_{\Sigma}, \zetas, \bs\cup \bs')$.
Using the fact that the top-graded generator is well defined
on the level of homology, as in Lemma~\ref{lem:topgradedgenerator},
it follows that $\Psi_{\Ds_{\Sigma}}^{\zetas\to \bs\cup \bs'}
\left(\Psi^{\zetas}_{\xis\to \Ds_{\Sigma}}(\Theta_{\xis,\zetas})\right)$
and $\Theta_{\Ds_{\Sigma},\bs\cup \bs'}$ differ by a boundary, so the last line of
equation~\eqref{eq:associativityfor2-handlemap} is
chain homotopic to $F_2$, establishing equation~\eqref{eq:F2=2-handle}.

Using the fact that $F_2$ is chain homotopic
to the 2-handle map for $\bL$, it follows that the composition
$F_3\circ F_2$ is chain homotopic to the map from naturality.
\end{proof}

To finish the proof of Lemma~\ref{lem:changeofdiagramsdoubleddiagram},
it is sufficient to show the following:

\begin{lem}\label{lem:Heegaardtripleissubordinatetobouquet}
The Heegaard triple $(D_{\bs}(\Sigma), \as\cup \bs', \Ds_{\Sigma}, \bs\cup \bs')$
is related to a triple subordinate to the framed link
$\bL$ by a sequence of handle slides and isotopies of the
$\Ds_{\Sigma}$ and $\bs\cup \bs'$ curves.
\end{lem}

\begin{proof}
We recall that the $\Ds_{\Sigma}$ curves were constructed
by picking a set of arcs $\ve{s}=\{s_1,\dots, s_n\}$ forming
a basis of $H_1(\Sigma,\cI_0)$, and then doubling the arcs of
$\ve{s}$ across $\cI_0$ onto $\bar{\Sigma}$ to obtain a
collection of $n := 2(g(\Sigma) + |\d \Sigma| - |\Sigma|)$ closed curves.
By Lemma~\ref{lem:allowablearcslides}, any two such bases $\ve{s}$
of $H_1(\Sigma, \cI_0)$ can be connected by a sequence of
allowable arc slides. An allowable arc slide of two arcs
in $\ve{s}$ induces a handle slide of the corresponding
doubled curves in $\Ds_{\Sigma}$. Consequently, we can assume
that $\Ds_{\Sigma}$ is constructed using any convenient basis of $H_1(\Sigma,\cI_0)$.

The curves $\ds_{\bs}$ are obtained by picking a set of
\[
k:=\rank H_1(R_{\bs}, \cI_1)=2(g(R_{\bs}) + |\d \Sigma| - |R_{\bs}|)
\]
arcs $d_1,\dots, d_k$ which form a basis of $H_1(R_{\bs},\cI_1)$.
Noting that we view $R_{\bs}$ as the result of surgering
$\Sigma$ on $\bs$, we can assume that $\Ds_\Sigma$ and $\ds_{\bs}$
are constructed using  arcs $\ve{s}$ and $\ve{d}$ satisfying the following:
\begin{enumerate}
\item For each $\beta\in \bs$, there is a pair of arcs
$s$, $s'\in \ve{s}$ such that $|s\cap \beta|=1$, and $s$ is disjoint
from all the other $\bs$ and $\ve{d}$ curves. The curve $s'$
is obtained by taking $\beta$, and isotoping a neighborhood
of the point $\beta\cap s$ along $s$ until it intersects
$\d \Sigma$. In particular, $s'$ is disjoint from $\ve{d}\cup \bs$.
Furthermore, we can handle slide $\beta$ across its image $\bar{\beta}$ on $\bar{\Sigma}$
along the arc $s$ to obtain a curve isotopic to $s' \cup \bar{s}'$, where $\bar{s}'$ is the image
of $s'$ on $\bar{\Sigma}$.
\item For each $d\in \ve{d}$, there is a corresponding arc $s\in \ve{s}$
satisfying $|s\cap d|=1$. Furthermore, $s$ is disjoint from $(\ve{d}\setminus \{d\})\cup \bs$.
\end{enumerate}
See Figure~\ref{fig::74} for an example of the arcs $\ve{s}$ and $\ve{d}$, and the resulting attaching curves
$\Ds_{\Sigma}$ and $\ds_{\bs}$.

\begin{figure}[ht!]
 \centering
 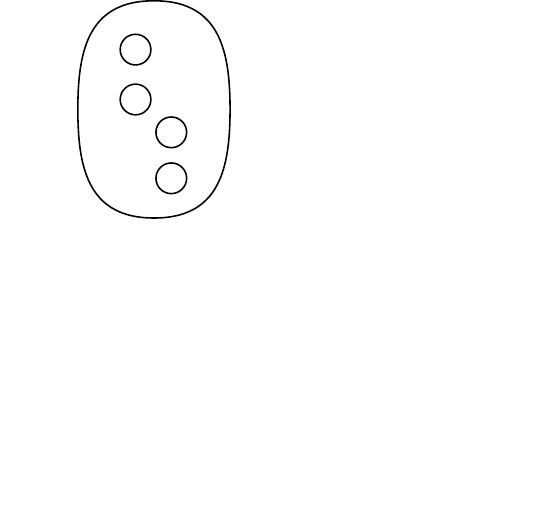
 \caption{On the top left is an example of a monodiagram
 $(\Sigma,\bs)$ with $|\bs|=1$, $g(\Sigma)=2$, and $|\d \Sigma|=1$.
 On the top right are the arcs $\ve{s}$ used to form $\Ds_{\Sigma}$,
 as well as the arcs $\ve{d}$ used to form $\ds_{\bs}$.
 On the bottom is the diagram
 $(D_{\bs}(\Sigma), \Ds_\Sigma, \bs \cup \bs')$.}
 \label{fig::74}
\end{figure}

We handle slide each $\beta\in \bs$ across the corresponding
$\bar{\beta}\in \bar{\bs}$ along the arcs $\ve{s}$ to form the curve $\beta^H$,
and let us call the resulting set of curves $\bs^{H}$.
We note  that each curve in $\bar{\bs} \cup \ds_{\bs}$
is a longitude of a component of $\bL$, and each curve in
$\Ds_{\Sigma}$ is either a meridian, or is a small Hamiltonian
isotopy of a curve in $\bs^H$. It follows that the triple
\[
(D_{\bs}(\Sigma), \as\cup \bar{\bs}\cup \ds_{\bs}, \Ds_{\Sigma}, \bs^H\cup \bar{\bs}\cup \ds_{\bs})
\]
is subordinate to a bouquet for $\bL$.
\end{proof}

\subsection{The change of diagrams map from \texorpdfstring{$C(\cH)$}{C(H)}
to \texorpdfstring{$D_{\b}(\cH)$}{D(H)}.}
Analogously to the formula in Lemma~\ref{lem:changeofdiagramsdoubleddiagram}, we can
 describe the change of diagrams map from $C(\cH)$ to
$D_{\bs}(\cH)$ in concrete terms.

Pick a collection of pairwise disjoint arcs on $R_{\as}$ with endpoints in
$\cI_0$ that form a basis of  $H_1(R_{\as}, \cI_0)$. These induce arcs on
$\S$ disjoint from $\as$, well defined up to handle slides across $\as$, and let $\Ds_{\as} \subset
\S\,\natural_{\cI_0} \bar{\S}$ be the curves formed by doubling them
across $\cI_0$. By construction, the curves in $\Ds_{\as}$ are disjoint from
$\as$ and $\bar{\as}$. Also, let $\ds_{\as}$ denote the images of the curves $\Ds_{\as}$ on
$R_{\as}\natural_{\cI_0} \bar{\Sigma}$. Let $\as' := \bar{\as}\cup \Ds_{\as}$
and $\bs' := \bar{\bs} \cup \ds_{\bs}$.

Recall that the surface $D_{\bs}(\S)$ is $\S\,\natural_{\cI_0}\,
\bar{\S}\,\natural_{\cI_1} R_{\bs}$, and $C(\S)$ is defined as
$R_{\as}\,\natural_{\cI_0} \bar{\S}\, \natural_{\cI_1} R_{\bs}$. Since
surgering the surface $\S$ along the $\as$ curves yields $R_{\as}$, we see
that there is a 1-handle map
\[
G_1 := F_1^{\as,\as} \colon \CF(C(\cH)) =
\CF(C(\S),\bs',\bar{\as} \cup \ds_{\as}) \to \CF(D_{\bs}(\S),\as \cup \bs',\as \cup \as').
\]

\begin{lem}\label{lem:diagramforS1S2}
The diagram $(D_{\bs}(\S), \as \cup \as', \Ds_{\S})$ represents the sutured manifold
\[
(S^1\times S^2)^{\# n}(|\d\S|)
\,\natural_{\cI_1}\,  (R_{\bs} \times I, \d R_{\bs} \times I)
\]
for some $n$, where $(S^1\times S^2)^{\# n}(|\d\S|)$
denotes the sutured manifold obtained by removing $|\d \S|$ balls from $(S^1\times S^2)^{\# n}$
and adding a connected suture to each boundary component.
Furthermore, $\natural_{\cI_1}$ denotes the boundary connected sum taken along $\cI_1$
by adding $|\cI_1|$ product 1-handles.
\end{lem}

\begin{proof}
All the attaching curves are disjoint from
$R_{\bs}$. If we cut $R_{\bs}$ off of $D_{\bs}(\S)$, we are left with the
diagram $(\S \, \natural_{\cI_0} \bar{\S}, \as \cup \as',
\Ds_{\S})$. Recall that $\as'=\bar{\as}\cup \Ds_{\as}$, where $\Ds_{\as}$ is
obtained by choosing a basis of arcs for $H_1(R_{\as},\cI_0)$, and doubling
the induced curves on $\S$ onto $\S\, \natural_{\cI_0}  \bar{\S}$.
We can assume that $\Ds_\Sigma$ is constructed by starting with small Hamiltonian
isotopes of the curves in $\Ds_{\as}$, and then adjoining additional $2|\as|$ curves that are disjoint from
$\Ds_{\as}$. Then the curves in $\Ds_{\as}$ together with
their isotopes in $\Ds_\Sigma$ determine $|\Ds_{\as}|$
embedded 2-spheres, which may or may not be separating. After surgering these out,
we are left with the double of a diagram for a disjoint union of some number of copies of
$(S^1\times S^2)^{\# l}(m)$ for various $l$ and $m$.
It follows that $(D_{\bs}(\S), \as \cup \as', \Ds_{\S})$ represents
\[
(S^1\times S^2)^{\# n}(|\d\S|)
\, \natural_{\cI_1}\,  (R_{\bs} \times I, \d R_{\bs} \times I)
\]
for some $n$.
\end{proof}

As a consequence of Lemma~\ref{lem:diagramforS1S2}, there is a top-graded generator
\[
\Theta_{\as \cup \as', \Ds_{\S}} \in \SFH(D_{\bs}(\S), \as \cup \as', \Ds_{\S}),
\]
and hence we can also define a triangle map
\[
G_2 := F_{\as \cup \bs', \as \cup \as', \Ds_{\S}}(-, \Theta_{\as\cup \as', \Ds_{\S}}).
\]

\begin{lem}\label{lem:changeofdiagramsdoubletoconj}
The composition $G_2 \circ G_1$ is chain homotopic to the change of diagrams map from $C(\cH)$ to
$D_{\bs}(\cH)$.
\end{lem}

\begin{proof}
The proof is similar to the proof of Lemma~\ref{lem:changeofdiagramsdoubleddiagram}.
We can interpret the composition $G_2\circ G_1$ as the cobordism map
for a canceling collection of $|\as|$ pairs of 4-dimensional
1- and 2-handles. We now describe the attaching spheres of the
1-handles and 2-handles (see also \cite{ZemDualityMappingTori}*{Section~7.2}
for a detailed account of a similar topological manipulation
in the setting of closed 3-manifolds). Let $D_1,\dots, D_{|\as|}$ denote
compressing disks attached along the curves $\as\subset \Sigma$,
where $\Sigma\subset M$ denotes the original Heegaard surface.

The two feet of a 1-handle are obtained by pushing off
the center point of the disk $D_i$ in both normal directions.
The canceling 2-handle for this 1-handle is attached along the core
of the 1-handle concatenated with an arc connecting the two feet that intersects the center
of $D_i$ transversely. By adapting the proof of
Lemma~\ref{lem:Heegaardtripleissubordinatetobouquet}, it is
not hard to see that, after a sequence of handle slides, the
triple $(D_{\bs}(\Sigma), \as\cup \bs', \as\cup \as', \Ds_{\Sigma})$
becomes subordinate to the link described above.
\end{proof}

\subsection{A diagram for \texorpdfstring{$M_{\a,\b} \cup_{R_\b} M_{\b,\g}$}{M},
and a formula for the special cobordism map \texorpdfstring{$(W_{\a,\b,\g})^s$}{W}}
\label{sec:glueddiagram}

Let $\cT = (\S, \as, \bs, \gs)$ be an admissible sutured triple diagram. We
now describe a diagram $A(\cT)$ for the sutured manifold
$(M_{{\as},{\bs}}\cup_{R_{\bs}} M_{{\bs},{\gs}}, \g_{\as,\gs})$,
which has sutures along $\d R_{\bs}$, where the two manifolds are glued together.

Analogous to the doubled diagram, we let $\cI_0$, $\cI_1 \subset \d \S$
be disjoint collections of subintervals, such that each component of $\d \S$
contains exactly one subinterval from $\cI_0$ and $\cI_1$. We form the diagram
\[
A(\cT) = (D_{\gs}(\S), D_{\gs}(\as), A(\bs)) =
(\S\,\natural_{\cI_0} \bar{\S}\,\natural_{\cI_1} R_{\gs},\as \cup \bar{\gs} \cup
\ds_{\gs},\bs \cup \bar{\bs} \cup \Ds_{\bs}),
\]
where the component $\S$ of $D_{\gs}(\S)$ is embedded in $M_{\as,\bs}$ and
the component $\bar{\S}$ in $M_{\bs,\gs}$. We call $A(\cT)$ the
\emph{amalgamation of $\cT$ along $R_{\bs}$}. Here
$\ds_{\gs} \subset \bar{\S}\,\natural_{\cI_1} R_{\gs}$ is obtained by
doubling a collection of arcs forming a basis of $H_1(R_{\gs}, \cI_1)$.
Similarly, $\Ds_{\bs} \subset \S\,\natural_{\cI_0} \bar{\S}$ is
obtained by choosing a basis of arcs in $H_1(R_{\bs}, \cI_0)$,
and doubling a lift of them to $\S$ across $\cI_0$.

Note that the doubling construction can be viewed as an instance of amalgamation,
in the sense that
\[
D_{\bs}(\S,\as,\bs) = A(\S,\as,\emptyset,\bs).
\]
Here we interpret $R_{\emptyset}$
as $\Sigma$, so that $\Ds_{\bs}$ is the collection
$\Ds_{\Sigma}$ defined in the construction of a doubled diagram.

\begin{lem}
The diagram $A(\cT)$ is a diagram for $(M_{{\as},{\bs}}\cup_{R_{\bs}} M_{{\bs},{\gs}}, \g_{\as,\gs})$.
\end{lem}

\begin{proof}
The diagram $A(\cT)$ is obtained by first replacing $R_{\bs}$
with $\S$ in
\[
C(\cH_{\bs,\gs}) = (R_{\bs}\,\natural_{\cI_0}
\bar{\S}\, \natural_{\cI_1} R_{\gs}, \bar{\gs}\cup \ds_{\gs}, \bar{\bs}
\cup \ds_{\bs}),
\]
after which $\ds_{\bs}$ becomes $\Ds_{\bs}$. As compressing $\S$ along $\bs$
gives $R_{\bs}$, if we add $\bs$, the resulting diagram
\[
(\S \,\natural_{\cI_0} \bar{\S}\, \natural_{\cI_1} R_{\gs}, \bar{\gs}\cup \ds_{\gs},
\bs \cup \bar{\bs} \cup \Ds_{\bs})
\]
represents $(-U_{\bs} \cup_{R_{\bs}} U_{\bs} \cup_{\S} -U_{\gs}, \d \S)$.
Finally, we add $\as$, which amounts to gluing $U_{\as}$ to
$-U_{\bs} \cup_{R_{\bs}} U_{\bs} \cup_{\S} -U_{\gs}$
by identifying $\S \subset U_{\as}$ with $\bar{\S} \subset \d (-U_{\bs})$.
\end{proof}

By Lemma~\ref{lem:Wabgspecialcobordism2handle}, the special cobordism
\[
(W_{\as,\bs,\gs})^s \colon (M_{\as,\bs} \cup_{R_{\bs}} M_{\bs,\gs}, \g_{\as,\gs})
\to (M_{\as,\gs}, \g_{\as,\gs})
\]
is a 2-handle cobordism, for surgery on a framed link $\bL \subset
M_{\as,\bs} \cup_{R_{\bs}} M_{\bs,\gs}$. We recall the description of the
framed link $\bL$. One takes a Morse function $f_{\bs}$ on $U_{\bs}$ that
is $0$ on $\bar{R}_{\bs}$ and $1$ on $\S$, and has $\bs$ as the intersection of
the ascending manifolds of the critical points of $f_{\bs}$ with $\S$.
Then the descending manifolds of the critical points of $f_{\bs}$ determine a
collection of $|\bs|$ properly embedded arcs $\lambda_i \subset U_{\bs}$
that have both ends on $R_{\bs}$. The link $\bL$ is obtained by taking the
union of the arcs $\lambda_i \subset U_{\bs} \subset M_{\bs,\gs}$, together with
their images in $-U_{\bs} \subset M_{\as,\bs}$.

Let $\Ds_{\S}\subset \S\,\natural_{\cI_0} \bar{\S}$ be a
collection of curves obtained by picking arcs forming a basis of $H_1(\S,
\cI_0)$, and then doubling them across $\cI_0$. One can assume that the chosen
basis of $H_1(\S, \cI_0)$ extends the lift of the basis of $H_1(R_{\bs},\cI_0)$ to $\S$ we
chose in the construction of the $\Ds_{\bs}$ curves, so that
$\Ds_{\bs}\subset \Ds_{\S}$, though this is not essential.

We note that the diagram
\[
(D_{\gs}(\S), A(\bs), \Ds_{\S}) =
(\S\,\natural_{\cI_0} \bar{\S} \, \natural_{\cI_1} R_{\gs},
\bs \cup \bar{\bs} \cup \Ds_{\bs}, \Ds_{\S})
\]
represents
\[
(S^1\times S^2)^{\# n}(|\d \Sigma|)
\,\natural_{\cI_1}\, (I \times R_{\gs}, I \times \d R_{\gs})
\]
for some $n$ by Lemma~\ref{lem:diagramforS1S2}. Consequently, there is a top-graded generator
\[
\Theta_{A(\bs), \Ds_{\S}}\in \SFH(D_{\gs}(\S), A(\bs), \Ds_{\S}) =
\SFH(\S\,\natural_{\cI_0} \bar{\S} \, \natural_{\cI_1} R_{\gs},
\bs \cup \bar{\bs} \cup \Ds_{\bs}, \Ds_{\S}).
\]
Note that $(D_{\gs}(\S), D_{\gs}(\as), \Ds_{\S})$
is a doubled diagram for $(M_{\as,\gs}, \g_{\as,\gs})$. We have the following:

\begin{lem}
The special cobordism map
\[
F_{(W_{\as,\bs,\gs})^s} \colon \CF(M_{\as,\bs} \cup_{R_{\bs}} M_{\bs,\gs}, \g_{\as,\gs})
\to \CF(M_{\as,\gs}, \g_{\as,\gs})
\]
is chain homotopic to the triangle map
\[
F_{D_{\gs}(\as), A(\bs), \Ds_{\S}}(-,\Theta_{A(\bs), \Ds_{\S}}).
\]
\end{lem}

\begin{proof}
By Lemma~\ref{lem:Wabgspecialcobordism2handle}, the special cobordism
$(W_{\as,\bs,\gs})^s$ is a 2-handle cobordism, for a framed link $\bL \subset
-U_{\bs} \cup_{R_{\bs}} U_{\bs}$ described above. By adapting the proof of
Lemma~\ref{lem:Heegaardtripleissubordinatetobouquet}, we see that the triple
\[
(D_{\gs}(\S), D_{\gs}(\as), A(\bs), \Ds_{\S})
\]
can be related by a sequence of handle slides and isotopies to a triple that
is subordinate to a bouquet for $\bL$. Consequently,  $F_{(W_{\as,\bs,\gs})^s}$
is chain homotopic to the triangle map above.
\end{proof}

\subsection{A holomorphic triangle description of the gluing map}

Let $\cT = (\S,\as,\bs,\gs)$ be a sutured Heegaard triple. In this
section, we present a natural candidate for the map for gluing
$(Z_{\as,\bs,\gs},-\xi_{\as,\bs,\gs})$ to $(-M_{\as,\bs},-\g_{\as,\bs}) \sqcup (-M_{\bs,\gs},-\g_{\bs,\gs})$,
and prove that it is indeed the gluing map.

Let
\[
C(\cH_{\bs,\gs}) =
(R_{\bs}\,\natural\bar{\S}\,\natural R_{\gs}, \bar{\gs}\cup \ds_{\gs},\bar{\bs}\cup \ds_{\bs})
\]
be a weak conjugate of $\cH_{\bs,\gs} = (\S,\bs,\gs)$,
as described in Section~\ref{sec:weaklyconjugated}. We now construct our candidate
\[
\begin{split}
\Phi\colon \CF(\cH_{\as,\bs})\otimes \CF(C(\cH_{\bs,\gs}))
\to \CF(A(\cT))
\end{split}
\]
for the gluing map. Notice that the domain of the map $\Phi$ is $\CF(M_{\as,\bs}, \g_{\as,\bs}) \otimes
\CF(M_{\bs,\gs}, \g_{\bs,\gs})$, and its range is $\CF(M_{{\as},{\bs}} \cup_{R_{\bs}}
M_{{\bs},{\gs}}, \g_{\as,\gs})$.

The definition of the map $\Phi$ is formally similar to the map for connected
sums due to Ozsv\'ath and Szab\'o~\cite{OSProperties}. We will call $\Phi$
the \emph{amalgamation map}, since its image is in the Floer homology of the
amalgamated diagram from the previous section. We also remark that the
Ozsv\'{a}th--Szab\'{o} maps that inspire the construction of $\Phi$ were
called the \emph{intertwining maps} in~\cite{ZemDualityMappingTori},
where they played a similar role as in our present context.

Note that there is a 1-handle map
\[
F_1^{\bar{\gs} \cup \ds_{\gs}, \bar{\gs}\cup \ds_{\gs}} \colon  \CF(\cH_{\as,\bs}) \to
\CF(D_{\gs}(\S), D_{\gs}(\as), D_{\gs}(\bs)),
\]
where $D_{\gs}(\as) = \as \cup \bar{\gs} \cup \ds_{\gs}$ and $D_{\gs}(\bs) =
\bs \cup \bar{\gs} \cup \ds_{\gs}$.

Recall that the underlying Heegaard surface of
$C(\cH_{\bs,\gs})$ is $R_{\bs}\,\natural \bar{\Sigma}\, \natural R_{\gs}$,
and the Heegaard surface $D_{\gs}(\Sigma)$ is
$\Sigma\, \natural \bar{\Sigma}\,  \natural R_{\gs}$.
Since the surface $R_{\bs}$ is the result of surgering
$\Sigma$ along the $\bs$ curves, it follows that there is a 1-handle map
\[
F_1^{\bs,\bs} \colon \CF(C(\cH_{\bs,\gs}))
\to \CF(D_{\gs}(\S), D_{\gs}(\bs), A(\bs)),
\]
obtained by adding the $\bs$ curves back into $R_{\bs}\subset C(\Sigma)$.
Note that, after we add these 1-handles, turning $R_{\bs}$ into $\S$,
the curves $\ds_{\bs}$ become $\Ds_{\bs}$.

Finally, we define
\[
\Phi := F_{D_{\gs}(\as), D_{\gs}(\bs), A(\bs)} \circ
\left(F_1^{\bar{\gs} \cup \ds_{\gs}, \bar{\gs}\cup \ds_{\gs}} \otimes F_1^{\bs,\bs} \right).
\]

A key ingredient in our analysis of the sutured cobordism $\cW_{\as,\bs,\gs}$ is the following:

\begin{prop}\label{prop:gluingmapformula}
The amalgamation map
\[
\Phi \colon \SFH(M_{\as,\bs}, \g_{\as,\bs}) \otimes \SFH(M_{\bs,\gs}, \g_{\bs,\gs}) \to
\SFH(M_{\as,\bs} \cup_{R_{\bs}} M_{\bs,\gs}, \g_{\as,\gs})
\]
defined above is chain homotopic
to the contact gluing map for gluing $(Z_{\as,\bs,\gs}, -\xi_{\as,\bs,\gs})$ to
\[
(-M_{\as,\bs}, - \g_{\as,\bs}) \sqcup (-M_{\bs,\gs}, -\g_{\bs,\gs}).
\]
\end{prop}

\begin{proof}
First, we describe a contact handle decomposition of
$(Z_{\as,\bs,\gs}, -\xi_{\as,\bs,\gs})$, relative to $R_{\as,\bs} \sqcup
R_{\bs,\gs}$.
We pick a collection of subintervals $\cI_0\subset \d R_{\bs}$
such that each component of $\d R_{\bs}$ contains exactly one subinterval.
We pick an arc basis $\lambda_1,\dots, \lambda_n$ for $(R_{\bs},\cI_0)$.
We claim that a contact handle decomposition of
$(Z_{\as,\bs,\gs},\xi_{\as,\bs,\gs})$, relative to $R_{\as,\bs}\sqcup
R_{\bs,\gs}$, can be constructed as follows:

\begin{enumerate}
\item The contact 1-handles are the components of $N(I\times \cI_0)$ (in particular, we have one 1-handle
for each component of $\d R_{\bs}$);
\item The contact 2-handles are $N(I\times \lambda_i)$ for $i \in \{\, 1, \dots, n\,\}$.
\end{enumerate}
The 1-handles are simply added with feet along the
corresponding subintervals of $\cI_0 \subset R_{\bs}\subset \d M_{\as,\bs}$
and $\bar{\cI}_0 \subset \bar{R}_{\bs}\subset \d M_{\bs,\gs}$.
The 2-handles are attached along curves
obtained by concatenating an arc $\lambda_i$
on $R_{\bs}$ with its mirror on $\bar{R}_{\bs}$.

To show that the above description determines a contact handle decomposition,
we use Lemma~\ref{lem:xi-abg-istwoIinvariantgluedtogether},
 which allows us to decompose $Z_{\as,\bs,\gs}$ along the convex surface
 obtained by extending $\{\tfrac{1}{2}\}\times R_{\bs}\subset Z_{\bs}$
across the solid tori $Z_0$ to get a decomposing surface which is diffeomorphic to $R_{\bs}$.
See Figures~\ref{fig::51} and~\ref{fig::60} for schematics of the decomposing surface.
Phrased another way, we can write
 $(Z_{\as,\bs,\gs}, -\xi_{\as,\bs,\gs})$
 as $(I\times R_{\as,\bs}, -\xi_{\as,\bs})$ and $(I\times R_{\bs,\gs}, -\xi_{\bs,\gs})$ glued together
along two subsurfaces $\bar{R}_{\bs}' \subset \{0\}\times  R_{\as,\bs}$ and
$R_{\bs}' \subset \{0\}\times R_{\bs,\gs}$ that are small perturbations of
$\bar{R}_{\bs} \subset \{0\}\times \bar{R}_{\as,\bs} $ and
$R_{\bs} \subset\{0\}\times  R_{\bs,\gs} $, respectively.
We pick $R_{\bs}'$ and $\bar{R}_{\bs}'$ such that
$\cI_0 \subset R_{\bs}'$ and $\bar{\cI}_0 \subset \bar{R}_{\bs}'$. We identify
$R_{\bs}'$ and $\bar{R}_{\bs}'$ with their images on $\d M_{\as,\bs}$ and
$\d M_{\bs,\gs}$, respectively.

The above description implies that the cores of
the 1-handles in our contact handle decomposition are Legendrian, and the
attaching circles of the 2-handles cross the dividing set exactly twice.
Using Legendrian realization after attaching the contact 1-handles, we can
assume the attaching curves of the 2-handles are Legendrian, and hence have
$\tb = -1$. It follows that the above description determines a contact handle
decomposition of $Z_{\as,\bs,\gs}$, starting at $R_{\as,\bs} \sqcup
R_{\bs,\gs}$ and ending at $R_{\as,\gs}$.
Using the above contact handle decomposition of
$(Z_{\as,\bs,\gs}, -\xi_{\as,\bs,\gs})$, we can now give a description of the
gluing map associated to $(Z_{\as,\bs,\gs}, -\xi_{\as,\bs,\gs})$ on the level
of Heegaard diagrams.

The contact 1-handle maps have a simple description in terms of diagrams; one
simply adds a band with ends at the feet of the 1-handle. There are two
dividing arcs on the boundary of the band, and they are distinguished: One intersects
$I\times R_{\bs}\subset Z_{\as,\bs,\gs}$ nontrivially, while the other  is
disjoint from $I\times R_{\bs}$; see Figure~\ref{fig::52}. Let us call the
edge that intersects $I\times R_{\bs} $ the ``interior'' dividing arc of the
1-handle. The other edge we call the ``exterior'' dividing arc. We will write
$E_{\Int}$ for the interior arcs, and $E_{\ext}$ for the exterior arcs.

\begin{figure}[ht!]
 \centering
 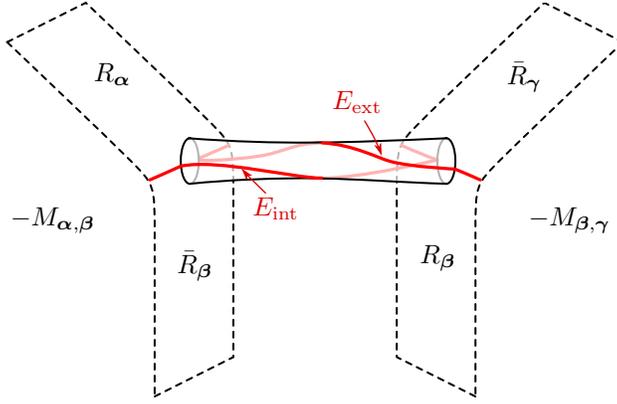
 \caption{A contact 1-handle added to $-M_{\as,\bs} \sqcup -M_{\bs,\gs}$.
 On the 1-handle, the interior edge of the dividing set (in $E_{\Int}$) is
 contained in $I\times R_{\bs}$, and the exterior edge
(in  $E_{\ext}$) is disjoint from $I\times R_{\bs}$. \label{fig::52}}
\end{figure}

Note that the attaching circles of the 2-handles only intersect the dividing
set along the contact 1-handles, and only along the interior arcs $E_{\Int}$.
Hence, on the level of diagrams, the contact 2-handle map
involves adding a band to $E_{\Int}$,
and also adding an $\a$-curve and a $\b$-curve.

We will write $\Sigma_0$ for the new portion of the Heegaard surface, which
is added by the contact 1-handles and 2-handles. We will shortly see that
$\Sigma_0$ can be identified with $\bar{R}_{\bs}$; see
equation~\eqref{eq:phidiffeomorphism} below. We call the collection of
$\b$-curves that we add with the contact 2-handles $\bar{\ds}_{\bs}$
(the notation will be justified below; see equation~\eqref{eq:phidiffeomorphism}),
and we call the $\a$-curves that we add $\es$.

By construction, the diagram for the domain of the amalgamation map is the disjoint union of
$\cH_{\as,\bs}=(\S,\as,\bs)$ and a weak conjugate
\[
C(\cH_{\bs,\gs}) = (R_{\bs}\,\natural \bar{\S}\,\natural
R_{\gs}, \bar{\gs}\cup \ds_{\gs},\bar{\bs}\cup \ds_{\bs})
\]
of $\cH_{\bs,\gs} = (\S,\bs,\gs)$.
After adding all the 1-handles and 2-handles, we obtain the Heegaard diagram
\[
G(\cT) = (\S(\cT), \as(\cT), \bs(\cT)) :=
(\S\,\natural \Sigma_0 \,\natural R_{\bs}\,\natural \bar{\S}\,\natural R_{\gs},\as
\cup \es \cup \bar{\gs} \cup \ds_{\gs},
\bs \cup \bar{\ds}_{\bs} \cup \bar{\bs} \cup \ds_{\bs})
\]
for $M_{\as,\bs}\cup_{R_{\bs}} M_{\bs,\gs}$.
Note that the diagram $G(\cT)$ is similar to, but not equal to, the Heegaard diagram
\[
A(\cT) =
(\S\,\natural \bar{\S}\,\natural R_{\gs},\as \cup \bar{\gs} \cup
\ds_{\gs},\bs \cup \bar{\bs} \cup \Ds_{\bs}),
\]
of $M_{\as,\bs} \cup_{R_{\bs}} M_{\bs,\gs}$ that appears in the target of the map in
Proposition~\ref{prop:gluingmapformula}. We now describe the curves
$\ds_{\bs}$, $\es$, and $\bar{\ds}_{\bs}$, and the surface $\Sigma_0$
more explicitly. For each contact 1-handle, there is a corresponding 0-handle of the
surface $\Sigma_0$. For each contact 2-handle, we attach a 1-handle to $\Sigma_0$,
along the interior edges $E_{\Int}$ of the portion of $\Sigma_0$ built when attaching the
contact 1-handles. For each contact 2-handle, we also add a
curve in~$\es$ and a curve in~$\bar{\ds}_{\bs}$.

Let $c_i\subset \Sigma_0$ be the core of the
band of the contact 2-handle associated to the arc $\lambda_i$ on
$R_{\bs}$. We extend the curves $c_i$ across the bands of the 1-handles
such that they have both ends on $E_{\ext}$; see the top left and top right of
Figure~\ref{fig::23}. If we isotope each $c_i$ near $E_{\ext}$ such that its ends
are in $\cI_0$, and then double it across $\cI_0$ onto $R_{\bs}\subset -M_{\bs,\gs}$,
in a sense specified in equation \eqref{eq:epsdiffeo} below, then we get
the $\es$ curves. If we isotope each $c_i$ near $E_{\ext}$ in the opposite
direction until its ends are in $\bar{\cI}_0$, and then we double it across
$\bar{\cI}_0$ onto $\S\subset -M_{\as,\bs}$,
in a sense specified in equation~\eqref{eq:deltadiffeo} below, then we get the $\bar{\ds}_{\bs}$ curves.
Examples of the curves $\bar{\ds}_{\bs}$ and $\es$ on $\bar{R}_{\bs}$
are shown in Figure~\ref{fig::23}.

Let $\cI_1$ denote the collection of subarcs of $\d R_{\bs}$
that are used to connect $R_{\bs}$ to $\bar{\S}$ in the weakly
conjugated diagram $C(\cH_{\bs,\gs})$.  Since the curves $\es$
were added by the contact 2-handles, and the 2-handles have attaching circles
constructed using the basis of arcs $\lambda_1,\dots, \lambda_n$
for $(R_{\bs},\cI_0)$, the arcs $\es\cap R_{\bs}$ cut $R_{\bs}$
into a collection of disks, each of which contains a
single component of $\cI_1$. Similarly, since $\es\cap \Sigma_0$
are the cores of the 1-handles used to build $\Sigma_0$,
the arcs $\es\cap \Sigma_0$ cut $\Sigma_0$ into a
collection of disks, each of which contains exactly
one arc of $\bar{\cI}_0$ along its boundary. Consequently, there is
an orientation reversing diffeomorphism
\begin{equation}
\phi\colon \Sigma_0\to R_{\bs},\label{eq:phidiffeomorphism}
\end{equation}
specified up to isotopy by the property that
\begin{equation}\label{eq:epsdiffeo}
\phi(\es\cap \Sigma_0)=\es\cap R_{\bs},
\end{equation}
and also that  $\cI_0\subset \Sigma_0$ is
mapped to $\cI_0 \subset R_{\bs}$, and $\bar{\cI}_0 \subset \Sigma_0$ is mapped to
$\cI_1 \subset R_{\bs}$.

Since we have freedom to assume that the curves $\ds_{\bs}$ are constructed
with any convenient basis of $H_1(R_{\bs}, \cI_1)$, we can assume that they
are constructed using the images of the arcs $\bar{\ds}_{\bs}\cap
\Sigma_0$ under $\phi$. Consequently, we assume
\begin{equation}\label{eq:deltadiffeo}
\phi(\bar{\ds}_{\bs}\cap \Sigma_0)=\ds_{\bs}\cap R_{\bs}.
\end{equation}
In light of equation~\eqref{eq:phidiffeomorphism},
we will henceforth write $\bar{R}_{\bs}$ for $\Sigma_0$.

\begin{figure}[ht!]
 \centering
 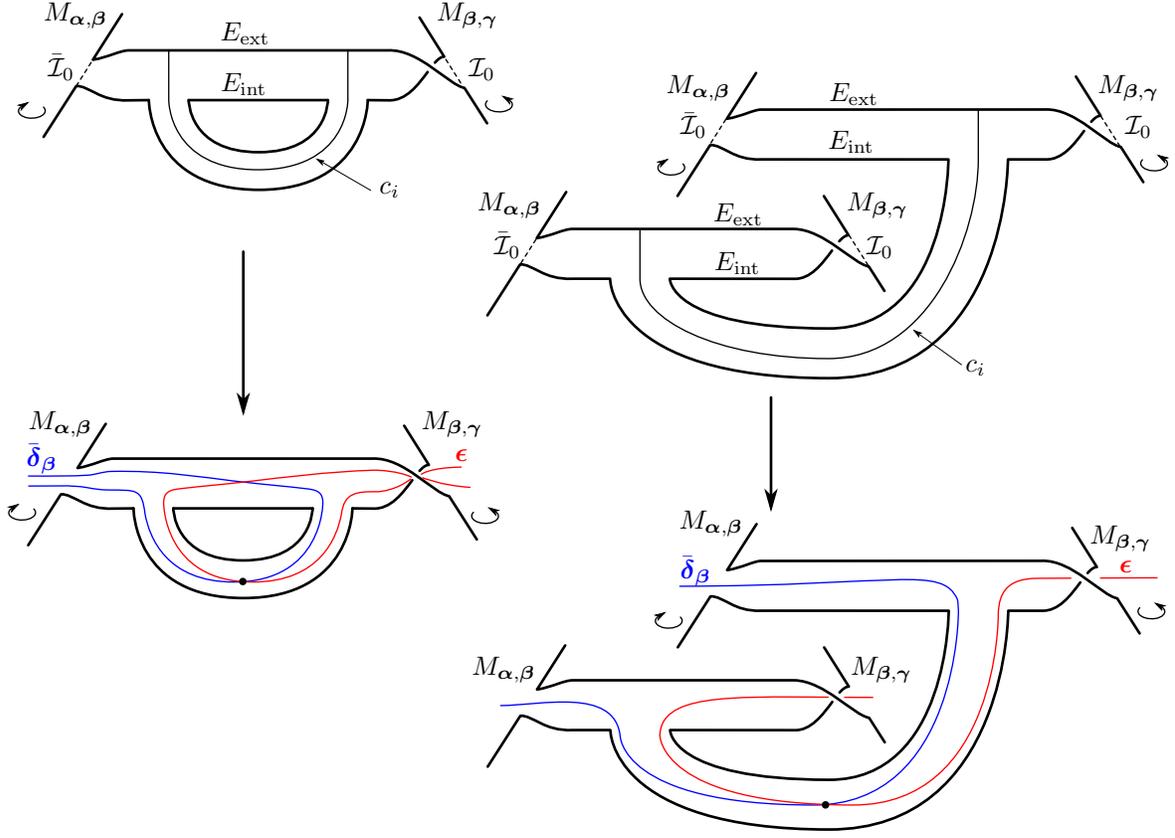
 \caption{The effect on the Heegaard diagram of attaching a single contact 2-handle in our decomposition of
 $Z_{\as,\bs,\gs}$. The regions shown consist of a band for a contact
 2-handle with both ends attached to a single contact 1-handle (left) or a
 pair of contact 1-handles (right). Viewing the band from the contact
 2-handle as a 1-handle added to the Heegaard surface, the core $c_i$ is
 shown in the top row. Isotoping the ends of $c_i$ in a neighborhood of $E_{\ext}
 \cup \cI_0$ so they lie in $\cI_0$, and then doubling across $\cI_0$, we get
 $\es$. Isotoping the ends of $c_i$ in a neighborhood of $E_{\ext} \cup
 \bar{\cI}_0$ so they lie in $\bar{\cI}_0$, we get $\ds_{\bs}$.  The
 orientation of the Heegaard surfaces for $M_{\as,\bs}$ and $M_{\bs,\gs}$ are
 shown.
 \label{fig::23}}
\end{figure}

The next step to understanding the gluing map is to destabilize the region
$\bar{R}_{\bs}\natural R_{\bs}$. Unfortunately, the curves $\bar{\ds}_{\bs}$
and $\es$ are not suitable for this (even if we use the compound
destabilization operation from Section~\ref{sec:compoundstabilization}). In
order to present the curves in a reasonable manner, we need to do some
handle slides amongst the $\ds_{\bs}$, $\bar{\ds}_{\bs}$, and $\es$ curves, as
we now describe.

We perform two moves. First, we modify the $\es$ curves, as follows. Recall
that we obtained $\bar{\ds}_{\bs}$ and $\es$ by isotoping the cores of the
2-dimensional 1-handles near the $E_{\ext}$ boundary arcs of the contact 1-handles. Let us write
$c_i^*$ for the co-core of the handle with core $c_i$. By construction,
$|c_i^* \cap c_j| = \delta_{ij}$.

Since the bands associated to the contact 2-handles
are all attached to $E_{\Int}$, we can subdivide each
component of $\d \bar{R}_{\bs}$ into four subarcs, which we
label as $\cI_0$, $E_{\ext}$, $\bar{\cI}_0$, or $E_{\Int}$,
in a way which is compatible with the analogous designation of the boundaries
of the contact 1-handle bands.

We now isotope each $c_i^*$ in a neighborhood of $E_{\Int}\cup \cI_0$ such that
its ends lie in $\cI_0$. Let $\es'$ denote the closed curves obtained by
doubling the resulting curves across $\cI_0$, onto $R_{\bs}$. Note that we
perform this isotopy after \emph{all} the 2-handle bands have been added (not
after just the corresponding 2-handle has been attached). An example is shown
in Figure~\ref{fig::24}.

\begin{figure}[ht!]
 \centering
 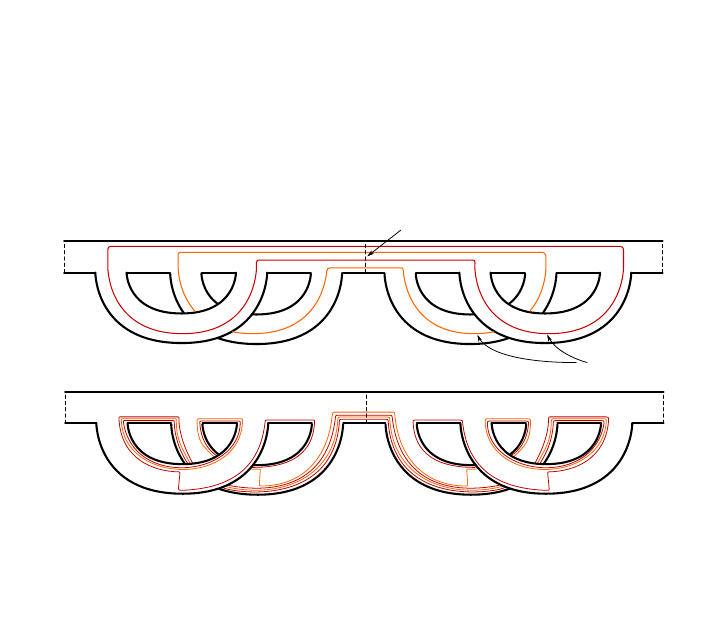
 \caption{The portion of the Heegaard surface identified with $\bar{R}_{\bs}\natural R_{\bs}$.
 In this case, $R_{\bs}$ is a genus 1 surface with one boundary component.
 The curves $\es$ are shown in the second row. The curves $\es'$ are shown in the third.
 The curves $\ds_{\bs}$ and $\bar{\ds}_{\bs}$ are shown in the last row. \label{fig::24}}
\end{figure}

Since
\[
|(\ds_{\bs})_i\cap \es'_j| = |(\bar{\ds}_{\bs})_i\cap \es'_j|=\delta_{ij},
\]
we can handle slide $\ds_{\bs}$ over $\bar{\ds}_{\bs}$ along $\es'$ in such a way
that the resulting curves $\hat{\ds}_{\bs}$ do not intersect the $\es'$ curves.
With this configuration, we note that $\es'$ intersects only $\bar{\ds}_{\bs}$,
and, furthermore, the two sets of curves come in pairs.
A sequence of destabilizations can then be used to surger out the $\es'$ curves,
while removing the $\bar{\ds}_{\bs}$ curves. Once we do this,
we are left with the diagram $A(\cT)$ of $M_{\as,\bs}\cup_{R_{\bs}} M_{\bs,\gs}$
described in Section~\ref{sec:glueddiagram}.

Note that $\es$ and $\es'$ are both obtained by
picking a set of arcs which form a basis of
$H_1(\bar{R}_{\bs},\cI_0)$, and then doubling them onto
$\bar{R}_{\bs}\natural R_{\bs}$. By Lemma~\ref{lem:allowablearcslides},
two such collections of arcs can be connected by a
sequence of allowable arc slides (Definition~\ref{def:allowablearcslice}).
An allowable arc slide with respect to $\cI_0$ induces
a handle slide of the corresponding curves obtained
by doubling across $\cI_0$. Consequently, it follows that
$\es$ and $\es'$ can be connected by a sequence of handle
slides on $\bar{R}_{\bs}\natural R_{\bs}$, though
the particular sequence of handle
slides is not of importance for us.

We now describe the effect of these Heegaard moves on the sutured Floer complexes,
and relate them to the desired triangle map formula.
We write
\[
C^{\es,\bar{\ds}_{\bs}} \colon \CF(\cH_{\as,\bs}) \otimes \CF(C(\cH_{\bs,\gs})) \to \CF(G(\cT))
\]
for the contact gluing map, obtained by adding the surface $\bar{R}_{\bs}$ to
$\Sigma \sqcup (R_{\bs} \,\natural \bar{\Sigma} \,\natural R_{\gs})$ to get
$\S(\cT) = \S\,\natural \bar{R}_{\bs} \,\natural R_{\bs}\,\natural \bar{\S}\,\natural R_{\gs}$,
and then adding in the curves $\es$ and $\bar{\ds}_{\bs}$. If
$\xs\in \CF(\cH_{\as,\bs})$ and $\ys\in \CF(C(\cH_{\bs,\gs}))$, the map
$C^{\es,\bar{\ds}_{\bs}}$ is defined by the formula $\xs\times \ys\mapsto
\xs\times \ve{c}\times \ys$, where $\ve{c}$ is the distinguished intersection
point of $\bT_{\es}\cap \bT_{\bar{\ds}_{\bs}}$.

On the other hand, we can precompose the map $C^{\es,\bar{\ds}_{\bs}}$ with a
transition map for a small isotopy of the curves $\bs$ on $\cH_{\as,\bs}$,
and of the curves $\bar{\gs} \cup \ds_{\gs}$ on $C(\cH_{\bs,\gs})$.
By a small abuse of notation, we also write $\bs$ and $\bar{\gs} \cup \ds_{\gs}$ for the
translates. Hence
\begin{equation}
\Phi_{-\xi}(\xs\times \ys) := C^{\es,\bar{\ds}_{\bs}} \circ \left(\Psi_{\as}^{\bs \to \bs}(\xs) \otimes
\Psi_{\bar{\gs} \cup \ds_{\gs} \to \bar{\gs} \cup \ds_{\gs}}^{\bar{\bs} \cup \ds_{\bs}}(\ys)\right).
\label{eq:first-formula-gluing-map-glued-diagram}
\end{equation}
The transition maps in equation~\eqref{eq:first-formula-gluing-map-glued-diagram} can be computed
using holomorphic triangle maps. We note that, even though we are doing a move
of the curves $\bs$ on $\cH_{\as,\bs}$ and a move of the curves $\bar{\gs} \cup \ds_{\gs}$ on
$C(\cH_{\bs,\gs})$, since the two diagrams are disjoint, we can compute the
holomorphic triangles simultaneously to arrive at the formula
\begin{equation}
\Phi_{-\xi}(\xs\times \ys) = C^{\es,\bar{\ds}_{\bs}} \circ
F_{\as\cup \bar{\gs}\cup \ds_{\gs}, \bs\cup \bar{\gs} \cup \ds_{\gs}, \bs \cup \bar{\bs} \cup \ds_{\bs}}
\left(\xs\times \Theta_{\bar{\gs} \cup \ds_{\gs}, \bar{\gs} \cup \ds_{\gs}}, \ys\times \Theta_{\bs,\bs}\right).
\label{eq:second-formula-gluing-map-glued-diagram}
\end{equation}

Using the local computation shown in Figure~\ref{fig::50}, we can commute the
map $C^{\es,\bar{\ds}_{\bs}}$ to the right of the triangle map. When we do this,
the top-graded generators $\Theta_{\bar{\gs} \cup \ds_{\gs}, \bar{\gs}\cup \ds_{\gs}}$
and $\Theta_{\bs,\bs}$ can be rewritten as the images of sequences of 1-handle maps.
Hence, we can write equation~\eqref{eq:second-formula-gluing-map-glued-diagram} as
\begin{equation}
\Phi_{-\xi}(\xs\times \ys) =
F_{\as\cup \es\cup \bar{\gs}\cup \ds_{\gs}, \bs\cup\es\cup \bar{\gs}\cup \ds_{\gs}, \bs\cup\bar{\ds}_{\bs} \cup \bar{\bs} \cup \ds_{\bs}}
\left(F_1^{\es\cup \bar{\gs}\cup \ds_{\gs}, \es\cup \bar{\gs}\cup \ds_{\gs}}(\xs),
F_1^{\bs,\bs}\left(C^{\es,\bar{\ds}_{\bs}}(\ys)\right)\right)
\label{eq:third-formula-gluing-map-glued-diagram}
\end{equation}
To compactify the notation, we rewrite Equation~\eqref{eq:third-formula-gluing-map-glued-diagram} as
\begin{equation}
\Phi_{-\xi} = F_{\as(\cT),\bs \cup \gs', \bs(\cT)} \circ
\left(F_1^{\gs',\gs'} \otimes \left(F_1^{\bs,\bs} \circ C^{\es,\bar{\ds}_{\bs}}\right)\right),\label{eq:gluingmaptrianglesv1}
\end{equation}
where
\begin{align*}
\as(\cT)&:=\as\cup \es\cup \bar{\gs}\cup \ds_{\gs}\\
\bs(\cT)&:=\bs\cup \bar{\ds}_{\bs} \cup \bar{\bs} \cup \ds_{\bs}, \text{ and}\\
\gs'&:=\es\cup \bar{\gs}\cup \ds_{\gs}.
\end{align*}
In the above expression, the contact handle map $C^{\es,\bar{\ds}_{\bs}}$ is
defined by adding the surface $R_{\bs} \,\natural \bar{R}_{\bs}$ to $R_{\bs}
\,\natural \bar{\S} \,\natural R_{\gs}$ to get $R_{\bs} \,\natural \bar{R}_{\bs} \,\natural R_{\bs}
\,\natural \bar{\S} \,\natural R_{\gs}$, and then adding in the curves $\es$
and $\bar{\ds}_{\bs}$. On the level of complexes, it is defined by the
formula $\ve{y} \mapsto \ve{c} \times \ve{y}$, where $\ve{c}$ is the
distinguished intersection point of $\bT_{\es} \cap \bT_{\bar{\ds}_{\bs}}$.
The copy of $R_{\bs}$ that is added should be thought of as $\S$ surgered
along the $\bs$ curves. Note that $C^{\es, \bar{\ds}_{\bs}}$ can be described
by attaching contact 0-handles and 1-handles to add $R_{\bs}$, and then
attaching contact 2-handles to add $\bar{R}_{\bs}$ and the curves $\es$ and
$\bar{\ds}_{\bs}$.

We now change $\es$ to $\es'$. Define
\[
\gs'' := \es'\cup \bar{\gs}\cup \ds_{\gs}.
\]
Applying associativity to the 4-tuple $(\as\cup \gs'', \as(\cT), \bs\cup \gs', \bs(\cT))$, we conclude from equation~\eqref{eq:gluingmaptrianglesv1} that
\begin{equation}
\begin{split}
\Psi_{\as(\cT) \to \as \cup \gs''}^{\bs(\cT)} \circ \Phi_{-\xi}&\simeq \Psi_{\as(\cT) \to \as \cup \gs''}^{\bs(\cT)}\circ F_{\as(\cT),\bs \cup \gs', \bs(\cT)} \circ
\left(F_1^{\gs',\gs'} \otimes \left(F_1^{\bs,\bs} \circ C^{\es,\bar{\ds}_{\bs}}\right)\right)\\
&\simeq F_{\as\cup \gs'', \bs\cup \gs',\bs(\cT)}\circ \left(\left(\Psi_{\as(\cT)\to \as\cup \gs''}^{\bs\cup \gs'}\circ F_1^{\gs',\gs'}\right)\otimes \left(F_1^{\bs,\bs}\circ C^{\es,\bar{\ds}_{\bs}}\right)\right).
\end{split}
\label{eq:gluingmaptrianglesv2}
\end{equation}

Applying naturality of sutured Floer homology as well as associativity for the 4-tuple $(\as\cup \gs'', \bs\cup \gs', \bs\cup \gs'', \bs(\cT))$ to equation~\eqref{eq:gluingmaptrianglesv2}, we conclude that
\begin{equation}
\begin{split}
& \,F_{\as\cup \gs'', \bs\cup \gs',\bs(\cT)}\circ \left(\left(\Psi_{\as(\cT)\to \as\cup \gs''}^{\bs\cup \gs'}\circ F_1^{\gs',\gs'}\right)\otimes \left(F_1^{\bs,\bs}\circ C^{\es,\bar{\ds}_{\bs}}\right)\right)\\
\simeq & \,F_{\as\cup \gs'', \bs\cup \gs',\bs(\cT)}\circ \left(\left(\id \circ \Psi_{\as(\cT)\to \as\cup \gs''}^{\bs\cup \gs'}\circ F_1^{\gs',\gs'}\right)\otimes \left(F_1^{\bs,\bs}\circ C^{\es,\bar{\ds}_{\bs}}\right)\right)\\
\simeq & \,F_{\as\cup \gs'', \bs\cup \gs',\bs(\cT)}\circ \left(\left( \left(\Psi_{\as\cup \gs''}^{\bs\cup \gs''\to \bs\cup \gs'}\circ \Psi_{\as\cup \gs''}^{\bs\cup \gs'\to \bs\cup \gs''}\right)\circ \Psi_{\as(\cT)\to \as\cup \gs''}^{\bs\cup \gs'}\circ F_1^{\gs',\gs'}\right)\otimes \left(F_1^{\bs,\bs}\circ C^{\es,\bar{\ds}_{\bs}}\right)\right)\\
\simeq&\,F_{\as \cup \gs'', \bs \cup \gs'', \bs(\cT)}
\circ \left(
\left(\Psi_{\as(\cT) \to \as \cup \gs''}^{\bs \cup \gs' \to \bs \cup \gs''} \circ
F_1^{\gs',\gs'} \right) \otimes
\left(\Psi_{\bs \cup \gs' \to \bs \cup \gs''}^{\bs(\cT)} \circ
F_{1}^{\bs,\bs} \circ C^{ \es, \bar{\ds}_{\bs}} \right)\right)
\end{split}
\label{eq:gluingmaptrianglesv3}
\end{equation}

Combining equations~\eqref{eq:gluingmaptrianglesv1}, \eqref{eq:gluingmaptrianglesv2},
and~\eqref{eq:gluingmaptrianglesv3}, we conclude
\[
\Psi_{\as(\cT) \to \as \cup \gs''}^{\bs(\cT)} \circ \Phi_{-\xi} \simeq
F_{\as \cup \gs'', \bs \cup \gs'', \bs(\cT)}
\circ \left(
\left(\Psi_{\as(\cT) \to \as \cup \gs''}^{\bs \cup \gs' \to \bs \cup \gs''} \circ
F_1^{\gs',\gs'} \right) \otimes
\left(\Psi_{\bs \cup \gs' \to \bs \cup \gs''}^{\bs(\cT)} \circ
F_{1}^{\bs,\bs} \circ C^{ \es, \bar{\ds}_{\bs}} \right)\right).
\]
We claim that
\begin{equation}
\Psi_{\as(\cT) \to \as \cup \gs''}^{\bs \cup \gs' \to \bs \cup \gs''} \circ
F_1^{\gs',\gs'} \simeq F_1^{\gs'',\gs''}.
\label{eq:1--handlemapwelldefined}
\end{equation}

To establish equation~\eqref{eq:1--handlemapwelldefined}, first note that the domain of both maps is
$\SFH(\S,\as,\bs)$. The 1-handle maps $F_1^{\gs',\gs'}$ and
$F_1^{\gs'',\gs''}$ are defined by taking the boundary connected sum of $\S$ with
$\S' := \bar{R}_{\bs}\natural R_{\bs} \natural \bar{\S} \natural R_{\gs}$. The boundary
connected sum operation yields $|\d \S|$ arcs on $\S(\cT) =
\S\natural \S'$ that separate $\S$ from $\S'$, and intersect none of
the attaching curves in $\as$, $\bs$, $\gs'$, or $\gs''$. Hence, the change of
diagrams map appearing on the left-hand side of
equation~\eqref{eq:1--handlemapwelldefined} involves only counting
holomorphic triangles with image on the disjoint union of $\S$ and
$\S'$. Since there is a unique top-graded generator of
$\SFH(\S', \gs', \gs')$, and $\gs''$ is obtained from $\gs'$ by a sequence of
handle slides, that generator will be preserved by the change of
diagrams map from $\SFH(\S', \gs', \gs')$ to $\SFH(\S',\gs'',\gs'')$.
Hence, the 1-handle map will be preserved.
Equation~\eqref{eq:1--handlemapwelldefined} now follows.

We obtain that the gluing map is chain homotopic to
\[ F_{\as \cup \gs'', \bs \cup \gs'', \bs(\cT)}
\circ \left(
F_1^{\gs'',\gs''} \otimes
\left(\Psi_{\bs \cup \gs' \to \bs \cup \gs''}^{\bs(\cT)} \circ
F_{1}^{\bs,\bs} \circ C^{ \es, \bar{\ds}_{\bs}} \right)\right).
\]
Note that, on the diagram $(\S(\cT), \bs \cup \gs'', \bs(\cT))$, the $\bar{\ds}_{\bs}$
curves each have only one intersection point, which occurs with an $\es'$ curve.
The $\es'$ curves still intersect both $\ds_{\bs}$ and
$\bar{\ds}_{\bs}$.
Further, each $\es'$ curve intersects exactly one $\bar{\ds}_{\bs}$ curve.
Hence, we can consider the compound stabilization map $\sigma^{\es',\bar{\ds}_{\bs}}$ defined in
Section~\ref{sec:compoundstabilization}.
It agrees with the map from naturality by Proposition~\ref{prop:compoundstabilizationmap}.

We claim that
\[
\Psi_{\bs \cup \gs' \to \bs \cup \gs''}^{\bs(\cT)} \circ
F_{1}^{\bs,\bs} \circ C^{ \es, \bar{\ds}_{\bs}} \simeq F_1^{\bs,\bs} \circ \sigma^{\es',\bar{\ds}_{\bs}}
\colon \SFH(C(\cH_{\bs,\gs})) \to \SFH(\S(\cT), \bs \cup \gs'', \bs(\cT)).
\]
By Proposition~\ref{prop:1handletrianglecount}, it suffices to show the claim with the
$\bs$ curves surgered out, and with no 1-handle maps (note that it is easy to show that
on $\S$, there is a path from each $\bs$ curve to $\d \S$ that avoids $\bar{\ds}_{\bs}\cap \S$,
so the hypotheses of the previously mentioned proposition are satisfied).
Therefore, it suffices to show that
\[
\sigma^{\es',\bar{\ds}_{\bs}} \simeq
\Psi_{\gs' \to \gs''}^{\bar{\ds}_{\bs} \cup \bar{\bs} \cup \ds_{\bs}} \circ C^{ \es, \bar{\ds}_{\bs}},\]
or, equivalently, that
\[
\id \simeq \left(\sigma^{\es',\bar{\ds}_{\bs}}\right)^{-1} \circ
\Psi_{\gs' \to \gs''}^{\bar{\ds}_{\bs} \cup \bar{\bs} \cup \ds_{\bs}} \circ C^{ \es, \bar{\ds}_{\bs}}.
\]
However, this holds because of the functoriality of the gluing map; i.e.,
because the right-hand side represents the map
induced by gluing a trivial a copy of $I\times R_{\bs}'$ to $-M_{\bs,\gs}$.

Hence, we conclude that the gluing map $\Phi_{-\xi}$ is chain homotopic to
\[ F_{\as \cup \gs'', \bs \cup \gs'', \bs(\cT)}
\circ \left(
F_1^{\gs'',\gs''} \otimes \left(F_1^{\bs,\bs} \circ \sigma^{\es',\bar{\ds}_{\bs}}\right)
\right).
\]
We still need to handle slide the $\ds_{\bs}$ over $\bar{\ds}_{\bs}$ to become $\hat{\ds}_{\bs}$,
and then compound destabilize. That is, the gluing map is chain homotopic to
\[
\left( \sigma^{\es', \bar{\ds}_{\bs}} \right)^{-1} \circ
\Psi^{\bs(\cT) \to \bs(\cT)'}_{\as \cup \gs''} \circ
F_{\as \cup \gs'', \bs \cup \gs'', \bs(\cT)}
\circ \left(
F_1^{\gs'',\gs''} \otimes \left(F_1^{\bs,\bs} \circ \sigma^{\es',\bar{\ds}_{\bs}}\right)
\right),
\]
where $\bs(\cT)' = \bs \cup \bar{\ds}_{\bs} \cup \bar{\bs} \cup \hat{\ds}_{\bs}$.
By associativity applied to the 4-tuple $(\as \cup \gs'', \bs \cup \gs'', \bs(\cT), \bs(\cT)')$,
this is chain homotopic to
\[
\left(\sigma^{\es',\bar{\ds}_{\bs}}\right)^{-1} \circ
\left(F_{\as \cup \gs'', \bs \cup \gs'', \bs(\cT)'} \circ
\left(F_1^{\gs'', \gs''} \otimes
\left(\Psi_{\bs \cup \gs''}^{\bs(\cT) \to \bs(\cT)'} \circ
F_1^{\bs,\bs} \circ  \sigma^{\es',\bar{\ds}_{\bs}} \right)
\right)\right).
\]
By a simple generalization of Lemma~\ref{lem:compoundtrianglemap1}
to deal with multiple compound stabilizations of the Heegaard triple simultaneously,
 this now becomes
\[
F_{D_{\gs}(\as),D_{\gs}(\bs), A(\bs)} \circ
\left(
F_1^{C(\gs), C(\gs)} \otimes
\left(\left(\sigma^{\es',\bar{\ds}_{\bs}} \right)^{-1} \circ
\Psi_{\bs \cup \gs''}^{\bs(\cT) \to \bs(\cT)'} \circ
F_1^{\bs,\bs} \circ \sigma^{\es',\bar{\ds}_{\bs}}
\right)\right),
\]
where $C(\gs) = \bar{\gs}\cup \ds_{\gs}$.
Applying the triangle counts from Proposition~\ref{prop:1handletrianglecount}
to the map  $F_1^{\bs,\bs}$ for 1-handles added near the boundary,
and also commuting the map $F_1^{\bs,\bs}$ with the destabilization map, this becomes
\[
F_{D_{\gs}(\as), D_{\gs}(\bs), A(\bs)} \circ
\left(
F_1^{C(\gs), C(\gs)} \otimes
\left(
F_1^{\bs,\bs} \circ
\left(\sigma^{\es',\bar{\ds}_{\bs}} \right)^{-1} \circ
\Psi_{\gs''}^{\bs(\cT) \setminus \bs \to \bs(\cT)' \setminus \bs} \circ \sigma^{\es',\bar{\ds}_{\bs}}
\right)\right),
\]
We claim that
\[
\left(\sigma^{\es',\bar{\ds}_{\bs}} \right)^{-1} \circ
\Psi_{\gs''}^{\bs(\cT) \setminus \bs \to \bs(\cT)' \setminus \bs} \circ
\sigma^{\es',\bar{\ds}_{\bs}} \simeq \id,
\]
which follows simply from naturality, as it is a loop in the space of Heegaard diagrams.
We have now arrived at our desired formula, concluding the proof of Proposition \ref{prop:gluingmapformula}.
\end{proof}

\subsection{Computation of the triangle cobordism map}

We now prove that the sutured cobordism map for $\cW_{\as,\bs,\gs}$ is chain homotopic to the map that
counts holomorphic triangles on a single Heegaard triple, by using the formula
from the previous section for the contact gluing map.

\begin{thm}
If $(\S,\as,\bs,\gs)$ is a sutured Heegaard triple, then the sutured cobordism map
\[
F_{\cW_{\as,\bs,\gs}}\colon \CF(\S,\as,\bs)\otimes \CF(\S,\bs,\gs)\to \CF(\S,\as,\gs)
\]
is chain homotopic to the map that counts holomorphic triangles on the triple $(\S,\as,\bs,\gs)$.
\end{thm}

\begin{proof}
We first compose with the change of diagrams map
$\id \otimes \Psi_{\cH_{\bs,\gs} \to C(\cH_{\bs,\gs})}$.
The next step is to use the gluing map to glue the two copies of $R_{\bs}$ together.
Then we performs surgery on a $|\bs|$-component framed link.
See Section~\ref{sec:glueddiagram} for a description of the
triple diagram
\[
(D_{\gs}(\S), D_{\gs}(\as), A(\bs), \Ds_{\S})
\]
used for computing the 2-handle cobordism map.
This yields (omitting writing the first change of diagrams map)
\[F_{D_{\gs}(\as), A(\bs), \Ds_{\S}}\circ
\left(\left(
F_{D_{\gs}(\as), D_{\gs}(\bs), A(\bs)} \circ
\left(
F_1^{\bar{\gs}\cup\ds_{\gs},\bar{\gs}\cup\ds_{\gs}}\otimes F_1^{\bs,\bs}
\right)\right) \otimes
\Theta_{A(\bs), \Ds_{\S}}
\right).
\]
We now use associativity applied to the 4-tuple
$(D_{\gs}(\as), D_{\gs}(\bs), A(\bs), \Delta_{\Sigma})$
to see that this is chain homotopic to
\[
F_{D_{\gs}(\as), D_{\gs}(\bs), \Ds_{\S}} \circ
\left(
F_1^{\bar{\gs}\cup\ds_{\gs},\bar{\gs}\cup\ds_{\gs}} \otimes
\left(
F_{D_{\gs}(\bs), A(\bs), \Ds_\S} \circ
\left(
F_1^{\bs,\bs} \otimes \Theta_{A(\bs), \Ds_{\S}}
\right)
\right)\right).
\]

Note that
\[
\Psi := F_{D_{\gs}(\bs), A(\bs), \Ds_\S} \circ
\left(
F_1^{\bs,\bs} \otimes \Theta_{A(\bs), \Ds_{\S}}
\right)
\]
is the change of diagrams map $\Psi_{C(\cH_{\bs,\gs})\to D_{\gs}(\cH_{\bs,\gs})}$,
from a weakly conjugated diagram to a doubled diagram, by Lemma~\ref{lem:changeofdiagramsdoubletoconj}.
Thus, the cobordism map is
\begin{equation} \label{eq:cobordismmapeq1}
F_{D_{\gs}(\as), D_{\gs}(\bs), \Ds_{\S}} \circ
\left(
F_1^{\bar{\gs}\cup\ds_{\gs},\bar{\gs}\cup\ds_{\gs}} \otimes \Psi
\right).
\end{equation}
The range of this map is not $\SFH(\cH_{\as,\gs})$, but instead a double of this diagram along $R_{\gs}$,
so we must compose with the change of diagrams map
 from $D_{\gs}(\cH_{\as,\gs})$ to $\cH_{\as,\gs}$,
which is
\[
F_3^{\bar{\gs}\cup \ds_{\gs}, \bar{\gs}\cup \ds_{\gs}} \circ
\left(
F_{D_{\gs}(\as), \Ds_{\S}, D_{\gs}(\gs)} \circ
\left(- \otimes \Theta_{\Ds_{\S}, D_{\gs}(\gs)}
\right)\right),
\]
by Lemma~\ref{lem:changeofdiagramsdoubleddiagram}.
By post composing equation~\eqref{eq:cobordismmapeq1} with this expression,
and applying associativity of the 4-tuple
$(D_{\gs}(\as), D_{\gs}(\bs), \Delta_{\Sigma}, D_{\gs}(\gs)$, see conclude
that the composition is chain homotopic to
\[
F_3^{\bar{\gs}\cup \ds_{\gs}, \bar{\gs}\cup \ds_{\gs}} \circ
\left(
F_{D_{\gs}(\as), D_{\gs}(\bs), D_{\gs}(\gs)} \circ
\left(
F_1^{\bar{\gs}\cup\ds_{\gs},\bar{\gs}\cup \ds_{\gs}} \otimes
\left(
F_{D_{\gs}(\bs), \Ds_{\S}, D_{\gs}(\gs)} \circ
\left(\Psi \otimes \Theta_{\Ds_{\S}, D_{\gs}(\gs)}
\right)\right)\right)\right).
\]

Using the 3-handle and triangle map computation from Proposition~\ref{prop:1handletrianglecount}
(note that it is an easy exercise to verify that the hypotheses of that proposition are satisfied),
this is equal to
\[
F_{\as, \bs, \gs} \circ
\left(- \otimes F_3^{\bar{\gs}\cup \ds_{\gs}, \bar{\gs}\cup \ds_{\gs}} \circ
\left(
F_{D_{\gs}(\bs), \Ds_{\S}, D_{\gs}(\gs)} \circ
\left(\Psi \otimes \Theta_{\Ds_{\S}, D_{\gs}(\gs)}
\right)\right)\right).
\]
We note that
\[
F_3^{\bar{\gs}\cup \ds_{\gs}, \bar{\gs}\cup \ds_{\gs}} \circ
\left(
F_{D_{\gs}(\bs), \Ds_{\S}, D_{\gs}(\gs)} \circ
\left(\Psi \otimes \Theta_{\Ds_{\S}, D_{\gs}(\gs)}
\right)\right)
\]
is just a change of diagrams map. Writing $\Psi'$ for the compositions of all
the three change of diagrams maps (the initial one from $\cH_{\bs,\gs}$ to a
weakly conjugated diagram $C(\cH_{\bs,\gs})$, which we have been omitting
writing, and then the last two, going back to $\cH_{\bs,\gs}$ through a
doubled diagram), the composition is thus equal to
\[
F_{\as,\bs,\gs}(-\otimes \Psi')\simeq F_{\as,\bs,\gs}(-,-),
\]
since $\Psi'\simeq \id_{(\S,\bs,\gs)}$ by naturality.
\end{proof}

\section{Trace and cotrace cobordisms}

Let $\xi_{I \times \d M}$ denote the $I$-invariant contact structure
on $-I \times \d M$ that induces the dividing set~$\g$ on~$\d M$.
In this section, we consider the trace cobordism
\[
\Lambda_{(M,\g)} = (I \times M, -I \times \d M, [\xi_{I \times \d M}])
\]
from $(M,\g) \sqcup (-M,\g)$ to $\emptyset$,
and the cotrace cobordism
\[
V_{(M,\g)} = (I \times M, -I \times \d M, [\xi_{I \times \d M}])
\]
from $\emptyset \to (-M,\g) \sqcup (M,\g)$.
The main result of this section is the following:

\begin{thm}\label{thm:traceandcotracecobordismmaps}
The trace cobordism $\Lambda_{(M,\g)} \colon (M,\g)\sqcup (-M,\g) \to \emptyset$
induces the canonical trace map
\[
\tr \colon \SFH(M,\g) \otimes \SFH(-M,\g) \to \bF_2.
\]
The cotrace cobordism $V_{(M,\g)} \colon
\emptyset \to (-M,\g) \sqcup(M,\g)$ induces the canonical cotrace map
\[
\cotr \colon \bF_2 \to \SFH(-M,\g) \otimes \SFH(M,\g).
\]
\end{thm}

To prove the Theorem~\ref{thm:traceandcotracecobordismmaps},
we first need a convenient topological description of the trace cobordism.
Suppose that $(\S,\as,\bs)$ is a diagram for $(M,\g)$. We will write $(M_{\as,\as}, \g_{\as,\as})$
for the sutured manifold constructed from the diagram $(\S,\as,\as)$, and we note that
$(M_{\as,\as}, \g_{\as,\as})$ can be obtained by surgering $(I\times R_{\as}, I\times \d R_{\as})$
along $k$ 0-spheres. We consider the triangular sutured manifold cobordism
\[
\cW_{\as,\bs,\as} = (W_{\as,\bs,\as}, Z_{\as,\bs,\as}, [\xi_{\as,\bs,\as}])
\]
from $(M,\g) \sqcup (-M,\g)$ to $(M_{\as,\as}, \g_{\as,\as})$,
defined in Section~\ref{sec:trianglecobordisms}.

There is also a cobordism $\cW_{\as} = (W_{\as} , Z_{\as} , [\xi_{\as}])$ from
$(M_{\as,\as}, \g_{\as,\as})$ to $\emptyset$. The 4-manifold $W_{\as}$ is  defined as
\[
W_{\as} = (P_1 \times \S) \cup_{e_{\as} \times \S} (e_{\as} \times U_{\as}).
\]
Here $P_1$ denotes a monogon, viewed as having a single boundary edge $e_{\as}$.

\begin{lem}
The cobordism $\Lambda_{(M,\g)}$ is equivalent to the composition
$\cW_{\as} \circ \cW_{\as,\bs,\as}$.
\end{lem}

\begin{proof}
It follows from \cite{JCob}*{Proposition~6.6} that $W_{\as} \cup
W_{\as,\bs,\as}$ is diffeomorphic to $I \times M_{\as,\bs}$, and furthermore,
that $Z_{\as,\bs,\as} \cup Z_{\as}$ is diffeomorphic to $-I \times \d M$. The
fact that $\xi_{\as}$ and $\xi_{\as,\bs,\as}$ glue up to give $\xi_{I \times \d M}$
can be proven by adapting Lemma~\ref{lem:xi-abg-istwoIinvariantgluedtogether}.
Schematically, the decomposition of $\Lambda_{(M,\g)}$ into $\cW_{\as,\bs,\as}$
and $\cW_{\as}$ is shown in Figure~\ref{fig::31}.
\end{proof}

\begin{figure}[ht!]
 \centering
 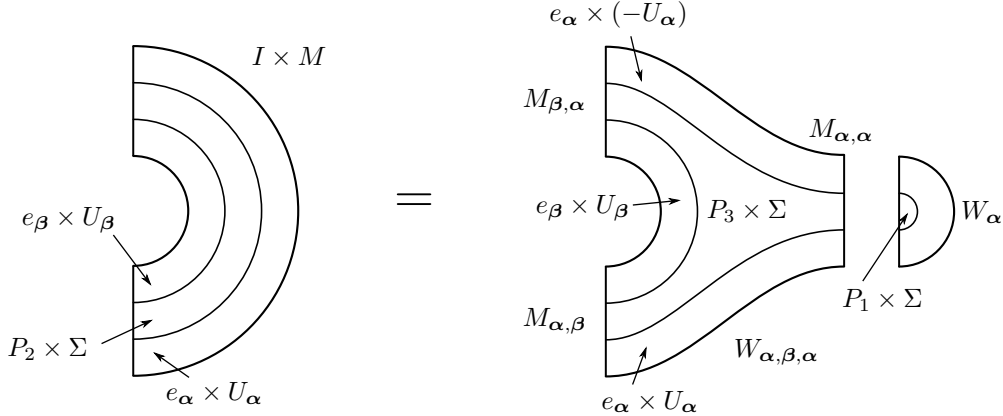
 \caption{The decomposition $I \times M \iso W_{\as,\bs,\as} \cup_{M_{\as,\as}} W_{\as}$.
 It is convenient to view the polygons $P_1$, $P_2$, and $P_3$ as having fattened vertices. \label{fig::31}}
\end{figure}

Let $\as'$ denote be a small Hamiltonian translate of $\as$.
Note that, of course, $(M_{\as,\as}, \g_{\as,\as})$ and $(M_{\as,\as'},\g_{\as,\as'})$ are homeomorphic.
It follows from Theorem~\ref{thm:trianglemapiscobmap} and the composition law that the map
\[
F_{\Lambda_{(M,\g)}} \colon \CF(\S,\as,\bs) \otimes \CF(\S,\bs,\as) \to \bF_2
\]
is equal to the composition
\begin{equation}
F_{\cW_{\as}} \circ F_{\as,\bs,\as'} \circ (\id_{\CF(\S,\as,\bs)} \otimes \Psi_{(\S,\bs,\as) \to (\S,\bs,\as')}),\label{eq:tracecobordismcomposition}
\end{equation}
where $F_{\as,\bs,\as'}$ is the map that counts holomorphic triangles on $(\S,\as,\bs,\as')$.

It remains to compute the cobordism map for $\cW_{\as}$.
Note that, on $(\S,\as,\as')$, there is a canonical bottom-graded intersection point $\Theta_{\as,\as'}^-$.

\begin{prop}\label{prop:Xascobordismmap}
The cobordism $\cW_{\as}$ from $(M_{\as,\as'}, \g_{\as,\as'})$ to $\emptyset$ induces the map
\[
\ve{x} \mapsto \begin{cases}1& \text{if } \ve{x} = \Theta_{\as,\as'}^-\\
0& \text{otherwise}.
\end{cases}
\]
\end{prop}
Before we prove the above result, we need to find a convenient Morse function
for the cobordism $\cW_{\as}$. As a first step, we prove the following:

\begin{lem}\label{lem:somemoreMorsetheory}
Let $U_{\as}$ be the sutured compression body formed by attaching
3-dimensional 2-handles to $I\times \S$ along the curves $\{0\}\times \as$.
After rounding corners,
we can view $U_{\as}$ as a (non-sutured) handlebody of genus $|\as| -
\chi(R_{\as}) + 1$ and boundary
\[
\d U_{\as} = (\{1\}\times \S) \cup \bar{R}_{\as},
\]
where $R_{\as}$ is the surface obtained by surgering $\S$ along
the $\as$ curves. Furthermore, a (possibly overcomplete) set of compressing
disks for $U_{\as}$ can be obtained by taking $|\as|$ compressing disks
$D_\a$ with boundary $\{1\}\times \a$ for $\a \in \as$, as well as disks of
the form $D_{c_i^*} := I\times c_i^*$, for pairwise disjoint, embedded arcs
$c_1^*, \dots, c_{b_1(R_{\as})}^*$ in $\S$ that avoid the $\as$ curves, and
form a basis of $H_1(R_{\as}, \d R_{\as})$. These cut $U_{\as}$ into
$b_0(R_{\as})$ 3-balls.
\end{lem}

\begin{rem}
A basis of arcs $c_1^*,\dots, c_{b_1(R_{\as})}^*$ can be obtained by picking a Morse function
$g \colon R_{\as} \to I$  (viewing $R_{\as}$ as a cobordism from $\emptyset$ to $\d R_{\as}$)
that has $b_0(R_{\as})$ local minima, and $b_1(R_{\as})$ index 1 critical points.
The Morse function $g$ determines a handle decomposition of $R_{\as}$ with $b_0(R_{\as})$
0-handles (i.e., disks) and $b_1(R_{\as})$ 1-handles. The arcs $c_1^*, \dots, c_{b_1(R_{\as})}^*$
can be taken as the co-cores of the 1-handles in this decomposition.
\end{rem}

\begin{proof}[Proof of Lemma \ref{lem:somemoreMorsetheory}]
View $U_{\as}$ as a cobordism from $\overline{\d U}_{\as} = (\{1\}\times \bar{\S}) \cup
R_{\as}$ to the empty set. The $\as$ curves determine 3-dimensional 2-handles
in $U_{\as}$. After attaching these 2-handles, the remaining cobordism is
homeomorphic to $I\times R_{\as}$ (with corners smoothed), viewed as a
cobordism from $R_{\as} \cup_{\d R_{\as}} \bar{R}_{\as}$ to $\emptyset$. In
such a way, we reduce the argument to the case when there are no $\as$
curves, where it is straightforward.
\end{proof}

We need an additional Morse theory argument:

\begin{lem}\label{lem:somemoreMorsetheory2}
Suppose that $U_{\as}$ is the sutured compression body induced by the sutured
monodiagram $(\S,\as)$. Then  $I \times U_{\as}$ can be viewed, after
smoothing corners, as a (non-sutured) cobordism from the closed manifold
$U_{\as} \cup_{\d U_{\as}}  -U_{\as}$ to $\emptyset$. Furthermore, there
is a Morse function $F \colon I \times U_{\as} \to I$ such that
\begin{itemize}
\item $F^{-1}(0) = U_{\as} \cup_{\d U_{\as}} -U_{\as}$,
\item $F$ has no index 0, 1, or 2 critical points,
\item $F$ has $|\as| + b_1(R_{\as})$ index 3 critical points, and
\item $F$ has $b_0(R_{\as})$ index 4 critical points.
\end{itemize}
The attaching spheres of the 3-handles are obtained taking the union of the disks $D_{\a_i}$
and $D_{c_i^*}$ (defined in Lemma~\ref{lem:somemoreMorsetheory}) in $U_{\as}$,
together with their images in $-U_{\as}$.
\end{lem}

\begin{proof}
A model for the 4-manifold obtained by rounding the corners of $I \times U_{\as}$
can be taken to be $X := I \times U_{\as}/{\sim}$, where
$(t,x) \sim (t',x)$ if $x \in \d U_{\as}$.
Using Lemma~\ref{lem:somemoreMorsetheory}, we can construct a Morse function $f \colon U_{\as} \to I$
that has $f^{-1}(0) = \d U_{\as}$, such that $f$ has no index 0 or 1 critical points,
$|\as| + b_1(R_{\as})$ index 2 critical points, and $b_0(R_{\as})$ index 3 critical points.
Furthermore, the descending manifolds of the index 2 critical points in $U_{\as}$
are the disks $D_{\a_i}$ and $D_{c_i^*}$.

To construct a Morse function on $X$ with the stated critical points,
the argument is now a modification of Lemma~\ref{lem:Morsetheoryturningaroundsuturedcobordism}.
More precisely, we will construct an auxiliary function
\[
G \colon (I \times I)/(I \times \{0\}) \to I.
\]
We view $(I \times I)/(I \times \{0\})$ as having the same smooth structure at
the point $I \times \{0\}$ as the upper half-plane; see Figure~\ref{fig::53}.
Furthermore, we assume that
\begin{itemize}
\item $G|_{I \times \{0\}} \equiv G|_{\{0\} \times I} \equiv G|_{\{1\} \times I} \equiv 0$,
\item $\d G/ \d s > 0$ in $(0,1) \times I$,
\item $(\d G /\d t)(t,s) = 0$ if and only if $(t,s) \in \{\tfrac{1}{2}\} \times I$.
\end{itemize}

An example of such a function $G$ is shown in Figure~\ref{fig::53}.
We then consider the function $F \colon X \to I$ defined by
\[
F(t,y) = G(t,f(y)).
\]
It is straightforward to verify that $F$ on $I \times U_{\as}$ with its corners rounded
is Morse, and has critical points with attaching spheres as stated.
\end{proof}

\begin{figure}[ht!]
 \centering
\begingroup%
  \makeatletter%
  \providecommand\color[2][]{%
    \errmessage{(Inkscape) Color is used for the text in Inkscape, but the package 'color.sty' is not loaded}%
    \renewcommand\color[2][]{}%
  }%
  \providecommand\transparent[1]{%
    \errmessage{(Inkscape) Transparency is used (non-zero) for the text in Inkscape, but the package 'transparent.sty' is not loaded}%
    \renewcommand\transparent[1]{}%
  }%
  \providecommand\rotatebox[2]{#2}%
  \newcommand*\fsize{\dimexpr\f@size pt\relax}%
  \newcommand*\lineheight[1]{\fontsize{\fsize}{#1\fsize}\selectfont}%
  \ifx\svgwidth\undefined%
    \setlength{\unitlength}{211.37383976bp}%
    \ifx\svgscale\undefined%
      \relax%
    \else%
      \setlength{\unitlength}{\unitlength * \real{\svgscale}}%
    \fi%
  \else%
    \setlength{\unitlength}{\svgwidth}%
  \fi%
  \global\let\svgwidth\undefined%
  \global\let\svgscale\undefined%
  \makeatother%
  \begin{picture}(1,0.62836145)%
    \lineheight{1}%
    \setlength\tabcolsep{0pt}%
    \put(0,0){\includegraphics[width=\unitlength,page=1]{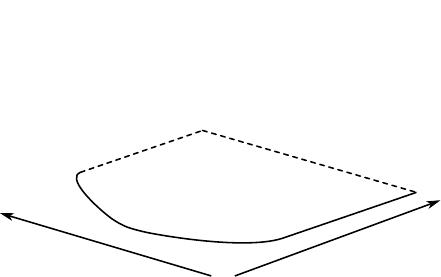}}%
    \put(0.76771578,0.05179151){\color[rgb]{0,0,0}\makebox(0,0)[lt]{\lineheight{1.25}\smash{\begin{tabular}[t]{l}$s$\end{tabular}}}}%
    \put(0.2423788,0.02665833){\color[rgb]{0,0,0}\makebox(0,0)[rt]{\lineheight{1.25}\smash{\begin{tabular}[t]{r}$t$\end{tabular}}}}%
    \put(0,0){\includegraphics[width=\unitlength,page=2]{fig53.pdf}}%
    \put(0.05269827,0.41605594){\color[rgb]{0,0,0}\makebox(0,0)[lt]{\lineheight{1.25}\smash{\begin{tabular}[t]{l}$G(t,s)$\end{tabular}}}}%
    \put(0,0){\includegraphics[width=\unitlength,page=3]{fig53.pdf}}%
  \end{picture}%
\endgroup%

 \caption{An example of an auxiliary function $G \colon (I \times I)/(I \times \{0\}) \to I$
 in Lemma~\ref{lem:somemoreMorsetheory2}. \label{fig::53}}
\end{figure}

We are now ready to prove Proposition~\ref{prop:Xascobordismmap}:

\begin{proof}[Proof of Proposition \ref{prop:Xascobordismmap}]
The sutured cobordism $\cW_{\as} = (W_{\as}, Z_{\as}, [\xi_{\as}])$
from $(M_{\as,\as'}, \g_{\as,\as'})$ to $\emptyset$ is the composition of the boundary
cobordism $\cW_{\as}^b$ obtained by gluing $Z_{\as} := I\times  R_{\as} $ to $\d
M_{\as,\as'} \iso R_{\as} \cup _{\d R_{\as}} \bar{R}_{\as}$, and the special cobordism
$\cW_{\as}^s$ (between closed 3-manifolds) diffeomorphic to $I\times U_{\as} $ from
$U_{\as} \cup -U_{\as}$ to $\emptyset$. We use
Lemma~\ref{lem:somemoreMorsetheory2} to give a handle decomposition of the
cobordism $\cW_{\as}^s$ consisting of $|\as| - \chi(R_{\as}) + 1$ 3-handles and
$b_0(R_{\as})$ 4-handles.

Note that this description of the cobordism $\cW_{\as}$ does not quite allow
us to compute the cobordism map $F_{\cW_{\as}}$, because when we glue
$Z_{\as} = I\times R_{\as} $ to $M_{\as,\as'}$, we obtain a closed (i.e.,
non-sutured manifold). The necessary modification is to instead remove
$b_0(R_{\as})$ 3-balls from $Z_{\as}$, to obtain a sutured maifold cobordism from
$(M_{\as,\as'}, \g_{\as,\as'})$ to $\bigsqcup_{i=1}^{b_0(R_{\as})} (B^3,\g_{0})$ (where
$\g_0 \subset \d B^3$ is a simple closed curve), and then compose
with the natural isomorphism
\[
\bigotimes_{i=1}^{b_0(R_{\as})} \SFH(B^3, \g_0) \iso \bF_2.
\]

Let us write $Z_{\as}'$ for $Z_{\as}$ with $b_0(R_{\as})$ 3-balls removed,
and $\cW'$ for the induced special cobordism from $\left(M_{\as,\as'} \cup
Z_{\as}', \bigsqcup_{i=1}^{b_0(R_{\as})} \g_{0}\right)$ to
$\bigsqcup_{i=1}^{b_0(R_{\as})} (B^3, \g_0)$. Note that $\cW'$ can be given a
handle decomposition that is the same as the handle decomposition for
$\cW_{\as}^s$ with the 4-handles removed.

We write $S_\a \subset M_{\as,\as'} \cup Z_{\as}'$ for the 2-spheres
obtained by doubling the compressing disk $D_\a$ for $\a \in \as$,
and $S_{c_i^*} \subset M_{\as,\as'} \cup Z_{\as}'$
for the 2-spheres obtained by doubling the compressing disk $D_{c_i^*} \subset U_{\as}$.
We recall that the curves $c_i^*$ were obtained by picking a Morse function on $R_{\as}$
that had $b_0(R_{\as})$ index 0 critical points and $b_1(R_{\as})$ index 1 critical points.
The index 1 critical points each determine a 2-dimensional 1-handle,
added to the handles of index 0. The curves $c_i^*$ are then the co-cores of these handles.

Note that, by the composition law for sutured cobordisms, we can commute the 3-handle maps
for the spheres $S_\a$ with the contact gluing map for gluing in $Z_{\as}'$.
As the composition of the 3-handle maps for the spheres $S_\a$ has the same formula
as the one in the statement of the proposition, under the identification of
\[
\SFH(I\times R_{\as} , I\times \d R_{\as}) \iso \bF_2,
\]
we thus reduce to the case when there are no $\as$ curves; i.e.,
when $(\S,\as)$ is the diagram $(R_{\as},\emptyset)$.

To see the claim when there are no $\as$ curves,
we note that the cobordism map is obtained by first gluing $Z'_{\as}$ to the product sutured manifold
$(I\times R_{\as} , I\times \d R_{\as})$,
and then attaching $b_1(R_{\as})$ 4-dimensional 3-handles. The contact manifold $Z'_{\as}$
is obtained by attaching $b_1(R_{\as})$ contact 2-handles.
If $c_i^*$ denotes the arc on $R_{\as}$ from Lemma~\ref{lem:somemoreMorsetheory},
obtained as the co-core of a handle decomposition of $R_{\as}$, as above,
then the attaching 1-spheres $s_i$ for the contact 2-handles forming $Z'_{\as}$ are given by
\[
s_i = (\{0,1\}\times c_i ) \cup (I\times \d c_i).
\]
Note that $s_i$ bounds the disk $D_{c_i^*}$, described in Lemma~\ref{lem:somemoreMorsetheory},
which is the descending manifold of the Morse function $f$,
constructed on $U_{\as}$ in the proof of Lemma~\ref{lem:somemoreMorsetheory2}.
The 2-sphere $S_{c_i^*}$ along which we attach the
$b_1(R_{\as})$ 3-handles are then equal to the union of $D_{c_i^*}$,
together with the core of the corresponding contact 2-handle.
Now an easy model computation shows that the composition of these
$b_1(R_{\as})$ contact 2-handle maps, and the $b_1(R_{\as})$ 3-handle maps,
sends $1 \in \SFH(I\times R_{\as} , I\times \d R_{\as} )$ to
$1 \in \bigotimes_{i=1}^{b_0(R_{\as})} \SFH(B^3, \g_0)$.
This model computation is shown in Figures~\ref{fig::33} and~\ref{fig::32}.

Using the composition law for sutured cobordisms, the formula for
the cobordism map $F_{\cW_{\as}}$ now follows.
\end{proof}

\begin{figure}[ht!]
 \centering
 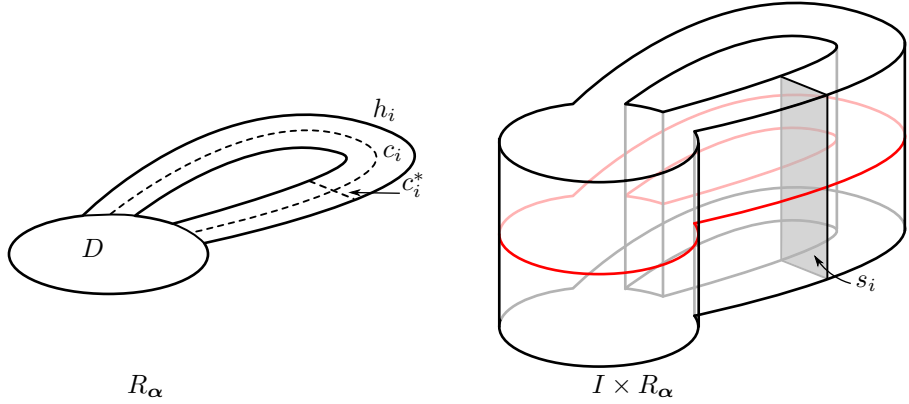
 \caption{$R_{\as}$ and $I\times R_{\as}$. On the left, a 1-handle $h_i$ from a handle decomposition
 of $R_{\as}$ is shown, as well as the core $c_i$ and the co-core $c_i^*$.
 On the right is the product manifold $(I\times R_{\as}, I\times \d R_{\as} )$,
 as well as the closed curve $s_i$, along which we attach a contact 2-handle.
 The red lines on the right indicate the sutures of $I\times R_{\as} $. \label{fig::33}}
\end{figure}

\begin{figure}[ht!]
 \centering
 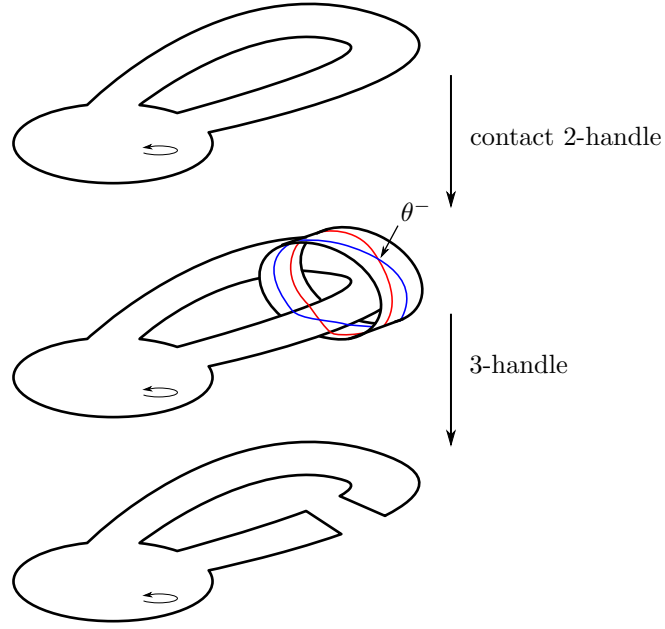
 \caption{The model computation from Proposition~\ref{prop:Xascobordismmap}.
 The contact 2-handle map takes the form $1 \mapsto \theta^-$,
 and the 3-handle map sends $\theta^-$ to $1$. \label{fig::32}}
\end{figure}

We can now prove the main theorem of this section:

\begin{proof}[Proof of Theorem~\ref{thm:traceandcotracecobordismmaps}]
Let us consider the trace cobordism map.
Recall from equation~\eqref{eq:tracecobordismcomposition} that $\Lambda_{(M,\g)}$
can be written as the composition of $\cW_{\as}$ with $\cW_{\as,\bs,\as}$.
Noting that $\cW_{\as,\bs,\as}$  is equivalent to $\cW_{\as,\bs,\as'}$
(where $\as'$ is a small Hamiltonian translate of $\as$),
and using the formula for $F_{\cW_{\as}}$ from Proposition~\ref{prop:Xascobordismmap},
we know that
\[
F_{\Lambda_{(M,\g)}} \colon  \CF(\S,\as,\bs) \otimes \CF(\S,\bs,\as) \to \bF_2
\]
takes the form
\[
F_{\Lambda_{(M,\g)}}(\xs \otimes \ys) =
\left\langle\, F_{\as,\bs,\as'}\left(\xs,\Psi_{\bs}^{\as\to \as'}(\ys)\right),
\Theta_{\as,\as'}^- \,\right\rangle,
\]
where
\[
\langle\, \zs,\zs' \,\rangle := \begin{cases} 1& \text{if } \zs = \zs', \text{ and} \\
0& \text{otherwise}.
\end{cases}
\]
Note, however, that the triangle map $F_{\as,\bs,\as'}$ counts the same triangles
as the triangle map $F_{\as',\as,\bs}$, and that $\Theta_{\as,\as'}^-\in \CF(\S,\as,\as')$
is the same intersection point as $\Theta_{\as',\as}^+ \in \CF(\S,\as',\as)$.
Hence
\[
\left\langle\, F_{\as,\bs,\as'}\left(\xs,\Psi_{\bs}^{\as\to \as'}(\ys)\right),
\Theta_{\as,\as'}^- \,\right\rangle =
\left\langle\, F_{\as',\as,\bs}\left(\Theta_{\as',\as}^+, \xs\right),
\Psi_{\bs}^{\as \to \as'}(\ys) \,\right\rangle.
\]
However, $F_{\as',\as,\bs}(\Theta_{\as',\as}^+, -)$ is the transition map
$\Psi_{\as\to \as'}^{\bs}(-)$, from naturality.
Hence the cobordism map becomes simply the composition
\[
\left\langle\, \Psi^{\bs}_{\as\to \as'}(\xs), \Psi_{\bs}^{\as\to \as'}(\ys) \,\right\rangle,
\]
which is easily seen to be $\left\langle\, \Psi_{\as'\to \as}^{\bs} \circ
\Psi_{\as\to \as'}^{\bs} (\xs), \ys \,\right\rangle =
\langle\, \xs, \ys \,\rangle$.

The formula for the cotrace cobordism map $F_{V_{(M,\g)}}$ follows from the above formula
for the trace cobordism map $F_{\Lambda_{(M,\g)}}$, combined with Theorem~\ref{thm:firstduality}.
\end{proof}

\section{Equivalence of two link cobordism map constructions}

In this section, we describe an application of the techniques of this paper to link Floer homology.

\subsection{Background on link Floer homology}

Knot Floer homology is an invariant of  knots embedded in 3-manifolds
constructed by Ozsv\'{a}th and Szab\'{o}~\cite{OSKnots},
and independently by Rasmussen~\cite{RasmussenKnots}.
Link Floer homology is a generalization to links,
constructed by Ozsv\'{a}th and Szab\'{o}~\cite{OSLinks}.

\begin{define} A \emph{multi-based link} $\bL = (L,\ws,\zs)$ in a 3-manifold $Y$ is an oriented link
$L\subset Y$, together with two disjoint collections of basepoints $\ws$, $\zs \subset L$ such that
\begin{enumerate}
\item each component of $L$ has at least two basepoints,
\item the basepoints along a link component of $L$ alternate between $\ws$ and $\zs$, as one traverses the link.
\end{enumerate}
\end{define}

To a multi-based link $\bL$ in $Y$, link Floer homology associates an $\bF_2$-module
\[
\hat{\HFL}(Y,\bL).
\]
To construct the modules, one picks a Heegaard diagram for the pair
$(Y,\bL)$, in the following sense:

\begin{define}A \emph{Heegaard diagram} $(\S,\as,\bs,\ws,\zs)$
\emph{for an oriented multi-based link} $(Y,(L,\ws,\zs))$ consists of the following:
\begin{enumerate}
\item A Heegaard diagram $(\S,\as,\bs)$ for $Y$, such that $Y\setminus \S$
is the union of two handlebodies $U_{\as}$ and $U_{\bs}$ that meet along $\S$.
\item $\S \cap L = \ws\cup \zs$.
\item Each component of $\S \setminus \as$ and $\S\setminus \bs$ contains exactly one $\ws$ basepoint and one $\zs$ basepoint.
\item $L \cap U_{\as}$ is isotopic in $U_{\as}$,
relative to $\d(L \cap U_{\as})$, to a collection of arcs in $\S \setminus \as$.
Similarly, $L \cap U_{\bs}$ is isotopic in $U_{\bs}$, relative to $\d(L \cap U_{\bs})$,
to a collection of arcs in $\S \setminus \bs$.
\item The link $L$ intersects $\S$ positively at the $\zs$ basepoints, and negatively at the $\ws$ basepoints.
\end{enumerate}
\end{define}

A somewhat more concise way of defining a Heegaard diagram for a multi-based links
is as a diagram for the sutured manifold $Y(\bL)$, obtained by removing a neighborhood of the link
$L$ from $Y$, and adding sutures to $\d (Y \setminus N(L))$ that are
positively oriented meridians of $L$ over the $\ws$ basepoints and negatively oriented meridians
of $L$ over the $\zs$ basepoints.
Link Floer homology, as defined by Ozsv\'{a}th and Szab\'{o}~\cite{OSLinks},
is easily seen to satisfy the isomorphism
\[
\hat{\HFL}(Y,\bL)\iso \SFH(Y(\bL)).
\]

If $(\S,\as,\bs,\ws,\zs)$ is a diagram for $(Y,\bL)$,
then the link Floer complex $\hat{\CFL}(\S,\as,\bs,\ws,\zs)$ is generated over
$\bF_2$ by intersection points $\xs\in \bT_{\as}\cap \bT_{\bs}$,
and the differential counts Maslov index~1
pseudo-holomorphic discs that go over none of the $\ws$ or $\zs$ basepoints.

For links in $S^3$ or null-homologous links in $(S^1\times S^2)^{\# n}$,
there is some additional structure on $\hat{\HFL}(S^3,\bL)$.
If $(\S,\as,\bs,\ws,\zs)$ is a diagram for a link in $S^3$,
one can define three relative gradings, $\gr_{\ws}$, $\gr_{\zs}$, and $A$,
on $\hat{\CFL}(\S,\as,\bs,\ws,\zs)$. If $\xs$ and $\ys$ are two intersection points,
then the gradings $\gr_{\ws}$ and $\gr_{\zs}$ are defined by picking a homology class
$\phi \in\pi_2(\xs,\ys)$ and setting
\[
\gr_{\ws}(\ve{x},\ve{y}) := \mu(\phi) - 2 n_{\ws}(\phi) \text{ and } \gr_{\zs}(\ve{x},\ve{y})
:= \mu(\phi) - 2n_{\zs}(\phi),
\]
where $n_{\ws}(\phi)$ and $n_{\zs}(\phi)$ denote the sum of the multiplicities of $\phi$
over the $\ws$ or $\zs$ basepoints, respectively. It is easy to see that the formulas
for $\gr_{\ws}$ and $\gr_{\zs}$ are independent of the choice of homology class $\phi$.
An absolute lift of the relative grading $\gr_{\ws}$ can be fixed
by requiring that $\hat{\HF}(\S,\as,\bs,\ws)$,
which is isomorphic as a relatively graded group to
$\bigotimes^{|\ws|-1}\left((\bF_2)_{-\frac{1}{2}} \oplus (\bF_2)_{\frac{1}{2}}\right)$,
have top-graded generator in grading $(|\ws|-1)/2$.
An absolute lift of the grading $\gr_{\zs}$ can be specified similarly.
Finally, the Alexander grading $A$ can be defined as
\[
A := \tfrac{1}{2}(\gr_{\ws} - \gr_{\zs}).
\]

\subsection{The link Floer homology TQFT}

In \cite{JCob}, the first author provided a construction of cobordism maps for decorated link cobordisms.
The construction used the following notion of cobordism between multi-based links:

\begin{define}
Let $Y_1$ and $Y_2$ be 3-manifolds containing multi-based links $\bL_1 = (L_1, \ws_1, \zs_1)$
and $\bL_2 = (L_2, \ws_2, \zs_2)$, respectively. A \emph{decorated link cobordism} from
$(Y_1, \bL_1)$ to $(Y_2, \bL_2)$ is a triple $(X,S,\cA)$, where
\begin{enumerate}
\item $X$ is an oriented cobordism from $Y_1$ to $Y_2$,
\item $S$ is a properly embedded oriented surface in $X$ with $\d S = -L_1 \cup L_2$, and
\item $\cA$ is a properly embedded 1-manifold in $S$
that divides $S$ into two subsurfaces $S_{\ws}$ and $S_{\zs}$
that meet along $\cA$, such that $\ws_1$, $\ws_2 \subset S_{\ws}$
and $\zs_1$, $\zs_2 \subset S_{\zs}$.
\end{enumerate}
\end{define}

Note that the above definition is slightly different from \cite[Definition~4.5]{JCob},
and follows \cite{ZemCFLTQFT}. The equivalence of the two definitions is
explained in ~\cite[Section~2.3]{JMComputeCobordismMaps}.

If $\cX = (X,S,\cA)$ from $(Y_1,\bL_1)$ to $(Y_2,\bL_2)$ is a decorated link cobordism,
then there is a well-defined cobordism $\cW(\cX) = (W,Z,[\xi])$ of sutured manifolds
from $Y(\bL_1)$ to $Y(\bL_2)$, as we now describe.
The 4-manifold $W$ is defined by the formula
\[
W := X \setminus N(S),
\]
where $N(S)$ denotes a regular neighborhood of $S$,
viewed as the unit normal disk bundle of $S$. The set $Z$ is defined as
the unit normal circle bundle of $S$, oriented as a submanifold of $\d W$.
Since $S$ is an oriented surface, $Z$ is a principal $S^1$-bundle over $S$.
According to Lutz~\cite{LutzS1InvariantContactStructures} and Honda~\cite{HondaClassII},
the dividing set $\cA$ uniquely determines an $S^1$-invariant contact structure on $Z$
with dividing set $\cA$ on $S$, up to isotopy.
The contact structure $\xi$ is defined to be this $S^1$-invariant contact structure.
The link cobordism map
\[
F^J_{\cX} \colon \HFLh(Y_1, \bL_1) \to \HFLh(Y_2, \bL_2)
\]
is defined to be the sutured cobordism map
\[
F_{\cW(\cX)} \colon \SFH(Y_1(\bL_1)) = \HFLh(Y_1, \bL_1) \to \SFH(Y_2(\bL_2)) = \HFLh(Y_2, \bL_2).
\]

The second author~\cite{ZemCFLTQFT} constructed another link cobordism map $F^Z_{\cX, \frs}$,
where $\frs \in \Spin^c(X)$ is a $\Spin^c$ structure on $X$,
that did not use the Honda--Kazez--Mati\'c gluing map.
It instead was defined by writing a link cobordism as a composition of elementary link cobordisms.
The map $F^Z_{\cX, \frs}$ is defined on a more general version of link Floer homology
than $F^J_{\cX}$, though it induces a map $\hat{F}_{\cX, \frs}^Z$ on the hat version.
Let
\[
\hat{F}_{\cX}^Z := \sum_{\frs \in \Spin^c(X)} \hat{F}_{\cX,\frs}^Z.
\]
It is not obvious that the maps $F^J_{\cX}$ and $\hat{F}^Z_{\cX}$ agree.

The first author and Marengon~\cite{JMComputeCobordismMaps} made some steps
towards computing the maps $F^J_{\cX}$ when $\cX$ is an elementary link
cobordism, though most of the computational results from
\cite{JMComputeCobordismMaps}  are still in terms of the
Honda--Kazez--Mati\'c gluing map, and hence it is challenging to directly
compare the maps $F^J_{\cX}$ and $\hat{F}^Z_{\cX}$. Nonetheless, combining
several results from \cite{JMComputeCobordismMaps} with the results of this
paper, we are able to prove the following:

\begin{thm}\label{thm:cobordismmapsagree}
Given a decorated link cobordism $\cX$, we have $F_{\cX}^J = \hat{F}_{\cX}^Z$.
\end{thm}

\subsection{Elementary link cobordisms}

In this section, we provide the following definition:
\begin{define}\label{def:elementarycobordisms}
We say a decorated link cobordism $\cX = (X,S,\cA) \colon (Y_1,\bL_1)\to (Y_2,\bL_2)$
is an \emph{elementary link cobordism} if one of the following is satisfied:
\begin{enumerate}
\item (\emph{Identity cobordism}) $(Y_1,\bL_1) = (Y_2,\bL_2) = (Y,\bL)$ and
  $(X,S,\cA) = (I \times Y, I \times L, I \times \ve{p})$, where $\ve{p} \subset L$
  consists of exactly one point in each component of $L \setminus (\ws\cup \zs)$.
\item (\emph{1-, 2-, or 3-handle attachment}) The cobordism $(X,S,\cA)$ is obtained
  by attaching a 4-dimensional 1-, 2-, or 3-handle, with framed attaching sphere disjoint from $\bL_1$.
\item (\emph{0- or 4-handle attachment}) The cobordism $\cX$ is obtained by attaching a
  4-dimensional 0-handle or 4-handle, viewed as a smooth 4-ball,
  that contains a standard disk intersecting the boundary of the 4-ball in an unknot,
  with dividing set consisting of a single arc on the disk.
\item (\emph{Saddle cobordism}) The cobordism $\cX$ has underlying 4-manifold $X = I \times Y$,
  and a surface $S$ such that projection to $I$ induces a Morse function that has a single index 1 critical point
  that occurs in a $\ws$-region or a $\zs$-region,
  and such that the dividing arcs all travel from $\bL_1$ to $\bL_2$.
\item (\emph{Stabilization cobordism}) The cobordism $\cX$ has underlying 4-manifold $X = I \times Y$
  and surface $S = I \times L$. Furthermore, exactly one arc of $\cA$ goes from $\bL_1$
  to $\bL_1$ or from $\bL_2$ to $\bL_2$. All other arcs are of the form $I \times \{p\}$
  for various $p \in L$. A stabilization cobordism is \emph{positive} if it adds
  two basepoints, and is \emph{negative} if it removes two basepoints.
\end{enumerate}
\end{define}

A schematic of a saddle cobordism can be found in Figure~\ref{fig::36}.
Examples of stabilization cobordisms are shown in Figure~\ref{fig::55}.

\begin{figure}[ht!]
 \centering
 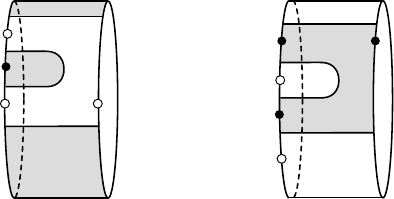
 \caption{Two examples of stabilization cobordisms. The two surfaces are each
 of the form $I \times L$, and sit inside $I \times Y$. The shaded regions
 are the $\ws$ regions, and the unshaded regions are the $\zs$
 regions.\label{fig::55}}
\end{figure}

It is straightforward to see that an arbitrary link cobordism $\cX$, such
that $\pi_0(S) \to \pi_0(X)$ is a surjection, can be decomposed into a
sequence of link cobordisms that are each diffeomorphic to one of the
elementary link cobordisms in the above list. We remark that the above list
is over-complete, in the following sense:

\begin{rem}\label{rem:stabilization=composition}
An elementary positive stabilization cobordism can be written as a
composition of a 0-handle cobordism (adding an unknot with two basepoints)
followed by a 1-handle, followed by a saddle cobordism. Similarly an
elementary negative stabilization cobordism can be written as a composition
of a saddle cobordism, followed by a 3-handle and a 4-handle.
\end{rem}

\subsection{Link triple diagrams and contact structures}

\begin{define}
A \emph{link triple diagram}
\[
(\S,\as,\bs,\gs,\ve{w},\zs)
\]
is a Heegaard triple $(\S,\as,\bs,\gs)$ with $2(n-g(\S) + 1)$ basepoints
$\ws \cup \zs \subset \S$, where $n = |\as| = |\bs| = |\gs|$,
such that each component of $\S \setminus \ve{\tau}$ for $\ve{\tau} \in \{\as,\bs,\gs\}$
is planar and contains exactly one $\ve{w}$ basepoint and exactly one $\zs$ basepoint.
\end{define}

Given a link triple diagram $\cT = (\S,\as,\bs,\gs,\ve{w},\ve{z})$,
we can construct a decorated link cobordism
\[
\cX_\cT = (X_\cT, S_\cT, \cA_\cT),
\]
as follows.
The 4-manifold $X_\cT$ is constructed as in \cite{OSTriangles}, by the formula
\[
X_\cT := (\Delta \times \S) \cup (e_{\as}\times U_{\as})
\cup (e_{\bs} \times U_{\bs}) \cup (e_{\gs} \times U_{\gs})/{\sim}.
\]
Here, the handlebody $U_{\ts}$ for $\ts \in \{\,\as,\bs,\gs\,\}$ is obtained by
attaching 3-dimensional 2-handles to $\S \times I$ along $\ts \times \{0\}$,
and filling in the resulting sphere boundary components with 3-dimensional 3-handles.

We obtain the surface $S_\cT$ as follows.
Pick Morse functions $f_{\as}$, $f_{\bs}$, and $f_{\gs}$ on $U_{\as}$, $U_{\bs}$, and $U_{\gs}$
that induce the attaching curves $\as$, $\bs$, and $\gs$, respectively.
By concatenating the ascending flow lines passing through the basepoints in $\ve{w}$ and $\ve{z}$,
we get a collection of $|\ve{w}|$ arcs $K_{\as}$, $K_{\bs}$, and $K_{\gs}$
in each of $U_{\as}$, $U_{\bs}$, and $U_{\gs}$, respectively.
Each arc has exactly one endpoint in $\ve{w}$ and one endpoint in $\ve{z}$,
and we orient it from $\ve{w}$ to $\ve{z}$.
Then the surface $S_\cT$ is defined as
\[
S_\cT := (\Delta \times (-\ve{w} \cup \ve{z})) \cup (e_{\as} \times K_{\as})
\cup (e_{\bs} \times K_{\bs}) \cup (e_{\gs} \times K_{\gs}).
\]

Finally, we describe the dividing set $\cA_\cT$ on $S_\cT$.
For $\ts \in \{\as, \bs, \gs\}$, let $p_{\ts} \subset K_{\ts}$
denote a collection of points obtained by picking a
single point in each arc of $K_{\ts}$. Then
\[
\cA_\cT := (e_{\as} \times p_{\as}) \cup (e_{\bs} \times p_{\bs}) \cup (e_{\gs} \times p_{\gs})
\]
is a dividing set on $S_\cT$.

Given a link triple diagram $\cT = (\S,\as,\bs,\gs,\ws,\zs)$,
we naturally obtain a sutured Heegaard triple $\cT_0 = (\S_0,\as,\bs,\gs)$, where
$\S_0 = \S \setminus N(\ws \cup \zs)$.
Let $\cW_{\cT_0} = (W_{\cT_0}, Z_{\cT_0}, [\xi_{\cT_0}])$
denote the associated sutured cobordism.

\begin{prop}\label{prop:contactstructuresagree}
Let $\cT = (\S, \as, \bs, \gs, \ws, \zs)$ be a link triple diagram,
and $\cT_0 = (\S_0, \as, \bs, \gs)$ the corresponding sutured triple diagram.
Then $Z_{\cT_0}$ is the unit normal circle bundle of $S_\cT$,
and there is a projection map $\pi \colon Z_{\cT_0} \to S_\cT$ with fiber $S^1$.
The contact structure $\xi_{\cT_0}$ on $Z_{\cT_0}$ described
in Section~\ref{sec:trianglecobordisms} is equivalent to the $S^1$-invariant contact structure
on $Z_{\cT_0}$ with respect to the dividing set $\cA_\cT$
and the projection map $\pi$.
\end{prop}

\begin{proof}
Let us write $\xi_{S^1}$ for the $S^1$-invariant contact structure on the unit
normal circle bundle $SN(S_{\cT})$ of $S_{\cT}$.
The proof of the proposition will be to describe a convex decomposition of
$(SN(S_{\cT}),\xi_{S^1})$ into the disjoint union of the contact manifolds
$(Z_0,\xi_0)$, $(Z_{\as}, \xi_{\as})$, $(Z_{\bs}, \xi_{\bs})$, and $(Z_{\gs}, \xi_{\gs})$,
whose union is $(Z_{\cT_0}, \xi_{\cT_0})$, by definition.
Note that $S_{\cT}$ has no closed components, so $SN(S_{\cT})$ is diffeomorphic
to $S_{\cT}\times S^1$, with the map $\pi$ given by projection onto the first
factor.

We will decompose $SN(S_{\cT}) \approx S_{\cT} \times S^1$ along $3|\d \S_0 |$ convex annuli. To
construct the annuli, it is convenient to view $\Delta$ as a smooth 2-disk,
and view the edges $e_{\as},$ $e_{\bs}$, and $e_{\gs}$ as being closed,
disjoint subintervals of $\d \Delta$. We let $\ve{A}_0$ denote the union of
annuli of the form $e_{\ts}\times \{w\}\times S^1$ and $e_{\ts}\times
\{z\}\times S^1$ inside $S_{\cT}\times S^1$ for $\ts \in \{\as, \bs, \gs\}$. Note that we cannot
decompose along the annuli in $\ve{A}_0$, since their boundaries are disjoint
from the dividing set, and hence we cannot use Legendrian
realization to ensure that $\d \ve{A}_0$ is Legendrian and $\ve{A}_0$ is convex.
Instead, we perform a finger move along each boundary component of each annulus in $\ve{A}_0$
until it intersects the dividing set. Let $\ve{A}$ denote the resulting collection of annuli.
The configuration of the annuli $\ve{A}_0$ and $\ve{A}$ are shown in Figure~\ref{fig::62}.

\begin{figure}[ht!]
 \centering
 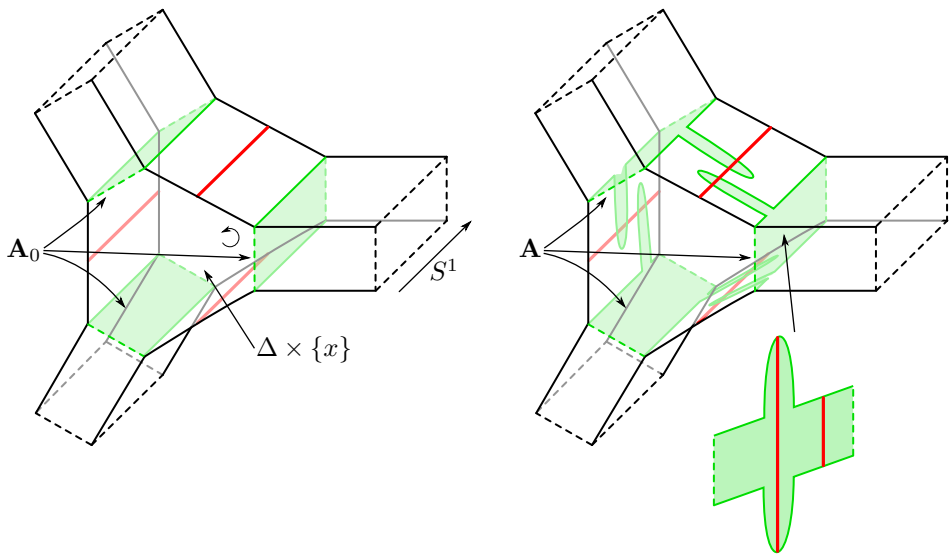
 \caption{The annuli $\ve{A}_0$ in $S_{\cT}\times S^1$ (left), and
 the annuli $\ve{A}$ obtained by performing finger moves along the boundaries
 toward the dividing set on $\d S_{\cT}\times S^1$ (right). We prove
 that the dividing set on each annulus in $\ve{A}$ is as shown in the bottom.
 The hexagonal region in the middle corresponds to a component of
 $\Delta\times \{x\}$, for a basepoint $x\in \ws\cup \zs$. The orientation of
 $\Delta\times \{x\}$ is shown, and we take the product orientation on
 $S_{\cT}\times S^1$ in the picture. \label{fig::62}}
\end{figure}

We can perturb the annuli in $\ve{A}$ so that they are convex with Legendrian
boundary (note this requires using Legendrian realization along $\d
S_{\cT} \times S^1$, and hence involves replacing $\xi_{S^1}$ with an
equivalent contact structure that is no longer $S^1$-invariant near the
boundary).

We now claim that the dividing sets on the annuli in $\ve{A}$ are as in the
bottom of Figure~\ref{fig::62}, consisting of two arcs that go from one
boundary component of the annulus to the other, and which do not wind around
the annulus. To see this, we first claim that, for an appropriately chosen
$S^1$-invariant contact 1-form, we can embed a neighborhood of each annulus
inside an $e_{\ts}$-invariant contact structure on $e_{\ts} \times I \times
S^1$ that has dividing set on $(\d e_{\ts}) \times I \times S^1$ equal to
$(\d e_{\ts}) \times \{\tfrac{1}{2}\} \times S^1$. To do this, we recall that
an $S^1$-invariant contact 1-form on $S_{\cT} \times S^1$ can be
defined as $\b + f \cdot d\theta$, where $\b$ is a 1-form on
$S_{\cT}$ and $f \colon S_{\cT}\to \R$ is a function that is zero
exactly on $\cA_{\cT}$. It is straightforward to write down conditions for
such a 1-form to be a contact form on $S_{\cT} \times S^1$. Since the
contact form $\b + f \cdot d\theta$ is essentially determined by the
characteristic foliation $\ker \b$ on $S_{\cT} \times \{p\}$, we
will focus on constructing a singular foliation with the appropriate dividing
set that is invariant under translation along a non-vanishing vector field
(thought of as $\d/ \d e_{\ts}$) in a neighborhood of each annulus in
$\ve{A}$. It is an elementary, though somewhat tedious, exercise to
explicitly construct an appropriate contact 1-form by picking $\b$ and $f$
appropriately, so we will leave that step to the reader.
An example of an appropriately chosen characteristic foliation on
$S_{\cT}\times \{p\}$ is shown in Figure~\ref{fig::61}.

\begin{figure}[ht!]
 \centering
 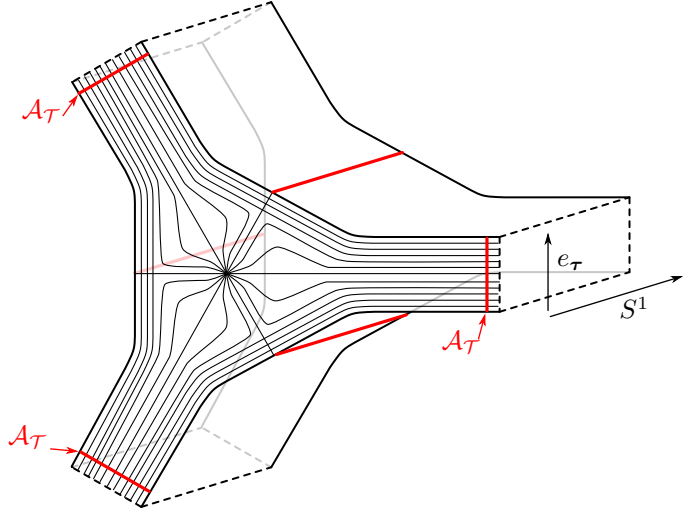
 \caption{The characteristic foliation of $S_{\cT} \times \{p\}$ of an $S^1$-invariant
 contact 1-form on $S_{\cT} \times S^1$ with dividing set $\cA_{\cT}$. \label{fig::61}}
\end{figure}

\begin{figure}[ht!]
 \centering
 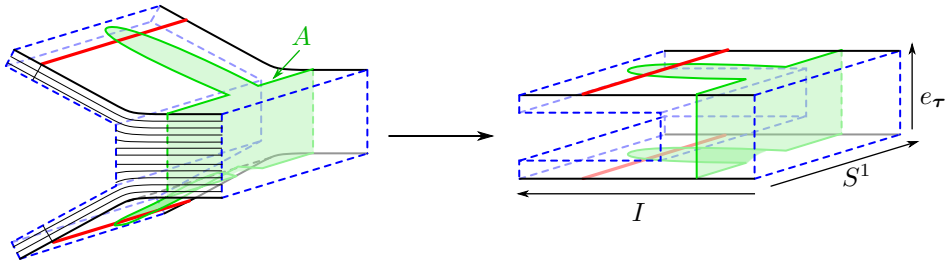
 \caption{A neighborhood of an annulus  $A \in \ve{A}$ in $S_{\cT} \times S^1$,
 which embeds into an $e_{\ts}$-invariant contact structure on
 $e_{\ts} \times I \times S^1$. \label{fig::63}}
\end{figure}

We can choose a neighborhood of an annulus in $\ve{A}$ that embeds into an
$e_{\ts}$-invariant contact structure on $e_{\ts} \times I \times S^1$
as in Figure~\ref{fig::63}.
Note that, inside $e_{\ts} \times I \times S^1$, the annulus $A$ pictured on
the right of Figure \ref{fig::63} is isotopic, relative to its boundary, to a
surface of the form $e_{\ts} \times s$, for a Legendrian $s$ in
$(\d e_{\ts}) \times I \times S^1$. The characteristic foliation and dividing set on
$e_{\ts} \times s$ are the same as on the annulus shown in
Figure~\ref{fig::59}. Namely, the characteristic foliation consists of
horizontal leaves lying in $\{t\} \times s$, as well as two vertical
singular sets of the form $e_{\ts} \times \{p\} $, for two points $p$ in $s$.
The dividing set on $e_{\ts} \times s$ consists of two vertical arcs, as well.

Note that, if we could show that $A$ and $e_{\ts} \times s$ were isotopic
through convex surfaces, we would be done, since the dividing set on $A$
would be isotopic to the one on $e_{\ts} \times s$, which is what we are trying to show.
This is somewhat geometrically hard to prove, so we argue as follows. Perturb
$e_{\ts} \times s$ slightly, such that there is a Legendrian loop $\ell$
intersecting one of the dividing arcs twice. This does not change the isotopy
type of the dividing set. Let $D_0 \subset e_{\ts} \times s$ be the disk bounded
by $\ell$. Now take a convex disk $D$ in $e_{\ts} \times I \times S^1$
with $D \cap (e_{\ts} \times s) = \ell$.
The dividing set on $D$ consists of a single arc, since we can assume
$D$ lies in a tight contact ball. We can replace
$e_{\ts} \times s$ by $((e_{\ts} \times s) \setminus D_0) \cup D$, and then
rounding along the Legendrian corner $\ell$. When we do this, we do not change
the isotopy type of the dividing set, and we can move $e_{\ts} \times s$ to
$A$ via a sequence of such moves. The move is shown in Figure~\ref{fig::64}.
This establishes that the annulus $A$ has a dividing set isotopic to the one
shown in Figure~\ref{fig::62}.

\begin{figure}[ht!]
  \centering
  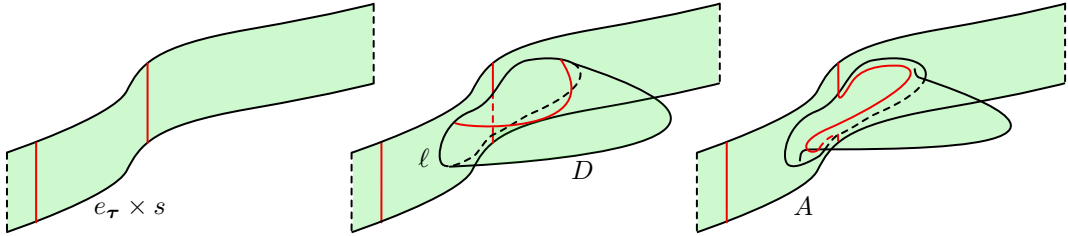
  \caption{Attaching a convex disk to move the annulus $e_{\ts} \times
  s \subset e_{\ts} \times I \times S^1$ to the position of an annulus $A \in
  \ve{A}$. On the left is (a $C^0$-small perturbation of) the annulus
  $e_{\ts} \times s$. In the middle is a convex disk with Legendrian boundary
  $\ell$ with $\tb(\ell) = -1$. On the right is a convex annulus that we
  can take to be $A$, obtained by edge rounding along $\ell$.\label{fig::64}}
\end{figure}

Having determined the dividing sets along the convex annuli in $\ve{A}$,
we can cut along the annuli in $\ve{A}$, then round the Legendrian corners.
The dividing set on the sutured manifold corresponding to $\Delta \times \d \S$
is shown in Figure~\ref{fig::65}. As a sutured manifold, this is the
same as $(Z_0,\g_0)$. Similarly, it is easy to see that rounding the
Legendrian corners on the other 3 pieces yields the sutured manifolds
$(Z_{\as},\g_{\as})$, $(Z_{\bs},\g_{\bs})$, and $(Z_{\gs}, \g_{\gs})$.
It follows that $SN(S_{\ts}) \approx Z_{\cT_0}$.
On the other hand, we note that $\xi_{S^1}$ is tight, since
the dividing set on $S_{\cT}$ contains no contractible components.
It follows that $\xi_{S^1}$ restricts to tight contact structures on
$(Z_0,\g_0)$, $(Z_{\as}, \g_{\as})$, $(Z_{\bs}, \g_{\bs})$, and $(Z_{\gs}, \g_{\gs})$.
However, each of these four sutured manifolds are product disk decomposable,
so, up to equivalence, there is a unique tight contact structure on each one.
Hence the restrictions of $\xi_{S^1}$  are equivalent to $\xi_0$,
$\xi_{\as}$, $\xi_{\bs}$, and $\xi_{\gs}$, respectively. Since $\xi_{\as,\bs,\gs}$ is
constructed by gluing together $\xi_0$, $\xi_{\as}$, $\xi_{\bs}$, and
$\xi_{\gs}$,  it follows that $\xi_{S^1}$ and $\xi_{\as,\bs,\gs}$ are
equivalent on~$Z_{\cT_0}$.
\end{proof}

\begin{figure}[ht!]
  \centering
  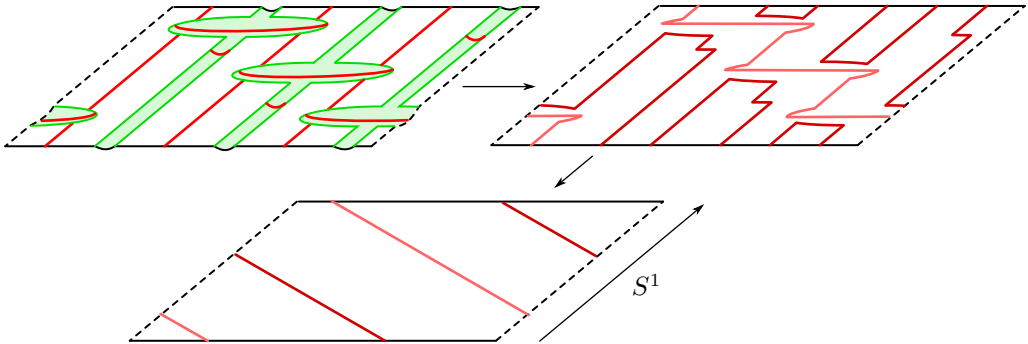
  \caption{Rounding corners after cutting $S_{\cT}\times S^1$ along
  the annuli in $\ve{A}$. Shown is the dividing set on the boundary of a
  solid torus component of $(S_{\cT}\times S^1)\setminus \ve{A}$
  corresponding to $\Delta\times \{x\}\times S^1$, for a basepoint $x\in
  \ws\cup \zs$. On the top left, we show the dividing before rounding the
  Legendrian corners. On the top right, we show the result of rounding corners.
  On the bottom, we show the result of isotoping the dividing set. We view
  $Z_0$ as being ``below'' the surface shown.\label{fig::65}}
\end{figure}

\subsection{Saddle cobordisms and link triple diagrams}

In this section, we review the construction of saddle cobordism
maps~\cite{ZemCFLTQFT}*{Section~6}. As an important step towards proving
Theorem~\ref{thm:cobordismmapsagree}, we show that the maps from
\cite{ZemCFLTQFT} and \cite{JCob} agree for such cobordisms.

\begin{define}
Suppose that $Y$ is a 3-manifold containing an oriented multi-based link $\bL=(L,\ws,\zs)$.
We say that an oriented square $B\subset Y$ is a \emph{$\b$-band}
for a multi-based link $\bL$ in $Y$ if $B$  is smoothly embedded in $Y$,
it is identified with $[-1,1] \times [-1,1]$, and
\begin{enumerate}
\item $B \cap L=[-1,1]\times \{-1, 1\}$,
\item the boundary orientation of $B$ agrees with the orientation of $-L$,
\item $B \cap (\ws \cup \zs)=\emptyset$, and both ends of $B$ are in regions of
$L \setminus (\ws \cup \zs)$ that go from $\zs$ to $\ws$.
\end{enumerate}
\end{define}

Note that if $B$ is a $\b$-band for the link $\bL$ in $Y$, then
there is a well-defined multi-based link
\[
\bL(B) = (L(B),\ve{w},\ve{z})
\]
obtained by band surgery on $B$.
The following is \cite{ZemCFLTQFT}*{Definition~6.4}:

\begin{define}
We say the link triple diagram
\[
(\S,\a_1,\dots, \a_n,\b_1,\dots, \b_n,\b'_1,\dots, \b_n',\ws,\zs),
\]
is \emph{subordinate to the $\b$-band} $B$ if the following hold:
\begin{enumerate}
\item $(\S_0,\a_1,\dots, \a_{n}, \b_1,\dots, \b_{n-1})$
is a diagram for the sutured manifold $Y(\bL)\setminus N(B)$,
where $\S_0 = \S \setminus N(\ws \cup \zs)$,
\item $(\S,\a_1,\dots, \a_n,\b_1,\dots, \b_n,\ws,\zs)$ is a diagram for $(Y,\bL)$,
\item the curves $\b_1',\dots, \b_{n-1}'$ are small Hamiltonian translates of the curves
$\b_1,\dots, \b_{n-1}$,
\item the curve $\b_{n}'$ is induced by the band $B$,
and  $(\S,\a_1,\dots, \a_n,\b_1',\dots, \b_n', \ws, \zs)$
is a diagram for $(Y,\bL(B))$.
\end{enumerate}
\end{define}

Notice that, if we ignore the basepoints $\ws$ and $\zs$,
then the curve $\b_n'$ is related to $\b_n$
by a sequence of handle slides and isotopies. It follows that the 4-manifold $X_{\as,\bs,\bs'}$
induced by a Heegaard triple subordinate to a $\b$-band is diffeomorphic to
$I \times Y$ with a neighborhood of $\{\tfrac{1}{2}\} \times U_{\bs}$ removed.

The band $B$ induces a saddle cobordism $S(B) \subset I \times Y$ from $L$ to $L(B)$.
The surface $S(B)$ is obtained by rounding the corners of the surface
\[
S_0(B) := \left([0,\tfrac{1}{2}] \times L\right) \cup
\left(\{\tfrac{1}{2}\} \times B \right) \cup \left([\tfrac{1}{2}, 1] \times L(B) \right).
\]
We note that $S(B)$ has two natural choices of dividing sets.
To construct them, pick points $\ve{p} \subset L \setminus (\ws \cup \zs)$
and $\ve{q} \subset L(B) \setminus (\ws \cup \zs)$, such that the following hold:
\begin{enumerate}
\item Each component of $L \setminus (\ws \cup \zs)$ contains exactly one point of $\ve{p}$,
and each component of $L(B) \setminus (\ws \cup \zs)$ contains exactly one point of $\ve{q}$.
\item If $C \subset L \setminus (\ve{w} \cup \ve{z})$ is a component disjoint from $B$,
then $\ve{p} \cap C = \ve{q} \cap C$.
\item If $C$ is a component of $L\setminus (\ve{w} \cup \ve{z})$ that intersects $B$,
and $\{p\} = C\cap \ve{p}$, then $p \in B$. Similarly, if $C'$ is a component of
$L(B) \setminus (\ve{w} \cup \ve{z})$ that intersects $B$, and $\{q\} = C'\cap \ve{q}$, then $q \in B$.
\end{enumerate}

A dividing $\cA_0$ on $S(B) \setminus \left(\{\tfrac{1}{2}\} \times B \right)$
is then specified by the equation
\[
\cA_0 := \left([0,\tfrac{1}{2}] \times \ve{p} \right) \cup
\left([\tfrac{1}{2}, 1] \times \ve{q} \right).
\]
There are two natural ways to extend the dividing set $\cA_0$ to $B$.
We write $\cA_{\ws}$ and $\cA_{\zs}$ for the two possible extensions, as shown in Figure~\ref{fig::36}.

\begin{figure}[ht!]
 \centering
 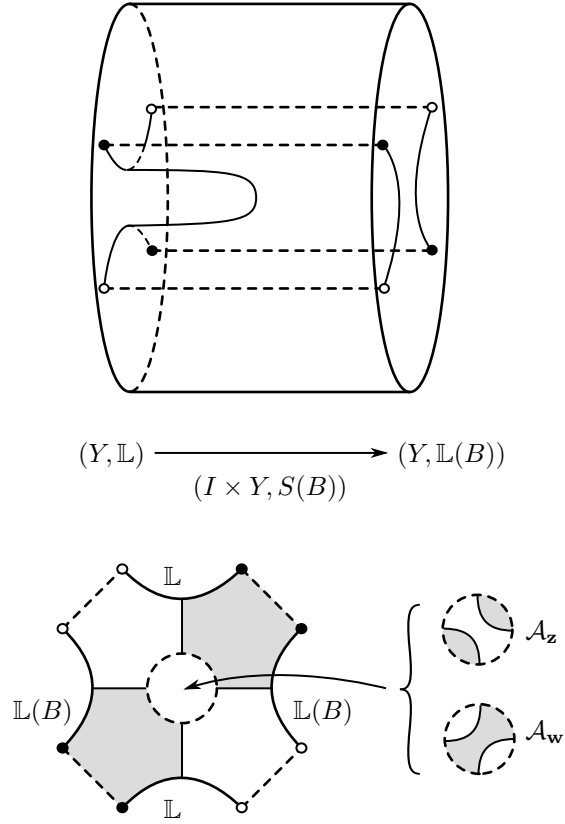
 \caption{A portion of the surface $S(B)\subset I \times Y$, as well as the two dividing sets
 $\cA_{\ws}$ and $\cA_{\zs}$ on $S(B)$. The $\ws$ basepoints are shown as solid dots,
 while the $\zs$ basepoints are open dots. The $\ws$ regions are shown as shaded,
 the $\zs$ regions are unshaded. \label{fig::36}}
\end{figure}

\begin{lem}\label{lem:topdegreegenerators}
If $(\S,\as,\bs,\bs',\ws,\zs)$ is a triple subordinate to a $\b$-band $B$,
then $(\S,\bs,\bs',\ws,\zs)$ represents an unlink $\bU$ in $(S^1\times S^2)^{\# g(\S)}$,
where all components of $\bU$ have two basepoints, except for one component
that has four basepoints. With respect to each of the Maslov gradings $\gr_{\ws}$ and $\gr_{\zs}$,
there is a top-graded generator of $\hat{\HFL}(\S, \bs, \bs',\ws,\zs)$,
for which we write $\Theta_{\bs,\bs'}^{\ws}$ and $\Theta_{\bs,\bs'}^{\zs}$.
\end{lem}

\begin{proof}
For $i \in \{\, 1, \dots, n-1 \,\}$, the curves $\b_i$ and $\b_i'$
are Hamiltonian translates of each other, and hence determine a 2-sphere in the manifold
$U_{\bs} \cup_{\S} U_{\bs'}$. After surgering all of these out,
we are left with $|\ws| - 2$ copies of $S^3$, each containing a doubly based unknot,
as well as one copy of $S^3$ with an unknot containing four basepoints.
After surgering all of these 2-spheres out, it is easy to see $\hat{\HFL}$ is generated by two elements,
one of which is in $(\gr_{\ws}, \gr_{\zs})$-grading $(+\tfrac{1}{2}, -\tfrac{1}{2})$,
and one of which is in grading $(-\tfrac{1}{2}, +\tfrac{1}{2})$.
The effect of undoing the surgeries we did on the 2-spheres corresponds to
adding back in a collection of 1-handles, which clearly preserves the property
of having a top $\gr_{\ws}$-graded element, and a top $\gr_{\zs}$-graded element.
\end{proof}

\begin{lem}\label{lem:computesimplecobordismmaps}
Consider the link cobordism $\cX=(X,S,\cA)$ from the empty set to an unlink
in $(S^1\times S^2)^{\# k}$, constructed by setting $X$ to be a 4-dimension
genus $k$ handlebody, $S$ to be a collection of $n$ standardly embedded disks
in $X$ intersecting $\d X\iso (S^1\times S^2)^{\# k}$ in $n$ disjoint
unknots, and letting $\cA$ consist  of a single arc on each component of $S$,
except for one component of $S$, where $\cA$ consists of two arcs. Then
\[
F_{\cX}^J(1)=\hat{F}_{\cX}^Z(1) =
\begin{cases}
\Theta^{\ws}& \text{ if } \chi(S_{\ws})=\chi(S_{\zs})+1,\\
\Theta^{\zs} & \text{ if } \chi(S_{\ws})=\chi(S_{\zs})-1,
\end{cases}
\]
where $1$ denotes the generator of $\hat{\HFL}(\emptyset,\emptyset)\iso \bF_2$.
\end{lem}
\begin{proof}We decompose the cobordism $\cX$ into a composition of elementary link cobordisms (Definition~\ref{def:elementarycobordisms}). Write $\cX = \cX_4\circ \cX_3\circ \cX_2\circ \cX_1$ where
\begin{itemize}
\item $\cX_1$ is a 0-handle cobordism, which adds a doubly based unknot in $S^3$;
\item $\cX_2$ is a stabilization cobordism, which adds two basepoints to the unknot in $S^3$;
\item $\cX_3$ consists of $(n-1)$ 0-handles, each adding a doubly based unknot;
\item $\cX_4$ consists of $(k+n-1)$ 1-handles.
\end{itemize}
If $\bU\subset S^3$ is a doubly-based unknot, then
$\hat{\HFL}(S^3,\bU)\iso \bF_2$. Noting that the 0-handle maps are nonzero,
since they can be canceled topologically, using multiplicitivity of link
Floer homology under disjoint unions, it follows that
\[
F^J_{\cX_i} = \hat{F}^Z_{\cX_i}
\]
for $i=1,3$.

If $\bU'$ is an unknot in $S^3$ with four basepoints, then,
by~Lemma~\ref{lem:topdegreegenerators}, we have $\hat{\HFL}(S^3,\bU') \iso \bF_2 \oplus
\bF_2$. Furthermore, $\hat{\HFL}(S^3,\bU')$ is generated by two elements,
$\Theta^{\ws}$ and $\Theta^{\zs}$, which are distinguished by grading. The
element $\Theta^{\ws}$ has $(\gr_{\ws},\gr_{\zs})$ bigrading
$(+\tfrac{1}{2},-\tfrac{1}{2})$, while $\Theta^{\zs}$ has bigrading
$(-\tfrac{1}{2}, +\tfrac{1}{2})$.

It follows immediately from the definition of the maps in \cite{ZemCFLTQFT} that
\[
\hat{F}_{\cX_2}^Z(1) =
\begin{cases}
\Theta^{\ws}& \text{ if } \chi(S_{\ws})=\chi(S_{\zs})+1,\\
\Theta^{\zs} & \text{ if } \chi(S_{\ws})=\chi(S_{\zs})-1,
\end{cases}
\]
since the cobordism map for a stabilization cobordism is defined using the
quasi-stabilization map \cite{ZemCFLTQFT}*{Section~3.2}. On the other hand, a
straightforward functoriality argument shows that $F_{\cX_2}^J$ must be
nonzero. By \cite{JMComputeCobordismMaps}*{Theorem~5.18}, it follows that
$F_{\cX_2}^J(1)$ has the same $\gr_{\ws}$ and $\gr_{\zs}$ grading as
$\hat{F}_{\cX_2}^Z(1)$. Since $\hat{\HFL}(S^3,\bU')$ is a 2-dimensional
vector space over $\bF_2$, it follows that
$F^{J}_{\cX_2}=\hat{F}^{Z}_{\cX_2}$.

Finally, the two cobordism map constructions use the same 4-dimensional handle attachment
maps, so \[F_{\cX_4}^J=\hat{F}_{\cX_4}^Z.\] Furthermore, as in
Lemma~\ref{lem:topdegreegenerators}, the 1-handle maps preserve the top
graded elements $\Theta^{\ws}$ and $\Theta^{\zs}$. Composing all the maps,
the claim now follows.
\end{proof}

In \cite{ZemCFLTQFT}, the link cobordism maps for $\cX_{\ws} = (I \times Y, S(B), \cA_{\ws})$
and $\cX_{\zs} = (I \times Y, S(B), \cA_{\zs})$
are defined to be
\begin{equation}
\hat{F}^Z_{\cX_{\ws}}(-) := F_{\as,\bs,\bs'}(- \otimes \Theta_{\bs,\bs'}^{\zs}) \text{ and }
\hat{F}^Z_{\cX_{\zs}}(-) := F_{\as,\bs,\bs'}(- \otimes \Theta_{\bs,\bs'}^{\ws}).\label{eq:bandmapsdef}
\end{equation}

\begin{lem}\label{lem:saddlecobordismmapsagree}
For the decorated saddle cobordisms $\cX_{\ws}$ and $\cX_{\zs}$, defined above,
we have $F_{\cX_{\ws}}^J = \hat{F}_{\cX_{\ws}}^Z$
and $F_{\cX_{\zs}}^J = \hat{F}_{\cX_{\zs}}^Z$.
\end{lem}

\begin{proof}
The key observation is that if $(\S,\as,\bs,\bs',\ws,\zs)$ is a Heegaard triple
subordinate to a $\b$-band, then the 4-manifold $X_{\as,\bs,\bs'}$ is equal to $I \times Y$
with a neighborhood of $\{\tfrac{1}{2}\} \times U_{\bs}$ removed, and the surface with divides
\[
(S_{\as,\bs,\bs'}, \cA_{\as,\bs,\bs'}) =
\left(S(B) \setminus \left(\{\tfrac{1}{2}\} \times B \right), \cA_0\right).
\]
There is a cobordism $X_{\bs} \colon \emptyset \to Y_{\bs,\bs'}$
consisting of 0-handles and 1-handles. Inside $X_{\bs}$, there is a surface $S_0$
that consists of $|\ws| - 1$ disks. We note that $|\ws| - 2$ of the disks have boundary equal
to a doubly based unknot in $Y_{\bs,\bs'}$, but one disk has boundary equal to an unknot with four basepoints.
It is clear that if we fill in the $Y_{\bs,\bs'}$ boundary of the link cobordism
$(X_{\as,\bs,\bs'}, S_{\as,\bs,\bs'})$ with $(X_{\bs}, S_0)$, then we obtain the (undecorated)
link cobordism $(I \times Y, S(B))$. On the other hand, as described above,
there are two natural dividing sets on $S_0$, shown in Figure~\ref{fig::36}.
Define
\[
\cA_{\ws,0} = S_0 \cap \cA_{\ws} \text{ and } \cA_{\zs,0} = S_0 \cap \cA_{\zs}.
\]
Using Lemma~\ref{lem:computesimplecobordismmaps}, we have that
\[
F^J_{X_{\bs}, S_0, \cA_{\ws,0}}(1) = \Theta_{\bs,\bs'}^{\zs} \text{ and }
F^J_{X_{\bs}, S_0, \cA_{\zs,0}}(1) = \Theta_{\bs,\bs'}^{\ws},
\]
as maps from $\hat{\HFL}(\emptyset,\emptyset) \iso \bF_2$ to
$\hat{\HFL}(\S, \bs, \bs', \ws, \zs)$.

Using Theorem~\ref{thm:trianglemapiscobmap} and Proposition~\ref{prop:contactstructuresagree},
we know that the sutured link cobordism map $F^J_{\cX_{\as,\bs,\bs'}}$
for the decorated link cobordism
\[
\cX_{\as,\bs,\bs'} = (X_{\as,\bs,\bs'}, S_{\as,\bs,\bs'}, \cA_{\as,\bs,\bs'})
\]
is the map $F_{\as,\bs,\bs'}$ that counts holomorphic triangles on
the Heegaard triple $(\S,\as,\bs,\bs')$.
It follows from the composition law that
\[
F^J_{\cX_{\zs}}(-) = F_{\as,\bs,\bs'}(- \otimes F_{X_{\bs}, S_0, \cA_{\zs,0}}(1)) =
F_{\as,\bs,\bs'}(- \otimes \Theta_{\bs,\bs'}^{\ws}) = \hat{F}^Z_{\cX_{\zs}}(-).
\]
The result for $F^J_{\cX_{\ws}}$ follows similarly.
\end{proof}

\subsection{Proof of Theorem~\ref{thm:cobordismmapsagree}}

We can now prove that the link cobordism map constructions from \cite{JCob}
and \cite{ZemCFLTQFT} agree: \begin{proof}[Proof of
Theorem~\ref{thm:cobordismmapsagree}] Given an arbitrary decorated link
cobordism $\cX=(X,S,\cA)$, it is easy to see that one can decompose $\cX$ into
a sequence of link cobordisms which are each diffeomorphic to an elementary link
cobordism (Definition~\ref{def:elementarycobordisms}). Using the composition
law, it remains to verify the claim for each type of elementary link
cobordism. The maps obviously agree for identity cobordisms, and 4-dimensional 0-, 1-, 2-,
3- or 4-handle cobordisms. By Lemma~\ref{lem:saddlecobordismmapsagree}, they
agree for decorated saddle cobordisms. Finally, as in
Remark~\ref{rem:stabilization=composition}, a stabilization cobordism can be
decomposed into a composition of the other elementary cobordisms, and hence
having established the claim for the other types of elementary cobordisms, it
follows as well for stabilization cobordisms.
\end{proof}

\bibliographystyle{custom}
\bibliography{biblio}
\end{document}